\documentclass[a4paper,11pt,reqno,noindent]{amsart}
\usepackage[centertags]{amsmath}
\usepackage{amsfonts,amssymb,amsthm,amscd,dsfont,cases,esint,enumerate,color,mathdots}
\usepackage{thm-restate}
\usepackage[T1]{fontenc}
\usepackage[english]{babel}
\usepackage[applemac]{inputenc}
\usepackage[body={15cm,21.5cm},centering]{geometry} 
\usepackage{fancyhdr}
\usepackage[colorlinks,citecolor=black,urlcolor=black]{hyperref}
\usepackage{tikz}
\usetikzlibrary{arrows.meta,patterns,decorations.markings}
\tikzset{bar/.style={midway,sloped,anchor=center,transform shape}}

\theoremstyle{plain}

\newtheorem{theor}{Theorem}[section]
\newtheorem{lem}[theor]{Lemma}

\newtheorem{prop}[theor]{Proposition}

\newtheorem*{keyrule}{Key dependence rule}

\newtheorem{rem}[theor]{Remark}

\mathchardef\emptyset="001F
\numberwithin{equation}{section}

\newcommand{\N}{\mathbb N}

\newcommand{\Pc}{\mathcal{P}}
\newcommand{\Gc}{\mathcal{G}}
\newcommand{\Lc}{\mathcal{L}}
\newcommand{\Jc}{\mathcal J}
\newcommand{\R}{\mathbb R}

\newcommand{\Ic}{\mathcal I}

\newcommand{\cvf}{\rightharpoonup}
\newcommand{\loc}{{\operatorname{loc}}}
\newcommand{\uloc}{{\operatorname{uloc}}}
\newcommand{\Id}{\operatorname{Id}}
\newcommand{\per}{{\operatorname{per}}}

\newcommand{\D}{\operatorname{D}}
\newcommand{\Ld}{{L}}
\newcommand{\dist}{\operatorname{dist}}
\newcommand{\supp}{\operatorname{supp}}
\newcommand{\Div}{{\operatorname{div}}}
\newcommand{\Sym}{{\operatorname{sym}}}
\newcommand{\step}[1]{\noindent \textbf{Step} #1.}
\newcommand{\substep}[1]{\noindent \textbf{Substep} #1.}
\newcommand{\E}{\mathbb{E}}
\newcommand{\Var}{\operatorname{Var}}
\newcommand{\expec}[1]{\mathbb{E}\left[ #1 \right]}
\newcommand{\expecm}[1]{\mathbb{E}\big[ #1 \big]}

\newcommand{\var}[1]{\mathrm{Var}\left[#1\right]}

\newcommand{\introheading}[1]{\par\medskip\noindent\textbf{#1.}\par\nobreak\smallskip}

\title{Batchelor's formula and infrared renormalization for sedimentation}

\author[M. Duerinckx]{Mitia Duerinckx}
\address[Mitia Duerinckx]{Universit\'e Libre de Bruxelles, D\'epartement de Math\'ematique, 1050~Brussels, Belgium}
\email{mitia.duerinckx@ulb.be}
\author[A. Gloria]{Antoine Gloria}
\address[Antoine Gloria]{Sorbonne Universit\'e, CNRS, Universit\'e de Paris, Laboratoire Jacques-Louis Lions, 75005~Paris, France  \& Universit\'e Libre de Bruxelles, D\'epartement de Math\'ematique, 1050~Brussels, Belgium}
\email{antoine.gloria@sorbonne-universite.fr}

\begin{document}
\selectlanguage{english}

\begin{abstract}
We study the sedimentation of stationary random suspensions of rigid particles in Stokes flow. Batchelor's formula predicts the first dilute correction to the infinite-volume mean settling speed due to hydrodynamic interactions between suspended particles. A rigorous derivation has long been obstructed by the long-range nature of the Stokes flow, which gives rise to infrared divergences in the large-volume limit.

In dimension~$d>2$, for stationary suspensions satisfying quantitative decorrelation assumptions, we construct the infinite-volume mean settling speed and show that it governs the relative settling speed of particles in large containers, independently of the container shape. We then establish a renormalized cluster expansion of this mean settling speed in the dilute regime and compute it up to the two-particle term, thereby justifying Batchelor's formula.

The proof is based on the infrared renormalization of hydrodynamic interactions. Infinite-volume observables are decomposed into an explicit singular part, carrying the non-integrable large-scale contribution, and a regular remainder controlled by elliptic estimates. The singular part is renormalized through counterterms that encode the diverging mean backflow generated by the suspension. At the level of the dilute cluster expansion, the renormalization is implemented cluster by cluster and the singular-regular decomposition is achieved through a finitary diagrammatic expansion of hydrodynamic interactions, inspired by the method of reflections, which isolates the leading divergent substructures and exposes the key cancellations.
\end{abstract}

\maketitle

\setcounter{tocdepth}{1}
\tableofcontents
\allowdisplaybreaks

\section{Introduction and main results}
This work is devoted to the sedimentation of stationary random suspensions of rigid particles in Stokes flow, which requires the development of a renormalization strategy for random elliptic PDEs with infrared singularities.
The issue is indeed the long-range nature of hydrodynamic interactions between suspended particles: the large-volume limit leads to non-integrable contributions and the physically-meaningful infinite-volume quantities are obtained only after the relevant counterterms have been identified. This is particularly delicate as hydrodynamic interactions are both non-explicit and genuinely multibody: they are generated by the Stokes equations with all particles entering as boundary data, rather than by a prescribed pair potential.

Our main result is a rigorous derivation of Batchelor's formula, which gives the leading dilute correction to the mean settling speed induced by hydrodynamic interactions; see Theorem~\ref{th:Batchelor} below.
Before we get there, this requires two preliminary steps. First, we construct the relevant renormalized notion of infinite-volume mean settling speed for stationary suspensions; this is the content of Theorem~\ref{th:mean-velocity}. Second, we justify its physical relevance by showing that it describes the relative local settling speed of suspended particles in the large-container limit, independently of the container shape; see Theorem~\ref{th:mean-velocityR}. The derivation of Batchelor's formula relies on a renormalized expansion of this infinite-volume quantity into many-body cluster contributions. One of our main realizations is to isolate all the potentially infrared-divergent components in this expansion and expose the explicit cancellations that make them finite, for which we introduce a finitary expansion of hydrodynamic interactions into {\it reflection blocks}; see Section~\ref{sec:elements}.

\subsection{Model and main results}
Consider a large container $U_R\subset\R^d$, in space dimension $d>2$, given by the rescaling $U_R:=RU$ of a given bounded Lipschitz domain $U\subset\R^d$. Assume that this container is filled with a Stokes fluid together with a suspension of rigid inertialess particles that settle under gravity. Following the theoretical physics literature starting with Batchelor~\cite{Batchelor-72}, we adopt a statistical perspective: given a statistical ensemble of particle positions, we aim to study the statistical ensemble of the associated instantaneous velocities given by the Stokes equations.

\introheading{Statistical ensemble of particle positions}
As in the physics literature, we consider the restriction to $U_R$ of a stationary point process in infinite volume $\R^d$, and discard particles that intersect the boundary~$\partial U_R$ of the container.  We focus on spherical particles for notational simplicity, but our results are easily adapted to random shapes.

We thus let~$\Pc$ denote a stationary random point process on~$\R^d$,
and we associate with it the spherical inclusion process
\begin{equation}\label{eq:def-Ic}
\Ic\,:=\,\bigcup_{x\in \Pc}B(x),
\end{equation}
where $B(x)=x+B$ stands for the unit ball centered at $x$ (and $B$ is the unit ball centered at the origin).
We assume throughout that the particles are uniformly separated: for some fixed $\delta>0$, the point process satisfies the $\delta$-hardcore condition
\begin{equation}\label{eq:hardcore}
\inf_{x,y\in\Pc:x\ne y}|x-y|\,\ge\,2(1+\delta)\qquad\text{almost surely.}
\end{equation}
In particular, the inclusions are pairwise disjoint.
Next, given a fixed boundary layer thickness $\rho>0$, we localize the infinite-volume ensemble in the container~$U_R$ by retaining only those particles at distance at least $\rho$ from the boundary: in other words, we define
\begin{gather*}
\Pc(U_R)\,:=\,\Big\{x\in \Pc\cap U_R\,:\, \dist(x,\partial U_R)\ge1+\rho\Big\},\qquad
\Ic(U_R)\,:=\,\bigcup_{x \in\Pc(U_R)} B(x).
\end{gather*}
Finally, we denote by $\Omega$ the probability space of all  $\delta$-hardcore point configurations, endowed with the standard sigma-algebra generated by the vague topology. The law of $\Pc$ is denoted by $\mathbb P$, the corresponding expectation by $\expec{\cdot}$, and $\Ld^2(\Omega)$ stands for the space of square-integrable random variables (that is, square-integrable measurable functions of the point configurations).

\introheading{Fluid equations and settling speed}
Following the classical works~\cite{Smoluchowski-11,Smoluchowski-06,Burgers-41,Burgers-42,Batchelor-72,Batchelor_1976}, we associate with each realization of the particle configuration $\Ic(U_R)$ in the container $U_R$ the instantaneous velocity field obtained from the steady Stokes equations. More precisely, letting $e\in\R^d$ denote the direction of gravity, we consider the solution $(u_R,p_R)$ of
\begin{equation}\label{eq:Dir}
\left\{\begin{array}{ll}
-\triangle u_{R}+\nabla p_{R}=0,&\text{in $U_R\setminus \Ic(U_R)$},\\
\Div(u_{R})=0,&\text{in $U_R\setminus \Ic(U_R)$},\\
\D(u_{R})=0,&\text{in $\Ic (U_R)$},\\
 e|B|+\int_{\partial B(x)}\sigma(u_{R},p_{R})\nu=0,&\forall x\in\Pc(U_R),\\
\int_{\partial B(x)}(\cdot- x)\times\sigma(u_{R},p_{R})\nu=0,&\forall x\in\Pc(U_R).
\end{array}\right.
\end{equation}
Here, $\nu$ denotes the exterior normal to $B(x_n)$, $p_R$ is the pressure, $\D(\cdot)$ denotes the symmetrized gradient, and
\[\sigma(u_R,p_R)=2\D(u_R)-p_R\Id\]
is the Cauchy stress tensor. The condition $\D(u_R)=0$ in each inclusion means that particles move as rigid bodies. The last two conditions express, respectively, the balance of the gravitational force and the vanishing of the total torque on each particle.\footnote{We use 3D notation for convenience: the relation $\int_{\partial B(x_n)}(\cdot-x_n)\times\sigma(u_R,p_R)\nu=0$ is to be understood as $\int_{\partial B(x_n)}\Theta(\cdot-x_n)\cdot\sigma(u_R,p_R)\nu=0$ for any skew-symmetric matrix $\Theta$.}
At this point, we do not specify boundary conditions on the external walls $\partial U_R$ of the container: the large-volume limits below are stable under several natural choices of homogeneous boundary conditions, which will be described explicitly later.

In this setting, for a particle centered at $x\in\Pc(U_R)$, we define its individual settling speed by
\begin{equation}\label{eq:defin-VR(x)}
V_R(x)\,:=\,\fint_{B(x)}u_{R},
\end{equation}
that is, the translational component of the rigid velocity field on $B(x)$.
Our aim is to understand the behavior of $V_R$ in the large-container limit $R\uparrow\infty$.

\medskip
\introheading{Density functions and correlations}
We next recall the statistical quantities entering the statement of the results.
For $j\ge1$, the \emph{$j$-point density} of the point process $\Pc$ is the non-negative function $f_j:(\R^d)^j\to\R_+$ characterized by
\begin{equation}\label{eq:multi-density}
\E\bigg[\sum_{x_1,\ldots,x_j\in\Pc}^{\ne}\zeta(x_1,\ldots,x_j)\bigg]\,=\,\int_{(\R^d)^j}\zeta f_j\qquad\text{for all $\zeta\in C^\infty_c((\R^d)^j)$.}
\end{equation}
Since the law of $\Pc$ is stationary, the one-point density is constant:
\[f_1\,\equiv\,\lambda\,:=\,\E\big[\sharp(\Pc\cap (0,1)^d)\big],\]
where $\lambda$ is the intensity of the point process.
We also introduce the associated Ursell functions $\{h_j\}_{j\ge1}$, or \emph{multi-point correlation functions}, with $h_j:(\R^d)^j\to\R$, defined through Mayer's cluster relations
\begin{equation}\label{eq:correl-expand}
f_j(x_1,\ldots,x_j)\,=\,\sum_{\pi}\prod_{H\in\pi}h_{\sharp H}(x_H),\qquad j\ge1.
\end{equation}
Here $\pi$ runs over all partitions of the index set $\{1,\ldots,j\}$, the product runs over the blocks~$H$ of~$\pi$, and we write $x_H:=(x_{i_1},\ldots,x_{i_l})$ for $H=\{i_1,\ldots,i_l\}$. Thus,
\begin{eqnarray*}
h_1(x)&:=&\lambda,\\
h_2(x,y)&:=&f_2(x,y)-\lambda^2,\\
h_3(x,y,z)&:=&f_3(x,y,z)-\lambda \big(f_2(x,y)+f_2(y,z)+f_2(z,x)\big)+2\lambda^3,
\end{eqnarray*}
and so on. The functions $h_j$ are symmetric and translation-invariant. Their decay at infinity, for $j\ge 2$, quantifies the mixing properties of the point process.

As in~\cite{DG-21b}, we further introduce a notion of multi-point intensities:  for $j\ge1$, we define the $j$-point intensity by
\[\lambda_j:=
\sup_{x_1,\ldots,x_j\in\R^d}
\fint_{B_\ell(x_1)\times\cdots\times B_\ell(x_j)}f_j,\]
in terms of the minimal distance $\ell:=\inf_{x,y\in\Pc:x\ne y}|x-y|\ge0$.
This is the largest intensity, locally averaged at the natural scale $\ell$ in $(\mathbb R^d)^j$, for finding ordered $j$-tuples in prescribed microscopic regions. In particular $\lambda_1=\lambda$. For a hardcore Poisson point process, we have $\lambda_j\simeq_{j,\ell}\lambda^j$,
whereas a scaling $\lambda_j\gg\lambda^j$ reflects short-range clustering: for a dimerized process, for instance, one may have $\lambda_2=O(\lambda)$. We also recall the following supermultiplicativity property from~\cite[Lemma~1.1(ii)]{DG-21b}: if $\Pc$ is strongly mixing, then
\begin{equation}\label{eq:submult-lambda}
\lambda_j\lambda_k\le\lambda_{j+k}\qquad\text{for all $j,k\ge1$}.
\end{equation}

\medskip
\introheading{Main results}
In what follows, we study settling velocities averaged locally on mesoscopic scales. For technical reasons, we use smooth spatial averages.
In addition, for physical reasons that are discussed in Section~\ref{sec:phys-intro}, we measure particle velocities relative to the local mean motion of the surrounding fluid. This amounts to subtracting the corresponding smooth average of the fluid velocity.
More precisely, consider an averaging function $\chi\in C^\infty_c(\R^d)$ with $\chi\ge0$ and $\int_{\R^d}\chi=1$. Setting~$\chi_r:=\chi(\frac\cdot r)$, we define the \emph{relative local mean settling speed}
\begin{equation}\label{eq:loc-settlsp-re}
\bar V_R^{\text{rel}}(\chi_r)\,:=\,\frac{\sum_{x\in \Pc(U_R)}\chi_r(x)V_R(x)}{\sum_{x\in\Pc(U_R)}\chi_r(x)}
-\fint_{\chi_r}u_R,
\end{equation}
provided $r\supp\chi\subset U_R$, where the individual particle velocity $V_R(x)$ is defined in~\eqref{eq:defin-VR(x)} and where we use the following short-hand notation, for any measurable function $g$,
\begin{equation}\label{eq:average-chir}
\fint_{\chi_r}g:=\frac{\int_{\R^d}\chi_r g}{\int_{\R^d}\chi_r}.
\end{equation}

The first step in the analysis of the large-volume limit of $\bar V_R^{\operatorname{rel}}(\chi_r)$ is to construct its infinite-volume counterpart. This requires a first use of an infrared renormalization, in a simple form that can be compared to the classical Wick renormalization for the $\Phi_2^4$ model; see Section~\ref{sec:infra-red-renorm}. In the result below, the renormalization manifests itself in the pressure: the field $p_{\infty,i}$ cannot be chosen stationary --- rather, as shown in the proof, only the renormalized pressure
\begin{equation}\label{eq:comp-pressure}
p_{\infty,i}-\frac{\lambda|B|}{1-\lambda|B|}\, e_i\cdot x
\end{equation}
is stationary. The subtracted linear term corresponds to the diverging mean backflow generated by the suspension. The proof is given in Section~\ref{sec:mean-settling}.

\begin{theor}\label{th:mean-velocity}
Let $d>2$, let~$\Pc$ be a stationary ergodic $\delta$-hardcore random point process, and assume that its two-point correlation function satisfies\footnote{with the standard notation $\langle x \rangle := (1+|x|^2)^\frac12$.}
\begin{equation}\label{eq:decay-h2}
|h_2(0,x)|\le C\langle x\rangle^{-2-\beta}\qquad\text{for some $\beta>0$}.
\end{equation}
Then, for each $1\le i\le d$, there exists a unique solution $u_{\infty,i} \in \Ld^2(\Omega;H^1_\loc(\R^d)^d)$ of the  infinite-volume problem with gravity $e_i$,
\begin{equation}\label{eq:correct-sed1}
\left\{\begin{array}{ll}
-\triangle u_{\infty,i}+\nabla p_{\infty,i} =0,&\text{in $\R^d\setminus\Ic$},\\
\Div( u_{\infty,i})=0,&\text{in $\R^d\setminus\Ic$},\\
\D( u_{\infty,i})=0,&\text{in $\Ic$},\\
e_i|B|+\int_{\partial B(x_n)}\sigma( u_{\infty,i}, p_{\infty,i})\nu=0,&\forall n\in\N,\\
\int_{\partial B(x_n)}(\cdot-x_n)\times\sigma(u_{\infty,i},p_{\infty,i})\nu=0,&\forall n\in\N,
\end{array}\right.
\end{equation}
such that $\nabla u_{\infty,i}$ is stationary and satisfies $\expec{\nabla u_{\infty,i}}=0$, $\expec{|\nabla u_{\infty,i}|^2}<\infty$, and the anchoring condition $\int_Bu_{\infty,i}=0$.
We may then define the infinite-volume mean settling speed as the symmetric matrix $\bar V_\infty=\{\bar V_{\infty,ij}\}_{1\le i,j\le d}$ given by
\begin{equation}\label{eq:mean-velocity-ws}
\bar V_{\infty,ij} \,:=\, \frac{1} {\lambda |B|} \expec{\nabla u_{\infty,i}:\nabla u_{\infty,j}}.
\end{equation}
It describes the ensemble-averaged settling speed in direction $e_j$ induced by gravity~$e_i$, relative to the averaged background fluid motion. More precisely,
we have almost surely
\begin{equation}\label{eq:def-V}
\lim_{r\uparrow\infty}\bar V_{\infty,ij}^{\operatorname{rel}}(\chi_r)\,=\,\bar V_{\infty,ij},
\end{equation}
where
\[\bar V_{\infty,ij}^{\operatorname{rel}}(\chi_r)\,:=\,e_j\cdot\bigg(\frac{\sum_{x\in\Pc}\,\chi_r(x) \fint_{B(x)}u_{\infty,i}}{\sum_{x\in\Pc}\chi_r(x)}-\fint_{\chi_r}u_{\infty,i}\bigg).\]
In particular, the limit is independent of the averaging function~$\chi$.
\end{theor}

We next connect this intrinsic infinite-volume quantity with the original finite-container problem~\eqref{eq:Dir}: we show that the relative local mean settling speed~\eqref{eq:loc-settlsp-re} converges to $\bar V_\infty e$ in the regime $1\ll r\ll R$, independently of the shape of the container. This absence of boundary effects is far from automatic in a problem governed by long-range hydrodynamic interactions; we refer to Section~\ref{sec:phys-intro} for the relevant physical background.

For concreteness, we first state the result in an infinite cylindrical geometry with no-slip boundary condition, which is viewed as the closest mathematical idealization of classical sedimentation experiments. The proof, given in Section~\ref{sec:inf-cyl}, combines the above infinite-volume construction with the qualitative large-scale regularity theory that we developed in~\cite{DG-21a} for random suspensions.

\begin{theor}\label{th:mean-velocityR}
Let the assumptions of Theorem~\ref{th:mean-velocity} hold. Let $U:=U'\times\R$, where $U'\subset\R^{d-1}$ is a bounded Lipschitz domain, and set $U_R=RU$. For $R\ge1$, the Stokes problem~\eqref{eq:Dir} is well-posed in $U_R$ with homogeneous Dirichlet boundary condition
\[u_{R} =0\qquad\text{on $\partial U_R$.}\]
Moreover, we have almost surely
\begin{equation}\label{e.conv-mean-velocityR}
\lim_{r\uparrow\infty}\lim_{R\uparrow\infty}\bar V_R^{\operatorname{rel}}(\chi_r)\,=\,\bar V_{\infty}e,
\end{equation}
where $\bar V_R^{\operatorname{rel}}(\chi_r)$ is the relative local mean settling speed~\eqref{eq:loc-settlsp-re} and where the limit $\bar V_\infty$ is defined in Theorem~\ref{th:mean-velocity} and is independent of both the cylindrical domain~$U$ and the averaging function~$\chi$.
The same also holds when the averaging scale is comparable with the container size, $r=R$: for $\chi\in C^\infty_c(U)$, we have almost surely
\begin{equation}\label{e.conv-mean-velocityR-re}
\lim_{R\uparrow\infty}\bar V_R^{\operatorname{rel}}(\chi_R)\,=\,\bar V_{\infty}e.
\end{equation}
\end{theor}

\begin{rem}\label{rem:bndary}
The same result applies to other natural geometries and homogeneous boundary conditions. In Section~\ref{sec:boundary}, we treat in particular the following two other cases:
\begin{enumerate}[---]
\item \emph{Snow-globe geometry.} Given a bounded Lipschitz domain $U\subset\R^d$, we set $U_R=RU$ and impose homogeneous Dirichlet boundary conditions on $\partial U_R$; see Section~\ref{sec:snow-globe}.
\smallskip\item \emph{Finite cylinder.} Given $U:=U'\times[a,b]$, with $U'\subset\R^{d-1}$ a bounded Lipschitz domain and with $a<b$, we set $U_R=RU$, we impose homogeneous traction boundary conditions $\sigma(u_R,p_R)\nu=0$ on the top face $R(U'\times\{b\})$, and homogeneous Dirichlet boundary conditions on the remaining part of the boundary, and we consider sedimentation under a vertical gravity $e\in\R(0,\ldots,0,1)$; see Section~\ref{sec:finite-cyl}.
\end{enumerate}
In the snow-globe case, incompressibility together with the Dirichlet boundary condition yields $\int_{U_R}u_R=0$. Hence, the relative average can be replaced by the global mean particle velocity, and we obtain almost surely
\[\bar V_{R}(U_R)\,:=\,\frac{\sum_{x\in\Pc(U_R)}\fint_{B(x)}u_R}{\sharp\Pc(U_R)}\,\xrightarrow{R\uparrow\infty}\,\bar V_{\infty}e.\]
 In the cylindrical case, the same argument only gives the vanishing of the mean velocity in the vertical direction, $\int_{U_R}u_R\cdot e=0$, hence $\bar V_R(U_R)\cdot e \to e\cdot\bar V_{\infty}e$.
\end{rem}

Finally, we turn to the dilute expansion of the infinite-volume mean settling speed and prove Batchelor's formula for the first correction to the Stokes velocity of a single settling particle. This correction arises as the two-particle contribution in a renormalized cluster expansion. We refer to Section~\ref{sec:phys-intro} for a discussion of the structure of Batchelor's formula, of its physical content, and of further approximations.
The formula is expressed in terms of the following one- and two-particle Stokes flows.
\begin{enumerate}[---]
\item For $x\in\R^d$, let $(\varphi^x_e,\pi_e^x)$ be the unique decaying solution of the single-particle Stokes problem
\[\left\{\begin{array}{ll}
-\triangle\varphi^x_e+\nabla\pi^x_e=0,&\text{in $\R^d\setminus B(x)$},\\
\Div(\varphi^x_e)=0,&\text{in $\R^d\setminus B(x)$},\\
\D(\varphi^x_e)=0,&\text{in $B(x)$},\\
e|B|+\int_{\partial B(x)}\sigma(\varphi^x_e,\pi^x_e)\nu=0,&\\
\int_{\partial B(x)}(\cdot-x)\times\sigma(\varphi^x_e,\pi^x_e)\nu=0.&
\end{array}\right.\]
\item For $x,y\in\R^d$ with $|x-y|>2$, let $(\varphi^{x,y}_e,\pi^{x,y}_e)$ be the unique decaying solution of the two-particle Stokes problem
\[\left\{\begin{array}{ll}
-\triangle\varphi^{x,y}_e+\nabla\pi^{x,y}_e=0,&\text{in $\R^d\setminus (B(x)\cup B(y))$},\\
\Div(\varphi^{x,y}_e)=0,&\text{in $\R^d\setminus (B(x)\cup B(y))$},\\
D(\varphi^{x,y}_e)=0,&\text{in $B(x)\cup B(y)$},\\
e|B|+\int_{\partial B(z)}\sigma(\varphi^{x,y}_e,\pi^{x,y}_e)\nu=0,&\text{for $z\in\{x,y\}$}\\
\int_{\partial B(z)}(\cdot-z)\times\sigma(\varphi^{x,y}_e,\pi^{x,y}_e)\nu=0,&\text{for $z\in\{x,y\}$}.
\end{array}\right.\]
\end{enumerate}
The one-particle flow $\varphi_e^x$ is explicit for spherical inclusions, cf.~Appendix~\ref{app:explicit}, whereas no closed formula is available for the two-body flow~$\varphi_e^{x,y}$.

The proof, given in Section~\ref{sec:cluster}, is based on a rigorous cluster expansion of the infinite-volume settling speed $\bar V_\infty$. A key difficulty is that the infrared renormalization must be carried out cluster by cluster. To this end, we introduce in Section~\ref{sec:elements} a finitary diagrammatic expansion of hydrodynamic interactions in terms of what we call \emph{reflection blocks}. This procedure isolates the leading divergent contributions and makes the relevant cancellations explicit. The cluster expansion must then be truncated, and the main remaining task is to estimate the truncation error, which again requires exploiting delicate cancellations. Since the remainder involves hydrodynamic interactions of arbitrary order, the reflection-block decomposition does not apply directly. The key ingredient is a nonlinear estimate that reduces the control of the full remainder to the control of finitely many higher-order cluster terms, to which the reflection-block decomposition can once again be applied.
Although the proof scheme can in principle be extended to arbitrary order, we restrict the analysis here to the second order needed for Batchelor's formula.

\begin{theor}\label{th:Batchelor}
The infinite-volume mean settling speed $\bar V_\infty$ defined in Theorem~\ref{th:mean-velocity} admits the second-order expansion
\[\bar V_\infty\,=\,\bar V_\infty^{(1)}+\bar V_\infty^{(2)}+E_\infty.\]
The leading term is the Stokes velocity of a single settling particle,
\begin{equation}\label{eq:stokes-vel}
\bar V_\infty^{(1)}e\,:=\,V_Se:=\fint_B\varphi_e^0\,=\,\frac{d-1}{d^2(d-2)}e.
\end{equation}
The second term is Batchelor's correction: it depends only on pair statistics and is the symmetric matrix characterized by
\begin{multline}\label{eq:Batch-form-pr}
e\cdot \bar V_\infty^{(2)}e\,:=\,
\frac1{\lambda}\int_{\R^d}\bigg(e\cdot\big(U_{e}(y)-V_Se\big)+\frac1{|B|}\int_{\partial B}\varphi^y_e\cdot \sigma(\varphi^0_e,\pi^0_e)\nu\bigg)\,f_{2}(0,y)\,dy\\
-\frac1{\lambda}\int_{\R^d}\bigg(\frac1{|B|}\int_{\partial B}\varphi^y_e\cdot \sigma(\varphi^0_e,\pi^0_e)\nu\bigg)\,h_{2}(0,y)\,dy,
\end{multline}
where $U_e(y):=\fint_B\varphi^{0,y}_e$ is the velocity of either sphere in the two-particle Stokes problem, and where both integrals are absolutely convergent.
Finally, the remainder $E_\infty$ satisfies the following estimates:
\begin{enumerate}[(i)]
\item Assume that the point process $\Pc$ is strongly mixing and that its correlation functions satisfy, for some $\beta>0$,
\begin{equation}\label{eq:estim-correl}
\qquad|h_{k}(x_1,\ldots,x_k)|\le C_k\min\Big\{\lambda_k\,;\,\min_{i\ne j}\langle x_i-x_j\rangle^{-2-\beta}\Big\}\qquad\text{for $2\le k\le 8$}.
\end{equation}
This condition holds, for instance, whenever $\Pc$ is $\alpha$-mixing with rate $\alpha(r)=r^{-2-\beta}$. Then
\[|\bar V_\infty^{(2)}|\,\lesssim\,\tfrac1\lambda(\lambda_2)^\frac{\beta}{2+\beta},\qquad |E_\infty|\,\lesssim\,\tfrac{1}{\lambda}\Big(\sqrt\lambda_5+(\sqrt\lambda_6)^{\frac{\beta}{2+\beta}}\Big).\]
\item Assume instead the following stronger correlation bound, for some $\beta>0$,
\[|h_{k}(x_1,\ldots,x_k)|\le C_k\lambda_k\min_{i\ne j}\langle x_i-x_j\rangle^{-2-\beta}\qquad\text{for $2\le k\le 8$}.\]
Then
\[|\bar V_\infty^{(2)}|\,\lesssim\,\tfrac1\lambda\lambda_2,\qquad|E_\infty|\,\lesssim\,\tfrac{1}{\lambda}\sqrt\lambda_5.\]
\end{enumerate}
\end{theor}

\begin{rem}[Optimality of the error estimates]\label{rem:optimal}
The above error estimates are not optimal. A refinement of the argument yields the sharper bounds
\begin{equation}\label{eq:Einfty-optimal}
|E_\infty|\,\lesssim\,
\left\{\begin{array}{lll}
\frac{1}{\lambda}(\sqrt\lambda_6)^{\frac{\beta}{2+\beta}}&:&\text{in case~(i)},\\[1mm]
\frac{1}{\lambda}\sqrt\lambda_6&:&\text{in case~(ii)}.
\end{array}\right.
\end{equation}
In particular, under assumption~(ii), if moreover $\lambda_k\lesssim\lambda^k$ for $k\le6$, as is the case for a hardcore Poisson process, then
\[|E_\infty|\lesssim\lambda^2.\]
This recovers the order of the error predicted by Batchelor~\cite{Batchelor-72}. By contrast, the estimate stated directly in Theorem~\ref{th:Batchelor} only gives $|E_\infty|\lesssim\lambda^{3/2}$ in this regime.
The loss comes from the nonlinear estimate used to control the remainder in the cluster expansion~\eqref{eq:VL-decomp}; see Section~\ref{sec:remainder}. The sharper bounds in~\eqref{eq:Einfty-optimal} can be obtained by carrying the cluster expansion~\eqref{eq:VL-decomp} to fourth order instead of third order --- this refinement is only technical.
\end{rem}

The above result builds upon preliminary work of the authors with Pertinand, as briefly described in
\cite[Chapter~3]{perti2022}, where the Stokes velocity~$V_S$ was derived as the leading contribution to~$\bar V_\infty$ in a more restrictive setting.

The rest of this introduction is organized as follows: In Section~\ref{sec:infra-red-renorm}, we describe the infrared renormalization procedure. In Section~\ref{sec:literature}, we place the work in relation to the mathematical literature on related problems. A discussion of the results and their implications from the perspective of  physics is postponed to Section~\ref{sec:phys-intro}.

\subsection{Infrared renormalization}\label{sec:infra-red-renorm}

Our analysis is guided by the general architecture of renormalization in singular stochastic PDEs. In that context, the obstruction is ultraviolet: nonlinear operations on rough random fields generate divergent small-scale contributions, which must be compensated by suitable counterterms. Following Bourgain~\cite{Bourgain96} and Da Prato and Debussche~\cite{MR2016604}, the basic strategy is to isolate explicit singular components, renormalize them, and then solve for a more regular remainder by deterministic PDE estimates. In more singular problems, such as $\Phi^4_3$, the extraction of the singular part involves a systematic expansion of nested divergent substructures, and the equation for the remainder has to be formulated in spaces adapted to the resulting singular local expansion; this is achieved by regularity structures or paracontrolled calculus~\cite{MR3274562,MR3406823}.

The obstruction in the present work is different in nature: the random data, namely the particle configuration, do not produce local singularities in the PDE; the divergence is instead infrared, due to the non-integrability at infinity of the Green kernels associated with the Stokes problem. Nevertheless, the same general strategy applies: the relevant infinite-volume observables are decomposed into an explicit singular part, carrying the non-integrable large-scale contribution, and a regular remainder controlled by elliptic estimates. The singular part is then renormalized by identifying the deterministic and correlation-induced cancellations that compensate the large-scale divergence.

We first illustrate this mechanism in the simplest setting needed for the construction of the infinite-volume settling speed in Theorem~\ref{th:mean-velocity}. Consider, for simplicity, the scalar electrostatic analogue of the sedimentation problem~\eqref{eq:correct-sed1},
\begin{equation}\label{e.corr}
\left\{\begin{array}{ll}
-\Delta u=0, &\text{in $\R^d\setminus\Ic$},\\
u|_{B(x_n)}\equiv u_n\in\R,&\forall n,\\[1mm]
|B|+\int_{\partial B(x_n)}\nu\cdot\nabla u=0, &\forall n.
\end{array}\right.
\end{equation}
This describes conducting spheres carrying prescribed unit charges, with unknown constant potentials $u_n$. In the point-force approximation, one is led to the linear model
\begin{equation}\label{e.model1}
-\Delta v=\mathds 1_{\mathcal I}
\qquad\text{in $\R^d$}.
\end{equation}
Formally, $v=G\star\mathds1_\Ic$, but this convolution with the Green function $G$ has no direct infinite-volume meaning since $G$ is not integrable at infinity. Introducing a massive cut-off,
\begin{equation}\label{e.model1-mu}
\mu v_\mu-\Delta v_\mu=\mathds1_\Ic
\qquad\text{in $\R^d$},
\end{equation}
one finds the divergent zero mode $\E[v_\mu]=\frac1\mu\E[\mathds1_\Ic]=\frac1\mu\lambda|B|$.
The corresponding renormalization is simply the subtraction of this mode: the difference $\tilde v_\mu:=v_\mu-\E[v_\mu]$ satisfies
\begin{equation}\label{e.model1-mu-renom}
\mu\tilde v_\mu-\Delta \tilde v_\mu=\mathds1_\Ic-\lambda |B|\qquad\text{in $\R^d$}.
\end{equation}
Under quantitative decorrelation assumptions, the centered right-hand side has sufficient large-scale cancellations to compensate for the non-integrability of the Green kernel, yielding uniform bounds on $\expec{|\nabla\tilde v_\mu|^2}$ as $\mu\downarrow0$. The weak limit $\nabla\tilde v$ is then the gradient solution of the renormalized equation
\begin{equation}\label{e.model1-renom}
-\Delta \tilde v=\mathds 1_{\mathcal I}-\lambda |B|
\qquad\text{in $\R^d$}.
\end{equation}
This is the simplest instance of an infrared renormalization: the diverging zero mode is compensated by a uniform background charge, as in the jellium model, and the remaining centered field is finite because of stochastic cancellations.
Returning to the nonlinear problem~\eqref{e.corr}, the corresponding renormalization consists in subtracting the large-scale charge density outside the inclusions:
\begin{equation}\label{e.corr-renom}
\left\{\begin{array}{ll}
-\Delta \tilde u=-\frac{\lambda |B|}{1-\lambda |B|},&\text{in $\R^d\setminus\Ic$},\\
\tilde u|_{B(x_n)}\equiv u_n\in\R,&\forall n,\\[1mm]
|B|+\int_{\partial B(x_n)}\nu\cdot\nabla\tilde u=0,&\forall n.
\end{array}\right.
\end{equation}
It can be shown that the linear field $\nabla\tilde v$ precisely captures the leading non-integrable component of $\nabla\tilde u$, in the sense that the difference $\nabla\tilde w:=\nabla\tilde u-\nabla\tilde v$
has better large-scale behavior and can be constructed by elliptic estimates. More precisely, $\nabla\tilde u$ is the Helmholtz projection of $\nabla\tilde v$ onto the space of gradient fields with constant potential inside the conducting spheres.

This simple argument may be viewed as an infrared analogue of the Wick renormalization in the~$\Phi^4_2$ model by Da Prato and Debussche~\cite{MR2016604}. It can be adapted to the sedimentation problem~\eqref{eq:correct-sed1}, with a more careful infrared cut-off, and yields the construction of the infinite-volume mean settling speed in Theorem~\ref{th:mean-velocity}. In that setting, the counterterm has a direct physical interpretation: it is the compensating backflow, or equivalently the uniform pressure gradient, needed to balance the nonzero mean force density in infinite volume. In the final formulation, this backflow is absorbed into the pressure field rather than added as an external term; this is reflected in the lack of stationarity of the pressure, cf.~\eqref{eq:comp-pressure}.

The proof of Theorem~\ref{th:Batchelor} requires a substantially finer version of the same renormalization strategy. There, the infinite-volume settling speed is expanded into many-body cluster contributions, and the infrared renormalization has to be carried out at the level of each cluster term. These terms are nonlinear functionals of finite particle configurations, built from auxiliary Stokes problems involving different subsets of particles, and their singular parts are not easily extracted.

The main technical novelty here is a finitary diagrammatic expansion of hydrodynamic many-body interactions, in the form of a resummed method of reflections; see Section~\ref{sec:elements}. This reflection-block expansion isolates the substructures that may produce infrared divergences and makes explicit the cancellations that remove them. Not every potentially divergent diagram gives rise to an independent counterterm: the genuine physical counterterms all come from the compensating backflow, while the remaining apparent divergences cancel through exact identities; see Lemma~\ref{lem:cancel}. Once these cancellations have been exposed, the remaining terms are shown to be absolutely convergent by classical elliptic estimates on finite-particle Stokes problems.

In contrast with the treatment of the $\Phi^4_2$ model by Da Prato and Debussche~\cite{MR2016604}, the extraction of the singular part is itself a substantial part of the analysis. In contrast with regularity structures or paracontrolled calculus, however, once this extraction is carried out, no modelled-distribution or noise-dependent solution theory is needed. We use finitary reflection-block expansions only as algebraic devices for extracting the infrared singular part and identifying the necessary cancellations, while the remaining terms are amenable to deterministic estimates.

\subsection{Relation to the mathematical literature}\label{sec:literature}
The results most closely related to the present work are our previous contributions~\cite{MR4259684,DG-22} on random fluctuations in sedimentation and on the Caflisch-Luke paradox, our work~\cite{DG-21b} on the cluster expansion of effective viscosity of particle suspensions, and the recent work of Hillairet and H\"ofer~\cite{MR4746426} on the settling speed of periodic arrays of particles.

On the other hand, it should be distinguished from the extensive literature on mean-field limits for sedimenting suspensions; see, for instance,~\cite{Hofer-19} and subsequent developments. Those works aim at deriving macroscopic evolution equations for heterogeneous suspensions, and therefore focus on the large-scale collective motion of the particle density. Here, by contrast, we consider statistically homogeneous suspensions, for which the particle density is constant. In this regime the macroscopic motion cancels, and the relevant object is instead the microscopic settling speed at the particle scale. This cancellation is not a harmless simplification: it is precisely reflected in the need for an infrared renormalization and is physically interpreted as the screening effect of the compensating fluid backflow.

\medskip\noindent{\it Infinite-volume mean settling speed.}
Because of the long-range nature of hydrodynamic interactions, even the definition of the infinite-volume problem~\eqref{eq:correct-sed1} is delicate. The first issue is therefore to construct an intrinsic infinite-volume settling speed.

There are two natural approximation procedures. One may either introduce an infrared cut-off in the operator, or solve the problem in bounded domains and send the size of the domain to infinity. For sedimentation, the second route is particularly subtle: long-range interactions make it difficult to exclude boundary effects that would persist in the large-volume limit. We therefore follow the first route.

In stochastic homogenization of divergence-form elliptic equations, adding an infrared cut-off in the form of a massive term is a standard device; see e.g.~\cite{PapaVara,GO1}. This strategy was used in~\cite{MR4259684} for the scalar analogue~\eqref{e.corr} of the sedimentation problem. Combined with a strong quantitative mixing assumption, namely a functional inequality in probability, it led to the construction of a scalar analogue of the infinite-volume settling speed. There are two main differences with Theorem~\ref{th:mean-velocity}. First, in the present Stokes setting, the divergence-free constraint and the associated Helmholtz projection retain a long-range component of the interaction. A naive massive regularization is therefore not sufficient to cut off the infrared behavior, and a more adapted regularization has to be introduced at the level of the Cauchy-stress tensor; see~\eqref{eq:cut-sigma-mu} below. Second, Theorem~\ref{th:mean-velocity} is proved under essentially sharp statistical assumptions: it only requires a decay of the two-point correlation, slightly faster than~$|x|^{-2}$. The functional inequality used in~\cite{MR4259684} is a much stronger quantitative mixing hypothesis that controls correlation functions of all orders.
 
\medskip\noindent{\it Large-container limit.}
Once the intrinsic infinite-volume settling speed has been constructed, the next question is whether it describes the relative settling speed observed in large containers. The point is again nontrivial because long-range interactions may transmit information from the boundary to the bulk.
To illustrate the difficulty, consider the scalar model~\eqref{e.corr} and, for simplicity, its renormalized linearization~\eqref{e.model1-renom} with homogeneous Dirichlet boundary conditions in~$U_R=RU$,
\[\left\{\begin{array}{ll}
-\triangle v_R=\mathds1_{\Ic(U_R)}-\lambda|B|,&\text{in $U_R$},\\
v_R=0,&\text{on $\partial U_R$}.
\end{array}\right.\]
Testing with~$v_R$ gives
\[\int_{U_R}|\nabla v_R|^2
\,=\,\int_{U_R}(\mathds1_{\Ic(U_R)}-\lambda|B|)v_R.\]
Using only Poincar\'e's inequality in~$U_R$, whose constant is of order $R$, yields the worst-case estimate
\begin{equation}\label{e:naive-apriori}
\fint_{U_R}|\nabla v_R|^2\lesssim R^2.
\end{equation}
This bound does not exploit the statistical homogeneity of the particle distribution, and it is optimal at this level of generality: it corresponds to the possible presence of a macroscopic motion in the container~\cite{Hofer-19}.
For statistically homogeneous suspensions, one expects a better scaling, but this improvement is not automatic.

One deterministic way to recover uniform bounds is to impose strong quantitative equidistribution of the particles. If the source term $\mathds1_{\Ic(U_R)}-\lambda|B|$ had vanishing local averages at unit scale away from the boundary, then a unit-scale Poincar\'e inequality in the bulk would replace the global Poincar\'e estimate and would yield a uniform bound instead of~\eqref{e:naive-apriori}. This is close in spirit to the strategy of Hillairet and H\"ofer~\cite{MR4746426}, who study well-separated deterministic configurations under strong separation and density-convergence assumptions, and identify leading hindered-settling corrections, including boundary and small-scale inhomogeneity effects. Their assumptions, which include a bound on the infinite-Wasserstein distance between the point process $\Pc$ and the Lebesgue measure $\lambda dx$, are well suited to periodic configurations. They do not, however, cover the hardcore Poisson model, which is the canonical model for well-stirred dilute suspensions in the physics literature. We also note that Hillairet and H\"ofer treat stratified configurations, which are physically relevant but outside the scope of the present work.

A second route is probabilistic. Under suitable probabilistic assumptions on $\Ic$, one may hope to prove an averaged estimate of the form
\[\expec{\fint_{U_R}|\nabla v_R|^2}\lesssim1.\]
This is the route taken in~\cite{DG-22}, where we established fluctuation estimates for random suspensions under correlation-decay assumptions comparable to those used here for the construction of the infinite-volume settling speed. There is, however, a major difficulty: because hydrodynamic interactions are long-ranged, finite-volume approximations are sensitive to the treatment of particles near the boundary. Deleting particles, cutting them, or modifying the point process in a boundary layer may produce non-negligible contributions. The fact that such boundary effects do not affect the relative settling speed is precisely part of what is proved here in Theorem~\ref{th:mean-velocityR}, and this proof relies on the intrinsic construction of Theorem~\ref{th:mean-velocity}. To avoid a circular argument, in~\cite{DG-22}, we made the strong assumption that the point process can be periodized in law and satisfies a functional inequality in probability, uniformly in the periodization scale --- which essentially restricts the results to processes built from a Poisson process. 

In the present work, we make more substantial use of the properties of the infinite-volume solution constructed in Theorem~\ref{th:mean-velocity} and we appeal to the qualitative large-scale regularity that we established in~\cite{DG-21a} for random suspensions. This allows us to prove that the infinite-volume solution governs the large-container relative settling speed and to prove Theorem~\ref{th:mean-velocityR} under minimal assumptions on the point process, without imposing periodization in law or any additional functional inequality in probability.

\medskip\noindent{\it Dilute cluster expansion and Batchelor's formula.}
Dilute expansions of homogenized coefficients, or effective medium expansions, are classical in the physics literature and have recently attracted substantial attention in mathematics. Examples include the Clausius-Mossotti formula and Einstein's effective viscosity formula; see, for instance,~\cite{MR3458165,MR4643267,Haines-Mazzu,MR4160802,MR4280836,MR4260456,MR4400909,DG-21b}. Batchelor's formula differs from these problems in two important respects.

First, both the Clausius-Mossotti formula and Einstein's formula are first-order dilute corrections, obtained by summing the effects of isolated particles. In the present sedimentation problem, the analogous first-order term is the Stokes velocity $\bar V_\infty^{(1)}=V_S$. These first-order terms are universal and do not depend on the statistics of the point process.
Batchelor's correction~$\bar V_\infty^{(2)}$, by contrast, is a second-order correction and depends on the pair statistics of the suspension. In this respect it is closer to the second-order correction in the expansion of the effective viscosity beyond Einstein's formula; compare~\cite{MR4280836,DG-21b}.
 
Second, and more importantly, the effective viscosity is well-defined without infrared renormalization. At the linearized level, the corrector equation for the effective viscosity problem has conservative form,
\[-\triangle v=\nabla\cdot b,\]
for some stationary random field $b$,
and the energy estimate directly yields
\[\expec{|\nabla v|^2}\lesssim \expec{|b|^2}.\]
This contrasts with~\eqref{e:naive-apriori}. As a consequence, for Einstein's formula, cluster expansions have at worst borderline singularities, which can be handled by Calder\'on-Zygmund estimates. For sedimentation, the cluster formulas associated with~$\bar V_\infty$ are genuinely infrared divergent and must be renormalized more carefully. For instance, whereas in the derivation of Einstein's formula in~\cite{DG-21b} arbitrary finite-volume approximations could be used to make the expansion rigorous, here it is critical to use an intrinsic approximation that avoids boundary effects. The renormalization of the cluster expansion is the central contribution of this paper.

\medskip\noindent{\it Connection with the random matching problem.}
Finally, we mention a connection with the random matching problem. The linearized equation~\eqref{e.model1-renom} coincides with the linearization of the Monge-Amp\`ere equation arising in the optimal coupling between a random measure and the Lebesgue measure; see~\cite{MR4373162}.
In particular, as in the matching problem between the Poisson and Lebesgue measures, the critical dimension is $d=2$; see~\cite{MR3112922}.
Theorem~\ref{th:mean-velocity} is more flexible with respect to the underlying point process, since it only assumes suitable decay of correlations and does not require the process to be Poisson. The problem treated here is nevertheless different: the main object is the sedimentation velocity in a Stokes suspension, and the central difficulty is the hydrodynamic infrared renormalization, rather than the optimal-transport structure.


\section{Infinite-volume mean settling speed}\label{sec:mean-settling}

This section is devoted to the proof of Theorem~\ref{th:mean-velocity}.

\subsection{Fredholm alternative, counterterm, and backflow}\label{sec:backflow}

We construct $u_{\infty,i}$ by an approximation procedure.
To this end, we first observe that one cannot solve~\eqref{eq:correct-sed1} with both $\nabla u_{\infty,i}$ and $p_{\infty,i}$ stationary. Indeed, the pressure $p_{\infty,i}$ necessarily contains a linear part to compensate the prescribed boundary data. This linear component has to be identified and treated explicitly as a counterterm, and is viewed as a mean backflow.

For some $b_i\in\R^d$ to be identified, assume that there exists a solution $(u_{\infty,i},p_{\infty,i})$ of~\eqref{eq:correct-sed1} such that $\nabla u_{\infty,i}$ and 
$\tilde p_{\infty,i} := p_{\infty,i}-b_i \cdot x$ are stationary.
In terms of this modified pressure, the Stokes problem~\eqref{eq:correct-sed1} can be rewritten as
\begin{equation}\label{eq:correct-sed1-re}
\left\{\begin{array}{ll}
-\triangle u_{\infty,i}+\nabla \tilde p_{\infty,i} =-b_i,&\text{in $\R^d\setminus\Ic$},\\
\Div( u_{\infty,i})=0,&\text{in $\R^d\setminus\Ic$},\\
\D( u_{\infty,i})=0,&\text{in $\Ic$},\\
(e_i-b_i)|B|+\int_{\partial B(x)}\sigma( u_{\infty,i}, \tilde p_{\infty,i})\nu=0,&\forall x\in\Pc,\\
\int_{\partial B(x)}(\cdot-x)\times\sigma(u_{\infty,i},\tilde p_{\infty,i})\nu=0,&\forall x\in\Pc.
\end{array}\right.
\end{equation}
Taking the expectation of the first equation in $\R^d\setminus\Ic$, and using the boundary conditions together with the assumed stationarity of both $\nabla u_{\infty,i}$ and $\tilde p_{\infty,i}$, we formally obtain
\[\expec{(e_i-b_i) \mathds1_{\Ic} - b_i \mathds1_{\R^d\setminus \Ic}}=0,\]
which yields the identity
\begin{equation}\label{eq:backflow}
b_i =  \lambda |B| e_i.
\end{equation}
On a bounded domain with periodic boundary conditions, this formal calculation would amount to the Fredholm alternative ensuring the existence of a periodic pressure.  In the present infinite-volume setting, the same compatibility mechanism suggests that this counterterm~$b_i$ is necessary in order for the modified pressure to be stationary.

Although~\eqref{eq:correct-sed1-re} is strictly equivalent to~\eqref{eq:correct-sed1} as far as the velocity $u_{\infty,i}$ is concerned, it will be more convenient to work with as it provides a better framework with stationary pressure.
For notational convenience, we perform a change of unknowns and set
\begin{equation}\label{e.uinfty-scaling}
u_{\infty,i}= (1-\lambda |B|) \varphi_i.
\end{equation}
The Stokes problem~\eqref{eq:correct-sed1-re} for $u_{\infty,i}$ is then equivalent to the following system for~$\varphi_i$,
\begin{equation}\label{eq:correct-sed}
\left\{\begin{array}{ll}
-\triangle\varphi_i+\nabla\pi_i=-\alpha e_i,&\text{in $\R^d\setminus\Ic$},\\
\Div(\varphi_i)=0,&\text{in $\R^d\setminus\Ic$},\\
\D(\varphi_i)=0,&\text{in $\Ic$},\\
e_i|B|+\int_{\partial B(x_n)}\sigma(\varphi_i,\pi_i)\nu=0,&\forall n\in\N,\\
\int_{\partial B(x_n)}(\cdot-x_n)\times\sigma(\varphi_i,\pi_i)\nu=0,&\forall n\in \N,
\end{array}\right.
\end{equation}
where the intensity of the backflow is given by
\begin{equation}\label{eq:def-alpha}
\alpha:=\frac{\lambda|B|}{1-\lambda|B|}.
\end{equation}
In what follows, we prove that this system~\eqref{eq:correct-sed} admits a unique solution
\[\varphi_i\in\Ld^2(\Omega;H^1_\loc(\R^d)^d)\]
such that $\nabla\varphi_i$ is stationary, $\E[\nabla\varphi_i]=0$, $\expec{|\nabla \varphi_i|^2}<\infty$, and with anchoring $\int_B\varphi_i=0$.  At the same time, the proof also yields the existence and uniqueness of the associated stationary pressure field $\pi_i\in \Ld^2(\Omega;\Ld^2_\loc(\R^d))$.

The analysis of~\eqref{eq:correct-sed} was the topic of our previous work~\cite{DG-22} on sedimentation. There, however, the form of~\eqref{eq:correct-sed} was only motivated by solvability conditions, and the relation~\eqref{e.uinfty-scaling} with the original whole-space problem~\eqref{eq:correct-sed1} was not made explicit.
Moreover, in order to solve~\eqref{eq:correct-sed} in~\cite{DG-22}, we assumed that we could periodize the point process~$\Pc$ in law and that these periodizations satisfied uniform quantitative mixing conditions. This is a severe restriction, which we completely remove here.

\subsection{Proof of Theorem~\ref{th:mean-velocity}}\label{sec:proof-th:mean-velocity}

We argue by approximation. As in stochastic homogenization for elliptic equations in divergence form, we introduce an infrared cut-off by adding a massive term.  For the Stokes system, because of incompressibility, this is not quite sufficient: the massive term alone only regularizes the transverse modes. To remedy this, we further need to relax incompressibility and penalize the divergence. More precisely, we consider the following regularized version of~\eqref{eq:correct-sed}: for $0<\mu\le 1$,
\begin{equation}\label{eq:correct-sed-mu}
\left\{\begin{array}{ll}
\mu \varphi_{\mu,i}-\Div(\sigma_\mu(\varphi_{\mu,i}))=-\alpha e_i,&\text{in $\R^d\setminus\Ic$},\\
\D(\varphi_{\mu,i})=0,&\text{in $\Ic$},\\
e_i|B|+\int_{\partial B(x)}\sigma_\mu(\varphi_{\mu,i})\nu=0,&\forall x\in\Pc,\\
\int_{\partial B(x)}(\cdot-x)\times\sigma_\mu(\varphi_{\mu,i})\nu=0,&\forall x\in\Pc,
\end{array}\right.
\end{equation}
where the modified stress tensor includes an $O(\frac1\mu)$ diverging bulk viscosity,
\begin{equation}\label{eq:cut-sigma-mu}
\sigma_\mu(\varphi)\,:=\,2\D(\varphi)+\frac1\mu\Div(\varphi)\Id.
\end{equation}
The massive term screens the transverse modes at length scale $\mu^{-1/2}$, whereas the divergence penalization only screens the longitudinal modes at the larger scale $\mu^{-1}$. This mismatch of scales is a genuinely nontrivial feature of the regularization and will have to be tracked carefully; see in particular Lemma~\ref{lem:massGreen} for the corresponding Green function estimates. One could instead introduce a higher-order divergence penalization, for instance through a term of the form $\frac1\mu(-\triangle)\Div(\varphi)\Id$, which would align the transverse and longitudinal screening scales, but this would lead to additional complications and we therefore retain the simpler regularization above.

The proof of Theorem~\ref{th:mean-velocity} is divided into five steps. First, we prove deterministic well-posedness for~\eqref{eq:correct-sed-mu}. We then establish estimates that are uniform in~$\mu$, and pass to the limit $\mu \downarrow 0$ to obtain the unique solution of~\eqref{eq:correct-sed}. The last two steps identify the infinite-volume mean settling speed and record the strong convergence of the infrared-regularized approximation.

\medskip
\step1 Deterministic well-posedness of~\eqref{eq:correct-sed-mu}.\\
In this step, the precise value of the counterterm $-\alpha e_i$ is not used.
It suffices to prove uniform local energy estimates for solutions: existence then follows by approximation on increasing domains, while uniqueness follows by comparison.  We work in the uniformly local space 
\begin{equation*}
H^1_{\uloc}:=\bigg\{v \in H^1_\loc(\R^d)^d \,:\,\sup_{x\in \R^d} \int_{B(x)} |v|^2 + |\nabla v|^2<\infty\bigg\}.
\end{equation*}
Let $\mu\in(0,1]$. For some $\kappa_0\ge1$ to be chosen below, starting from  $x\mapsto\exp(-\frac{\mu|x|}{\kappa_0})$ and modifying it near the inclusions, thanks to the $\delta$-hardcore assumption, we can construct a function~$\eta_\mu$ that is constant in a neighborhood of each inclusion $B(x)$, $x\in\Pc$, and satisfies
\begin{equation}\label{e.exp-cutoff}
\eta_\mu (x) \simeq  \exp(- \tfrac{\mu|x|}{\kappa_0}), \qquad |\nabla \eta_\mu| \lesssim \frac {\mu}{\kappa_0}\eta_\mu.
\end{equation}
Testing~\eqref{eq:correct-sed-mu} with $\eta_\mu^2  \varphi_{\mu,i}$, using the boundary conditions, using that $\eta_\mu$ is constant near the inclusions, and noting that $\D(\varphi_{\mu,i})=0$ in $\Ic$ implies $\Div(\varphi_{\mu,i})=\nabla\Div(\varphi_{\mu,i})=0$ in $\Ic$, we find
\begin{multline}\label{eq:test-eqn-phimu}
\mu\int_{\R^d\setminus\Ic}\eta_\mu^2|\varphi_{\mu,i}|^2
+2\int_{\R^d\setminus\Ic}\D(\eta_\mu^2\varphi_{\mu,i}):\D(\varphi_{\mu,i})
+\frac1\mu\int_{\R^d\setminus\Ic}\Div(\eta_\mu^2\varphi_{\mu,i})\,\Div(\varphi_{\mu,i})\\
=-\alpha e_i\cdot\int_{\R^d\setminus\Ic}\eta_\mu^2\varphi_{\mu,i}
+e_i\cdot\int_\Ic\eta_\mu^2\varphi_{\mu,i}.
\end{multline}
Rearranging the terms yields
\begin{multline}\label{e.rearranging}
\int_{\R^d \setminus \Ic}\eta_\mu^2\Big(\mu |\varphi_{\mu,i}|^2 +2 |\!\D( \varphi_{\mu,i})|^2+\frac1\mu|\Div(\varphi_{\mu,i})|^2\Big)\\
= e_i\cdot \int_{\R^d}(\mathds1_{\Ic}-\alpha \mathds1_{\R^d \setminus \Ic}) \eta_\mu^2 \varphi_{\mu,i} 
-4\int_{\R^d \setminus \Ic}  \eta_\mu ( \varphi_{\mu,i}\otimes \nabla \eta_\mu) : \D( \varphi_{\mu,i})
\\
-\frac2\mu\int_{\R^d\setminus\Ic}\eta_\mu(\varphi_{\mu,i}\cdot\nabla\eta_\mu)\,\Div(\varphi_{\mu,i}).
\end{multline}
By the properties~\eqref{e.exp-cutoff} of $\eta_\mu$, choosing $\kappa_0\gg1$ sufficiently large independently of~$\mu$, the last two terms in~\eqref{e.rearranging} can be absorbed into the left-hand side. For some constant $C\simeq1$, independent of~$\kappa_0$, we obtain
\begin{multline}\label{e.ener-estim-mu}
2\Big(1-\frac1{\kappa_0}\Big)\int_{\R^d\setminus\Ic}\eta_\mu^2|\!\D(\varphi_{\mu,i})|^2
+\frac1C\int_{\R^d\setminus\Ic}\eta_\mu^2\Big(\mu|\varphi_{\mu,i}|^2+\frac1{\mu}|\Div(\varphi_{\mu,i})|^2\Big)\\
\le e_i\cdot\int_{\R^d}(\mathds1_\Ic-\alpha\mathds1_{\R^d\setminus\Ic})\eta_\mu^2\varphi_{\mu,i}.
\end{multline}
Except in Step~5 below, where we shall use the flexibility in $\kappa_0$, we henceforth fix $1 \simeq \kappa_0 \ge 2$ sufficiently large for this estimate to hold.

Next, we observe that we can always control $\varphi_{\mu,i}$ inside the inclusions. For $x\in\Pc$, since~$\varphi_{\mu,i}$ is rigid in~$B(x)$, trace estimates and Korn's inequality yield
\begin{multline}\label{e.1.varphi-trou}
\|\varphi_{\mu,i}\|_{L^2(B(x))}
\lesssim \|\varphi_{\mu,i}\|_{H^{\frac12}(\partial B(x))}
\lesssim\|\varphi_{\mu,i}\|_{H^1(B_{1+\delta}(x) \setminus B(x))}\\
\lesssim\|\varphi_{\mu,i}\|_{L^2(B_{1+\delta}(x) \setminus B(x))}+\|\!\D(\varphi_{\mu,i})\|_{L^2(B_{1+\delta}(x) \setminus B(x))}.
\end{multline}
Using this to fill the holes in~\eqref{e.ener-estim-mu}, and recalling that we have $\D(\varphi_{\mu,i})=0$ and thus $\Div(\varphi_{\mu,i})=0$ in $\Ic$, we deduce
\begin{equation}\label{e.energy-mu+}
\int_{\R^d}\eta_\mu^2\Big(\mu |\varphi_{\mu,i}|^2 +|\!\D( \varphi_{\mu,i})|^2 +\frac1\mu|\Div(\varphi_{\mu,i})|^2\Big)
\,\lesssim \, e_i\cdot \int_{\R^d}(\mathds1_{\Ic}-\alpha \mathds1_{\R^d \setminus \Ic}) \eta_\mu^2 \varphi_{\mu,i}.
\end{equation}
Absorbing the occurrence of $\varphi_{\mu,i}$ on the right-hand side, this implies
\begin{equation}\label{e.energy-mu++}
\int_{\R^d} \eta_\mu^2\Big(\mu |\varphi_{\mu,i}|^2 +  |\!\D(\varphi_{\mu,i})|^2+\frac1\mu|\Div(\varphi_{\mu,i})|^2\Big)
\lesssim \, \frac1\mu  \int_{\R^d}\eta_\mu^2\,\lesssim \, \mu^{-d-1}.
\end{equation}
Since $\eta_\mu\gtrsim\mathds1_B$, and since the same estimates hold uniformly upon translating~$\eta_\mu$, we obtain
\begin{equation*}
\sup_{x\in\R^d}\int_{B(x)} \Big(\mu|\varphi_{\mu,i}|^2+ |\!\D(\varphi_{\mu,i})|^2+\frac1\mu|\Div(\varphi_{\mu,i})|^2\Big)\,\lesssim \, \mu^{-d-1}.
\end{equation*}
By a standard approximation argument, starting from solutions on large domains with homogeneous Dirichlet boundary conditions, we conclude that~\eqref{eq:correct-sed-mu} is well-posed for $\varphi_{\mu,i}$ in the space $H^1_{\uloc}(\R^d)$. Moreover, by uniqueness, the random field $\varphi_{\mu,i}$ is stationary; in particular,
\[\expec{\nabla \varphi_{\mu,i}}=0.\]
Taking expectations in~\eqref{e.energy-mu++}, using stationarity and the properties of~$\eta_\mu$, we immediately deduce
\begin{equation}\label{e.energy-mu+++}
\expec{|\varphi_{\mu,i}|^2} \lesssim \mu^{-2},\qquad\expec{|\nabla \varphi_{\mu,i}|^2} \lesssim \mu^{-1}.
\end{equation}
We shall also use the following centering property:
\begin{equation}\label{eq:aver-Icc}
\expecm{\varphi_{\mu,i} \mathds1_{\R^d\setminus \Ic}}=0.
\end{equation}
To prove it, let $g_R:\R^d \to [0,1]$ be such that $g_R\equiv 1$ in $Q_{R-4}$, $g_R\equiv 0$ in $\R^d \setminus Q_R$, $|\nabla g_R|\lesssim 1$, and $g_R$ is constant in a neighborhood of each inclusion $B(x)$, $x\in\Pc$. Testing~\eqref{eq:correct-sed-mu} with $\xi g_R$, for an arbitrary $\xi\in\R^d$, we get
\begin{multline*}
\int_{Q_R \setminus \Ic} \mu g_R \xi\cdot \varphi_{\mu,i} + \int_{Q_R}\xi\otimes\nabla g_R:2\D(\varphi_{\mu,i})+\frac1\mu\int_{Q_R}(\xi\cdot\nabla g_R)\Div(\varphi_{\mu,i})\\
= -\xi\cdot e_i\int_{Q_R}(\alpha \mathds1_{\R^d \setminus \Ic}- \mathds1_{\Ic}) g_R.
\end{multline*}
Dividing both sides by $R^d$, passing to the limit $R\uparrow\infty$, and using the ergodic theorem together with the stationarity of $\varphi_{\mu,i}$, we obtain
\[\mu\xi\cdot\expecm{\varphi_{\mu,i}\mathds1_{\R^d\setminus\Ic}}=-\xi\cdot e_i\,\expecm{\alpha\mathds1_{\R^d\setminus\Ic}-\mathds1_{\Ic}}.\]
The right-hand side vanishes by the definition~\eqref{eq:def-alpha} of~$\alpha$, and the claim~\eqref{eq:aver-Icc} follows since~$\xi$ is arbitrary.

The bound~\eqref{e.energy-mu+++}, however, is not uniform in~$\mu$. We shall exploit stochastic cancellations to obtain estimates that survive as $\mu\downarrow0$.

\medskip
\step2 Uniform bounds for~\eqref{eq:correct-sed-mu}.\\
We claim that for all $0<\mu\le 1$,
\begin{equation}\label{e.unif-bd-varphimu}
\expec{\mu|\varphi_{\mu,i}|^2+|\!\D(\varphi_{\mu,i})|^2+\frac1\mu|\Div(\varphi_{\mu,i})|^2} \,\lesssim\,1.
\end{equation}
This nontrivial estimate is the first instance in the paper of the general renormalization strategy described in Section~\ref{sec:infra-red-renorm}. At this stage, the precise value of the counterterm $-\alpha e_i$ is crucial. The auxiliary field $w_\mu$ introduced below plays the role of the linearization of~$\varphi_{\mu,i}$ with respect to the point process $\Pc$, which captures the leading infrared-divergent contribution. The bounds~\eqref{e.estimwmu} and~\eqref{e.estimgradwmu} on $w_\mu$ are obtained from stochastic cancellations that crucially use the explicit counterterm, and are subsequently transferred to $\varphi_{\mu,i}$ through the PDE estimate~\eqref{e.estim-varphimu}.

We now proceed with the proof of \eqref{e.unif-bd-varphimu}.
Taking the expectation in~\eqref{e.energy-mu+} and using the stationarity of $\varphi_{\mu,i}$, we get
\begin{equation}\label{eq:e.energy-mu+-exp}
\expec{\mu|\varphi_{\mu,i}|^2+|\!\D(\varphi_{\mu,i})|^2+\frac1\mu|\Div(\varphi_{\mu,i})|^2}\int_{\R^d}\eta_\mu^2
\lesssim e_i\cdot \expec{\int_{\R^d}(\mathds1_{\Ic}-\alpha \mathds1_{\R^d \setminus \Ic}) \eta_\mu^2 \varphi_{\mu,i}}.
\end{equation}
To make the cancellations explicit, we introduce the auxiliary field $w_\mu$ as the unique solution in $H^1_\uloc(\R^d)$ of 
\[\mu w_\mu - \triangle w_\mu = \mathds1_{\Ic}-\alpha \mathds1_{\R^d \setminus \Ic},\]
which defines a stationary random field. Since the choice~\eqref{eq:def-alpha} of $\alpha$ gives
\[\E[\mathds1_\Ic-\alpha\mathds1_{\R^d\setminus\Ic}]=0,\]
we have $\E[w_\mu]=0$.
In terms of $w_\mu$, the right-hand side of~\eqref{eq:e.energy-mu+-exp} can be rewritten as
\begin{equation*}
\expec{\int_{\R^d}(\mathds1_{\Ic}-\alpha \mathds1_{\R^d \setminus \Ic}) \eta_\mu^2 \varphi_{\mu,i}}\,
=\, \mu \expec{\int_{\R^d} w_\mu \eta_\mu^2  \varphi_{\mu,i}}
+ \expec{\int_{\R^d} \nabla w_\mu \cdot \nabla (\eta_\mu^2 \varphi_{\mu,i})}.
\end{equation*}
For the first term, by the Cauchy-Schwarz inequality, stationarity of $w_\mu$ and $\varphi_{\mu,i}$, and the fact that $\E[w_\mu]=0$, we have
\[\mu
\bigg|\expec{\int_{\R^d} w_\mu \eta_\mu^2  \varphi_{\mu,i} }\bigg|
\,\le\, \big(\mu \var{w_\mu}\big)^\frac12\, \expec{\mu |\varphi_{\mu,i}|^2}^\frac12 \int_{\R^d} \eta_\mu^2.\]
For the second term, we expand
\[\int_{\R^d}  \nabla w_\mu \cdot \nabla (\eta_\mu^2 \varphi_{\mu,i}) 
\,=\, \int_{\R^d} \eta_\mu^2 \nabla w_\mu \cdot \nabla \varphi_{\mu,i}
+ 2\int_{\R^d} \eta_\mu \varphi_{\mu,i} (\nabla w_\mu \cdot \nabla \eta_\mu),\]
and we use~\eqref{e.exp-cutoff} to get
\begin{equation*}
\bigg|\expec{\int_{\R^d}  \nabla w_\mu \cdot \nabla (\eta_\mu^2 \varphi_{\mu,i} )}\bigg|
\,
\lesssim \, \expec{|\nabla w_\mu|^2}^\frac12 \expec{\mu^2 |\varphi_{\mu,i}|^2+|\nabla \varphi_{\mu,i}|^2}^\frac12 \int_{\R^d} \eta_\mu^2 .
\end{equation*}
Returning to~\eqref{eq:e.energy-mu+-exp}, using Young's inequality and the stationary Korn identity
\[\E[|\nabla\varphi_{\mu,i}|^2]\,\le\,\E[|\nabla\varphi_{\mu,i}|^2]+\E[|\Div(\varphi_{\mu,i})|^2]\,=\,2\E[|\!\D(\varphi_{\mu,i})|^2],\]
we infer 
\begin{equation}\label{e.estim-varphimu}
\expec{\mu|\varphi_{\mu,i}|^2+|\nabla  \varphi_{\mu,i}|^2+\frac1\mu|\Div(\varphi_{\mu,i})|^2}
\lesssim\, \mu \var{w_\mu}+\expec{|\nabla w_\mu|^2}.
\end{equation}
It remains to estimate the right-hand side.
Let $G_\mu$ denote the massive Green function, or Yukawa kernel, solution of 
\begin{equation}\label{eq:Greenfct-mass}
\mu G_\mu-\triangle G_\mu \,=\, \delta_0\qquad\text{in $\R^d$}.
\end{equation}
It satisfies the pointwise bounds, for $d>2$,
\begin{equation}\label{e.mass-green}
0\le G_\mu (x) \lesssim |x|^{2-d} e^{-c\sqrt \mu |x|}, \quad |\nabla G_\mu (x)|\lesssim |x|^{1-d}e^{-c\sqrt \mu |x|}.
\end{equation}
By the Green representation formula,
\[
w_\mu  = G_\mu \star (\mathds1_{\Ic}-\alpha \mathds1_{\R^d \setminus \Ic}).
\]
Since
\[\mathds1_{\Ic}-\alpha\mathds1_{\R^d\setminus\Ic}=(1+\alpha)\mathds1_{\Ic}-\alpha,\]
and since $G_\mu\star\alpha=\alpha\mu^{-1}$ is constant,
we have
\begin{eqnarray*}
 \var{w_\mu}&=&(1+\alpha)^2\var{G_\mu \star \mathds1_{\Ic}},\\
 \expec{|\nabla w_\mu|^2}&=&(1+\alpha)^2\expec{|\nabla G_\mu \star \mathds1_{\Ic}|^2}.
\end{eqnarray*}
We estimate these two terms separately, and start with the variance. 
In terms of the intensity $\lambda$ and of the pair correlation function $h_{2}$, the moment formula~\eqref{eq:multi-density} yields
\begin{eqnarray*}
\var{G_\mu \star \mathds1_{\Ic}}
&=&\Var\bigg[\sum_{x\in\Pc} G_\mu \star \mathds1_{B}(x)\bigg]\\
&=&\lambda\int_{\R^d}(G_\mu \star \mathds1_{B})^2
+\iint_{\R^d\times\R^d}(G_\mu \star \mathds1_{B})(x)\,(G_\mu \star \mathds1_{B})(y)\,h_2(x,y)\,dxdy.
\end{eqnarray*}
By the pointwise decay~\eqref{e.mass-green} of the Green function, this implies
\begin{multline*}
\var{G_\mu \star \mathds1_{\Ic}}\\[-1mm]
\,\lesssim\,\lambda\int_{\R^d}  \langle x\rangle^{2(2-d)}e^{-c\sqrt\mu|x|}\,dx
+\iint_{\R^d\times \R^d}  \langle x\rangle^{2-d}\langle y\rangle^{2-d}e^{-c\sqrt\mu(|x|+|y|)}\,|h_2(x,y)|\,dxdy,
\end{multline*}
and thus, by the decay~\eqref{eq:decay-h2} of pair correlations, for $d>2$,
\begin{equation}\label{e.estimwmu}
\mu \var{w_\mu} \,\lesssim\,\mu \var{G_\mu\star\mathds1_\Ic} \,\lesssim\,1.
\end{equation}
We turn to the gradient term.
In terms of the pair density function~$f_{2}$ of $\Pc$, the moment formula~\eqref{eq:multi-density} now yields
\begin{multline*}
\expec{|\nabla G_\mu \star \mathds1_{\Ic}|^2}
\,=\,\E\bigg[\Big|\sum_{x\in\Pc} \nabla G_\mu\star \mathds1_{B}(x)\Big|^2\bigg]\\
\,=\,\lambda \int_{\R^d}|\nabla G_\mu\star \mathds1_B|^2
+\iint_{\R^d\times \R^d}\nabla (G_\mu \star \mathds1_B)(x)\cdot \nabla (G_\mu\star \mathds1_B)(y)\,f_{2}(x,y)\,dxdy.
\end{multline*}
Since $\int_{\R^d}\nabla G_\mu=0$, we may replace $f_{2}$ by the pair correlation function $h_{2}=f_{2}-\lambda^2$ in the second term. Using~\eqref{e.mass-green} and the decay~\eqref{eq:decay-h2} of pair correlations, we then get, for $d>2$,
\begin{equation}\label{e.estimgradwmu}
\expec{|\nabla w_\mu|^2}
\,\lesssim\,1.
\end{equation}
Combining this with~\eqref{e.estim-varphimu} and~\eqref{e.estimwmu}, we obtain the uniform bound~\eqref{e.unif-bd-varphimu}.

\medskip
\step3 Well-posedness of~\eqref{eq:correct-sed}.\\
This step relies on energy estimates and compactness.
By the uniform bounds~\eqref{e.unif-bd-varphimu}, after extraction of a subsequence, we have $\nabla\varphi_{\mu,i}\cvf\nabla\varphi_i$ in $\Ld^2(\Omega;\Ld^2_\loc(\R^d))$ as {$\mu \downarrow 0$}, for some stationary gradient random field $\nabla\varphi_i$ satisfying
\[\expec{\nabla \varphi_i}=0, \qquad\expec{|\nabla  \varphi_i|^2}\lesssim1.\]
Moreover,
\[
\D(\varphi_i)=0\quad\text{in $\Ic$}.
\]
As~\eqref{e.unif-bd-varphimu} also implies
$\E[{|\Div(\varphi_{\mu,i})|^2}]\lesssim \mu$,
we get $\Div(\varphi_i)=0$  in~$\R^d$.
We choose a representative
$\varphi_i\in \Ld^2(\Omega;H^1_\loc(\R^d)^d)$
uniquely by imposing the anchoring condition $\int_B\varphi_i=0$.

We now show that~$\varphi_i$ is a weak solution of~\eqref{eq:correct-sed}.
Let $\psi\in C^\infty_c(\R^d)^d$ satisfy $\Div(\psi)=0$ and $\D(\psi)|_{\Ic}=0$, and let $\chi$ be a square-integrable random variable. Testing~\eqref{eq:correct-sed-mu} with $\psi\chi$, similarly as in~\eqref{eq:test-eqn-phimu}, we find
\[
\E\bigg[\mu \chi \int_{\R^d\setminus\Ic} \psi\cdot \varphi_{\mu,i}\bigg] + \E\bigg[2\chi \int_{\R^d \setminus \Ic} \D(\psi): \D(\varphi_{\mu,i})\bigg]
=e_i\cdot \E\bigg[{\chi \Big(\int_{\Ic}\psi-\alpha\int_{\R^d\setminus\Ic}\psi\Big)}\bigg].
\]
The first term in the left-hand side vanishes as $\mu\downarrow0$, since by~\eqref{e.unif-bd-varphimu} we have
\[
\bigg|\E\bigg[\mu \chi \int_{\R^d\setminus\Ic} \psi\cdot \varphi_{\mu,i}\bigg]\bigg|
\,\le\, \mu^\frac12\,\expec{\mu |\varphi_{\mu,i}|^2}^\frac12 \expec{\chi^2} ^\frac12 \int_{\R^d}|\psi| \,\xrightarrow{\mu \downarrow 0}\, 0.
\]
Passing to the limit in the weak formulation and using the arbitrariness of $\chi$, we deduce that almost surely, for all test functions $\psi\in C^\infty_c(\R^d)^d$ with $\Div(\psi)=0$ and $\D(\psi)|_{\Ic}=0$,
\[2\int_{\R^d}\D( \psi) : \D( \varphi_i)= e_i\cdot\Big(\int_{\Ic}\psi-\alpha\int_{\R^d\setminus\Ic}\psi\Big).\]
Thus~$\varphi_i$ is an almost sure weak solution of the infinite-volume system~\eqref{eq:correct-sed}.
The stationarity of $\nabla\varphi_i$, the condition $\E[\nabla\varphi_i]=0$, and the anchoring $\int_B\varphi_i=0$ imply that~$\varphi_i$ is sublinear at infinity in quadratic average.
It is then standard to check that such an almost sure weak solution~$\varphi_i$ of~\eqref{eq:correct-sed} is necessarily unique.

Finally, arguing as in~\cite[Step~4 of the proof of Proposition~2.1]{DG-20}, the associated pressure~$\pi_i$ can be uniquely chosen as a stationary random field satisfying
\[\E[\pi_i\mathds1_{\R^d\setminus\Ic}]=0,\qquad \E[|\pi_i|^2\mathds1_{\R^d\setminus\Ic}]<\infty.\]
The regularity theory for the steady Stokes equations then ensures that $(\varphi_i,\pi_i)$ is in fact almost surely a classical solution of~\eqref{eq:correct-sed}, with the boundary conditions satisfied in a pointwise sense.

\medskip
\step4 Proof of~\eqref{eq:def-V}.\\
Let $\chi\in C^\infty_c(\R^d)$ be nonnegative with $\int_{\R^d} \chi =1$, and set $\chi_r:=\chi(\tfrac\cdot r)$.
By the $\delta$-hardcore assumption, for all $r\ge1$, we can construct a smooth cut-off function~$\tilde\chi_r$, supported in $\supp(\chi_r)+4B$, that is constant in a neighborhood of each inclusion $B(x)$, $x\in\Pc$, and satisfies
\begin{equation}\label{e.4-chir}
|\tilde\chi_r-\chi_r|+|\nabla\tilde\chi_r|\lesssim_\chi r^{-1},\qquad\tilde\chi_r|_{\Pc}=\chi_r|_\Pc.
\end{equation}
Given $E\in\R^d$, testing the Stokes system~\eqref{eq:correct-sed} for $\varphi_j$ with $\tilde\chi_r(\varphi_i-E)$, and using the properties of $\tilde\chi_r$, we find
\begin{multline*}
\int_{\R^d}\tilde\chi_r\nabla\varphi_i:\nabla\varphi_j
\,=\,
(\alpha+1) e_j\cdot\sum_{x\in\Pc}\int_{B(x)}\tilde\chi_r(\varphi_i-E)\\
-\int_{\R^d}(\varphi_i-E)\otimes \nabla\tilde\chi_r:(\nabla\varphi_j-\pi_j\Id\mathds1_{\R^d\setminus\Ic})
-\alpha e_j\cdot\int_{\R^d}\tilde\chi_r(\varphi_i-E).
\end{multline*}
Choose
\[E:=\fint_{\tilde\chi_r}\varphi_i:=\frac{\int_{\R^d}\tilde\chi_r\varphi_i}{\int_{\R^d}\tilde\chi_r}.\]
Then the last right-hand side term vanishes. Using~\eqref{e.4-chir} to estimate the remaining error term, we obtain
\begin{multline*}
r^{-d}\bigg|\int_{\R^d}\tilde\chi_r\nabla\varphi_i:\nabla\varphi_j
-(\alpha+1) e_j\cdot\sum_{x\in\Pc}\chi_r(x)\int_{B(x)}\Big(\varphi_i-\fint_{\tilde\chi_r}\varphi_i\Big)\bigg|\\
\,\lesssim\,
\Big(\int_{\supp\chi+4B}|r^{-1}\varphi_i(r\cdot)|^2\Big)^\frac12\Big(\int_{\supp\chi+4B}|\nabla\varphi_j(r\cdot)|^2+|(\pi_j\mathds1_{\R^d\setminus\Ic})(r\cdot)|^2\Big)^\frac12.
\end{multline*}
By sublinearity of $\varphi_i$ in $\Ld^2$, and since $\nabla\varphi_j$ and $\pi_j\mathds1_{\R^d\setminus\Ic}$ are stationary with finite second moments, the right-hand side converges to $0$ almost surely as $r\uparrow\infty$. In addition, using~\eqref{e.4-chir} and the joint stationarity of $\nabla\varphi_i$ and $\nabla\varphi_j$, we have almost surely
\[\fint_{\tilde\chi_r}\nabla\varphi_i:\nabla\varphi_j\,\to\,\E[\nabla\varphi_i:\nabla\varphi_j],\]
and, by sublinearity of $\varphi_i$,
\[\fint_{\tilde\chi_r}\varphi_i-\fint_{\chi_r}\varphi_i\,\to\,0.\]
Inserting these limits into the previous estimate gives almost surely
\begin{equation*}
\lim_{r\uparrow\infty}~e_j\cdot\frac1{\int_{\R^d}\chi_r}\sum_{x\in\Pc}\chi_r(x)\int_{B(x)}\Big(\varphi_i-\fint_{\chi_r}\varphi_i\Big)
~=~\frac1{\alpha+1}\E[\nabla\varphi_i:\nabla\varphi_j].
\end{equation*}
Since
\[\frac1{\alpha+1}=1-\lambda|B|,\qquad\lim_{r\uparrow\infty}\frac1{\int_{\R^d}\chi_r}\sum_{x\in\Pc}\chi_r(x) = \lambda,\]
this can be rewritten as
\begin{equation*}
\lim_{r\uparrow\infty}~e_j\cdot\frac1{\sum_{x\in\Pc}\chi_r(x)}\sum_{x\in\Pc}\chi_r(x)\fint_{B(x)}\Big(\varphi_i-\fint_{\chi_r}\varphi_i\Big)
~=~\frac{1-\lambda|B|}{\lambda|B|}\E[\nabla\varphi_i:\nabla\varphi_j].
\end{equation*}
The claim~\eqref{eq:def-V} follows from the relation~\eqref{e.uinfty-scaling} between $u_{\infty,i}$ and $\varphi_i$.
 
\medskip
\step5 Convergence of the infrared-regularized approximation.\\
In this last step, we record two convergence properties that will be used below: we first claim that
\begin{equation}\label{eq:massive-speed}
\lim_{\mu \downarrow 0}e_j\cdot \expecm{ (\varphi_{\mu,i}-\expec{\varphi_{\mu,i}}) \mathds1_{\Ic}}=\frac{\lambda|B|}{1-\lambda|B|}\bar V_{\infty,ij},
\end{equation}
and then that
\begin{equation}\label{eq:conv-mass}
\lim_{\mu \downarrow 0} \expec{\mu|\varphi_{\mu,i}|^2+|\!\D( \varphi_{\mu,i})-\D( \varphi_i)|^2+\frac1\mu|\Div(\varphi_{\mu,i})|^2}=0.
\end{equation}
We start with the proof of~\eqref{eq:massive-speed}.
To this end, let us test the equation~\eqref{eq:correct-sed} for $\varphi_j$ with~$\varphi_{\mu,i}$. More precisely, we test it with $\tilde\chi_r\varphi_{\mu,i}$, let $r\uparrow\infty$, and use the stationarity of~$\varphi_{\mu,i}$, $\nabla\varphi_j$, and $\pi_j\mathds1_{\R^d\setminus\Ic}$. This leads to
\[\expecm{\nabla\varphi_{\mu,i}:\nabla\varphi_j}
\,=\,\expecm{\Div(\varphi_{\mu,i})\pi_j\mathds1_{\R^d\setminus\Ic}}
+e_j\cdot\expecm{\varphi_{\mu,i}(\mathds1_{\Ic}-\alpha\mathds1_{\R^d\setminus\Ic})},\]
and thus, by~\eqref{eq:aver-Icc},
\[\expecm{\nabla\varphi_{\mu,i}:\nabla\varphi_j}
\,=\,\expecm{\Div(\varphi_{\mu,i})\pi_j\mathds1_{\R^d\setminus\Ic}}
+e_j\cdot\expecm{\varphi_{\mu,i}\mathds1_{\Ic}}.\]
Passing to the limit $\mu\downarrow0$ and using the weak convergence $\nabla \varphi_{\mu,i}\cvf\nabla \varphi_i$ and the strong convergence $\Div( \varphi_{\mu,i})\to 0$, cf.~Step~3, we get
\begin{equation}\label{e.5.convphi}
\lim_{\mu \downarrow 0}e_j\cdot \expecm{ \varphi_{\mu,i} \mathds1_{\Ic}}=\expecm{\nabla \varphi_i : \nabla \varphi_j}=2 \expec{\D(\varphi_i):\D(\varphi_j)}.
\end{equation}
In the last identity we used $\Div(\varphi_i)=0$.
By the relation between $u_{\infty,i}$ and $\varphi_i$, cf.~\eqref{e.uinfty-scaling}, and by the definition of $\bar V_{\infty,ij}$, cf.~\eqref{eq:mean-velocity-ws}, this entails
\begin{equation*}
\lim_{\mu \downarrow 0}e_j\cdot \expecm{ \varphi_{\mu,i} \mathds1_{\Ic}}=\frac{\lambda|B|}{(1-\lambda|B|)^2}\bar V_{\infty,ij}.
\end{equation*}
Since~\eqref{eq:aver-Icc} gives
\[
 \expecm{ (\varphi_{\mu,i}-\expec{\varphi_{\mu,i}}) \mathds1_{\Ic}}=\expec{\varphi_{\mu,i}\mathds1_{\Ic}}-\expec{\varphi_{\mu,i}}\lambda|B|= (1-\lambda |B|)\expec{\varphi_{\mu,i}\mathds1_{\Ic}},
\]
the claim~\eqref{eq:massive-speed} follows.

It remains to prove~\eqref{eq:conv-mass}.
Taking expectations in~\eqref{e.ener-estim-mu} and using stationarity, we obtain for all sufficiently large $\kappa_0 \gg 1$,
\begin{multline*}
2\Big(1-\frac1{\kappa_0}\Big)\expec{|\!\D(\varphi_{\mu,i})|^2}+\frac1C\expec{\mu\mathds1_{\R^d \setminus \Ic} |\varphi_{\mu,i}|^2+\frac1{\mu}|\Div(\varphi_{\mu,i})|^2}\\
\,\le  \,  e_i\cdot \expec{(\mathds1_{\Ic}-\alpha \mathds1_{\R^d \setminus \Ic})  \varphi_{\mu,i}}= e_i\cdot \expec{\varphi_{\mu,i}\mathds1_{\Ic}},
\end{multline*}
where we recalled that $\D(\varphi_{\mu,i})=0$ and $\Div(\varphi_{\mu,i})=0$ in~$\Ic$ and where we used~\eqref{eq:aver-Icc} in the last identity.
Since the constant $C$ in this estimate is independent of $\kappa_0$, we may now let~$\kappa_0 \uparrow\infty$ and obtain
\begin{equation*}
2\expec{ |\!\D(\varphi_{\mu,i})|^2}+\frac1C\expec{\mu\mathds1_{\R^d \setminus \Ic} |\varphi_{\mu,i}|^2+\frac1{\mu}|\Div(\varphi_{\mu,i})|^2}
\,\le\,e_i\cdot \expec{\varphi_{\mu,i}\mathds1_{\Ic}}.
\end{equation*}
Using the identity
\[
\expec{|\!\D (\varphi_{\mu,i})-\D( \varphi_i)|^2} = \expec{|\!\D( \varphi_{\mu,i})|^2}
+ \expec{|\!\D( \varphi_{i})|^2}-2 \expec{\D( \varphi_{\mu,i}):\D( \varphi_i)},
\]
we deduce
\begin{multline*}
2\expec{|\!\D (\varphi_{\mu,i})-\D( \varphi_i)|^2}+\frac1C\expec{\mu\mathds1_{\R^d \setminus \Ic} |\varphi_{\mu,i}|^2+\frac1{\mu}|\Div(\varphi_{\mu,i})|^2}
\\
\le  \, e_i\cdot \expec{\varphi_{\mu,i}\mathds1_{\Ic}}+2\expec{|\!\D( \varphi_{i})|^2}-4 \expec{\D( \varphi_{\mu,i}):\D( \varphi_i)}.
\end{multline*}
By the weak convergence $\nabla \varphi_{\mu,i} \cvf \nabla \varphi_i$ and by~\eqref{e.5.convphi}, the right-hand side converges to $0$. Hence,
\begin{equation*}
\lim_{\mu \downarrow 0}\expec{\mu\mathds1_{\R^d \setminus \Ic} |\varphi_{\mu,i}|^2+|\!\D (\varphi_{\mu,i})-\D( \varphi_i)|^2+\frac1{\mu}|\Div(\varphi_{\mu,i})|^2}\,=\,0.
\end{equation*}
Finally, the missing contribution of $\expec{\mu\mathds1_{\Ic}|\varphi_{\mu,i}|^2}$ is controlled by the trace estimate~\eqref{e.1.varphi-trou}, together with the preceding convergence and the uniform bound~\eqref{e.unif-bd-varphimu}. This concludes the proof of~\eqref{eq:conv-mass}.\qed


\section{Large-container limit}\label{sec:boundary}

This section is devoted to the proof of Theorem~\ref{th:mean-velocityR}. The argument uses the properties of the infinite-volume construction of Theorem~\ref{th:mean-velocity}, together with the large-scale Lipschitz regularity theory that we developed in~\cite{DG-21a} for the Stokes system with rigid inclusions.
In substance, Theorem~\ref{th:mean-velocityR} is a qualitative homogenization result for sedimenting suspensions, extending~\cite[Theorem~3]{DG-22} to the present setting. The main difference is that the proof no longer relies on the strong quantitative ergodicity assumptions and periodization in law imposed in~\cite{DG-22}; it only uses qualitative ergodicity, together with the minimal decay assumption on the pair correlation function~$h_2$.

We consider several finite-container geometries. We first treat the infinite cylindrical geometry of Theorem~\ref{th:mean-velocityR}. We then explain, in Sections~\ref{sec:snow-globe} and~\ref{sec:finite-cyl}, the minor modifications needed to cover the additional geometries described in Remark~\ref{rem:bndary}.

\subsection{Proof of Theorem~\ref{th:mean-velocityR} in an infinite cylinder}\label{sec:inf-cyl}
Let $U=U' \times \R$, where $U'\subset\R^{d-1}$ is a bounded Lipschitz domain. We set $U_R:=RU=U'_R\times\R$ with $U_R':=RU'$, and we decompose coordinates as~$x=(x',z)\in \R^{d-1}\times\R$.
The proof is divided into three steps.
We first show that the semi-infinite container problem~\eqref{eq:Dir} is well-posed on $U_R$ with the homogeneous Dirichlet condition $u_R=0$ on $\partial U_R$. We then compare this semi-infinite volume solution $u_R$ with the infinite-volume solution $u_\infty$ constructed in Theorem~\ref{th:mean-velocity}, by estimating an exponentially weighted $L^2$-norm of $\nabla u_R-\nabla u_\infty$.
Finally, we use the large-scale regularity theory of~\cite{DG-21a} to pass from this energy comparison to the convergence of the local averaged settling speed.

Throughout the section, for any $1\le i\le d$, we consider the problem~\eqref{eq:Dir} with forcing direction $e=e_i$ and denote the corresponding solution by $(u_{R,i},p_{R,i})$ whenever it exists.

\medskip
\step1 Deterministic well-posedness of~\eqref{eq:Dir}.\\
If $(u_{R,i},p_{R,i})$ solves~\eqref{eq:Dir} with $e=e_i$, after implicitly extending $p_{R,i}$ by zero in the inclusions,  it satisfies the following relation in the weak sense in the whole $U_R$,
\begin{equation}\label{e.cyl1}
-\triangle u_{R,i} +\nabla p_{R,i} =
-\sum_{x \in \Pc(U_R)}  \delta_{\partial B(x)} \sigma_{R,i} \nu,
\end{equation}
with the short-hand notation $\sigma_{R,i}=\sigma(u_{R,i},p_{R,i})$.
This formulation is convenient to obtain a priori estimates.
As for the well-posedness of~\eqref{eq:correct-sed-mu}, we use an exponential cut-off.
For some $\kappa_0\ge 1$ to be chosen below, we consider this time $\eta_R(x):=\exp(-\frac1{\kappa_0R} |z|)$,
and we modify it into a cut-off $\tilde\eta_R$ that is constant in a neighborhood of each inclusion $B(x)$, $x\in\Pc(U_R)$, and satisfies
\[\tilde\eta_R(x)\simeq\eta_R(x),\qquad|\nabla \tilde \eta_R(x)| \lesssim \frac1{\kappa_0R} \eta_R(x).\]
Testing~\eqref{e.cyl1} with $\tilde \eta_R^2 u_{R,i}$, and using boundary conditions for $u_{R,i}$ and properties of $\tilde\eta_R$, we find for any $\gamma \in\R$,
\begin{multline*}
\int_{U_R}\tilde \eta_R^2 |\nabla u_{R,i}|^2 +2 \tilde\eta_R (u_{R,i} \otimes \nabla \tilde\eta_R):\nabla u_{R,i}
-2\tilde\eta_R (p_{R,i}-\gamma)(u_{R,i}\cdot\nabla \tilde\eta_R) 
\\
=-\sum_{x\in\Pc(U_R)} \int_{\partial B(x)}\tilde \eta_R^2 u_{R,i}\cdot\sigma_{R,i}\nu
=\sum_{x\in\Pc(U_R)} e_i\cdot \int_{B(x)}\tilde \eta_R^2 u_{R,i},
\end{multline*}
which entails by the properties of $\tilde \eta_R$ and by Young's inequality,
\[
\int_{U_R} \eta_R^2 |\nabla u_{R,i}|^2 \,\lesssim\, \frac1{\kappa_0R^2} \int_{U_R} \eta_R^2 | u_{R,i}|^2  + \frac1{\kappa_0}  \int_{U_R\setminus\Ic(U_R)} \eta_R^2 (p_{R,i}-\gamma)^2+ \kappa_0R^2 \int_{U_R} \eta_R^2.
\]
Note that for any non-negative weight $w:\R\to\R_+$, we have the following Poincar\'e inequality on $U_R=U_R'\times\R$, for all~$\psi \in H^1_0(U_R)$,
\begin{equation}\label{e.Poincar-cyl}
\int_{U_R'\times\R} w(z) |\psi(x',z)|^2 dx'dz\,\lesssim\, R^2 \int_{U_R'\times\R} w(z) |\nabla \psi(x',z)|^2dx'dz.
\end{equation}
Using this with $w=\eta_R$ and $\psi=u_{R,i}$, and choosing $\kappa_0\simeq1$ large enough, it allows us to absorb the first right-hand side term in the above estimate, to the effect of
\begin{equation}\label{e.cyl-1}
\int_{U_R} \eta_R^2 |\nabla u_{R,i}|^2 \,\lesssim\,\frac1{\kappa_0}\int_{U_R\setminus\Ic(U_R)}\eta_R^2 (p_{R,i}-\gamma)^2+\kappa_0R^2 \int_{U_R} \eta_R^2.
\end{equation}
It remains to absorb the pressure term.
Choose
\[\gamma:=\frac1{\int_{U_R\setminus\Ic(U_R)} \tilde\eta_R}\int_{U_R\setminus\Ic(U_R)} \tilde\eta_R p_{R,i},\]
to the effect that $\int_{U_R}\tilde\eta_R(p_{R,i}-\gamma)\mathds1_{U_R\setminus\Ic(U_R)}=0$.
Using the Bogovskii operator as in~\cite[Step~4.2, Proof of Proposition~2.1]{DG-20}, we can construct a vector field $\zeta_R \in H^1_0(U_R)^d$ that is constant in each inclusion $B(x)$, $x\in\Pc(U_R)$, such that 
\[\Div(\zeta_R)=\tilde\eta_R (p_{R,i}-\gamma)\mathds1_{U_R\setminus\Ic(U_R)}\qquad\text{in $U_R$},\]
and
\begin{equation}\label{e.bogo0}
\int_{U_R} |\nabla \zeta_R|^2 \lesssim \int_{U_R\setminus\Ic(U_R)} \tilde\eta_R^2 (p_{R,i}-\gamma)^2.
\end{equation}
Testing~\eqref{e.cyl1} with $\tilde\eta_R \zeta_R$, using that the latter is constant in each inclusion, and using the boundary conditions for $u_{R,i}$, we find 
\begin{multline}\label{e.cyl9}
\int_{U_R}\nabla( \tilde\eta_R  \zeta_R):\nabla u_{R,i}
-\int_{U_R\setminus\Ic(U_R)}\tilde\eta_R^2(p_{R,i}-\gamma)^2
-\int_{U_R}(\zeta_R\cdot\nabla\tilde\eta_R)(p_{R,i}-\gamma)\\
=\sum_{x\in\Pc(U_R)} e_i\cdot \int_{B(x)}\tilde \eta_R \zeta_R.
\end{multline}
By the properties of $\tilde \eta_R$ and by Young's inequality, this entails 
\begin{multline*}
\int_{U_R\setminus\Ic(U_R)}\eta_R^2  (p_{R,i}-\gamma)^2 \, \lesssim \, 
\Big( \int_{U_R} |\nabla \zeta_R|^2\Big)^\frac12 \Big( \int_{U_R} \eta_R^2 |\nabla u_{R,i}|^2\Big)^\frac12
+\frac1{\kappa_0^2R^2}\int_{U_R}|\zeta_R|^2\\
+\Big(\frac1{R^2} \int_{U_R}  |\zeta_R|^2\Big)^\frac12 \Big(\int_{U_R} \eta_R^2| \nabla u_{R,i}|^2\Big)^\frac12
+\Big(R^2\int_{U_R} \eta_R^2\Big)^\frac12 \Big(\frac1{R^2} \int_{U_R}  |\zeta_R|^2\Big)^\frac12.
\end{multline*}
Appealing to Poincar\'e's inequality~\eqref{e.Poincar-cyl} for $\zeta_R \in H^1_0(U_R)^d$ with $w\equiv 1$, recalling~\eqref{e.bogo0}, and choosing $\kappa_0\simeq1$ large enough, this implies
\begin{equation}\label{e.cyl10}
\int_{U_R\setminus\Ic(U_R)}\eta_R^2  (p_{R,i}-\gamma)^2 \, \lesssim \, 
  \int_{U_R} \eta_R^2 |\nabla u_{R,i}|^2 +  R^2\int_{U_R} \eta_R^2.
\end{equation}
Inserting this into~\eqref{e.cyl-1}, and choosing $\kappa_0\simeq1$ large enough, we obtain
\[\int_{U_R} \eta_R^2 |\nabla u_{R,i}|^2 \lesssim  R^2\int_{U_R} \eta_R^2.\]
Hence, as $\eta_R\gtrsim\mathds1_B$, noting that the same estimates hold uniformly upon translating $\eta_R$,
\[\sup_{x\in \R^d}\int_{U_R\cap B(x)}|\nabla u_{R,i}|^2 \lesssim R^{d+2}.\]
By a standard approximation argument starting from a solution on a finite cylinder with homogeneous Dirichlet boundary conditions, we conclude that~\eqref{eq:Dir} is well-posed for $u_{R,i}$ in the space $H^1_{0,\uloc}(U_R)$ of uniformly locally $H^1$ vector fields with zero trace on $\partial U_R$.

\medskip

\step2 Proof that 
\begin{equation}\label{e.conv-DirInt}
\lim_{R\uparrow\infty} \frac1{\int_{U_R} \eta_R^2}\int_{U_R}\eta_R^2 |\nabla (u_{R,i}-u_{\infty,i})|^2=0.
\end{equation}
Recall the equation~\eqref{eq:correct-sed1-re} for $u_{\infty,i}$ with modified stationary pressure $\tilde p_{\infty,i}=p_{\infty,i}-\lambda|B|x_i$, and similarly define $\tilde p_{R,i}:=p_{R,i}-\lambda|B|x_i$. We implicitly extend $\tilde p_{\infty,i}$ and $\tilde p_{R,i}$ by zero in the inclusions. Changing the pressure in~\eqref{e.cyl1}, we get the following relation in the weak sense in $U_R$,
\begin{equation}\label{e.cyl1-re}
-\triangle u_{R,i} +\nabla\tilde p_{R,i} =-\lambda|B|e_i\mathds1_{U_R\setminus\Ic(U_R)}
-\sum_{x \in \Pc(U_R)}  \delta_{\partial B(x)}\tilde \sigma_{R,i} \nu,
\end{equation}
with the short-hand notation $\tilde\sigma_{R,i}=\sigma(u_{R,i},\tilde p_{R,i})$.
Comparing with the corresponding equation for $u_{\infty,i}$, we obtain in the weak sense in $U_R$,
\begin{multline}\label{e.cyl3}
-\triangle (u_{R,i}-u_{\infty,i}) +\nabla(\tilde p_{R,i}-\tilde p_{\infty,i})=
-\sum_{x \in \Pc(U_R)}  \delta_{\partial B(x)} (\tilde\sigma_{R,i}-\tilde\sigma_{\infty,i}) \nu\\
-\lambda|B|e_i\sum_{x\in\Pc\setminus\Pc(U_R)}\mathds1_{U_R\cap B(x)}
+\sum_{x \in \Pc\setminus\Pc(U_R)}  \delta_{U_R\cap\partial B(x)} \tilde\sigma_{\infty,i}\nu,
\end{multline}
with the short-hand notation $\tilde\sigma_{\infty,i}=\sigma(u_{\infty,i},\tilde p_{\infty,i})$.
Following the strategy in Step~1, we wish to test this equation with $\tilde \eta_R^2 (u_{R,i}-u_{\infty,i})$, but for that we need to make~$u_{\infty,i}$ vanish on $\partial U_R$ first.
Given $K\ge1$, let $\omega_R$ be a cut-off function in $U_R$ such that~$\omega_R$ is constant in a neighborhood of each inclusion $B(x)$,  $x\in\Pc(U_R)$, and such that
\[\omega_R\equiv 1\quad\text{in $U_{R,K}:=\{x\in U_R: \dist(x,\partial U_R)\ge\tfrac RK\}$},\quad
\omega_R|_{\partial U_R}= 0,\quad
|\nabla \omega_R|\lesssim \frac KR.\]
Testing~\eqref{e.cyl3} with $\tilde \eta_R^2 (u_{R,i}-\omega_R u_{\infty,i})$, we obtain for any $\gamma\in\R$,
\begin{multline*}
\int_{U_R} \tilde \eta_R^2 \nabla (u_{R,i}-\omega_R u_{\infty,i}):\nabla (u_{R,i}-u_{\infty,i})
+2 \tilde\eta_R  (u_{R,i}-\omega_R u_{\infty,i})\otimes\nabla \tilde\eta_R: \nabla (u_{R,i}-u_{\infty,i})
\\
-\int_{U_R}\Big(2\tilde \eta_R (u_{R,i}-\omega_R u_{\infty,i})\cdot\nabla \tilde\eta_R-\tilde \eta_R^2 u_{\infty,i}\cdot\nabla \omega_R\Big)(\tilde p_{R,i}-\tilde p_{\infty,i}-\gamma) 
\\
=-\sum_{x\in \Pc(U_R)} \int_{\partial B(x)}\tilde \eta_R^2(u_{R,i}-\omega_R u_{\infty,i})\cdot(\tilde\sigma_{R,i}-\tilde\sigma_{\infty,i})\nu\\
-\lambda|B|e_i\cdot \sum_{x\in\Pc\setminus\Pc(U_R)}\int_{U_R\cap B(x)} \tilde \eta_R^2 (u_{R,i}-\omega_R u_{\infty,i})\\
+ \sum_{x\in\Pc\setminus \Pc(U_R)}\int_{U_R\cap\partial B(x)} \tilde \eta_R^2(u_{R,i}-\omega_R u_{\infty,i}) \cdot\tilde\sigma_{\infty,i} \nu.
\end{multline*}
Since $\tilde \eta_R^2(u_{R,i}-\omega_R u_{\infty,i})$ is rigid in $B(x)$ for all $x\in \Pc(U_R)$, the boundary conditions for~$u_{\infty,i}$ and~$u_{R,i}$ ensure that the first right-hand side term vanishes identically.
Using properties of $\tilde\eta_R$ and Young's inequality, we then get
\begin{multline*} 
\int_{U_R} \eta_R^2 |\nabla (u_{R,i}-u_{\infty,i})|^2
\,\lesssim\,\int_{U_R}  \eta_R^2 |\nabla( (1-\omega_R) u_{\infty,i})|^2
+\frac1{\kappa_0R^2}\int_{U_R}\eta_R^2 |u_{R,i}-\omega_R u_{\infty,i}|^2\\
+\kappa_0\int_{U_R} \eta_R^2|\nabla \omega_R|^2|u_{\infty,i}|^2
+\frac1{\kappa_0}\int_{U_R}\eta_R^2(\tilde p_{R,i}-\tilde p_{\infty,i}-\gamma)^2\\
+\sum_{x\in\Pc\setminus\Pc(U_R)}\int_{U_R\cap B(x)}\eta_R^2|u_{R,i}-\omega_Ru_{\infty,i}|\\
+\sum_{x\in\Pc\setminus \Pc(U_R)}\Big|\int_{U_R\cap\partial B(x)} \tilde \eta_R^2(u_{R,i}-\omega_R u_{\infty,i})\cdot\tilde\sigma_{\infty,i} \nu\Big|.
\end{multline*}
Appealing to the Poincar\'e inequality~\eqref{e.Poincar-cyl} and choosing $\kappa_0\simeq1$ large enough, we can absorb the second right-hand side term into the left-hand side, to the effect of
\begin{multline}\label{e.cyl4}
\int_{U_R}  \eta_R^2 |\nabla (u_{R,i}-u_{\infty,i})|^2
\,\lesssim\,\int_{U_R}  \eta_R^2 |\nabla( (1-\omega_R) u_{\infty,i})|^2
+\kappa_0\int_{U_R} \eta_R^2|\nabla \omega_R|^2|u_{\infty,i}|^2\\
+\frac1{\kappa_0}\int_{U_R} \eta_R^2(\tilde p_{R,i}-\tilde p_{\infty,i}-\gamma)^2
+\sum_{x\in\Pc\setminus\Pc(U_R)}\int_{U_R\cap B(x)}\eta_R^2|u_{R,i}-\omega_Ru_{\infty,i}|\\
+\sum_{x\in\Pc\setminus \Pc(U_R)}\Big|\int_{U_R\cap\partial B(x)} \tilde \eta_R^2(u_{R,i}-\omega_R u_{\infty,i})\cdot\tilde\sigma_{\infty,i} \nu\Big|.
\end{multline}
We now estimate the right-hand side term by term. For the first two terms, by the choice of $\omega_R$, we can bound
\begin{multline}\label{e.cyl5}
\int_{U_R}\eta_R^2 |\nabla ((1-\omega_R) u_{\infty,i})|^2+\kappa_0 \int_{U_R} \eta_R^2 |\nabla \omega_R|^2|u_{\infty,i}|^2\\
\,\lesssim\, 
\int_{U_R \setminus U_{R,K}} \eta_R^2|\nabla u_{\infty,i}|^2 + \frac{\kappa_0K^2}{R^2}\int_{U_R}  \eta_R^2 |u_{\infty,i}|^2.
\end{multline}
For the last right-hand side term of~\eqref{e.cyl4}, by a standard trace estimate, we get
\begin{multline*}
\Big|\int_{U_R\cap \partial B(x)} \tilde \eta_R^2(u_{R,i}-\omega_R u_{\infty,i})\cdot\tilde\sigma_{\infty,i} \nu\Big|
\,\lesssim\,\Big(\int_{U_R\cap B_{1+\delta}(x)}\eta_R^2(|\nabla  u_{\infty,i}|^2+|\tilde p_{\infty,i}|^2)\Big)^\frac12\\
\times\Big(\int_{U_R\cap B_{1+\delta}(x)} \eta_R^2\big(|u_{R,i}-\omega_R u_{\infty,i}|^2+|\nabla (u_{R,i}-\omega_R u_{\infty,i})|^2\big)\Big)^\frac12.
\end{multline*}
Summing over $x\in\Pc\setminus\Pc(U_R)$, noting that for such~$x$ we have $U_R\cap B_{1+\delta}(x)\subset \partial^4U_R:=\{x\in U_R:\dist(x,\partial U_R)\le 4\}$, using Poincar\'e's inequality in $\{\psi \in H^1(\partial^4U_R):\psi|_{\partial U_R}=0\}$ with weight $\eta_R$ and with $O(1)$ constant, and noting that~$\partial^4U_R\subset U_R\setminus U_{R,K}$ for $R\ge4K$, we are led to
\begin{multline}\label{e.cyl6}
\sum_{x\in\Pc\setminus\Pc(U_R)}\Big|\int_{U_R\cap \partial B(x)} \tilde \eta_R^2(u_{R,i}-\omega_R u_{\infty,i})\cdot\tilde\sigma_{\infty,i} \nu\Big|\\
\,\lesssim\,\Big(\int_{U_{R}\setminus U_{R,K}}\eta_R^2(|\nabla  u_{\infty,i}|^2+|\tilde p_{\infty,i}|^2)\Big)^\frac12
\Big(\int_{U_R} \eta_R^2|\nabla (u_{R,i}-\omega_R u_{\infty,i})|^2\Big)^\frac12.
\end{multline}
For the penultimate right-hand side term of~\eqref{e.cyl4}, using the same Poincar\'e inequality, we similarly get
\begin{equation}\label{e.cyl7}
\sum_{x\in\Pc\setminus\Pc(U_R)}\int_{U_R\cap B(x)} \tilde \eta_R^2 |u_{R,i}-\omega_R u_{\infty,i}|\,\lesssim\,(R^{d-1})^\frac12\Big(\int_{U_R} \eta_R^2 |\nabla (u_{R,i}-\omega_R u_{\infty,i})|^2\Big)^\frac12.
\end{equation}
It remains to treat the pressure term in~\eqref{e.cyl4}. The argument is as for~\eqref{e.cyl10} in Step~1. This time, choosing
\begin{equation}\label{eq:gamma-R-choice}
\gamma=\frac1{\int_{U_R\setminus\Ic(U_R)}\tilde\eta_R}\int_{U_R\setminus\Ic(U_R)}\tilde\eta_R(\tilde p_{R,i}-\tilde p_{\infty,i}),
\end{equation}
we consider a vector field $\zeta_R \in H^1_0(U_R)^d$ that is constant in each inclusion $B(x)$, $x\in\Pc(U_R)$, such that
\[\Div (\zeta_R)=\tilde \eta_R (\tilde p_{R,i}-\tilde p_{\infty,i}-\gamma)\mathds1_{U_R\setminus\Ic(U_R)}\qquad\text{in $U_R$},\]
and
\[\int_{U_R}|\nabla\zeta_R|^2\lesssim\int_{U_R\setminus\Ic(U_R)}\eta_R^2(\tilde p_{R,i}-\tilde p_{\infty,i}-\gamma)^2.\]
Testing~\eqref{e.cyl3} with $\tilde\eta_R\zeta_R$, we now find
\begin{multline*}
\int_{U_R}\nabla(\tilde\eta_R\zeta_R):\nabla (u_{R,i}-u_{\infty,i})
-\int_{U_R\setminus\Ic(U_R)}\tilde\eta_R^2(\tilde p_{R,i}-\tilde p_{\infty,i}-\gamma)^2\\
-\int_{U_R}(\zeta_R\cdot\nabla\tilde\eta_R) (\tilde p_{R,i}-\tilde p_{\infty,i}-\gamma)\\
=
-\lambda|B|e_i\cdot\sum_{x\in\Pc\setminus \Pc(U_R)}\int_{U_R\cap B(x)}\tilde\eta_R\zeta_R
+\sum_{x\in\Pc\setminus \Pc(U_R)} \int_{U_R\cap\partial B(x)}\tilde\eta_R\zeta_R\cdot\tilde\sigma_{\infty,i} \nu.
\end{multline*}
All the terms can be treated as in the proof of~\eqref{e.cyl10} in Step~1, except for the right-hand side.
By a trace estimate and Poincar\'e's inequality, as for~\eqref{e.cyl6}, for $R\ge4K$, we find
\begin{eqnarray*}
\lefteqn{\sum_{x\in\Pc\setminus \Pc(U_R)}\Big|\int_{U_R\cap\partial B(x)}\tilde\eta_R\zeta_R\cdot\tilde\sigma_{\infty,i} \nu\Big|}\\
&\lesssim&  \Big(\int_{U_{R} \setminus U_{R,K}}  \eta_R^2( |\nabla u_{\infty,i}|^2+|\tilde p_{\infty,i}|^2)\Big)^\frac12\Big(\int_{U_R} |\nabla \zeta_R |^2\Big)^\frac12\\
&\lesssim &  \Big(\int_{U_{R} \setminus U_{R,K}} \eta_R^2 (|\nabla u_{\infty,i}|^2+|\tilde p_{\infty,i}|^2)\Big)^\frac12\Big(\int_{U_R\setminus\Ic(U_R)}\eta_R^2 (\tilde p_{R,i}-\tilde p_{\infty,i}-\gamma)^2\Big)^\frac12.
\end{eqnarray*}
Likewise,
\begin{eqnarray*}
\sum_{x\in\Pc\setminus \Pc(U_R)}\Big|\int_{U_R\cap B(x)}\tilde\eta_R\zeta_R\Big|
&\lesssim&(R^{d-1})^\frac12\Big(\int_{U_R}|\nabla\zeta_R|^2\Big)^\frac12\\
&\lesssim&(R^{d-1})^\frac12\Big(\int_{U_R\setminus\Ic(U_R)}\eta_R^2(\tilde p_{R,i}-\tilde p_{\infty,i}-\gamma)^2\Big)^\frac12.
\end{eqnarray*}
This leads us to the following pressure estimate, in the spirit of~\eqref{e.cyl10}, for $\kappa_0\simeq1$ large enough,
\begin{multline}\label{e.cyl8}
\int_{U_R\setminus\Ic(U_R)}\eta_R^2 (\tilde p_{R,i}-\tilde p_{\infty,i}-\gamma)^2 \,\lesssim\, \int_{U_R} \eta_R^2 |\nabla (u_{R,i}-u_{\infty,i})|^2\\
+\int_{U_{R} \setminus U_{R,K}}  \eta_R^2 (|\nabla u_{\infty,i}|^2+|\tilde p_{\infty,i}|^2)
+R^{d-1}.
\end{multline}
The combination of~\eqref{e.cyl4}--\eqref{e.cyl8} finally entails, for $\kappa_0\simeq1$ large enough,
\begin{equation*}
\int_{U_R} \eta_R^2 |\nabla (u_{R,i}- u_{\infty,i})|^2\\
\,\lesssim\,
\int_{U_R \setminus U_{R,K}} \eta_R^2(|\nabla u_{\infty,i}|^2+|\tilde p_{\infty,i}|^2)
+\frac{K^2}{R^2}\int_{U_R}  \eta_R^2 |u_{\infty,i}|^2+R^{d-1}.
\end{equation*}
Passing to the limit $R\uparrow\infty$, recalling the sublinearity of $u_{\infty,i}$ at infinity, the stationarity of $\nabla u_{\infty,i}$ and $\tilde p_{\infty,i}$, and applying the ergodic theorem, we get
\[
\limsup_{R\uparrow\infty} \frac1{\int_{U_R} \eta_R^2}\int_{U_R} \eta_R^2 |\nabla (u_{R,i}- u_{\infty,i})|^2
\,\lesssim\, \frac1K.
\]
The claim~\eqref{e.conv-DirInt} then follows by the arbitrariness of $K\ge1$.
For future reference, observe that we also have, as a by-product of the proof,
\begin{equation}\label{e.cyl4++}
\lim_{R\uparrow\infty} \frac1{\int_{U_R} \eta_R^2}\int_{U_R} \eta_R^2 \frac1{R^2}|u_{R,i}-u_{\infty,i}|^2=0,
\end{equation}
 and, with the choice~\eqref{eq:gamma-R-choice} of $\gamma=\gamma_R$,
\begin{equation}\label{e.cyl4++pres}
\lim_{R\uparrow\infty}\frac1{\int_{U_R} \eta_R^2}\int_{U_R\setminus\Ic(U_R)}\eta_R^2 (\tilde p_{R,i}-\tilde p_{\infty,i}-\gamma_R)^2=0.
\end{equation}

\medskip
\step3 Global mean settling speed: proof of \eqref{e.conv-mean-velocityR-re}. \\
Let $\chi\in C^\infty_c(U)$ with $\int_U \chi=1$, $\chi_R(x):=\chi(\frac xR)$, and let us modify it into some~$\tilde\chi_R$ as in Step~4 of the proof of Theorem~\ref{th:mean-velocity}.
Testing the equation~\eqref{e.cyl1-re} for $u_{R,j}$ with \mbox{$\tilde \chi_R (u_{R,i}-u_{\infty,i}-\tilde E_R)$}, for some $\tilde E_R \in \R^d$ to be chosen below, we obtain for any~$\gamma \in \R$,
\begin{multline}\label{eq:test-sumvel}
\int_{U_R} \tilde \chi_R\nabla (u_{R,i}- u_{\infty,i}):\nabla u_{R,j}
+ \int_{U_R}(u_{R,i}-u_{\infty,i}-\tilde E_R)\otimes\nabla \tilde\chi_R : \nabla u_{R,j}\\
-\int_{U_R} (u_{R,i}- u_{\infty,i}-\tilde E_R)\cdot\nabla \tilde\chi_R(\tilde p_{R,j}-\gamma)
=
-\lambda|B|e_j\cdot\int_{U_R\setminus\Ic(U_R)}\tilde\chi_R(u_{R,i}-u_{\infty,i}-\tilde E_R)\\
-\sum_{x\in \Pc(U_R)} \int_{\partial B(x)}\tilde\chi_R (u_{R,i} -u_{\infty,i}-\tilde E_R)\cdot\tilde\sigma_{R,j} \nu.
\end{multline}
Choosing
\[\tilde E_R:=\frac1{\int_{U_R} \tilde \chi_R}\int_{U_R} \tilde \chi_R (u_{R,i}- u_{\infty,i}),\]
recalling that $\tilde\chi_R$ is constant in the inclusions, and using the boundary conditions for $u_{R,j}$, cf.~\eqref{eq:Dir}, with $\tilde\sigma_{R,j}=\sigma_{R,j}+\lambda|B|x_j \Id$, the right-hand side term is equal to
\begin{multline*}
-\lambda|B|e_j\cdot\int_{U_R\setminus\Ic(U_R)}\tilde\chi_R(u_{R,i}-u_{\infty,i}-\tilde E_R)
-\sum_{x\in \Pc(U_R)} \int_{\partial B(x)}\tilde\chi_R (u_{R,i} -u_{\infty,i}-\tilde E_R)\cdot\tilde\sigma_{R,j} \nu\\
=e_j\cdot\sum_{x\in \Pc(U_R)} \int_{B(x)}\tilde\chi_R (u_{R,i} -u_{\infty,i}-\tilde E_R).
\end{multline*}
Hence, the above becomes
\begin{multline*}
e_j\cdot \sum_{x\in \Pc(U_R)} \int_{B(x)}\tilde\chi_R (u_{R,i} -u_{\infty,i}-\tilde E_R)
\,=\,\int_{U_R} \tilde \chi_R\nabla (u_{R,i}- u_{\infty,i}):\nabla u_{R,j}\\
+ \int_{U_R}(u_{R,i}-u_{\infty,i}-\tilde E_R)\otimes\nabla \tilde\chi_R : \nabla u_{R,j}
-\int_{U_R} (u_{R,i}- u_{\infty,i}-\tilde E_R)\cdot\nabla \tilde\chi_R (\tilde p_{R,j}-\gamma).
\end{multline*}
Replacing~$\tilde E_R$ in the right-hand side by $E_R:=\fint_{\chi_R} (u_{R,i}-u_{\infty,i})$, and using the Poincar\'e inequality
\[R^{-2}\int_{\supp\tilde\chi_R}|u_{R,i}-u_{\infty,i}-E_R|^2\lesssim\int_{\supp\tilde\chi_R}|\nabla(u_{R,i}-u_{\infty,i})|^2,\] which can be proved by scaling since $E_R$ is defined using the \emph{scaled} cut-off function $\chi_R$ (as opposed to $\tilde E_R$ defined with the modified cut-off $\tilde\chi_R$), we deduce
\begin{multline}\label{eq:est-uR-uinfty-vel}
\bigg|e_j\cdot\frac1{\int_{U_R} \tilde \chi_R }\sum_{x\in\Pc(U_R)} \int_{B(x)} \tilde \chi_R (u_{R,i}-u_{\infty,i}-\tilde E_R)\bigg|\\
\,\lesssim\,
\bigg( R^{-2}|\tilde E_R-E_R|^2+\frac1{\int_{U_R} \tilde \chi_R}\int_{\supp\tilde\chi_R}|\nabla(u_{R,i}-u_{\infty,i})|^2\bigg)^\frac12\\
\times\bigg(\frac1{\int_{U_R} \tilde \chi_R}\int_{\supp\tilde\chi_R}|\nabla u_{R,j}|^2+|\tilde p_{R,j}-\gamma|^2\bigg)^\frac12.
\end{multline}
Let us show that 
\begin{equation}\label{e.ag.diffE}
\lim_{R\uparrow\infty} |\tilde E_R-E_R| =0.
\end{equation}
Recalling the properties~\eqref{e.4-chir} of $\tilde\chi_R$, noting that $\mathds1_{\supp(\chi_R)+4B} \lesssim \eta_R^2$ with $\eta_R$ as in Step~1, and that $\int_{U_R}\tilde\chi_R\simeq\int_{U_R}\eta_R^2\simeq R^d$, we find
\[
 |\tilde E_R-E_R| \lesssim \frac1{R^{d+1}} \int_{U_R} \eta_{R}^2|u_{R,i}-u_{\infty,i}|
\lesssim\bigg(\frac1{\int_{U_R} \eta_R^2} \int_{U_R}  \eta_{R}^2 \frac1{R^2}|u_{R,i}- u_{\infty,i}|^2\bigg)^\frac12.
\]
By~\eqref{e.cyl4++}, this proves the claim~\eqref{e.ag.diffE}.
Now combining~\eqref{eq:est-uR-uinfty-vel} with~\eqref{e.ag.diffE} and~\eqref{e.conv-DirInt}, using~\eqref{e.cyl4++pres} to replace the pressure $\tilde p_{R,j}$ by $\tilde p_{\infty,j}$, and appealing to the stationarity of the latter and to the ergodic theorem in the form
\[\frac1{\int_{U_R}\tilde\chi_R}\int_{\supp\tilde\chi_R}|\tilde p_{\infty,j}|^2\to\E[|\tilde p_{\infty,j}|^2]<\infty,\]
we deduce, by the arbitrariness of direction $e_j$,
\[
\lim_{R\uparrow\infty}\frac1{\int_{U_R} \tilde \chi_R }\sum_{x\in\Pc(U_R)} \int_{B(x)} \tilde \chi_R (u_{R,i}-u_{\infty,i}-\tilde E_R)=0.
\]
Replacing $\tilde E_R$ by $E_R$ using \eqref{e.ag.diffE}, replacing $\tilde \chi_R$ by $\chi_R$ as in Step~4 of the proof of Theorem~\ref{th:mean-velocity}, and recalling~\eqref{eq:def-V} for $u_{\infty,i}$,
this yields the conclusion~\eqref{e.conv-mean-velocityR-re}.

\medskip
\step4 Local mean settling speed: proof of~\eqref{e.conv-mean-velocityR}.\\
Without loss of generality, assume that $B\subset U$ and let $\chi\in C^\infty_c(B)$ so that $r\supp\chi\subset B_r\subset RU$ for all $1\le r\le R$.
Let $\bar V_{R,ij}^{\text{rel}}(\chi_r):=e_j\cdot\bar V_{R,i}^{\text{rel}}(\chi_r)$, where $\bar V_{R,i}^{\text{rel}}(\chi_r)$ is the mean settling speed as defined in the statement but associated to $u_{R,i}$. For $R\gg r$, the sum over~$\Pc(U_R)$ in the definition can be replaced by a sum over all $\Pc$, hence
\[\bar V_{R,ij}^{\text{rel}}(\chi_r)\,=\,e_j\cdot\frac1{\sum_{x\in \Pc}\chi_r(x)}\sum_{x\in \Pc}\,\chi_r(x)\fint_{B(x)}\Big(u_{R,i}-\fint_{\chi_r}u_{R,i}\Big).\]
Comparing this to $\bar V_\infty^{\text{rel}}(\chi_r)$ as defined in the statement of Theorem~\ref{th:mean-velocity}, we get
\[\bar V_{R,ij}^{\text{rel}}(\chi_r)-\bar V_{\infty,ij}^{\text{rel}}(\chi_r)\,=\,e_j\cdot\frac1{\sum_{x\in \Pc}\chi_r(x)}\sum_{x\in \Pc}\,\chi_r(x)\fint_{B(x)}\Big(u_{R,i}-u_{\infty,i}-\fint_{\chi_r}(u_{R,i}-u_{\infty,i})\Big).\]
Hence, by Poincar\'e's inequality on $r \supp (\chi)$,
\[
|\bar V_{R,ij}^{\text{rel}}(\chi_r)-\bar V_{\infty,ij}^{\text{rel}}(\chi_r)| \,\lesssim_\chi \, r^2 \int_{B_r} |\nabla (u_{R,i}-u_{\infty,i})|^2.
\]
Now note that on $B_{R-2(1+\delta)}\subset\{x\in U_R:\dist(x,\partial U_R)>2(1+\delta)\}$, the difference $w:=u_{R,i}-u_{\infty,i}$ satisfies
\begin{equation*} 
\left\{\begin{array}{ll}
-\triangle w+\nabla p=0,&\text{in $B_{R-2(1+\delta)}\setminus \Ic $},\\
\Div(w)=0,&\text{in $B_{R-2(1+\delta)}\setminus \Ic$},\\
\D(w)=0,&\text{in $B(x)\cap B_{R-2(1+\delta)}$,  $x \in \Pc$},\\
\int_{\partial B(x)}\sigma(w,p)\nu=0,&\forall x \in \Pc:B(x)\cap B_{R-2(1+\delta)}\ne\varnothing,\\
\int_{\partial B(x)}(\cdot- x)\times\sigma(w,p )\nu=0,&\forall x \in \Pc:B(x)\cap B_{R-2(1+\delta)}\ne\varnothing.
\end{array}\right.
\end{equation*}
By the large-scale Lipschitz regularity property of~\cite[Theorem~3]{DG-21a} (which holds under mere ergodicity of $\Pc$), there exists an almost surely finite random variable $r_*$ such that
for all $r_* \le r \le R$,
\[
\fint_{B_r} |\nabla (u_{R,i}-u_{\infty,i})|^2 \,\lesssim\, \fint_{B_R} |\nabla (u_{R,i}-u_{\infty,i})|^2 .
\]
Letting $R \uparrow\infty$ and using~\eqref{e.conv-DirInt}, this entails for all $r\ge r_*$,
\[
\lim_{R\uparrow\infty} |\bar V_{R,ij}^{\text{rel}}(\chi_r)-\bar V_{\infty,ij}^{\text{rel}}(\chi_r)|=0.
\]
In combination with~\eqref{eq:def-V}, this yields the conclusion~\eqref{e.conv-mean-velocityR}.
\qed

\subsection{The case of a snow globe}\label{sec:snow-globe}
Consider the first case of Remark~\ref{rem:bndary}. Let $U\subset\R^d$ be a bounded Lipschitz domain and consider the Stokes problem~\eqref{eq:Dir} in $U_R=RU$ with homogeneous Dirichlet boundary conditions on $\partial U_R$.
Since $U_R$ is bounded, the well-posedness of~\eqref{eq:Dir} is not an issue. Moreover, compared with the infinite-cylinder case treated in Section~\ref{sec:inf-cyl}, no exponential weight is needed and the proof is therefore simpler. In particular,~\eqref{e.conv-DirInt} reduces to
\[
\lim_{R\uparrow\infty} \fint_{U_R} |\nabla (u_{R,i}-u_{\infty,i})|^2 =0.
\]
To prove this convergence, we consider a cut-off function $\omega_R$ that vanishes on $\partial U_R$ and is constant in the inclusions, and we test the equation for $u_{R,i}-u_{\infty,i}$ with $u_{R,i}-\omega_R u_{\infty,i}\in H^1_0(U_R)^d$. The rest of the proof is identical.\qed

\subsection{The case of a finite cylinder}\label{sec:finite-cyl}
We finally turn to the second case of Remark~\ref{rem:bndary}. Let $U=U'\times[a,b]$, where $U'\subset\R^{d-1}$ is a bounded Lipschitz domain and $a<b$. We denote by $\Sigma_{R,1}:=R(U'\times\{b\})$ the upper face of the cylinder, and by $\Sigma_{R,2}:=\partial U_R\setminus\Sigma_{R,1}$ the remaining part of the boundary.
We consider the Stokes problem~\eqref{eq:Dir} in $U_R=RU$ with mixed homogeneous boundary conditions,
\[\sigma(u_R,p_R)\nu=0\quad\text{on $\Sigma_{R,1}$},\qquad u_R=0\quad\text{on $\Sigma_{R,2}$},\]
and with gravity $e\in\R(0,\ldots,0,1)$. We then adapt the above analysis in the vertical direction~$e_i=e$.
As in the snow-globe case, well-posedness of~\eqref{eq:Dir} is not an issue and no exponential weight is needed, which makes the proof simpler.
Yet, as boundary conditions are mixed, we need to proceed slightly differently to prove the convergence 
\begin{equation}\label{eq:main-conv-finitecyl}
\lim_{R\uparrow \infty} \fint_{U_R} |\nabla (u_{R}-u_{\infty})|^2 =0.
\end{equation}
Let $\omega_R$ be a smooth cut-off that vanishes on $\partial U_R$ and is constant in the inclusions.
We start by decomposing
\begin{multline}\label{eq:decomp-ND}
{\int_{U_R} |\nabla (u_{R}-u_{\infty})|^2}
=
\int_{U_R} |\nabla u_{R}|^2 + \int_{U_R} |\nabla u_{\infty}|^2
-2 \int_{U_R} \nabla u_{\infty}:\nabla u_{R}
\\
=\int_{U_R} |\nabla u_{R}|^2 + \int_{U_R} \nabla (\omega_R u_{\infty}) : \nabla u_{\infty}
-2 \int_{U_R} \nabla (\omega_R u_{\infty}):\nabla u_{R}
\\
+2\int_{U_R} \nabla  (\omega_R u_{\infty}-u_{\infty}): \nabla (u_{R}-u_{\infty})-
\int_{U_R} \nabla  (\omega_R u_{\infty}-u_{\infty}): \nabla  u_{\infty}.
\end{multline}
The first three right-hand side terms can be computed by testing the equation for $u_{R}$ with~$u_{R}$ and with~$\omega_R u_{\infty}$, and the equation for $u_{\infty}$ with $\omega_R u_{\infty}$. In all cases, boundary terms vanish. After recombination, this gives rise to the same terms we already analyzed in Section~\ref{sec:inf-cyl}.
The only apparent difference is that, for the first term, the pressure representative is no longer irrelevant  to the traction boundary condition: replacing $p_{R}$ by its renormalized version $p_{R}-\lambda|B|e\cdot x$, or further adding a constant $\gamma$, changes the traction by~$(\lambda|B|e\cdot x+\gamma)\nu$ on~$\Sigma_{R,1}$. However, this has no effect in the weak formulation. Indeed, since the test function~$u_{R}$ satisfies $u_{R}=0$ on~$\Sigma_{R,2}$ and $\Div(u_{R})=0$ in~$U_R$, we have
\[\int_{\Sigma_{R,1}}u_{R}\cdot\nu=\int_{\partial U_R}u_{R}\cdot\nu=\int_{U_R}\Div(u_{R})=0.\]
As $\Sigma_{R,1}=R(U'\times\{b\})$ is flat and orthogonal to $e\in\R(0,\ldots,0,1)$, this also implies
\[\int_{\Sigma_{R,1}}(\lambda|B|e\cdot x+\gamma)(u_{R}\cdot \nu)=0.\]
Hence, the pressure can indeed be replaced by its renormalized version, and constants may still be chosen freely in the energy identities.
It remains to examine the last two boundary layer terms in~\eqref{eq:decomp-ND}. The last one has already been treated in Section~\ref{sec:inf-cyl}. For the penultimate one, we appeal to Young's inequality in the form
\begin{multline*}
\Big|2\int_{U_R} \nabla  (\omega_R u_{\infty,i}-u_{\infty,i}): \nabla (u_{R,i}-u_{\infty,i})\Big|
\\
\le \frac12 \int_{U_R} | \nabla (u_{R,i}-u_{\infty,i})|^2+ 2 \int_{U_R} |\nabla  (\omega_R u_{\infty,i}-u_{\infty,i})|^2,
\end{multline*}
where we can absorb the first term and estimate the last one as before.
This proves the desired convergence~\eqref{eq:main-conv-finitecyl}.\qed

\section{Renormalized cluster expansion}\label{sec:cluster}
This section is devoted to the proof of Theorem~\ref{th:Batchelor}.
Fix a unit vector $e\in\R^d$.
Our starting point is the infrared-regularized approximation established in~\eqref{eq:massive-speed},
\begin{equation}\label{eq:massive-speed-re}
\lim_{\mu \downarrow 0}e\cdot \expecm{ (\varphi_{\mu,e}-\expec{\varphi_{\mu,e}} )\mathds1_{\Ic}}=\frac{\lambda|B|}{1-\lambda|B|}e\cdot\bar V_{\infty}e,
\end{equation}
where $\varphi_{\mu,e}$ denotes the unique stationary solution of~\eqref{eq:correct-sed-mu} with direction $e_i$ replaced by $e$.
For notational simplicity, we drop the dependence on $e$, and set
\[\varphi_{\mu}:=\varphi_{\mu,e},\qquad\bar V_{\infty}:=e\cdot\bar V_{\infty}e.\]
The goal is to extract the dilute expansion of the mean settling speed $\bar V_\infty$ from a cluster expansion of $\varphi_{\mu}$. As explained in the introduction, the main difficulty is that each cluster contribution is infrared divergent: hydrodynamic interactions are long-ranged, and the limit $\mu\downarrow0$ can only be taken after a suitable large-scale renormalization. A second difficulty is to control the truncation error for the cluster expansion, with bounds strong enough to survive the limit $\mu\downarrow0$.

For brevity, we only give the proof under the assumptions of case~(i) in Theorem~\ref{th:Batchelor}, that is, we assume that the correlation functions for the point process $\Pc$ satisfy~\eqref{eq:estim-correl}. The adaptation to case~(ii) is straightforward.

\subsection{Cluster expansion, renormalization, and counterterms}
Before formulating the cluster expansion of $\varphi_{\mu}$, we introduce notation for finite-particle problems.
We say that a finite set $Y\subset\R^d$ is \emph{hardcore} if
\begin{equation}\label{eq:Yseparation}
\dist(B(x),B(y))>2\delta,\quad\text{for all $x,y\in Y$, $x\ne y$}.
\end{equation}
For such a finite hardcore set $Y$, we denote by $\varphi_\mu^Y$ the unique finite-energy decaying solution of the infrared-regularized Stokes problem with inclusions centered at the points of~$Y$,
\begin{equation}\label{eq:phimuY}
\left\{\begin{array}{ll}
\mu\varphi_{\mu}^Y-\Div( \sigma_\mu(\varphi_{\mu}^Y))=0,&\text{in $\R^d\setminus\cup_{y\in Y}B(y)$},\\
\D(\varphi_{\mu}^Y)=0,&\text{in $\cup_{y\in Y}B(y)$},\\
e|B|+\int_{\partial B(x)} \sigma_\mu(\varphi_{\mu}^Y)\nu=0,&\forall x\in Y,\\
\int_{\partial B(x)}(\cdot-x)\times  \sigma_\mu(\varphi_{\mu}^Y)\nu=0,&\forall x\in Y,
\end{array}\right.
\end{equation}
where we recall the modified stress tensor~\eqref{eq:cut-sigma-mu},
\[\sigma_\mu(\varphi)=2\D(\varphi)+\frac1\mu\Div(\varphi)\Id.\]
We next introduce finite-difference operators. If $X\cup Y\subset\R^d$ is finite and hardcore, we set for~$x\in X$,
\[\delta^x\varphi^Y_{\mu}:=\varphi^{Y\cup\{x\}}_{\mu}-\varphi^{Y}_{\mu},\]
which we iterate to define
\begin{equation}\label{eq:def-deltaX}
\delta^X\varphi^Y_{\mu}:=\Big(\prod_{x\in X}\delta^x\Big)\varphi^{Y}_{\mu}=\sum_{Z\subset X}(-1)^{\sharp X\setminus Z}\varphi^{Y\cup Z}_{\mu}.
\end{equation}
By convention,
\[\varphi_{\mu}^\varnothing=0,\qquad \delta^\varnothing\varphi^Y_\mu:=\varphi^Y_\mu.\]
We shall also use the short-hand notation
\[\varphi^{x_1,\ldots,x_n}_{\mu}:=\varphi^{\{x_1,\ldots,x_n\}}_{\mu},\qquad\delta^{x_1,\ldots,x_n}:=\delta^{\{x_1,\ldots,x_n\}}.\]

In these terms, for a finite hardcore set $X$, the inclusion-exclusion identity gives the exact expansion
\[\varphi_{\mu}^X=\sum_{Y\subset X}\delta^Y\varphi_\mu^\varnothing=\sum_{n=1}^{\sharp X}~\sum_{Y\subset X:\sharp Y=n}\delta^Y\varphi_\mu^\varnothing,\]
where the term of order~$n$ describes the contribution of $n$-particle hydrodynamic interactions.
Formally applying the same expansion to the infinite random hardcore configuration~$\Pc$, an additional difficulty appears: the equation~\eqref{eq:correct-sed-mu} contains the deterministic mean backflow $-\alpha e$. Since adding the constant vector field $\frac1\mu\alpha e$ to $\varphi_\mu$ removes this backflow from the equation, the corresponding formal infinite-volume cluster expansion takes the form
\begin{equation*}
\varphi_{\mu}+\frac1\mu\alpha e
= \sum_{n=1}^\infty\sum_{Y\subset\Pc:\sharp Y=n}\delta^{Y}\varphi_{\mu}^\varnothing.
\end{equation*}
For fixed~$\mu>0$, the infrared regularization provides enough decay to actually make this expansion convergent in $\Ld^2(\Omega)$. However, this is not uniform as $\mu\downarrow0$.

The infrared regularization~\eqref{eq:massive-speed-re} of the infinite-volume mean settling speed only involves the centered field~$\varphi_\mu-\expec{\varphi_\mu}$. We shall therefore use the centered version of the above cluster expansion, where the diverging mean backflow is removed,
\begin{eqnarray}
\varphi_{\mu}-\expec{\varphi_{\mu}}
&=&\sum_{n=1}^\infty\bigg(\sum_{Y\subset\Pc:\sharp Y=n}\delta^Y\varphi_{\mu}^\varnothing-\E\bigg[\sum_{Y\subset\Pc:\sharp Y=n}\delta^Y\varphi_\mu^\varnothing\bigg]\bigg).
\label{eq:cluster-exp-centered}
\end{eqnarray}
As the analysis will reveal, the expectation subtracted at each cluster order is precisely the required counterterm: it cancels, term by term, the deterministic $O(\mu^{-1})$ divergence. More explicitly, for each fixed $n$, it can be shown that the counterterm associated with the $n$th cluster term is equal to
\[\E\bigg[\sum_{Y\subset\Pc:\sharp Y=n}\delta^Y\varphi_\mu^\varnothing\bigg]=\frac1\mu \Big((\lambda|B|)^ne+o(1)\Big),\]
see indeed~\eqref{eq:exp-phix}--\eqref{eq:exp-phixyz} for the first three orders. The counterterms thus reconstruct the mean backflow in the limit $\mu \downarrow 0$ through the geometric series $\sum_{n\ge1}(\lambda|B|)^ne=\alpha e$.

Truncating this expansion to order~$3$,
we obtain the following decomposition of the infrared regularization~\eqref{eq:massive-speed-re} of the infinite-volume mean settling speed,
\begin{equation}\label{eq:VL-decomp}
e\cdot\expec{(\varphi_{\mu} -\expec{\varphi_{\mu}})\mathds1_{\Ic}}
\,=\,T_{\mu}^1+T_\mu^2+T_\mu^3+R_\mu,
\end{equation}
in terms of
\begin{eqnarray*}
T_\mu^1&:=&e\cdot\E\bigg[\mathds1_{\Ic}\bigg(\sum_{x\in\Pc} \varphi_\mu^x-\E\Big[\sum_{x\in\Pc} \varphi_\mu^x\Big]\bigg)\bigg],\\
T_\mu^2&:=&e\cdot\E\bigg[\frac12\mathds1_{\Ic}\bigg(\sum_{x,y\in\Pc}^{\ne}\delta^{x,y}\varphi_\mu^\varnothing-\E\Big[\sum_{x,y\in\Pc}^{\ne}\delta^{x,y}\varphi_\mu^\varnothing\Big]\bigg)\bigg],\\
T_\mu^3&:=&e\cdot\E\bigg[\frac1{3!}\mathds1_{\Ic}\bigg(\sum_{x,y,z\in\Pc}^{\ne}\delta^{x,y,z}\varphi_\mu^\varnothing-\E\Big[\sum_{x,y,z\in\Pc}^{\ne}\delta^{x,y,z}\varphi_\mu^\varnothing\Big]\bigg)\bigg],\\
R_\mu&:=&e\cdot\E\big[\mathds1_{\Ic}(u_\mu-\E[u_\mu])\big],
\end{eqnarray*}
where $\sum^{\ne}$ stands for the sum over pairwise distinct points
and where we have set for abbreviation
\begin{equation}\label{eq:def-uL}
u_\mu\,:=\,\varphi_\mu+\frac1\mu\alpha e-\sum_{x\in\Pc}\varphi_\mu^x-\frac12\sum_{x,y\in\Pc}^{\ne}\delta^{x,y}\varphi_\mu^\varnothing-\frac1{3!}\sum_{x,y,z\in\Pc}^{\ne}\delta^{x,y,z}\varphi_\mu^\varnothing.
\end{equation}
Since we are only interested in the leading correction to the Stokes velocity of a single settling particle, one might expect a second-order cluster expansion to suffice. The need to go one order further comes from the estimate of the remainder. Indeed, by definition~\eqref{eq:def-uL}, $u_\mu$ contains hydrodynamic interactions of all orders, just as $\varphi_\mu$ itself. To quantify the remainder, we use a nonlinear estimate and close the argument by buckling; see Step~2 in the proof of Lemma~\ref{lem:estim-RL}.
This entails an effective loss of a square root relative to the expected size of the error. To compensate for this loss, we need to expand up to third order and use sharp estimates on the three-particle cluster contribution.
 
Theorem~\ref{th:Batchelor} follows from~\eqref{eq:massive-speed-re} together with the following estimates on the four terms in the decomposition~\eqref{eq:VL-decomp}, and further noting that~\eqref{eq:submult-lambda} entails 
$\lambda_3 \le \sqrt \lambda_6$.

\begin{prop}\label{prop:Batchelor}
Under the assumptions of case~(i) in Theorem~\ref{th:Batchelor},
we have as $\mu\downarrow0$,
\begin{eqnarray*}
T_\mu^1&=&\lambda e\cdot\int_B\varphi^0
+e\cdot\int_{\R^d}\Big(\int_B\varphi^y\Big)\,h_2(0,y)\,dy+o(1),\\
T_\mu^2 &=&
\lambda^2|B|\Big(e\cdot\int_B\varphi^0\Big)
+\int_{\R^d}\bigg(e\cdot\int_B(\varphi^{0,y}-\varphi^0)+\int_{\partial B(y)}\varphi^0\cdot \sigma^y\nu\bigg)\,f_{2}(0,y)\,dy\\
&&-\int_{\R^d}\bigg(\int_{\partial B}\Big(\varphi^y-\fint_B\varphi^y\Big)\cdot \sigma^0\nu\bigg)\,h_{2}(0,y)\,dy+O(\lambda_3^\frac{\beta}{2+\beta})+o(1),\\
T_\mu^3&=&O\big(\lambda_3^\frac{\beta}{2+\beta}\big)+o(1),\\
R_\mu&=&O\big(\sqrt\lambda_5+(\sqrt\lambda_6)^{\frac\beta{2+\beta}}\big)+o(1),
\end{eqnarray*}
where all the integrals are absolutely convergent, where $O(\lambda)$ stands for an expression that is uniformly bounded by $\lambda$ as $\mu\downarrow0$, and $o(1)$ for an expression that tends to $0$.
\end{prop}

The rest of this section is organized as follows. We start by establishing the convergence of infrared-regularized approximations of finite-particle flows in Section~\ref{sec:finite-part-conv}, as well as a trace estimate in Section~\ref{sec:trace}, which will be repeatedly used in the sequel. Section~\ref{sec:subsec-1} is then devoted to the convergence of the first cluster term~$T_\mu^1$, for which the renormalization is immediate.
For the convergence of the higher-order cluster terms $T_\mu^2$ and $T_\mu^3$, as well as for the uniform control of the remainder $R_\mu$, we need to unravel cancellations that are no longer straightforward.
To this aim, we devise in Sections~\ref{sec:elements}--\ref{sec:diag} a finite diagrammatic decomposition of finite-particle flows in terms of what we call \emph{reflection blocks}. These are non-explicit but allow us to explicitly identify divergences and cancellations. It constitutes one of our main technical achievements in this work.
With this tool at hand, we identify the leading-order contribution of $T_\mu^2$ in Section~\ref{sec:subsec-T2}, and then estimate $T_\mu^3$ and $R_\mu$ in Sections~\ref{sec:subsec-T3} and~\ref{sec:remainder}, respectively.

\subsection{Convergence of finite-particle flows}\label{sec:finite-part-conv}
We establish the following local boundedness and convergence result for infrared-regularized Stokes flows with finite number of inclusions. This will be used repeatedly in the sequel.

\begin{lem}\label{lem:conv-cor-finite}
For any finite hardcore configuration $Y$,
\begin{equation}\label{eq:bounded-cor}
\sup_{x\in\R^d}\int_{B(x)}|\varphi_\mu^Y|^2+|\nabla_\mu\varphi_\mu^Y|^2\,\lesssim_{\sharp Y}\,1,
\end{equation}
where the multiplicative constant only depends on the cardinality of $Y$ and where we use the short-hand notation
\begin{equation}\label{eq:nabmu}
\nabla_\mu\varphi:=\big(\D(\varphi),\tfrac1{\sqrt\mu}\Div(\varphi)\big).
\end{equation}
Moreover, as $\mu\downarrow0$, we have
\begin{equation}\label{eq:conv-cor}
\varphi_\mu^Y\to\varphi^Y\quad\text{in $H^1_\loc(\R^d)$},\qquad
\frac1{\sqrt\mu}\Div(\varphi_\mu^Y)\to0\quad\text{in $L^2(\R^d)$}.
\end{equation}
\end{lem}

\begin{proof}
Let $Y\subset\R^d$ be a finite hardcore configuration and set $\Ic_Y:=\cup_{y\in Y}B(y)$.
The result is trivial if $Y=\varnothing$, so assume $Y\neq\varnothing$. We split the proof into four steps.

\medskip\noindent
\step1 Local energy bounds.\\
Testing equation~\eqref{eq:phimuY} with the solution $\varphi_\mu^Y$ itself yields
\begin{equation*}
\mu\int_{\R^d\setminus\Ic_Y}|\varphi_\mu^Y|^2
+\int_{\R^d}2|\!\D(\varphi_\mu^Y)|^2+\frac1\mu|\Div(\varphi_\mu^Y)|^2
=e\cdot\sum_{y\in Y}\int_{B(y)}\varphi_\mu^Y.
\end{equation*}
As $\varphi_\mu^Y$ is rigid in $\Ic_Y$, we can argue as in~\eqref{e.1.varphi-trou} to fill the holes, to the effect of
\begin{equation}\label{eq:energy-finite-Y-global}
\int_{\R^d}\mu|\varphi_\mu^Y|^2+2|\!\D(\varphi_\mu^Y)|^2+\frac1\mu|\Div(\varphi_\mu^Y)|^2
\lesssim\Big|\sum_{y\in Y}\int_{B(y)}\varphi_\mu^Y\Big|.
\end{equation}
In order to estimate the right-hand side, we argue by duality. In terms of the massive Green function $G_\mu$, cf.~\eqref{eq:Greenfct-mass}, we can write for any $x\in\R^d$,
\begin{equation*}
\int_{B(x)}\varphi_\mu^Y=\mu\int_{\R^d}\varphi_\mu^Y(G_\mu\ast\mathds1_{B(x)})
+\int_{\R^d}\nabla\varphi_\mu^Y\cdot
\nabla(G_\mu\ast\mathds1_{B(x)}).
\end{equation*}
Using the bounds~\eqref{e.mass-green}, we deduce for $d>2$,
\begin{equation}\label{eq:estimate-mean-finite-Y-global}
\sup_{x\in\R^d}\Big|\int_{B(x)}\varphi_\mu^Y\Big|
\lesssim
\Big(\int_{\R^d}\mu|\varphi_\mu^Y|^2+|\nabla\varphi_\mu^Y|^2\Big)^{\frac12}.
\end{equation}
Combining this with~\eqref{eq:energy-finite-Y-global} and recalling Korn's inequality
\[\int_{\R^d}|\nabla\varphi|^2\le\int_{\R^d}|\nabla\varphi|^2+\int_{\R^d}|\Div(\varphi)|^2=\int_{\R^d}2|\!\D(\varphi)|^2,\]
we are led to
\begin{equation}\label{eq:global-energy-bound-finite-Y}
\int_{\R^d}\mu|\varphi_\mu^Y|^2+|\nabla\varphi_\mu^Y|^2+\frac1\mu|\Div(\varphi_\mu^Y)|^2\lesssim_{\sharp Y}1.
\end{equation}
Combined with~\eqref{eq:estimate-mean-finite-Y-global}, this yields
\begin{equation*}
\sup_{x\in\R^d}
\Big|\int_{B(x)}\varphi_\mu^Y\Big|
\lesssim_{\sharp Y}1,
\end{equation*}
and thus, by~\eqref{eq:global-energy-bound-finite-Y} and Poincar\'e's inequality,
\begin{equation}\label{eq:mean-phi-uniform-x}
\sup_{x\in\R^d}\int_{B(x)}|\varphi_\mu^Y|^2\lesssim_{\sharp Y}1.
\end{equation}
This concludes the proof of~\eqref{eq:bounded-cor}.

\medskip
\noindent
\step2 Convergence.\\
As we have shown that $\varphi^Y_\mu$ is bounded in $H^1_\loc(\R^d)$, cf.~\eqref{eq:bounded-cor}, we deduce by weak compactness that, along a subsequence, $\varphi_\mu^Y$ converges weakly to some $\tilde\varphi^Y\in H^1_\loc(\R^d)$. Moreover, the energy estimate~\eqref{eq:global-energy-bound-finite-Y} yields $\Div(\varphi_\mu^Y)\to0$ in $L^2(\R^d)$, hence $\Div(\tilde\varphi^Y)=0$. Since~$\D(\varphi_\mu^Y)=0$ in $\Ic_Y$, we also have $\D(\tilde\varphi^Y)=0$ in $\Ic_Y$. In addition, \eqref{eq:global-energy-bound-finite-Y} entails $\nabla\tilde\varphi^Y\in L^2(\R^d)$. Passing to the limit in the weak formulation of equation~\eqref{eq:phimuY} for $\varphi_\mu^Y$, we find that $\tilde\varphi^Y$ solves the finite-particle Stokes problem with inclusions centered at $Y$, with the prescribed force and torque balances. By uniqueness of the finite-energy decaying solution, we conclude~$\tilde\varphi^Y=\varphi^Y$, hence~$\varphi^Y_\mu\cvf\varphi^Y$ in $H^1_\loc(\R^d)$.

It remains to prove the strong convergence. To this aim, we note that the weak formulation of the equations for $\varphi_\mu^Y$ and $\varphi^Y$ yield
\begin{equation*}
\mu\int_{\R^d\setminus\Ic_Y}|\varphi_\mu^Y|^2+\int_{\R^d}\Big(2|\!\D(\varphi_\mu^Y-\varphi^Y)|^2+\frac1\mu|\Div(\varphi_\mu^Y)|^2\Big)
=e\cdot\sum_{y\in Y}\int_{B(y)}(\varphi_\mu^Y-\varphi^Y).
\end{equation*}
The weak convergence $\varphi^Y_\mu\cvf\varphi^Y$ in $H^1_\loc(\R^d)$ ensures that the right-hand side tends to $0$, which then implies the desired strong convergences in~\eqref{eq:conv-cor}.
\end{proof}

\subsection{Trace estimate}\label{sec:trace}
We establish the following trace estimate adapted to the infrared-regularized problem with the penalized stress~\eqref{eq:cut-sigma-mu}.

\begin{lem}\label{lem:trace-est}
Let $0<\mu\le1$ and recall the notation $B_+=B_{1+\delta}$. For any $u\in H^1(B_+)^d$ satisfying
\[\left\{\begin{array}{ll}
\mu u-\Div(\sigma_\mu(u))=0, &\text{in $B_+\setminus B$},\\
\D(u)=0, &\text{in $B$},\\
\int_{\partial B}x\times\sigma_\mu(u)\nu=0,&
\end{array}\right.\]
we have for all $h\in H^1(B)^d$,
\begin{equation}\label{eq:trace-new}
\bigg|\int_{\partial B}\Big(h-\fint_B h\Big)\cdot\sigma_\mu(u)\nu\bigg|
\,\lesssim\,\|\nabla_\mu h\|_{L^2(B)}\|\nabla_\mu u\|_{L^2(B_+)},
\end{equation}
where we recall the short-hand notation $\nabla_\mu h=(\D(h),\frac1{\sqrt\mu}\Div(h))$.
\end{lem}

\begin{proof}
We shall show more precisely that, for all $h\in H^1(B)^d$,
\begin{multline}\label{eq:trace-new-prec}
\bigg|\int_{\partial B}\Big(h-\fint_B h\Big)\cdot\sigma_\mu(u)\nu\bigg|\\
\lesssim
\bigg(\int_B|\D(h)|^2+\frac1\mu\Big|\int_B\Div (h)\Big|^2\bigg)^{\frac12}
\bigg(\int_{A}|\D(u)|^2+\frac1\mu|\Div (u)|^2\bigg)^{\frac12},
\end{multline}
where we have set for abbreviation $A:=B_+\setminus B$. We split the proof into two steps.

\medskip
\step1 Construction of a lifting.\\
Let $h\in H^1(B)^d$. By Korn's inequality, there exists a skew-symmetric matrix $S_h\in\R^{d\times d}_{\rm skew}$ such that, setting
\[g:=h-\fint_Bh-S_hx,\]
we have
\begin{equation}\label{eq:trace-korn-new}
\|g\|_{H^1(B)}\lesssim\|\!\D(h)\|_{L^2(B)}.
\end{equation}
Since the torque of $u$ averaged on $\partial B$ vanishes, the skew-symmetric part does not contribute, hence
\begin{equation}\label{eq:trace-torque-new}
\int_{\partial B}\Big(h-\fint_Bh\Big)\cdot\sigma_\mu(u)\nu
=\int_{\partial B}g\cdot\sigma_\mu(u)\nu.
\end{equation}
Let
\[m:=\int_{\partial B}g\cdot\nu.\]
Since $g=h-\fint_Bh-S_hx$ and $\Div(S_hx)=0$, the divergence theorem gives
\begin{equation}\label{eq:trace-m-new}
m=\int_B\Div(h).
\end{equation}
Choose an extension $\bar g\in H^1(A)^d$ such that
\[\bar g=g\quad\hbox{on }\partial B,
\qquad
\bar g=0\quad\hbox{on }\partial B_+,\]
and
\begin{equation}\label{eq:trace-extension-new}
\|\bar g\|_{H^1(A)}\,\lesssim\,\|g\|_{H^{1/2}(\partial B)}\stackrel{\eqref{eq:trace-korn-new}}{\lesssim}\|\!\D(h)\|_{L^2(B)}.
\end{equation}
We then have
\[\int_A\Div(\bar g)=-\int_{\partial B}g\cdot\nu=-m.\]
Choose a fixed function $\zeta_A\in C^\infty_c(A)$ with $\int_A\zeta_A=1$ and set
\[s:=-m\zeta_A,\]
so that $\int_A(s-\Div(\bar g))=0$. Applying the Bogovskii operator on $A$, we find $\psi\in H^1_0(A)^d$ such that
\[\Div(\psi)=s-\Div(\bar g),\]
and
\begin{equation}\label{eq:trace-bogo-new}
\|\nabla\psi\|_{L^2(A)}\,\lesssim\,\|s-\Div(\bar g)\|_{L^2(A)}
\stackrel{\eqref{eq:trace-extension-new}}\lesssim |m|+\|\!\D(h)\|_{L^2(B)}.
\end{equation}
Now define
\[\phi:=\bar g+\psi.\]
By construction,
\[\phi=g\quad\text{on $\partial B$},
\qquad
\phi=0\quad\text{on $\partial B_+$},
\qquad
\Div(\phi)=s.\]
In particular, since $s$ is compactly supported in $A$, we have
\[\Div(\phi)=0\quad\text{on $\partial A$.}\]
Moreover, by~\eqref{eq:trace-extension-new} and~\eqref{eq:trace-bogo-new},
\begin{equation}\label{eq:phi-bound}
\|\!\D(\phi)\|_{L^2(A)}+\frac1{\sqrt\mu}\|\Div(\phi)\|_{L^2(A)}
\lesssim
\|\!\D(h)\|_{L^2(B)}+\frac1{\sqrt\mu}|m|.
\end{equation}
Finally, as in the standard Korn argument on an annulus, we may impose the orthogonality
\begin{equation}\label{eq:lifting-orth-new}
\int_A\phi\cdot r=0
\qquad\text{for all rigid motions $r$.}
\end{equation}
To do so, choose a basis $(r_j)_j$ of the finite-dimensional space of rigid motions, and choose divergence-free fields $(\xi_j)_j\subset C^\infty_c(A)^d$ such that the matrix $(\int_A\xi_i\cdot r_j)_{ij}$ is invertible. Subtracting a suitable linear combination of the $\xi_j$'s from $\phi$, we may impose~\eqref{eq:lifting-orth-new} while preserving the value of $\phi$ at the boundary and preserving $\Div(\phi)$, hence preserving~\eqref{eq:phi-bound} up to a multiplicative constant.

\medskip
\step2 Conclusion.\\
Testing the equation for $u$ with $\phi$ on $A$, and recalling $\phi=0$ on $\partial B_+$ and $\phi=g$ on $\partial B$, we get
\begin{equation*}
-\int_{\partial B}g\cdot\sigma_\mu(u)\nu
=\int_A \mu u\cdot\phi+2\D(u):\D(\phi)
+\frac1\mu \Div(u)\,\Div(\phi).
\end{equation*}
Using~\eqref{eq:phi-bound} to bound the last two terms, we deduce
\begin{multline*}
\Big|\int_{\partial B}g\cdot\sigma_\mu(u)\nu\Big|
\le\Big|\int_A \mu u\cdot\phi\Big|\\
+\bigg(\int_B|\!\D(h)|^2+\frac1\mu\Big|\int_B\Div(h)\Big|^2\bigg)^\frac12\bigg(\int_A|\!\D(u)|^2+\frac1\mu|\Div(u)|^2\bigg)^\frac12
\end{multline*}
It remains to estimate the mass term. Using~\eqref{eq:lifting-orth-new} and Korn's inequality on $A$, we get
\[\Big|\int_Au\cdot\phi\Big|
=\inf_{r}\Big|\int_A(u-r)\cdot\phi\Big|
\le\inf_{r}\|u-r\|_{L^2(A)}\|\phi\|_{L^2(A)}
\lesssim\|\!\D(u)\|_{L^2(A)}\|\!\D(\phi)\|_{L^2(A)},\]
where the infimum runs over the set of rigid motions.
The conclusion~\eqref{eq:trace-new-prec} follows.
\end{proof}

\subsection{Convergence of first cluster term}\label{sec:subsec-1}
This section is devoted to the proof of the following lemma that characterizes the limit of the first cluster term in~\eqref{eq:VL-decomp}, thus concluding the first line of Proposition~\ref{prop:Batchelor}.

\begin{lem}\label{lem:limTmu1}
The first cluster term $T_\mu^1$ in~\eqref{eq:VL-decomp} satisfies
\[\lim_{\mu\downarrow0}T_\mu^1\,=\,\lambda e\cdot\int_B\varphi^0
+e\cdot\int_{\R^d}\Big(\int_B\varphi^y\Big)\,h_2(0,y)\,dy,\]
where the integrals are absolutely convergent.
\end{lem}

\begin{proof}
By definition of $T_\mu^1$, writing $\mathds1_\Ic=\sum_{x\in\Pc}\mathds1_{B(x)}$ and using translation invariance $\varphi_\mu^x=\varphi_\mu^0(\cdot-x)$, we can write
\begin{eqnarray*}
T_\mu^1
&=&e\cdot\E\bigg[\sum_{x\in\Pc}\mathds1_{B(x)}\varphi_\mu^x\bigg]+e\cdot\E\bigg[\sum_{x,y\in\Pc}^{\ne}\mathds1_{B(y)}\varphi_\mu^x\bigg]-\lambda|B|e\cdot\E\bigg[\sum_{x\in\Pc}\varphi_\mu^x\bigg]\\
&=&\lambda e\cdot\int_B\varphi_\mu^0+e\cdot\int_{\R^d}\Big(\int_B\varphi_\mu^y\Big)f_2(0,y)\,dy-\lambda^2|B|e\cdot\int_{\R^d}\varphi_\mu^0,
\end{eqnarray*}
in terms of the intensity $\lambda$ and the $2$-point density function $f_2$ of $\Pc$. 
All these integrals make sense since $\varphi_\mu^0$ decays exponentially at infinity. Decomposing $f_2=h_2+\lambda^2$ in the second right-hand side term, we are led to
\begin{equation*}
T_\mu^1\,=\,\lambda e\cdot\int_B\varphi_\mu^0
+e\cdot\int_{\R^d}\Big(\int_B\varphi_\mu^y\Big)h_2(0,y)\,dy
+\lambda^2e\cdot\int_{\R^d}\Big(\int_B\varphi_\mu^y\Big)\,dy
-\lambda^2|B|e\cdot\int_{\R^d}\varphi_\mu^0.
\end{equation*}
The last two terms would diverge as $\mu\downarrow0$, but they cancel out exactly with each other by translation invariance. We then obtain
\begin{equation}\label{eq:decomp-Tmu1}
T_\mu^1\,=\,\lambda e\cdot\int_B\varphi_\mu^0
+e\cdot\int_{\R^d}\Big(\int_B\varphi_\mu^y\Big)h_2(0,y)\,dy.
\end{equation}
In order to pass to the limit in this expression, we shall prove the following uniform-in-$\mu$ spatial decay of the single-particle flow,
\begin{equation}\label{eq:decay-phimu}
\Big(\int_{B(x)}|\varphi_\mu^0|^2\Big)^\frac12\,\lesssim\,\langle x\rangle^{2-d}.
\end{equation}
Combining this with the decay of correlations and with the convergence of finite-particle flows, cf.\@ Lemma~\ref{lem:conv-cor-finite}, the conclusion follows by dominated convergence.

It remains to prove~\eqref{eq:decay-phimu}. This is a particular case of Lemma~\ref{lem:decay-JG} below, but we include a short proof for completeness.
The equation for $\varphi_\mu^0$ implies the following in the weak sense in $\R^d$,
\begin{equation*} 
 \mu\mathds1_{\R^d\setminus B} \varphi_{\mu}^0 -\Div(\sigma_\mu(\varphi_{\mu}^0))=-\delta_{\partial B} \sigma_\mu(\varphi_{\mu}^0) \nu.
\end{equation*}
Denote by $\Gamma_\mu$ the fundamental solution of the infrared-regularized Stokes operator; see Lemma~\ref{lem:massGreen}.
By the Green representation formula, we can bound for all $x\in\R^d$,
\[|\varphi_\mu^0(x)|\le\Big|\int_{\partial B}  \Gamma_\mu(x-\cdot) \cdot\sigma_\mu(\varphi_{\mu}^0) \nu\Big| +\mu\Big|\int_B\Gamma_\mu(x-\cdot)\varphi_\mu^0\Big|,\]
and thus, recalling the boundary conditions satisfied by $\varphi_{\mu}^0$,
\[|\varphi_\mu^0(x)|
\le \Big|\int_{\partial B}\Big(\Gamma_\mu(x-\cdot)-\fint_{B(x)}\Gamma_\mu\Big)\cdot  \sigma_\mu(\varphi_{\mu}^0) \nu\Big|+\Big|\int_{B(x)}\Gamma_\mu\Big| +\mu\Big|\int_B\Gamma_\mu(x-\cdot)\varphi_\mu^0\Big|.\]
By the trace estimate in Lemma~\ref{lem:trace-est}, we deduce for $|x|\ge2$,
\[|\varphi_\mu^0(x)|
\le \Big(\int_{B(x)}|\nabla_\mu\Gamma_\mu|^2\Big)^\frac12\Big(\int_{B_+}|\nabla_\mu\varphi_{\mu}^0|^2\Big)^\frac12+\Big(\int_{B(x)}|\Gamma_\mu|^2\Big)^\frac12\Big(1+\mu^2\int_B|\varphi_\mu^0|^2\Big)^\frac12,\]
with the short-hand notation $B_+=B_{1+\delta}$.
Hence, by the local boundedness of finite-particle flows in Lemma~\ref{lem:conv-cor-finite} and by the decay bounds on $\Gamma_\mu$ in Lemma~\ref{lem:massGreen}, for $|x|\ge2$,
\[|\varphi_\mu^0(x)|\lesssim|x|^{2-d}.\]
As the claim~\eqref{eq:decay-phimu} is trivial for $|x|\le2$ by Lemma~\ref{lem:conv-cor-finite}, the conclusion follows.
\end{proof}
 
\subsection{Expansion into reflection blocks}\label{sec:elements}
The treatment of the higher-order terms in~\eqref{eq:VL-decomp} requires us to unravel cancellations that are no longer visible at the level of the raw cluster expansion, unlike the elementary cancellation leading to~\eqref{eq:decomp-Tmu1} for the first cluster term~$T_\mu^1$.
To this aim, we introduce below a new expansion of finite-particle flows in terms of what we call \emph{reflection blocks}.

The terminology is inspired by the classical method of reflections, which expands finite-particle flows by iterating single-particle responses. While such an expansion leads to explicit expressions where divergences and cancellations could be identified explicitly, the convergence of the infinite reflection series fails for close particles, see e.g.~\cite{Hofer-19}, thus making it unsuitable for our purposes. Our reflection blocks should be viewed instead as resummed reflection chains, leading to {\it finite} expansions of multi-particle flows. Although involving non-explicit multi-body kernels, they are designed to precisely isolate the infrared-divergent contributions and make the relevant cancellations explicit.

More precisely, reflection blocks are obtained by suitable iteration of Green-type representation formulas. They are expressed in terms of Green operators for the infrared-regularized Stokes problem in the presence of a finite configuration of rigid background particles; see elementary operators below.
This strategy to capture the leading infrared-divergent contributions is analogous, at a structural level, to the iterative Duhamel expansions used in the analysis of singular SPDEs: one applies suitable solution operators to extract explicitly the most singular components, while the remaining terms acquire improved regularity or decay and can be treated perturbatively.

A related expansion of finite-particle flows was used in our previous work~\cite{DG-21b} in the context of the dilute expansion of the effective viscosity of particle suspensions. Yet, the present sedimentation problem is more singular: because the particles are buoyant, the corresponding infrared divergences are stronger than in~\cite{DG-21b}, and this forces us to introduce different Green operators.

\medskip\noindent
{\bf Elementary operators.}
We introduce below two Green-type operators for the Stokes problem in the presence of a finite background configuration. The operator $\Jc_{\mu;Y}^z$ describes the response to a localized stress imbalance, that is, a dipole or stresslet-type source, placed near a point~$z$ in a background configuration~$Y$. The operator $\Gc_{\mu;Y}^z$ describes the response to a localized force, that is, a monopole or Stokeslet-type source.

\begin{enumerate}[$\bullet$]
\item {\bf Stresslet-type operator $\Jc_{\mu;Y}^z$.}
Let $\zeta\in H^1_{\loc}(\R^d)$ satisfy the following admissibility conditions around $B(z)$:\footnote{The first and last conditions in~\eqref{eq:admissibility} are not needed to define $\Jc_{\mu;Y}^z\zeta$, but they are crucial to establish the sharp decay estimates in Lemma~\ref{lem:decay-JG} below.} for some~$A\in\R^d$,
\begin{equation}\label{eq:admissibility}
\left\{\begin{array}{ll}
\mu \zeta-\Div(\sigma_\mu(\zeta))=0,&\text{in $B_+(z)\setminus B(z)$},\\
\D(\zeta)=0,&\text{in $B(z)$},\\
A|B|+\int_{\partial B(z)}\sigma_\mu(\zeta)\nu=0,&\\
\int_{\partial B(z)}(\cdot-z)\times \sigma_\mu(\zeta)\nu=0,&
\end{array}\right.
\end{equation}
with the short-hand notation $B_+(z)=B_{1+\delta}(z)$.
Let~$Y\subset \R^d$ be a finite hardcore set in the sense of~\eqref{eq:Yseparation}, viewed as a background configuration.
We first describe the operator when $z$ is separated from~$Y$,
that is, when $\{z\}\cup Y$ is also hardcore. In that case, we define $\Jc_{\mu;Y}^z\zeta\in H^1(\R^d)$ as the solution of
\begin{equation}\label{eq:def-JLYz-0}
\quad\left\{\begin{array}{ll}
\mu\Jc_{\mu;Y}^z\zeta-\Div(\sigma_\mu(\Jc_{\mu;Y}^z\zeta))&\\
\hspace{1.5cm}=-\delta_{\partial B(z)}\sigma_\mu(\zeta)\nu-A\mathds1_{B(z)}+\mu\mathds1_{B(z)}\zeta,&\text{in $\R^d\setminus\cup_{y\in Y}B(y)$},\\[1mm]
\D(\Jc_{\mu;Y}^z\zeta)=0,&\text{in $\cup_{y\in Y}B(y)$},\\
\int_{\partial B(y)}\sigma_\mu(\Jc_{\mu;Y}^z\zeta)\nu=0,&\forall y\in Y,\\
\int_{\partial B(y)}(\cdot-y)\times\sigma_\mu(\Jc_{\mu;Y}^z\zeta)\nu=0,&\forall y\in Y.
\end{array}\right.
\end{equation}
Thus $\Jc_{\mu;Y}^z\zeta$ is the (infrared-regularized) velocity field generated by the stress imbalance of~$\zeta$ on~$\partial B(z)$, in the presence of rigid background inclusions at points of~$Y$.\\[1mm]
For later use, we need a definition that remains meaningful even when $z$ is close to~$Y$. For arbitrary $z\in\R^d$, we define $\Jc_{\mu;Y}^z\zeta$ as the unique element of $H^1(\R^d)$ satisfying
\begin{equation}\label{eq:def-JLYz+}
\D(\Jc_{\mu;Y}^z\zeta)=0\quad\text{in $\cup_{y\in Y}B(y)$},
\end{equation}
and such that, for all $\psi\in H^1(\R^d)$ with $\D(\psi)=0$ in $\cup_{y\in Y}B(y)$,
\begin{multline}\label{eq:def-JLYz}
\quad\int_{\R^d\setminus\cup_{y\in Y}B(y)}\mu\psi\cdot\Jc^z_{\mu;Y}\zeta+2\D(\psi):\D(\Jc^z_{\mu;Y}\zeta)+\frac1\mu \Div(\psi) \Div(\Jc^z_{\mu;Y}\zeta)\\
=-\int_{\partial B(z)}\Big(\psi-\fint_{B(z)}\psi\Big)\cdot\sigma_\mu(\zeta)\nu+\mu\int_{B(z)}\psi\cdot\zeta.
\end{multline}
When $z$ is separated from $Y$, the admissibility condition~\eqref{eq:admissibility} ensures that~\eqref{eq:def-JLYz} is precisely the weak formulation of~\eqref{eq:def-JLYz-0}.

\smallskip\item{\bf Stokeslet-type operator $\Gc_{\mu;Y}^z$.}
Let again $Y\subset \R^d$ be a finite hardcore background configuration. If $z\in\R^d$ is separated from $Y$, in the sense that $\{z\}\cup Y$ is also hardcore, we define $\Gc_{\mu;Y}^z\in H^1(\R^d)$ as the solution of
\begin{equation}\label{eq:def-GLYz-0}
\left\{\begin{array}{ll}
\mu\Gc_{\mu;Y}^z-\Div(\sigma_\mu(\Gc_{\mu;Y}^z))=e\mathds1_{B(z)},&\text{in $\R^d\setminus\cup_{y\in Y}B(y)$},\\
\D(\Gc_{\mu;Y}^z)=0,&\text{in $\cup_{y\in Y}B(y)$},\\
\int_{\partial B(y)}\sigma_\mu(\Gc_{\mu;Y}^z)\nu=0,&\forall y\in Y,\\
\int_{\partial B(y)}(\cdot-y)\times\sigma_\mu(\Gc_{\mu;Y}^z)\nu=0,&\forall y\in Y.
\end{array}\right.
\end{equation}
This is the (infrared-regularized) Stokeslet-type response, in the background configuration~$Y$, to a localized force in~$B(z)$.
\\[1mm]
As above, we extend $\Gc_{\mu;Y}^z$ to arbitrary $z\in\R^d$ by a weak formulation. Namely, we define~$\Gc_{\mu;Y}^z$ as the unique element of $H^1(\R^d)$ satisfying
\begin{equation}\label{eq:def-GLYz+}
\D(\Gc_{\mu;Y}^z)=0\quad\text{in $\cup_{y\in Y}B(y)$},
\end{equation}
and such that, for all $\psi\in H^1(\R^d)$ with~$\D(\psi)=0$ in $\cup_{y\in Y}B(y)$,
\begin{equation}\label{eq:def-GLYz}
\quad\int_{\R^d\setminus\cup_{y\in Y}B(y)}\mu\psi\cdot\Gc_{\mu;Y}^z+2\D(\psi):\D(\Gc_{\mu;Y}^z)+\frac1\mu \Div(\psi)\Div(\Gc_{\mu;Y}^z)
\,=\,e\cdot\int_{B(z)}\psi.
\end{equation}
This is the weak formulation of~\eqref{eq:def-GLYz-0} when $z$ is separated from~$Y$.
Note the following crucial link with $\Jc_{\mu;Y}^z$:
\begin{equation}\label{eq:ling-GJ}
\Jc_{\mu;Y}^z e=\mu\Gc_{\mu;Y}^z.
\end{equation}
\end{enumerate}

\medskip\noindent
{\bf Expansion of finite-particle flows.}
The above elementary operators provide the building blocks for our decomposition of finite-particle flows into reflection blocks. The expansion is obtained by suitably iterating, a finite number of times, the following Green-type representation formulas, which extract the monopole and dipole contributions generated when one particle is added to a finite background configuration.

\begin{lem}\label{lem:element-dec}
For $|x-y|>2$, we have
\begin{eqnarray}
\varphi_\mu^x&=&\Jc_\mu^x\varphi_\mu^x+\Gc_\mu^x,\label{eq:ident-phiL0}\\
\delta^{x,y}\varphi_\mu^\varnothing&=&
\Jc_{\mu}^x\delta^y\varphi_\mu^x+\Jc_{\mu}^y\delta^x\varphi_\mu^y
\label{eq:ident-phiL01}\\
&=&\Jc_{\mu}^x\big(\Jc_{\mu;x}^y\varphi_\mu^{x,y}+\Gc_{\mu;x}^y\big)
+\Jc_{\mu}^y\big(\Jc_{\mu;y}^x\varphi_\mu^{x,y}+\Gc_{\mu;y}^x\big).\label{eq:ident-phiL01+}
\end{eqnarray}
More generally, for any disjoint finite sets $X,Y\subset \R^d$ such that $X\cup Y$ is hardcore, we have
\begin{equation}\label{eq:ident-phiL03}
\delta^{Y}\varphi_\mu^X
\,=\,\sum_{y\in Y}\Big(\Jc_{\mu;X}^y\delta^{Y\setminus\{y\}}\varphi_\mu^{X\cup\{y\}}+\mathds1_{\sharp Y=1}\Gc_{\mu;X}^y\Big)
+\mathds1_{\sharp Y=0}\varphi_\mu^X.
\end{equation}
\end{lem}

\begin{proof}
We split the proof into three parts.

\medskip
\step1 Proof of~\eqref{eq:ident-phiL0}.\\
The equation for~$\varphi_\mu^x$ implies the following in the weak sense in $\R^d$,
\begin{equation}\label{eq:phimu0-Rd}
\mu\mathds1_{\R^d\setminus B(x)}\varphi_\mu^x-\Div(\sigma_\mu(\varphi_\mu^x))=-\delta_{\partial B(x)}\sigma_\mu(\varphi_\mu^x)\nu,
\end{equation}
or equivalently,
\[\mu\varphi_\mu^x
-\Div(\sigma_\mu(\varphi_\mu^x))=-\Big(e\mathds1_{B(x)}+\delta_{\partial B(x)}\sigma_\mu(\varphi_\mu^x)\nu\Big)+\mu\mathds1_{B(x)}\varphi_\mu^x+e\mathds1_{B(x)}.\]
By definition of $\Jc_{\mu}^x$ and $\Gc_\mu^x$, as $\varphi_\mu^x$ satisfies the admissibility condition~\eqref{eq:admissibility} with $z=x$ and $A=e$, the claim~\eqref{eq:ident-phiL0} follows by linearity.

\medskip
\step2 Proof of~\eqref{eq:ident-phiL01} and~\eqref{eq:ident-phiL01+}.\\
Comparing the equations for $\varphi_\mu^{x,y},\varphi_\mu^x,\varphi_\mu^y$, we find that the second-order difference ~$\delta^{x,y}\varphi_\mu^\varnothing$ satisfies in the weak sense in $\R^d$,
\begin{multline*}
\mu\Big(\mathds1_{\R^d\setminus (B(x)\cup B(y))}\varphi_\mu^{x,y}-\mathds1_{\R^d\setminus B(x)}\varphi_\mu^{x}-\mathds1_{\R^d\setminus B(y)}\varphi_\mu^{y}\Big)
-\Div(\sigma_\mu( \delta^{x,y}\varphi_\mu^\varnothing))\\
\,=\,
-\delta_{\partial B(x)}\sigma_\mu(\delta^{y}\varphi_\mu^x)\nu
-\delta_{\partial B(y)}\sigma_\mu(\delta^x\varphi_\mu^{y})\nu,
\end{multline*}
or equivalently,
\begin{multline}\label{eq:delphimu0-Rd}
\mu\delta^{x,y}\varphi_\mu^\varnothing
-\Div(\sigma_\mu(\delta^{x,y}\varphi_\mu^\varnothing))\\
\,=\,
-\delta_{\partial B(x)}\sigma_\mu(\delta^{y}\varphi_\mu^{x})\nu
+\mu\mathds1_{B(x)}\delta^y\varphi_\mu^{x}
-\delta_{\partial B(y)}\sigma_\mu(\delta^{x}\varphi_\mu^{y})\nu
+\mu\mathds1_{B(y)}\delta^x\varphi_\mu^{y}.
\end{multline}
The claim~\eqref{eq:ident-phiL01} then follows by linearity and definition of $\Jc_\mu^x$ and $\Jc_\mu^y$, noting that
\[\int_{\partial B(x)} \sigma_\mu(\delta^y \varphi_\mu^x)\nu=0.\]
In order to prove~\eqref{eq:ident-phiL01+}, it remains to further expand $\delta^y\varphi_\mu^x$. For that purpose, we note that the latter satisfies the following in the weak sense in~$\R^d\setminus B(x)$,
\begin{equation}\label{eq:delphixmu0-Rd}
\mu\delta^y\varphi_\mu^{x}-\Div(\sigma_\mu(\delta^y\varphi_\mu^x))
=\Big(-\delta_{\partial B(y)}\sigma_\mu(\varphi_\mu^{x,y})\nu-e\mathds1_{B(y)}+\mu\mathds1_{B(y)}\varphi_\mu^{x,y}\Big)+e\mathds1_{B(y)}.
\end{equation}
Combining this with the rigidity constraint $\D(\delta^y\varphi_\mu^x)=0$ in $B(x)$, together with
\[\int_{\partial B(x)}\sigma_\mu(\delta^y\varphi_\mu^x)\nu=0,\qquad
\int_{\partial B(x)}(\cdot-x)\times\sigma_\mu(\delta^y\varphi_\mu^x)\nu=0,\]
we deduce, by definition of $\Jc_{\mu;x}^y$ and $\Gc_{\mu;x}^y$,
\begin{equation}\label{eq:decomp-delphi}
\delta^y\varphi_\mu^x\,=\,\Jc_{\mu;x}^y\varphi_\mu^{x,y}+\Gc_{\mu;x}^y.
\end{equation}
Inserting this into~\eqref{eq:ident-phiL01}, the claim~\eqref{eq:ident-phiL01+} follows.

\medskip
\step3 Proof of~\eqref{eq:ident-phiL03}.\\
By definition~\eqref{eq:def-deltaX}, arguing similarly as above for lower-order differences, we find that~$\delta^Y\varphi_\mu^X$ satisfies the following in the weak sense in~$\R^d\setminus\cup_{x\in X}B(x)$,
\begin{eqnarray}
\mu\delta^{Y}\varphi_\mu^X
\lefteqn{-\Div(\sigma_\mu(\delta^{Y}\varphi_\mu^X))}\nonumber\\
&=&-\sum_{Z\subset Y}(-1)^{Y\setminus Z}\sum_{y\in Z}\delta_{\partial B(y)}\sigma_\mu(\varphi_\mu^{X\cup Z})\nu
+\mu\sum_{Z\subset Y}(-1)^{Y\setminus Z}\sum_{y\in Z}\mathds1_{B(y)}\varphi_\mu^{X\cup Z}\nonumber\\
&=&\sum_{y\in Y}\Big(-\delta_{\partial B(y)}\delta^{Y\setminus\{y\}}\sigma_\mu(\varphi_\mu^{X\cup\{y\}})\nu+\mu\mathds1_{B(y)}\delta^{Y\setminus\{y\}}\varphi_\mu^{X\cup\{y\}}\Big),\label{eq:delYphi-dev}
\end{eqnarray}
and the claim~\eqref{eq:ident-phiL03} follows by definition of $\Jc_{\mu;X}^y$ and $\Gc_{\mu;X}^y$.
\end{proof}

There is much freedom in how the above identities can be iterated to expand finite-particle flows. We shall use an iteration scheme that separates the potentially infrared-divergent contributions explicitly, while ensuring that all remaining terms acquire enough spatial decay to be estimated directly. The following key rule is the bookkeeping principle behind this scheme. It will be interpreted diagrammatically in Section~\ref{sec:diag} below, where it becomes particularly transparent.

\begin{keyrule}
Expansions of finite-particle flows are chosen as sums of terms of the form
\[\Jc_{\mu;X_1}^{y_1}\ldots \Jc_{\mu;X_n}^{y_n}\varphi_\mu^{X_{n+1}}\qquad\text{or}\qquad \Jc_{\mu;X_1}^{y_1}\ldots \Jc_{\mu;X_{n-1}}^{y_{n-1}} \Gc_{\mu;X_{n}}^{y_{n}},\]
such that for all $1\le k\le n$ we have $y_k\in X_{k+1}$ and
\begin{equation}\label{eq:keyrule}
\{y_1,\ldots,y_{k-1}\}\cap\{y_k,\ldots,y_{n}\}=\varnothing\quad\implies\quad\Big(\bigcup_{j\le k} X_j\Big)\cap\Big(\bigcup_{j\ge k+1} X_j\Big)=\varnothing.
\end{equation}
\end{keyrule}

We claim that every finite-particle flow admits a finite expansion satisfying this rule; such an expansion will be called a \emph{reflection-block expansion}. Let us illustrate the mechanism on the two-particle flow difference $\delta^{x,y}\varphi_\mu^\varnothing$. A first application of Lemma~\ref{lem:element-dec} gives
\begin{eqnarray*}
\delta^{x,y}\varphi_\mu^\varnothing
&=&\Jc^x_\mu\delta^y\varphi_\mu^x+\Jc^y_\mu\delta^x\varphi_\mu^y\\
&=&\Jc^x_\mu(\Jc^y_{\mu;x}\varphi_\mu^{x,y}+\Gc^y_{\mu;x})+\Jc^y_\mu(\Jc^x_{\mu;y}\varphi_\mu^{x,y}+\Gc^x_{\mu;y}).
\end{eqnarray*}
This decomposition, however, does not yet satisfy the key dependence rule.
Indeed, consider the first term. In the above notation, it corresponds to $n=2$, $X_1=\varnothing$, $X_2=\{x\}$, $X_3=\{x,y\}$, $y_1=x$, and $y_2=y$. Although $y_1 \in X_2$ and $y_2 \in X_3$, the condition~\eqref{eq:keyrule} fails at $k=2$: we have $\{y_1\} \cap \{y_2\}=\varnothing$ but $X_2 \cap X_3 = \{x\} \ne\varnothing$.
To restore the required structure, we further decompose $\varphi_\mu^{x,y}=\varphi_\mu^y+\delta^x\varphi_\mu^y$ and apply Lemma~\ref{lem:element-dec} once more. This yields
\begin{eqnarray}
\delta^{x,y}\varphi_\mu^\varnothing
&=&\Jc^x_\mu(\Jc^y_{\mu;x}\varphi_\mu^{y}+\Gc^y_{\mu;x})
+\Jc^x_\mu\Jc^y_{\mu;x}\delta^x\varphi_\mu^y+\Sym_{x,y}\nonumber\\
&=&\Jc^x_\mu(\Jc^y_{\mu;x}\varphi_\mu^{y}+\Gc^y_{\mu;x})
+\Jc^x_\mu\Jc^y_{\mu;x}(\Jc^x_{\mu;y}\varphi_\mu^{x,y}+\Gc^x_{\mu;y})+\Sym_{x,y},\label{eq:expand-deltxy}
\end{eqnarray}
where $\Sym_{x,y}$ denotes the sum of the preceding terms with $x$ and $y$ exchanged.
This decomposition now satisfies the key dependence rule. For instance, in the second term, we have $n=3$, $X_1=\varnothing$, $X_2=\{x\}$, $X_3=\{y\}$, $X_4=\{x,y\}$, $y_1=x$, $y_2=y$, and $y_3=x$. Then $y_1\in X_2$, $y_2 \in X_3$, and $y_3 \in X_4$.  Moreover, the condition~\eqref{eq:keyrule} is trivially satisfied for $k=2,3$ as $\{y_1,\ldots,y_{k-1}\}\cap\{y_k,\ldots,y_{n}\} \ne \varnothing$, and for $k=1$ we have $X_1\cap X_2=\varnothing$ as required.
The same expansion procedure will be applied below to higher-order differences of finite-particle flows.

\medskip\noindent
{\bf Decay of elementary operators.}
The following result unravels the optimal large-scale decay properties of the elementary operators. The estimate for $\Jc_{\mu;Y}^z$ is analogous to~\cite[Lemmas~4.4--4.5]{DG-22}, although the object considered here is different; in particular, the infrared regularization in~\cite{DG-22} was implemented through periodization (which is simpler but not possible here).
In essence, this result shows that the Green function of the Stokes problem in the presence of a finite background configuration has the same decay properties as the usual Stokeslet and stresslet, uniformly with respect to the finite background configuration.
The large-scale behavior is modified by our specific infrared regularization~\eqref{eq:correct-sed-mu}--\eqref{eq:cut-sigma-mu}. Indeed, the massive term screens transverse modes at scale~$\mu^{-1/2}$, while the artificial bulk-viscosity penalization of the divergence only screens longitudinal modes at the larger scale~$\mu^{-1}$. This produces the slow exponential factor~$e^{-c\mu|x|}$. However, the longitudinal contribution is not a raw Yukawa tail: the cancellation of the singular longitudinal projection yields an additional crossover factor at the transverse scale. More precisely, each algebraic power of~$\langle x\rangle$ above the borderline decay~$\langle x\rangle^{-d}$ is accompanied by a compensating factor~$\langle\sqrt\mu x\rangle^{-1}$. This is encoded in the kernels
\begin{equation}\label{eq:not-Ki}
K_i(x):=\langle x\rangle^{i-d}\langle\sqrt\mu x\rangle^{-i}e^{-c\mu|x|},\qquad i\ge0.
\end{equation}
Thus, $K_i$ behaves like $\langle x\rangle^{i-d}$ up to the transverse scale $\mu^{-1/2}$, and like $\mu^{-i/2}\langle x\rangle^{-d}$ between the transverse scale $\mu^{-1/2}$ and the longitudinal scale $\mu^{-1}$.
In particular, up to a logarithmic loss coming from this intermediate annulus, the large-scale integrability of $K_i$ is the same as that of $\nabla^{2-i}G_\mu$, for $i=0,1,2$, where $G_\mu$ denotes the scalar Yukawa kernel~\eqref{eq:Greenfct-mass} with mass~$\mu$. The proof is given in Appendix~\ref{app:elementary-pieces}.

\begin{lem}\label{lem:decay-JG} 
Let $0<\mu \le 1$, let $\zeta$ satisfy the admissibility condition~\eqref{eq:admissibility} for some $z,A\in\R^d$, and let~$Y\subset \R^d$ be a finite hardcore configuration. 
Then, for all $x\in\R^d$,
\begin{eqnarray*}
\Big(\int_{B(x)}|\Gc_{\mu;Y}^z|^2\Big)^\frac12&\lesssim_{\sharp Y}&K_2(x-z),\\
\Big(\int_{B(x)}|\nabla_\mu\Gc_{\mu;Y}^z|^2\Big)^\frac12&\lesssim_{\sharp Y}&K_1(x-z),\\
\Big(\int_{B(x)}|\Jc_{\mu;Y}^z\zeta|^2\Big)^\frac12&\lesssim_{\sharp Y}&K_1(x-z)\Big(\int_{B_{+}(z)}\mu|\zeta|^2+|\nabla_\mu\zeta|^2\Big)^\frac12,\\
\Big(\int_{B(x)}|\nabla_\mu\Jc_{\mu;Y}^z\zeta|^2\Big)^\frac12&\lesssim_{\sharp Y}&K_0(x-z)\Big(\int_{B_{+}(z)}\mu|\zeta|^2+|\nabla_\mu\zeta|^2\Big)^\frac12,
\end{eqnarray*}
where we recall the short-hand notation~\eqref{eq:nabmu} and $B_+(z)=B_{1+\delta}(z)$, and where $K_i$ is defined in~\eqref{eq:not-Ki} above.
\end{lem}

\medskip\noindent
{\bf Integral identities and cancellations.}
Although the elementary operators involving finite background configurations are not explicit beyond the single-particle case, we show that some specific integrals have an explicit structure. These identities will play a crucial role in isolating the divergent contributions in the cluster expansion.

Compared with the periodized setting of our previous work~\cite{DG-22} on the effective viscosity of suspensions, the present situation is more delicate. In~\cite[Lemma~4.3]{DG-22}, some integrals of the corresponding elementary operators vanished identically, while here the analogous integrals do not vanish in general. Instead, they either exhibit a hidden cancellation or produce an explicit divergence. More precisely, identity~\eqref{eq:ident-intJ} shows that suitable integrals of the stresslet-type operator~$\Jc_{\mu;Y}^z$ remain of order~$O(1)$, whereas the pointwise decay estimate of Lemma~\ref{lem:decay-JG} would only give the weaker bound~$O(\mu^{-1/2}|\!\log\mu|)$ by taking absolute values. On the other hand, identity~\eqref{eq:ident-intG} shows that the integral of the Stokeslet-type operator~$\Gc_{\mu;Y}^z$ contains an explicit~$O(\mu^{-1})$ divergent contribution.

These identities will be used repeatedly to extract the explicit divergences and the cancellations hidden in the cluster formulas.
The key dependence rule~\eqref{eq:keyrule} in the reflection-block expansion is designed precisely so that these cancellation identities can be applied to all terms that are a priori non-integrable.

We also emphasize that the particular extensions~\eqref{eq:def-JLYz} and~\eqref{eq:def-GLYz} of the elementary operators are essential here: they ensure that the integrals below remain well defined when the source point approaches, or overlaps with, the background configuration.

\begin{lem}[Integral cancellation identities]\label{lem:cancel}
Let $\zeta^0$ satisfy the admissibility condition~\eqref{eq:admissibility} at $z=0$, for some $A\in\R^d$, and set $\zeta^z:=\zeta^0(\cdot-z)$. Let also $Y\subset\R^d$ be a finite hardcore configuration. Then we have
\begin{eqnarray}
\int_{\R^d}(\Jc_{\mu;Y}^z\zeta^z)\,dz&=&\sum_{i=1}^d\Big(\int_B\zeta^0\Big)_i\,\big(e_i+\mu\varphi_{\mu,i}^Y\big),\label{eq:ident-intJ}\\
\int_{\R^d}\Gc_{\mu;Y}^z\,dz&=&\tfrac1\mu |B|\big(e+\mu\varphi_{\mu}^Y\big).\label{eq:ident-intG}
\end{eqnarray}
where $\varphi_{\mu,i}^Y$ denotes the finite-particle solution with forcing direction~$e_i$, while $\varphi_\mu^Y$ denotes the one with forcing direction~$e$.
Consequently,  for any finite hardcore configuration $X\subset \R^d$, for all $x\in\R^d$ with $Y\cup\{x\}$ hardcore, we have
\begin{eqnarray}
\int_{\R^d}\big(\Jc_{\mu;X}^x\Jc_{\mu;Y\cup\{x\}}^z\zeta^z\big)\,dz&=&\mu\sum_{i=1}^d\Big(\int_B\zeta^0\Big)_i\big(\Jc_{\mu;X}^x\varphi_{\mu,i}^{Y\cup\{x\}}+\Gc_{\mu,i;X}^x\big),\label{eq:ident-intJ+}\\
\int_{\R^d}\big(\Jc_{\mu;X}^x\Gc_{\mu;Y\cup\{x\}}^z\big)\,dz&=&|B|\big(\Jc_{\mu;X}^x\varphi_{\mu}^{Y\cup\{x\}}+\Gc_{\mu;X}^x\big),\label{eq:ident-intJG+}
\end{eqnarray}
so that
\begin{multline}\label{eq:ident-intJG-re}
\int_{\R^d}\big(\Jc_{\mu;X}^x(\Jc_{\mu;Y\cup\{x\}}^z\varphi_\mu^z+\Gc_{\mu;Y\cup\{x\}}^z)\big)\,dz\\[-3mm]
=\sum_{i=1}^d\Big(\int_B(e+\mu\varphi_\mu^0)\Big)_i\big(\Jc_{\mu;X}^x\varphi_{\mu,i}^{Y\cup\{x\}}+\Gc_{\mu,i;X}^x\big).
\end{multline}
\end{lem}

\begin{proof}
We start with the proof of~\eqref{eq:ident-intJ}.
Set
\[J_{\mu;Y}:=\int_{\R^d}(\Jc_{\mu;Y}^z\zeta^z)\,dz.\]
Integrating in $z$ the weak formulation~\eqref{eq:def-JLYz} for $\Jc^z_{\mu;Y}\zeta^z$, we find for all $\psi\in H^1(\R^d)$ with $\D(\psi)=0$ in $\cup_{y\in Y}B(y)$,
\begin{eqnarray*}
\lefteqn{\int_{\R^d\setminus\cup_{y\in Y}B(y)}\mu\psi\cdot J_{\mu;Y}+2\D(\psi):\D(J_{\mu;Y})+\tfrac1\mu\Div(\psi)\Div(J_{\mu;Y})}\\
&=&-\int_{\partial B}\bigg(\int_{\R^d}\Big(\psi(\cdot+z)-\fint_{B}\psi(\cdot+z)\Big)dz\bigg)\cdot\sigma_\mu(\zeta^0)\nu +\mu\int_{B}\Big(\int_{\R^d}\psi(\cdot+z)\,dz\Big)\cdot\zeta^0\\
&=&\mu\Big(\int_{\R^d}\psi\Big)\cdot\Big(\int_{B}\zeta^0\Big).
\end{eqnarray*}
Decomposing the integral of $\psi$ in the right-hand side, we can reorganize this equation as
\begin{multline*}
\int_{\R^d\setminus\cup_{y\in Y}B(y)}\mu\psi\cdot \Big(J_{\mu;Y}-\int_B\zeta^0\Big)+2\D(\psi):\D(J_{\mu;Y})+\tfrac1\mu \Div(\psi)\Div(J_{\mu;Y})\\[-2mm]
=\mu\Big(\int_{B}\zeta^0\Big)\cdot\sum_{y\in Y}\int_{B(y)}\psi.
\end{multline*}
This coincides with the weak formulation of the equation~\eqref{eq:phimuY} for $b+\mu\varphi_{\mu,b}^Y$ with forcing direction $b=\int_B\zeta^0$. By uniqueness, this proves~\eqref{eq:ident-intJ}.
We turn to the proof of~\eqref{eq:ident-intG}. Set
\[G_{\mu;Y}:=\int_{\R^d}\Gc_{\mu;Y}^zdz.\]
Integrating in $z$ the weak formulation~\eqref{eq:def-GLYz} for $\Gc_{\mu;Y}^z$,
we find for all $\psi\in H^1(\R^d)$ with~$\D(\psi)=0$ in $\cup_{y\in Y}B(y)$,
\begin{equation*}
\int_{\R^d\setminus\cup_{y\in Y}B(y)}\mu\psi\cdot G_{\mu;Y}+2\D(\psi):\D(G_{\mu;Y})+\tfrac1\mu\Div(\psi)\Div(G_{\mu;Y})
\,=\,e|B|\cdot\int_{\R^d}\psi,
\end{equation*}
and the claim~\eqref{eq:ident-intG} then follows arguing as above.
Finally, \eqref{eq:ident-intJ+} and \eqref{eq:ident-intJG+} are direct consequences of~\eqref{eq:ident-intJ} and~\eqref{eq:ident-intG}, respectively, using identity~\eqref{eq:ling-GJ}. Identity~\eqref{eq:ident-intJG-re} follows by summing~\eqref{eq:ident-intJ+} and~\eqref{eq:ident-intJG+}, for the choice $\zeta^0=\varphi_\mu^0$.
\end{proof}

\subsection{Strategy and diagrammatic bookkeeping}\label{sec:diag}
With the above ingredients in place, we now describe the general strategy used to analyze the cluster expansion~\eqref{eq:VL-decomp}. The argument proceeds in four steps:
\begin{enumerate}[(1)]
\item We iterate Lemma~\ref{lem:element-dec} to expand finite-particle flow differences into reflection blocks satisfying the key dependence rule~\eqref{eq:keyrule}; see e.g.~\eqref{eq:expand-deltxy} for the two-particle flow.
\smallskip\item We rewrite expectations with respect to the point process $\Pc$ as integrals over particle positions, and expand the multi-point densities in terms of correlation functions; cf.~\eqref{eq:correl-expand}. This produces couplings between different subsets of integration variables.
\smallskip\item We exploit the explicit cancellation identities of Lemma~\ref{lem:cancel} whenever they apply. These identities either isolate explicit divergent contributions or turn quantities that are a priori non-integrable, at the level of pointwise bounds, into convergent expressions. The key dependence rule~\eqref{eq:keyrule} in the reflection-block expansion is precisely designed to ensure that such cancellations can be unravelled in all potentially divergent terms.
\smallskip\item We finally combine the decay estimates for the elementary operators, cf.~Lemma~\ref{lem:decay-JG}, with the decay assumptions on the correlation functions~\eqref{eq:estim-correl}, in order to bound the resulting integrals.
\end{enumerate}
To implement this strategy and keep the estimates transparent, we introduce a graphical shorthand notation. This notation is deliberately schematic: a diagram is not meant to determine uniquely the underlying analytic expression. Rather, it records only the features that are relevant for the estimates, namely the decay structure of the integrand, the variables over which integration is performed, and the dependence structure imposed by the reflection blocks.
The conventions are as follows.
\begin{enumerate}[$\bullet$]
\item {\bf Vertices.}
Each vertex represents a spatial variable. Frozen variables are drawn as small white circles, optionally labelled. Variables that are integrated over (modulo translation invariance) are drawn as small black circles.
\smallskip\item {\bf Solid edges.}
A solid edge between two vertices $x$ and $y$, decorated with $k$ short bars, represents a kernel bounded by
\[{\tiny\begin{tikzpicture}[baseline={([yshift=.3ex]current bounding box.center)},scale=0.7]
\begin{scope}[every node/.style={circle,draw,inner sep=0pt,minimum size=8pt}]
    \node (1) at (0,0) {x};
    \node (2) at (2,0) {y};
\end{scope}
\begin{scope}[>={Stealth[black]},every edge/.style={draw=black,very thick}]
    \path [-] (1) edge node[bar]{\rule{1pt}{6pt}\hspace{11pt}\rule{1pt}{6pt}} (2);
\end{scope}
\path [dotted] (1-0.19,0.07) edge (1+0.19,0.07);
\path [dotted] (1-0.19,-0.07) edge (1+0.19,-0.07);
\node (12) at (1,-0.3) {\text{$i$ times}};
\end{tikzpicture}}
\equiv
\left\{\begin{array}{lll}
K_i(x-y)&:&0\le i\le d,\\[1mm]
\mu^{-\frac12(i-d)}K_d(x-y)&:&i\ge d,
\end{array}\right.\]
where we recall that $K_i$ is defined in~\eqref{eq:not-Ki}. The saturation for $i\ge d$ is chosen to ensure the simple product rule~\eqref{eq:calc-rule} below.
Thus, for instance,
\begin{equation*}
{\tiny\begin{tikzpicture}[baseline={([yshift=-.8ex]current bounding box.center)},scale=0.7]
\begin{scope}[every node/.style={circle,draw,inner sep=0pt,minimum size=8pt}]
    \node (1) at (0,0) {x};
    \node (2) at (1,0) {y};
\end{scope}
\begin{scope}[>={Stealth[black]},every edge/.style={draw=black,very thick}]
    \path [-] (1) edge (2);
\end{scope}
\end{tikzpicture}}
\,\equiv\,\langle x-y\rangle^{-d}e^{-c\mu|x-y|},\qquad
{\tiny\begin{tikzpicture}[baseline={([yshift=-.8ex]current bounding box.center)},scale=0.7]
\begin{scope}[every node/.style={circle,draw,inner sep=0pt,minimum size=8pt}]
    \node (1) at (0,0) {x};
    \node (2) at (1,0) {y};
\end{scope}
\begin{scope}[>={Stealth[black]},every edge/.style={draw=black,very thick}]
    \path [-] (1) edge node[bar]{\rule{1pt}{6pt}}(2);
\end{scope}
\end{tikzpicture}}
\,\equiv\,\langle x-y\rangle^{1-d}\langle\sqrt\mu(x-y)\rangle^{-1}e^{-c\mu|x-y|}.
\end{equation*}

\item {\bf Dashed edges.}
A dashed edge between vertices $x$ and $y$ represents a correlation function of some order involving variables $x$ and $y$. By assumption~\eqref{eq:estim-correl}, for a correlation function of order $n\ge2$, it is then bounded by $\lambda_n\wedge | x-y|^{-2-\beta}$. The coefficient $\lambda_n$ is recorded as a prefactor~$[\lambda_n]$ in front of the diagram:
\[[\lambda_n]\,{\tiny\begin{tikzpicture}[baseline={([yshift=-.8ex]current bounding box.center)},scale=0.7]
\begin{scope}[every node/.style={circle,draw,inner sep=0pt,minimum size=8pt}]
    \node (1) at (0,0) {x};
    \node (2) at (1.2,0) {y};
\end{scope}
\begin{scope}[>={Stealth[black]},every edge/.style={draw=black,very thick,dashed}]
    \path [-] (1) edge (2);
\end{scope}
\end{tikzpicture}}
\,\equiv\,\lambda_n\wedge| x-y|^{-2-\beta}.\]
For correlations of order greater than two, there is some flexibility in choosing which pair of variables to connect with a dashed line in the diagram. The convention is to select a pair that generates the largest possible cycle: indeed, as we shall see, covering as much of a diagram as possible with cycles improves integrability and thus helps suppress divergences in the limit~$\mu\downarrow0$.

\smallskip\noindent
If a diagram contains several dashed edges, say $s\ge2$, corresponding to~$s$ different correlation functions, a symbolic prefactor $[\lambda_n]$ means that we sum over all splittings
\[n_1+\cdots+n_s=n,\qquad n_i\ge2,\]
where the $i$th dashed edge carries the factor
\[\lambda_{n_i}\wedge\langle x_i-y_i\rangle^{-2-\beta}.\]
For instance,
\[[\lambda_n]\,{\tiny\begin{tikzpicture}[baseline={([yshift=-.8ex]current bounding box.center)},scale=0.7]
\begin{scope}[every node/.style={circle,draw,inner sep=0pt,minimum size=8pt}]
    \node (1) at (0,0) {x};
    \node (2) at (1.2,0.6) {y};
    \node (3) at (1.2,-0.6) {z};
\end{scope}
\begin{scope}[>={Stealth[black]},every edge/.style={draw=black,very thick,dashed}]
    \path [-] (1) edge (2);
    \path [-] (1) edge (3);
\end{scope}
\end{tikzpicture}}
\,\equiv\,\sum_{n_1,n_2\ge2\atop n_1+n_2=n}\big(\lambda_{n_1}\wedge| x-y|^{-2-\beta}\big)\big(\lambda_{n_2}\wedge| x-z|^{-2-\beta}\big).\]
Note that the supermultiplicativity property~\eqref{eq:submult-lambda} ensures that $\lambda_{n_1}\lambda_{n_2}\le\lambda_n$ for all $n_1,n_2$ in this sum.
\smallskip\item {\bf Logarithmic factors.}
Some of the elementary integrations below generate logarithmic factors.
We denote them schematically by a prefactor $\Lc^m$. This means that the integrand is multiplied either by a logarithmic factor of the form
\[\log^m\Big(2+\sqrt\mu\max_{i,j}|x_i-x_j|\Big),\]
where the maximum is taken over the solid edges $(i,j)$ of the diagram, or by the global factor
\[|\!\log\mu|^m\]
whenever a solid edge of saturated order $i\ge d$ is present.
The precise logarithmic factor is irrelevant for the estimates below.
\end{enumerate}
Let us illustrate these conventions on a simple example. The diagram below represents the corresponding integral expression. In the first representation, the vertex labelled~$y$ is kept fixed, while the two black
vertices correspond to the integration variables~$x$ and~$z$:
\begin{eqnarray*}
[\lambda_n]\Lc{\tiny
\begin{tikzpicture}[baseline={([yshift=-.8ex]current bounding box.center)},scale=0.7]
\begin{scope}[every node/.style={circle,fill,inner sep=0pt,minimum size=3pt}]
    \node (1) at (0,0) {};
    \node (3) at (0.6,1) {};
\end{scope}
\begin{scope}[every node/.style={circle,fill,inner sep=0pt,minimum size=8pt}]
    \node[draw,fill=white] (2) at (1.2,0) {y};
\end{scope}
\begin{scope}[>={Stealth[black]},every edge/.style={draw=black,very thick}]
    \path [-] (1) edge node[bar]{\rule{1pt}{6pt}}(2);
    \path [-] (1) edge[bend left=40] node[bar]{\rule{1pt}{6pt}\hspace{1.5pt}\rule{1pt}{6pt}} (3);
    \path [-] (1) edge[bend right=20] (3);
\end{scope}
\begin{scope}[>={Stealth[black]},every edge/.style={draw=black,very thick,dashed}]
    \path [-] (2) edge (3);
\end{scope}
\end{tikzpicture}}
&=&
\int_{(\R^d)^2}K_2(x-z)K_0(x-z)K_1(x-y)\big(\lambda_n\wedge|y-z|^{-2-\beta}\big)\nonumber\\[-3mm]
&&\hspace{3cm}\times\log\big(2+\sqrt\mu(|x-y|+|x-z|+|y-z|)\big)\,dxdz
\nonumber\\
&=&\int_{(\R^d)^2}K_2(x-z)K_0(x-z)K_1(x)\big(\lambda_n\wedge|z|^{-2-\beta}\big)\nonumber\\[-3mm]
&&\hspace{3cm}\times\log\big(2+\sqrt\mu(|x|+|x-z|+|z|)\big)\,dxdz
\nonumber\\
&=&[\lambda_n]\Lc{\tiny\begin{tikzpicture}[baseline={([yshift=-.8ex]current bounding box.center)},scale=0.7]
\begin{scope}[every node/.style={circle,fill,inner sep=0pt,minimum size=3pt}]
    \node (1) at (0,0) {};
    \node (2) at (1.2,0) {};
    \node (3) at (0.6,1) {};
\end{scope}
\begin{scope}[>={Stealth[black]},every edge/.style={draw=black,very thick}]
    \path [-] (1) edge node[bar]{\rule{1pt}{6pt}}(2);
    \path [-] (1) edge[bend left=40] node[bar]{\rule{1pt}{6pt}\hspace{1.5pt}\rule{1pt}{6pt}} (3);
    \path [-] (1) edge[bend right=20] (3);
\end{scope}
\begin{scope}[>={Stealth[black]},every edge/.style={draw=black,very thick,dashed}]
    \path [-] (2) edge (3);
\end{scope}
\end{tikzpicture}}.
\end{eqnarray*}
In the second line, we have used translation invariance to eliminate the fixed variable $y$. After this change of variables, the same integral is represented by the fully black diagram in the last line.
We indeed recall our convention that black vertices are integrated {\it modulo the available translation invariance}: when a diagram contains a cycle with $k$ black vertices, one of the corresponding variables is fixed, and only the remaining $k-1$ variables are integrated. Thus the fully black diagram in the last line should not be read as an integral over three independent variables, but rather as the translation-reduced integral displayed in the preceding line.

With this diagrammatic notation, the key dependence rule~\eqref{eq:keyrule} has the following geometric interpretation. Consider a diagram associated with a term
\[\Jc_{\mu;X_1}^{y_1}\ldots \Jc_{\mu;X_n}^{y_n}\varphi_\mu^{X_{n+1}}\qquad\text{or}\qquad\Jc_{\mu;X_1}^{y_1}\cdots \Jc_{\mu;X_{n-1}}^{y_{n-1}}\Gc_{\mu;X_n}^{y_n}.\]
Whenever an edge corresponding to a factor $\Jc_{\mu;X_k}^{y_k}$ is a cut edge, which is equivalent to the condition $\{y_1,\ldots,y_{k-1}\}\cap\{y_k,\ldots,y_n\}=\varnothing$,
then we require that this cut is also a cut at the level of background dependencies: the background labels carried by the factors on the two sides of the cut are disjoint, namely
\[\Big(\bigcup_{j\le k}X_j\Big)\cap\Big(\bigcup_{j\ge k+1}X_j\Big)=\varnothing.\]
In other words, a cut edge separates not only the vertices in the diagram, but also the background variables entering the elementary kernels.
\begin{keyrule}[diagrammatic version]
At the level of solid edges, every cut edge in a diagram is a genuine dependence cut: after removing such an edge, the source labels and the background labels carried by the two resulting components are disjoint.
\end{keyrule}

\medskip\noindent
{\bf Diagrammatic calculation rules.}
The diagrams are estimated by repeated use of elementary convolution bounds for the kernels $K_i$. The only rules needed below are the following.
\begin{enumerate}[---]
\item {\it Integration of solid edges:} for all $i,j,m\ge0$,
\begin{eqnarray}
\Lc^m{\tiny\begin{tikzpicture}[baseline={([yshift=.6ex]current bounding box.center)},scale=0.7]
\begin{scope}[every node/.style={circle,fill,inner sep=0pt,minimum size=3pt}]
    \node (2) at (2,0) {};
\end{scope}
\begin{scope}[every node/.style={circle,draw,inner sep=0pt,minimum size=3pt}]
    \node (1) at (0,0) {};
\end{scope}
\begin{scope}[>={Stealth[black]},every edge/.style={draw=black,very thick}]
    \path [-] (1) edge node[bar]{\rule{1pt}{6pt}\hspace{11pt}\rule{1pt}{6pt}} (2);
\end{scope}
\path [dotted] (1-0.19,0.07) edge (1+0.19,0.07);
\path [dotted] (1-0.19,-0.07) edge (1+0.19,-0.07);
\node (12) at (1,-0.3) {\text{$i$ times}};
\end{tikzpicture}}
&\lesssim&|\!\log\mu|^{m+1}\,\mu^{-\frac i2}\nonumber\\
\Lc^m{\tiny\begin{tikzpicture}[baseline={([yshift=.6ex]current bounding box.center)},scale=0.7]
\begin{scope}[every node/.style={circle,draw,inner sep=0pt,minimum size=3pt}]
    \node (1) at (0,0) {};
    \node (3) at (4,0) {};
\end{scope}
\begin{scope}[every node/.style={circle,fill,inner sep=0pt,minimum size=3pt}]
    \node (2) at (2,0) {};
\end{scope}
\begin{scope}[>={Stealth[black]},every edge/.style={draw=black,very thick}]
    \path [-] (1) edge node[bar]{\rule{1pt}{6pt}\hspace{11pt}\rule{1pt}{6pt}} (2);
    \path [-] (2) edge node[bar]{\rule{1pt}{6pt}\hspace{11pt}\rule{1pt}{6pt}} (3);
\end{scope}
\path [dotted] (1-0.19,0.07) edge (1+0.19,0.07);
\path [dotted] (1-0.19,-0.07) edge (1+0.19,-0.07);
\node (12) at (1,-0.3) {\text{$i$ times}};
\path [dotted] (3-0.19,0.07) edge (3+0.19,0.07);
\path [dotted] (3-0.19,-0.07) edge (3+0.19,-0.07);
\node (12) at (3,-0.3) {\text{$j$ times}};
\end{tikzpicture}}
&\lesssim&
\Lc^{m+1}\,{\tiny\begin{tikzpicture}[baseline={([yshift=.6ex]current bounding box.center)},scale=0.7]
\begin{scope}[every node/.style={circle,draw,inner sep=0pt,minimum size=3pt}]
    \node (1) at (0,0) {};
    \node (2) at (2,0) {};
\end{scope}
\begin{scope}[>={Stealth[black]},every edge/.style={draw=black,very thick}]
    \path [-] (1) edge node[bar]{\rule{1pt}{6pt}\hspace{11pt}\rule{1pt}{6pt}} (2);
\end{scope}
\path [dotted] (1-0.19,0.07) edge (1+0.19,0.07);
\path [dotted] (1-0.19,-0.07) edge (1+0.19,-0.07);
\node (12) at (1,-0.3) {\text{$i+j$ times}};
\end{tikzpicture}}
\label{eq:calc-rule}
\end{eqnarray}
In addition, for the diagonal contribution,
\begin{equation}\label{eq:calc-rule-diag}
\Lc^m{\tiny\begin{tikzpicture}[baseline={([yshift=-.8ex]current bounding box.center)},scale=0.7]
\begin{scope}[every node/.style={circle,draw,inner sep=0pt,minimum size=3pt}]
    \node (1) at (0,0) {};
\end{scope}
\begin{scope}[every node/.style={circle,fill,inner sep=0pt,minimum size=3pt}]
    \node (2) at (2,0) {};
\end{scope}
\begin{scope}[>={Stealth[black]},every edge/.style={draw=black,very thick}]
    \path [-] (1) edge [bend left=30] node[bar]{\rule{1pt}{6pt}\hspace{11pt}\rule{1pt}{6pt}} (2);
    \path [-] (1) edge [bend right=30] node[bar]{\rule{1pt}{6pt}\hspace{11pt}\rule{1pt}{6pt}} (2);
\end{scope}
\path [dotted] (1-0.19,0.07-0.3) edge (1+0.19,0.07-0.3);
\path [dotted] (1-0.19,-0.07-0.3) edge (1+0.19,-0.07-0.3);
\path [dotted] (1-0.19,0.07+0.3) edge (1+0.19,0.07+0.3);
\path [dotted] (1-0.19,-0.07+0.3) edge (1+0.19,-0.07+0.3);
\node (12) at (1,-0.6) {\text{$i$ times}};
\node (12) at (1,0.6) {\text{$j$ times}};
\end{tikzpicture}}
~\lesssim~|\!\log\mu|^m\mu^{-\frac12(i+j-d)\vee0}\times\left\{\begin{array}{lll}
|\!\log\mu|&:&i\wedge d+j\wedge d=d,\\
1&:&i\wedge d+j\wedge d\ne d.
\end{array}\right.
\end{equation}
\item {\it Integration of dashed edges:} given $d>2$ and $\beta>0$, we have for $i\le2$ and $m\ge0$,
\begin{equation}\label{eq:int-dash}
[\lambda_n]\Lc^m
{\tiny\begin{tikzpicture}[baseline={([yshift=.6ex]current bounding box.center)},scale=0.7]
\begin{scope}[every node/.style={circle,draw,inner sep=0pt,minimum size=3pt}]
    \node (1) at (0,0) {};
    \node (3) at (3,0) {};
\end{scope}
\begin{scope}[every node/.style={circle,fill,inner sep=0pt,minimum size=3pt}]
    \node (2) at (2,0) {};
\end{scope}
\begin{scope}[>={Stealth[black]},every edge/.style={draw=black,very thick}]
    \path [-] (1) edge node[bar]{\rule{1pt}{6pt}\hspace{11pt}\rule{1pt}{6pt}} (2);
\end{scope}
\begin{scope}[>={Stealth[black]},every edge/.style={draw=black,very thick,dashed}]
    \path [-] (2) edge (3);
\end{scope}
\path [dotted] (1-0.19,0.07) edge (1+0.19,0.07);
\path [dotted] (1-0.19,-0.07) edge (1+0.19,-0.07);
\node (12) at (1,-0.3) {\text{$i$ times}};
\end{tikzpicture}}
\,\lesssim\,
o(1)+\left\{\begin{array}{lll}
(\lambda_n)^\frac{2-i+\beta}{2+\beta}&:&1\le i\le 2,\\
\lambda_n|\!\log\lambda_n|&:&i=0,
\end{array}\right.
\end{equation}
where $o(1)$ denotes a quantity that tends to $0$ as $\mu\downarrow0$, uniformly for $\lambda_n>0$ and for fixed $m$. This $o(1)$ term can be removed in the case $m=0$. For~$i>2$, the $\mu$-scaling is given by
\begin{eqnarray}
[\lambda_n]\Lc^m
{\tiny\begin{tikzpicture}[baseline={([yshift=.6ex]current bounding box.center)},scale=0.7]
\begin{scope}[every node/.style={circle,draw,inner sep=0pt,minimum size=3pt}]
    \node (1) at (0,0) {};
    \node (3) at (3,0) {};
\end{scope}
\begin{scope}[every node/.style={circle,fill,inner sep=0pt,minimum size=3pt}]
    \node (2) at (2,0) {};
\end{scope}
\begin{scope}[>={Stealth[black]},every edge/.style={draw=black,very thick}]
    \path [-] (1) edge node[bar]{\rule{1pt}{6pt}\hspace{11pt}\rule{1pt}{6pt}} (2);
\end{scope}
\begin{scope}[>={Stealth[black]},every edge/.style={draw=black,very thick,dashed}]
    \path [-] (2) edge (3);
\end{scope}
\path [dotted] (1-0.19,0.07) edge (1+0.19,0.07);
\path [dotted] (1-0.19,-0.07) edge (1+0.19,-0.07);
\node (12) at (1,-0.3) {\text{$i$ times}};
\end{tikzpicture}}
&\lesssim&
\big(1+\mathds1_{i\ge d}|\!\log\mu|^m\big)
\left\{\begin{array}{lll}
1&:&i\wedge d<2+\beta,\\
|\!\log\mu|&:&i\wedge d=2+\beta,\\
\mu^{1+\frac\beta2-\frac{i\wedge d}2}&:&i\wedge d>2+\beta,\\
\end{array}\right.
\nonumber\\
&\lesssim&|\!\log\mu|^{m+1}\mu^{-(\frac{i\wedge d}2-\frac\beta2-1)\vee0}.
\nonumber
\end{eqnarray}
\end{enumerate}
These estimates follow from straightforward integral calculations; we omit the details for brevity. The logarithmic loss in the convolution estimate~\eqref{eq:calc-rule} occurs for all $i,j\ge0$ and comes from integration on the borderline annulus between the transverse scale $\mu^{-1/2}$ and the longitudinal scale $\mu^{-1}$. By contrast, in the diagonal estimate~\eqref{eq:calc-rule-diag}, this infrared logarithm disappears; only the genuine critical case $i\wedge d+j\wedge d=d$ produces an extra logarithm. Finally, in the dashed-edge estimate~\eqref{eq:int-dash}, the contribution of the spatial logarithmic factor disappears in the limit $\mu\downarrow0$.

\medskip\noindent
{\bf Cancellation rules.}
We now explain how the integral identities of Lemma~\ref{lem:cancel} enter the diagrammatic bookkeeping.
By the key dependence rule, every cut edge in a diagram is a genuine dependence cut: the variables associated with one side of the cut do not appear in the elementary operators on the other side. Whenever the correlation factors are compatible with such a cut, the identities of Lemma~\ref{lem:cancel} apply. With the convention that gray boxes denote arbitrary fully integrated subdiagrams, the identities translate into the following diagrammatic rules.
\begin{enumerate}[---]
\item The identity~\eqref{eq:ident-intJ+} corresponds to:
\begin{eqnarray*}
{\tiny\begin{tikzpicture}[baseline={([yshift=-.8ex]current bounding box.center)},scale=0.7]
\filldraw[gray] (1.5-0.5,0) rectangle (1.5+0.5,0.5);
\filldraw[gray] (-0.5-0.5,0) rectangle (-0.5+0.5,0.5);
\begin{scope}[every node/.style={circle,fill,inner sep=0pt,minimum size=3pt}]
    \node (0) at (-1,0) {};
    \node (1) at (0,0) {};
    \node (3) at (2,0) {};
    \node (2) at (1,0) {};
\end{scope}
\begin{scope}[>={Stealth[black]},every edge/.style={draw=black,very thick}]
    \path [-] (0) edge node[bar]{\rule{1pt}{6pt}} (1);
    \path [-] (1) edge (2);
    \path [-] (2) edge node[bar]{\rule{1pt}{6pt}} (3);
\end{scope}
\end{tikzpicture}}
&\lesssim&
\mu~{\tiny\begin{tikzpicture}[baseline={([yshift=-.8ex]current bounding box.center)},scale=0.7]
\filldraw[gray] (1.5-0.5,0) rectangle (1.5+0.5,0.5);
\filldraw[gray] (-0.5-0.5,0) rectangle (-0.5+0.5,0.5);
\begin{scope}[every node/.style={circle,draw,fill,inner sep=0pt,minimum size=3pt}]
    \node (0) at (-1,0) {};
    \node (1) at (0,0) {};
    \node (2) at (1,0) {};
    \node (3) at (2,0) {};
\end{scope}
\begin{scope}[>={Stealth[black]},every edge/.style={draw=black,very thick}]
    \path [-] (0) edge node[bar]{\rule{1pt}{6pt}\hspace{1.5pt}\rule{1pt}{6pt}} (1);
    \path [-] (2) edge node[bar]{\rule{1pt}{6pt}\hspace{1.5pt}\rule{1pt}{6pt}} (3);
\end{scope}
\end{tikzpicture}}
\\
{\tiny\begin{tikzpicture}[baseline={([yshift=-.8ex]current bounding box.center)},scale=0.7]
\filldraw[gray] (1.5-0.5,0) rectangle (1.5+0.5,0.5);
\filldraw[gray] (-0.5-0.5,0) rectangle (-0.5+0.5,0.5);
\begin{scope}[every node/.style={circle,draw,fill,inner sep=0pt,minimum size=3pt}]
    \node (0) at (-1,0) {};
    \node (1) at (0,0) {};
    \node (3) at (2,0) {};
    \node (2) at (1,0) {};
\end{scope}
\begin{scope}[>={Stealth[black]},every edge/.style={draw=black,very thick}]
    \path [-] (0) edge node[bar]{\rule{1pt}{6pt}} (1);
    \path [-] (1) edge (2);
    \path [-] (2) edge (3);
\end{scope}
\end{tikzpicture}}
&\lesssim&
\mu~{\tiny\begin{tikzpicture}[baseline={([yshift=-.8ex]current bounding box.center)},scale=0.7]
\filldraw[gray] (1.5-0.5,0) rectangle (1.5+0.5,0.5);
\filldraw[gray] (-0.5-0.5,0) rectangle (-0.5+0.5,0.5);
\begin{scope}[every node/.style={circle,draw,fill,inner sep=0pt,minimum size=3pt}]
    \node (0) at (-1,0) {};
    \node (1) at (0,0) {};
    \node (2) at (1,0) {};
    \node (3) at (2,0) {};
\end{scope}
\begin{scope}[>={Stealth[black]},every edge/.style={draw=black,very thick}]
    \path [-] (0) edge node[bar]{\rule{1pt}{6pt}\hspace{1.5pt}\rule{1pt}{6pt}} (1);
    \path [-] (2) edge node[bar]{\rule{1pt}{6pt}} (3);
\end{scope}
\end{tikzpicture}}
\\
{\tiny\begin{tikzpicture}[baseline={([yshift=-.8ex]current bounding box.center)},scale=0.7]
\filldraw[gray] (1.5-0.5,0) rectangle (1.5+0.5,0.5);
\filldraw[gray] (-0.5-0.5,0) rectangle (-0.5+0.5,0.5);
\begin{scope}[every node/.style={circle,draw,fill,inner sep=0pt,minimum size=3pt}]
    \node (0) at (-1,0) {};
    \node (1) at (0,0) {};
    \node (3) at (2,0) {};
    \node (2) at (1,0) {};
\end{scope}
\begin{scope}[>={Stealth[black]},every edge/.style={draw=black,very thick}]
    \path [-] (0) edge (1);
    \path [-] (1) edge (2);
    \path [-] (2) edge (3);
\end{scope}
\end{tikzpicture}}
&\lesssim&
\mu~{\tiny\begin{tikzpicture}[baseline={([yshift=-.8ex]current bounding box.center)},scale=0.7]
\filldraw[gray] (1.5-0.5,0) rectangle (1.5+0.5,0.5);
\filldraw[gray] (-0.5-0.5,0) rectangle (-0.5+0.5,0.5);
\begin{scope}[every node/.style={circle,draw,fill,inner sep=0pt,minimum size=3pt}]
    \node (0) at (-1,0) {};
    \node (1) at (0,0) {};
    \node (2) at (1,0) {};
    \node (3) at (2,0) {};
\end{scope}
\begin{scope}[>={Stealth[black]},every edge/.style={draw=black,very thick}]
    \path [-] (0) edge node[bar]{\rule{1pt}{6pt}} (1);
    \path [-] (2) edge node[bar]{\rule{1pt}{6pt}} (3);
\end{scope}
\end{tikzpicture}}
\end{eqnarray*}
\item The identities~\eqref{eq:ident-intJG+}--\eqref{eq:ident-intJG-re} corresponds to:
\begin{eqnarray*}
{\tiny\begin{tikzpicture}[baseline={([yshift=-.8ex]current bounding box.center)},scale=0.7]
\filldraw[gray] (-0.5-0.5,0) rectangle (-0.5+0.5,0.5);
\begin{scope}[every node/.style={circle,draw,fill,inner sep=0pt,minimum size=3pt}]
    \node (0) at (-1,0) {};
    \node (1) at (0,0) {};
    \node (2) at (1,0) {};
\end{scope}
\begin{scope}[>={Stealth[black]},every edge/.style={draw=black,very thick}]
    \path [-] (0) edge node[bar]{\rule{1pt}{6pt}} (1);
    \path [-] (1) edge node[bar]{\rule{1pt}{6pt}} (2);
\end{scope}
\end{tikzpicture}}
&\lesssim&
{\tiny\begin{tikzpicture}[baseline={([yshift=-.8ex]current bounding box.center)},scale=0.7]
\filldraw[gray] (-0.5-0.5,0) rectangle (-0.5+0.5,0.5);
\begin{scope}[every node/.style={circle,draw,fill,inner sep=0pt,minimum size=3pt}]
    \node (0) at (-1,0) {};
    \node (1) at (0,0) {};
\end{scope}
\begin{scope}[>={Stealth[black]},every edge/.style={draw=black,very thick}]
    \path [-] (0) edge node[bar]{\rule{1pt}{6pt}\hspace{1.5pt}\rule{1pt}{6pt}} (1);
\end{scope}
\end{tikzpicture}}\\
{\tiny\begin{tikzpicture}[baseline={([yshift=-.8ex]current bounding box.center)},scale=0.7]
\filldraw[gray] (-0.5-0.5,0) rectangle (-0.5+0.5,0.5);
\begin{scope}[every node/.style={circle,draw,fill,inner sep=0pt,minimum size=3pt}]
    \node (0) at (-1,0) {};
    \node (1) at (0,0) {};
    \node (2) at (1,0) {};
\end{scope}
\begin{scope}[>={Stealth[black]},every edge/.style={draw=black,very thick}]
    \path [-] (0) edge (1);
    \path [-] (1) edge node[bar]{\rule{1pt}{6pt}} (2);
\end{scope}
\end{tikzpicture}}
&\lesssim&
{\tiny\begin{tikzpicture}[baseline={([yshift=-.8ex]current bounding box.center)},scale=0.7]
\filldraw[gray] (-0.5-0.5,0) rectangle (-0.5+0.5,0.5);
\begin{scope}[every node/.style={circle,draw,fill,inner sep=0pt,minimum size=3pt}]
    \node (0) at (-1,0) {};
    \node (1) at (0,0) {};
\end{scope}
\begin{scope}[>={Stealth[black]},every edge/.style={draw=black,very thick}]
    \path [-] (0) edge node[bar]{\rule{1pt}{6pt}} (1);
\end{scope}
\end{tikzpicture}}
\end{eqnarray*}
\end{enumerate}
In a nutshell, integrating a cut edge adds one bar to each adjacent edge. For a two-legged~$\Jc$-type edge, this means one additional bar on both adjacent edges, compensated by the gain of a factor~$\mu$, whereas for a terminal~$\Gc$-type edge it adds one bar to the preceding edge.

\subsection{Convergence of second cluster term}\label{sec:subsec-T2}
This section identifies the leading contribution of the second cluster term $T_\mu^2$ in the decomposition~\eqref{eq:VL-decomp}. This is the content of the following lemma, which proves the second line of Proposition~\ref{prop:Batchelor}.
Together with Lemma~\ref{lem:limTmu1}, it completes the identification of the main terms in the dilute expansion of the infinite-volume mean settling speed in Theorem~\ref{th:Batchelor}.

For clarity, we give the proof in a fully explicit form, spelling out the relevant integral expressions and calculations, while only indicating the corresponding diagrammatic notation alongside them.  In the subsequent sections, we shall progressively switch to fully diagrammatic proofs in order to simplify lengthy expansions and integral calculations.

\begin{lem}
The second cluster term $T_\mu^2$ in~\eqref{eq:VL-decomp} satisfies
\begin{multline*}
\limsup_{\mu\downarrow0}\bigg|T_\mu^2
-\lambda^2|B|\Big(e\cdot\int_B\varphi^0\Big)
-\int_{\R^d}\bigg(e\cdot\int_B(\varphi^{0,y}-\varphi^0)+\int_{\partial B(y)}\varphi^0\cdot\sigma^y\nu\bigg)\,f_{2}(0,y)\,dy\\
+\int_{\R^d}\bigg(\int_{\partial B}\Big(\varphi^y-\fint_B\varphi^y\Big)\cdot \sigma^0\nu\bigg)\,h_{2}(0,y)\,dy\bigg|\,\lesssim\,\lambda_3^\frac{\beta}{2+\beta},
\end{multline*}
where all the integrals are absolutely convergent. 
\end{lem}

\begin{proof}
By definition of $T_\mu^2$, writing $\mathds1_\Ic=\sum_{x\in\Pc}\mathds1_{B(x)}$, and distinguishing the different intersection cases, we get
\begin{multline*}
T_\mu^2
\,=\,e\cdot\E\bigg[\sum_{x,y\in\Pc}^{\ne}\mathds1_{B(x)}\delta^{x,y}\varphi_\mu^\varnothing\bigg]
+\frac12e\cdot\E\bigg[\sum_{x,y,z\in\Pc}^{\ne}\mathds1_{B(z)}\delta^{x,y}\varphi_\mu^\varnothing\bigg]
-\frac12\lambda|B|e\cdot\E\bigg[\sum_{x,y}^{\ne}\delta^{x,y}\varphi_\mu^\varnothing\bigg].
\end{multline*}
In terms of the multi-point density functions of $\Pc$, cf.~\eqref{eq:multi-density}, using translation invariance, we get
\begin{equation}\label{eq:decomp-TL20}
T_\mu^2\,=\,T_\mu^{2,1}+T_\mu^{2,2},
\end{equation}
in terms of
\begin{eqnarray*}
T_\mu^{2,1}
&:=&\int_{\R^d}\Big(e\cdot\int_{B}\delta^{0,y}\varphi_\mu^\varnothing\Big) f_2(0,y)\,dy,\\
T_\mu^{2,2}
&:=&\frac12\int_{(\R^d)^2}\Big(e\cdot\int_{B}\delta^{x,y}\varphi_\mu^\varnothing\Big)\Big(f_3(x,y,0)-\lambda f_2(x,y)\Big)\,dxdy.
\end{eqnarray*}
We analyze these two contributions in two separate steps.

\medskip
\step1 Proof that
\begin{multline}\label{eq:limTmu21}
\lim_{\mu\downarrow0}T_\mu^{2,1}\,=\,
\lambda^2|B|\Big(e\cdot\int_B\varphi^0\Big)
-\int_{\R^d}\bigg(\int_{\partial B(y)}\Big(\varphi^0-\fint_{B(y)}\varphi^0\Big)\sigma(\delta^{0}\varphi^{y})\nu\bigg)\,f_{2}(0,y)\,dy\\
-\int_{\R^d}\bigg(\int_{\partial B}\Big(\varphi^y-\fint_B\varphi^y\Big)\cdot \sigma(\varphi^0)\nu\bigg)\,h_{2}(0,y)\,dy,
\end{multline}
where all the integrals over $\R^d$ are absolutely convergent.
Testing the equation~\eqref{eq:phimu0-Rd} for~$\varphi_\mu^0$ with $\delta^{0,y}\varphi_\mu^\varnothing$, and recalling the boundary conditions for $\varphi_\mu^0$, we get
\begin{multline*}
{\int_{\R^d}\mu\delta^{0,y}\varphi_\mu^\varnothing\cdot\varphi_\mu^0+2\D(\delta^{0,y}\varphi_\mu^\varnothing):\D(\varphi_\mu^0)+\tfrac1\mu \Div(\delta^{0,y}\varphi_\mu^\varnothing)\Div(\varphi_\mu^0)}\\
=-\int_{\partial B}\Big(\delta^{0,y}\varphi_\mu^\varnothing-\fint_B\delta^{0,y}\varphi_\mu^\varnothing\Big)\cdot\sigma_\mu(\varphi_\mu^0)\nu
+e\cdot\int_B\delta^{0,y}\varphi_\mu^\varnothing
+\mu\int_{B}\delta^{0,y}\varphi_\mu^\varnothing\cdot\varphi_\mu^0,
\end{multline*}
and thus, decomposing $\delta^{0,y}\varphi_\mu^\varnothing=\delta^y\varphi_\mu^{0}-\varphi_\mu^{y}$, where $\delta^y\varphi_\mu^{0}$ is a rigid motion in $B$,
\begin{multline*}
{\int_{\R^d}\mu\delta^{0,y}\varphi_\mu^\varnothing\cdot\varphi_\mu^0+2\D(\delta^{0,y}\varphi_\mu^\varnothing):\D(\varphi_\mu^0)+\tfrac1\mu\Div(\delta^{0,y}\varphi_\mu^\varnothing)\Div(\varphi_\mu^0)}\\
=\int_{\partial B}\Big(\varphi_\mu^{y}-\fint_B\varphi_\mu^{y}\Big)\cdot\sigma_\mu(\varphi_\mu^0)\nu
+e\cdot\int_B\delta^{0,y}\varphi_\mu^\varnothing
+\mu\int_{B}\delta^{0,y}\varphi_\mu^\varnothing\cdot\varphi_\mu^0.
\end{multline*}
Conversely, testing the equation~\eqref{eq:delphimu0-Rd} for $\delta^{0,y}\varphi_\mu$ with $\varphi_\mu^0$, and using the rigidity of $\varphi_\mu^0$ in~$B$ and the boundary conditions, we get
\begin{multline*}
\int_{\R^d}\mu\delta^{0,y}\varphi_\mu^\varnothing\cdot\varphi_\mu^0
+2\D( \delta^{0,y}\varphi_\mu^\varnothing):\D(\varphi_\mu^0)
+\tfrac1\mu\Div(\delta^{0,y}\varphi_\mu^\varnothing)\Div(\varphi_\mu^0)\\
\,=\,
-\int_{\partial B(y)}\Big(\varphi_\mu^0-\fint_{B(y)}\varphi_\mu^0\Big)\cdot\sigma_\mu(\delta^{0}\varphi_\mu^y)\nu
+\mu\int_{B}\varphi_\mu^0\cdot\delta^{y}\varphi_\mu^{0}
+\mu\int_{B(y)}\varphi_\mu^0\cdot\delta^0\varphi_\mu^{y}.
\end{multline*}
Combining these two identities, we get after straightforward simplifications,
\begin{multline*}
e\cdot\int_{B}\delta^{0,y}\varphi_\mu^\varnothing
\,=\,
\mu\int_{B(y)}\varphi_\mu^0\cdot\delta^0\varphi_\mu^{y}
+\mu\int_{B}\varphi_\mu^0\cdot\varphi_\mu^{y}
\\
-\int_{\partial B(y)}\Big(\varphi_\mu^0-\fint_{B(y)}\varphi_\mu^0\Big)\cdot\sigma_\mu(\delta^{0}\varphi_\mu^{y})\nu
-\int_{\partial B}\Big(\varphi_\mu^{y}-\fint_B\varphi_\mu^{y}\Big)\cdot\sigma_\mu(\varphi_\mu^0)\nu.
\end{multline*}
Inserting this into the definition of $T_\mu^{2,1}$, cf.~\eqref{eq:decomp-TL20}, and noting that we can replace $f_2$ by $h_2=f_2-\lambda^2$ in the last term as $\int_{\R^d}(\varphi_\mu^y-\fint_B\varphi_\mu^y)\,dy=0$, we get
\begin{multline*}
T_\mu^{2,1}\,=\,
\mu\int_{\R^d}\Big(\int_{B(y)}\varphi_\mu^0\cdot\delta^0\varphi_\mu^{y}\Big) f_2(0,y)\,dy
+\mu\int_{\R^d}\Big(\int_{B}\varphi_\mu^0\cdot\varphi_\mu^{y}\Big) f_2(0,y)\,dy\\
-\int_{\R^d}\bigg(\int_{\partial B(y)}\Big(\varphi_\mu^0-\fint_{B(y)}\varphi_\mu^0\Big)\cdot\sigma_\mu(\delta^{0}\varphi_\mu^y)\nu\bigg) f_2(0,y)\,dy\\
-\int_{\R^d}\bigg(\int_{\partial B}\Big(\varphi_\mu^{y}-\fint_B\varphi_\mu^{y}\Big)\cdot\sigma_\mu(\varphi_\mu^0)\nu\bigg) h_2(0,y)\,dy.
\end{multline*}
For the second right-hand side term, we further decompose $f_2=h_2+\lambda^2$ and we appeal to the cancellation rules of Lemma~\ref{lem:cancel} in the form
\begin{equation}\label{eq:int-phi0mu}
\int_{\R^d}\varphi^y_\mu\,dy
\,=\,\tfrac1\mu e|B|+\int_B\varphi^0_\mu.
\end{equation}
(Without even invoking Lemma~\ref{lem:cancel}, note that this identity actually follows by integrating the equation~\eqref{eq:phimu0-Rd} for $\varphi_\mu^0$ on $\R^d$.)
The above then becomes
\[T_\mu^{2,1}\,=\,\sum_{i=1}^5A_\mu^{i},\]
in terms of
\begin{eqnarray*}
A_\mu^{1}&:=&\mu\int_{\R^d}\Big(\int_{B(y)}\varphi_\mu^0\cdot\delta^0\varphi_\mu^{y}\Big) f_2(0,y)\,dy,\\
A_\mu^{2}&:=&\mu\int_{\R^d}\Big(\int_{B}\varphi_\mu^0\cdot\varphi_\mu^{y}\Big) h_2(0,y)\,dy,\\
A_\mu^{3}&:=&\lambda^2\Big(\int_{B}\varphi_\mu^0\Big)\cdot\Big(|B|e+\mu\int_B\varphi_\mu^0\Big),\\
A_\mu^{4}&:=&-\int_{\R^d}\bigg(\int_{\partial B(y)}\Big(\varphi_\mu^0-\fint_{B(y)}\varphi_\mu^0\Big)\cdot\sigma_\mu(\delta^{0}\varphi_\mu^{y})\nu\bigg) f_2(0,y)\,dy,\\
A_\mu^{5}&:=&-\int_{\R^d}\bigg(\int_{\partial B}\Big(\varphi_\mu^{y}-\fint_B\varphi_\mu^{y}\Big)\cdot\sigma_\mu(\varphi_\mu^0)\nu\bigg) h_2(0,y)\,dy.
\end{eqnarray*}
We start by examining $A_\mu^{1}$.
Using Lemma~\ref{lem:element-dec} to expand finite-particle flows into elementary operators,
\[\varphi_\mu^0=\Jc_\mu^0\varphi_\mu^0+\Gc_\mu^0,\qquad\delta^0\varphi_\mu^y=\Jc_{\mu;y}^0\varphi_\mu^{0,y}+\Gc^0_{\mu;y},\]
using the decay estimates of Lemma~\ref{lem:decay-JG}, and recalling the local boundedness of finite-particle flows, cf.\@ Lemma~\ref{lem:conv-cor-finite}, we find
\[|A_\mu^{1}|\,\lesssim\,\lambda_2\mu\int_{\R^d}K_2(y)^2dy\,\lesssim\,\lambda_2\mu\times\left\{\begin{array}{lll}
\sqrt\mu^{-1}&:&d=3,\\
|\!\log\mu|&:&d=4,\\
1&:&d>4,
\end{array}\right.\]
which entails $\lim_{\mu\downarrow0}A_\mu^{1}=0$.
Alternatively, using diagrammatic notation,
\[|A_\mu^{1}|\,\lesssim\,\lambda_2\mu{\tiny\begin{tikzpicture}[baseline={([yshift=-.8ex]current bounding box.center)},scale=0.7]
\begin{scope}[every node/.style={circle,fill,inner sep=0pt,minimum size=3pt}]
    \node (1) at (0,0) {};
    \node (2) at (0,1) {};
\end{scope}
\begin{scope}[>={Stealth[black]},every edge/.style={draw=black,very thick}]
    \path [-] (1) edge[bend left=40] node[bar]{\rule{1pt}{6pt}\hspace{1.5pt}\rule{1pt}{6pt}} (2);
    \path [-] (1) edge[bend right=40] node[bar]{\rule{1pt}{6pt}\hspace{1.5pt}\rule{1pt}{6pt}} (2);
\end{scope}
\end{tikzpicture}}\]
and the same estimate follows from the calculation rule~\eqref{eq:calc-rule-diag}. Similarly, for $A_\mu^2$, further using the decay of correlations, we can estimate
\begin{eqnarray*}
|A_\mu^2|
&\lesssim&
\mu\int_{\R^d}K_2(y)\big(\lambda_2\wedge\langle y\rangle^{-2-\beta}\big)\,dy
\,\equiv\,\mu{\tiny\begin{tikzpicture}[baseline={([yshift=-.8ex]current bounding box.center)},scale=0.7]
\begin{scope}[every node/.style={circle,fill,inner sep=0pt,minimum size=3pt}]
    \node (1) at (0,0) {};
    \node (2) at (0,1) {};
\end{scope}
\begin{scope}[>={Stealth[black]},every edge/.style={draw=black,very thick}]
    \path [-] (1) edge[bend left=40] node[bar]{\rule{1pt}{6pt}\hspace{1.5pt}\rule{1pt}{6pt}} (2);
\end{scope}
\begin{scope}[>={Stealth[black]},every edge/.style={draw=black,very thick,dashed}]
    \path [-] (1) edge[bend right=40] (2);
\end{scope}
\end{tikzpicture}}\\
&\le&\mu\int_{\R^d}\langle y\rangle^{-\beta-d}dy\,\lesssim\,\mu,
\end{eqnarray*}
hence $\lim_{\mu\downarrow0}A_\mu^2=0$.
Next, by the convergence of the single-particle problem, cf.~Lemma~\ref{lem:conv-cor-finite}, we can directly pass to the limit in $A_\mu^3$, to the effect of
\[\lim_{\mu\downarrow0}A_\mu^3\,=\,\lambda^2|B|e\cdot\int_B\varphi^0.\]
We turn to the analysis of $A_\mu^4$. Appealing to the trace estimate of Lemma~\ref{lem:trace-est}, using again Lemmas~\ref{lem:element-dec} and~\ref{lem:decay-JG}, and the local boundedness of finite-particle flows, cf.~Lemma~\ref{lem:conv-cor-finite}, we find that the integrand is bounded by
\begin{equation}\label{eq:decay-key-corr-Batch}
\bigg|\int_{\partial B(y)}\Big(\varphi_\mu^0-\fint_{B(y)}\varphi_\mu^0\Big)\cdot\sigma_\mu(\delta^{0}\varphi_\mu^{y})\nu\bigg|\,\lesssim\,K_1(y)^2\,\le\,\langle y\rangle^{2(1-d)},
\end{equation}
which is integrable in dimension $d>2$. By dominated convergence, it remains to pass to the limit in this integrand for fixed $y\in\R^d$. To this aim, we first note that, by the trace estimate of Lemma~\ref{lem:trace-est}, we can bound
\begin{multline*}
\bigg|\int_{\partial B(y)}\Big(\varphi_\mu^0-\fint_{B(y)}\varphi_\mu^0\Big)\cdot\sigma_\mu(\delta^{0}\varphi_\mu^{y})\nu-\int_{\partial B(y)}\Big(\varphi^0-\fint_{B(y)}\varphi^0\Big)\cdot\sigma_\mu(\delta^{0}\varphi_\mu^{y})\nu\bigg|\\
\,\lesssim\,\Big(\int_{B(y)}|\nabla_\mu(\varphi_\mu^0-\varphi^0)|^2\Big)^\frac12\Big(\int_{B_+(y)\setminus B(y)}|\nabla_\mu\delta^{0}\varphi_\mu^{y}|^2\Big)^\frac12,
\end{multline*}
and thus, by Lemma~\ref{lem:conv-cor-finite},
\begin{equation}\label{eq:replace-phimu-phi}
\lim_{\mu\downarrow0}\bigg|\int_{\partial B(y)}\Big(\varphi_\mu^0-\fint_{B(y)}\varphi_\mu^0\Big)\cdot\sigma_\mu(\delta^{0}\varphi_\mu^{y})\nu-\int_{\partial B(y)}\Big(\varphi^0-\fint_{B(y)}\varphi^0\Big)\cdot\sigma_\mu(\delta^{0}\varphi_\mu^{y})\nu\bigg|\,=\,0.
\end{equation}
Next, appealing to the Bogovskii operator as in the proof of Lemma~\ref{lem:trace-est}, cf.~\eqref{eq:phi-bound}, but recalling here $\Div(\varphi^0)=0$, we can construct a lifting $\theta\in H^1(B_+(y)\setminus B(y))^d$ such that
\[\theta=\varphi^0-\fint_{B(y)}\varphi^0\quad\text{on $\partial B(y)$},\quad\theta=0\quad\text{on $\partial B_+(y)$},\quad\Div(\theta)=0\quad\text{in $B_+(y)\setminus B(y)$}.\]
In these terms, an integration by parts and the equation for $\delta^0\varphi_\mu^y$ yield
\begin{equation*}
\int_{\partial B(y)}\Big(\varphi^0-\fint_{B(y)}\varphi^0\Big)\cdot\sigma_\mu(\delta^{0}\varphi_\mu^{y})\nu
\,=\,-\int_{B_+(y)\setminus B(y)}\mu\theta\cdot \delta^{0}\varphi_\mu^{y}+2\D(\theta):\D(\delta^{0}\varphi_\mu^{y}),
\end{equation*}
and thus, appealing to Lemma~\ref{lem:conv-cor-finite} to pass to the limit, and then integrating by parts to return to a boundary integral,
\begin{multline*}
\lim_{\mu\downarrow0}\int_{\partial B(y)}\Big(\varphi^0-\fint_{B(y)}\varphi^0\Big)\cdot\sigma_\mu(\delta^{0}\varphi_\mu^{y})\nu\\
\,=\,-2\int_{B_+(y)\setminus B(y)}\D(\theta):\D(\delta^{0}\varphi^{y})
\,=\,\int_{\partial B(y)}\Big(\varphi^0-\fint_{B(y)}\varphi^0\Big)\cdot\sigma(\delta^{0}\varphi^{y})\nu.
\end{multline*}
Together with~\eqref{eq:replace-phimu-phi}, this entails
\[\lim_{\mu\downarrow0}A_\mu^4=-\int_{\R^d}\bigg(\int_{\partial B(y)}\Big(\varphi^0-\fint_{B(y)}\varphi^0\Big)\cdot \sigma(\delta^{0}\varphi^{y})\nu\bigg)f_2(0,y)\,dy,\]
where we emphasize that the right-hand side is absolutely convergent by~\eqref{eq:decay-key-corr-Batch}. It remains to slightly reorganize the integrand. Testing the two-particle flow equation with the one-particle flow, and vice versa, we find
\[e\cdot\int_B\varphi^{0,y}
\,=\,\int_{\R^d}\nabla\varphi^0:\nabla\varphi^{0,y}
\,=\,e\cdot\int_B\varphi^0-\int_{\partial B(y)}\varphi^0\cdot\sigma(\varphi^{0,y})\nu,\]
and therefore
\begin{equation}\label{eq:identi-batch}
e\cdot\int_B\delta^y\varphi^{0}\,=\,-\int_{\partial B(y)}\varphi^0\cdot\sigma(\varphi^{0,y})\nu.
\end{equation}
Decomposing $\sigma(\varphi^{0,y})=\sigma(\varphi^y)+\sigma(\delta^0\varphi^{y})$,
and using the force balance identities, we deduce
\begin{equation*}\label{eq:decomp-Batchform}
e\cdot\int_B\delta^y\varphi^{0}
\,=\,
-\int_{\partial B(y)}\Big(\varphi^0-\fint_{B(y)}\varphi^0\Big)\cdot\sigma(\delta^0\varphi^y)\nu
-\int_{\partial B(y)}\varphi^0\cdot \sigma(\varphi^y)\nu.
\end{equation*}
The above limit of $A_\mu^4$ then becomes
\[\lim_{\mu\downarrow0}A_\mu^4=\int_{\R^d}\bigg(e\cdot\int_B\delta^y\varphi^{0}+\int_{\partial B(y)}\varphi^0\cdot \sigma(\varphi^y)\nu\bigg)f_2(0,y)\,dy,\]
while the absolute integrability of the integrand is still ensured by~\eqref{eq:decay-key-corr-Batch},
\begin{equation}\label{eq:decomp-Batchform}
\bigg|e\cdot\int_B\delta^y\varphi^{0}+\int_{\partial B(y)}\varphi^0\cdot \sigma(\varphi^y)\nu\bigg|\,\lesssim\,\langle y\rangle^{2(1-d)}.
\end{equation}
Finally, a similar argument allows us to pass to the limit in the expression for $A_\mu^5$, to the effect of
\[\lim_{\mu\downarrow0}A_\mu^5=-\int_{\R^d}\bigg(\int_{\partial B}\Big(\varphi^y-\fint_B\varphi^y\Big)\cdot\sigma(\varphi^0)\nu\bigg)h_2(0,y)\,dy,\]
where the integral is absolutely convergent provided $|h_2(0,y)|\lesssim\langle y\rangle^{-2-\beta}$, $\beta>0$.
Combining the different pieces, the claim~\eqref{eq:limTmu21} follows.

\medskip
\step2 Proof that
\begin{equation}\label{eq:estim-TL22}
\limsup_{\mu\downarrow0}|T_\mu^{2,2}|\,\lesssim\,\lambda_3^{\frac{\beta}{2+\beta}}.
\end{equation}
We appeal to our general strategy described in Section~\ref{sec:diag}.
First, we expand $\delta^{x,y}\varphi_\mu^\varnothing$ into reflection blocks, cf.~\eqref{eq:expand-deltxy},
\begin{equation*}
\delta^{x,y}\varphi_\mu^\varnothing
=\Jc^x_\mu(\Jc^y_{\mu;x}\varphi_\mu^{y}+\Gc^y_{\mu;x})
+\Jc^x_\mu\Jc^y_{\mu;x}(\Jc^x_{\mu;y}\varphi_\mu^{x,y}+\Gc^x_{\mu;y})+\Sym_{x,y},
\end{equation*}
where we recall that $\Sym_{x,y}$ stands for the sum of the preceding terms over permutations of the set $\{x,y\}$. Using diagrammatic notation, recalling the decay of the elementary kernels, cf.~Lemma~\ref{lem:decay-JG}, this reads
\begin{equation*}
\delta^{x,y}\varphi_\mu^\varnothing(0)
\,=\,
{\tiny\begin{tikzpicture}[baseline={([yshift=-.8ex]current bounding box.center)},scale=0.7]
\begin{scope}[every node/.style={circle,draw,inner sep=0pt,minimum size=8pt}]
    \node (1) at (0,0) {0};
    \node (2) at (0,1) {x};
    \node (3) at (0,2) {y};
\end{scope}
\begin{scope}[>={Stealth[black]},every edge/.style={draw=black,very thick}]
    \path [-] (1) edge node[bar]{\rule{1pt}{6pt}} (2);
    \path [-] (2) edge node[bar]{\rule{1pt}{6pt}} (3);
\end{scope}
\end{tikzpicture}}
+
{\tiny\begin{tikzpicture}[baseline={([yshift=-.8ex]current bounding box.center)},scale=0.7]
\begin{scope}[every node/.style={circle,draw,inner sep=0pt,minimum size=8pt}]
    \node (1) at (0,0) {x};
    \node (2) at (0,1) {y};
    \node (3) at (0,-1) {0};
\end{scope}
\begin{scope}[>={Stealth[black]},every edge/.style={draw=black,very thick}]
    \path [-] (1) edge[bend left=30] node[bar]{\rule{1pt}{6pt}} (2);
    \path [-] (1) edge[bend right=30] (2);
    \path [-] (1) edge node[bar]{\rule{1pt}{6pt}} (3);
\end{scope}
\end{tikzpicture}}
+
\Sym_{x,y},
\end{equation*}
where we emphasize that the key dependence rule~\eqref{eq:keyrule} is satisfied.
Inserting this into the definition of $T_\mu^{2,2}$, cf.~\eqref{eq:decomp-TL20}, we get
\begin{multline*}
T_\mu^{2,2}
=\int_{(\R^d)^2}\Big(e\cdot\int_{B} \Jc^x_\mu(\Jc^y_{\mu;x}\varphi_\mu^{y}+\Gc^y_{\mu;x})\Big)\Big(f_3(x,y,0)-\lambda f_2(x,y)\Big)\,dxdy\\
+\int_{(\R^d)^2}\Big(e\cdot\int_{B}\Jc^x_\mu\Jc^y_{\mu;x}(\Jc^x_{\mu;y}\varphi_\mu^{x,y}+\Gc^x_{\mu;y})\Big)\Big(f_3(x,y,0)-\lambda f_2(x,y)\Big)\,dxdy.
\end{multline*}
Next, we decompose the $3$-particle density in terms of correlation functions,
\begin{equation}\label{eq:3-cluster}
f_{3}(0,x,y)-\lambda f_2(x,y)\,=\,\lambda \Big(h_{2}(0,x)+h_{2}(0,y)\Big)+h_{3}(0,x,y),
\end{equation}
which leads us to
\begin{multline*}
T_\mu^{2,2}
=\lambda\int_{(\R^d)^2}\Big(e\cdot\int_{B} \Jc^x_\mu(\Jc^y_{\mu;x}\varphi_\mu^{y}+\Gc^y_{\mu;x})\Big) h_2(0,x)\,dxdy\\
+\int_{(\R^d)^2}\Big(e\cdot\int_{B} \Jc^x_\mu(\Jc^y_{\mu;x}\varphi_\mu^{y}+\Gc^y_{\mu;x})\Big)\Big(\lambda h_2(0,y)+h_3(x,y,0)\Big)\,dxdy\\
+\lambda\int_{(\R^d)^2}\Big(e\cdot\int_{B}\Jc^x_\mu\Jc^y_{\mu;x}(\Jc^x_{\mu;y}\varphi_\mu^{x,y}+\Gc^x_{\mu;y})\Big)h_2(0,x)\,dxdy\\
+\int_{(\R^d)^2}\Big(e\cdot\int_{B}\Jc^x_\mu\Jc^y_{\mu;x}(\Jc^x_{\mu;y}\varphi_\mu^{x,y}+\Gc^x_{\mu;y})\Big)\Big(\lambda h_2(0,y)+h_3(x,y,0)\Big)\,dxdy,
\end{multline*}
or diagrammatically,
\[T_\mu^{2,2}=[\lambda_3]\bigg(
{\tiny\begin{tikzpicture}[baseline={([yshift=-.8ex]current bounding box.center)},scale=0.7]
\begin{scope}[every node/.style={circle,fill,inner sep=0pt,minimum size=3pt}]
    \node (1) at (0,0) {};
    \node (2) at (0,1) {};
    \node (3) at (0,2) {};
\end{scope}
\begin{scope}[>={Stealth[black]},every edge/.style={draw=black,very thick}]
    \path [-] (1) edge[bend right=-40] node[bar]{\rule{1pt}{6pt}} (2);
    \path [-] (2) edge node[bar]{\rule{1pt}{6pt}}(3);
\end{scope}
\begin{scope}[>={Stealth[black]},every edge/.style={draw=black,very thick,dashed}]
    \path [-] (1) edge[bend right=40] (2);
\end{scope}
\end{tikzpicture}}
+
{\tiny\begin{tikzpicture}[baseline={([yshift=-.8ex]current bounding box.center)},scale=0.7]
\begin{scope}[every node/.style={circle,fill,inner sep=0pt,minimum size=3pt}]
    \node (1) at (0,0) {};
    \node (2) at (0,1) {};
    \node (3) at (0,2) {};
\end{scope}
\begin{scope}[>={Stealth[black]},every edge/.style={draw=black,very thick}]
    \path [-] (1) edge node[bar]{\rule{1pt}{6pt}} (2);
    \path [-] (2) edge node[bar]{\rule{1pt}{6pt}}(3);
\end{scope}
\begin{scope}[>={Stealth[black]},every edge/.style={draw=black,very thick,dashed}]
    \path [-] (1) edge[bend right=60] (3);
\end{scope}
\end{tikzpicture}}
+
{\tiny\begin{tikzpicture}[baseline={([yshift=-.8ex]current bounding box.center)},scale=0.7]
\begin{scope}[every node/.style={circle,fill,inner sep=0pt,minimum size=3pt}]
    \node (1) at (0,0) {};
    \node (2) at (0,1) {};
    \node (3) at (0,2) {};
\end{scope}
\begin{scope}[>={Stealth[black]},every edge/.style={draw=black,very thick}]
    \path [-] (1) edge[bend left=40] node[bar]{\rule{1pt}{6pt}} (2);
    \path [-] (2) edge[bend left=40] node[bar]{\rule{1pt}{6pt}} (3);
    \path [-] (2) edge[bend right=40] (3);
\end{scope}
\begin{scope}[>={Stealth[black]},every edge/.style={draw=black,very thick,dashed}]
    \path [-] (1) edge[bend right=40] (2);
\end{scope}
\end{tikzpicture}}
+
{\tiny\begin{tikzpicture}[baseline={([yshift=-.8ex]current bounding box.center)},scale=0.7]
\begin{scope}[every node/.style={circle,fill,inner sep=0pt,minimum size=3pt}]
    \node (1) at (0,0) {};
    \node (2) at (0,1) {};
    \node (3) at (0,2) {};
\end{scope}
\begin{scope}[>={Stealth[black]},every edge/.style={draw=black,very thick}]
    \path [-] (1) edge node[bar]{\rule{1pt}{6pt}} (2);
    \path [-] (2) edge[bend left=30] node[bar]{\rule{1pt}{6pt}}(3);
    \path [-] (2) edge[bend right=30] (3);
\end{scope}
\begin{scope}[>={Stealth[black]},every edge/.style={draw=black,very thick,dashed}]
    \path [-] (1) edge[bend right=60] (3);
\end{scope}
\end{tikzpicture}}\bigg).\]
Note that the first right-hand side term is not absolutely convergent in the limit $\mu\downarrow0$. 
Yet, our suitable expansion into reflection blocks allows us to appeal to cancellation rules: by \eqref{eq:ident-intJG-re} in Lemma~\ref{lem:cancel}, we find
\begin{equation*}
\int_{\R^d}\Jc_\mu^x(\Jc_{\mu;x}^y\varphi_\mu^{y}+\Gc_{\mu;x}^y)\,dy
\,=\,\sum_{i=1}^d\Big(\int_B(e+\mu\varphi_\mu^{0})\Big)_i(\Jc_\mu^x\varphi_{\mu,i}^x+\Gc_{\mu,i}^x),
\end{equation*}
so that the above reduces to
\begin{multline*}
T_\mu^{2,2}
=\lambda\sum_{i=1}^d\Big(\int_B(e+\mu\varphi_\mu^0)\Big)_i\int_{\R^d}\Big(e\cdot\int_{B} (\Jc^x_\mu\varphi_{\mu,i}^{x}+\Gc^x_{\mu,i})\Big) h_2(0,x)\,dx\\
+\int_{(\R^d)^2}\Big(e\cdot\int_{B} \Jc^x_\mu(\Jc^y_{\mu;x}\varphi_\mu^{y}+\Gc^y_{\mu;x})\Big)\Big(\lambda h_2(0,y)+h_3(x,y,0)\Big)\,dxdy\\
+\lambda\int_{(\R^d)^2}\Big(e\cdot\int_{B}\Jc^x_\mu\Jc^y_{\mu;x}(\Jc^x_{\mu;y}\varphi_\mu^{x,y}+\Gc^x_{\mu;y})\Big)h_2(0,x)\,dxdy\\
+\int_{(\R^d)^2}\Big(e\cdot\int_{B}\Jc^x_\mu\Jc^y_{\mu;x}(\Jc^x_{\mu;y}\varphi_\mu^{x,y}+\Gc^x_{\mu;y})\Big)\Big(\lambda h_2(0,y)+h_3(x,y,0)\Big)\,dxdy,
\end{multline*}
or diagrammatically,
\[T_\mu^{2,2}=[\lambda_3]\bigg(
{\tiny\begin{tikzpicture}[baseline={([yshift=-.8ex]current bounding box.center)},scale=0.7]
\begin{scope}[every node/.style={circle,fill,inner sep=0pt,minimum size=3pt}]
    \node (1) at (0,0) {};
    \node (2) at (0,1) {};
\end{scope}
\begin{scope}[>={Stealth[black]},every edge/.style={draw=black,very thick}]
    \path [-] (1) edge[bend left=40] node[bar]{\rule{1pt}{6pt}\hspace{1.5pt}\rule{1pt}{6pt}} (2);
\end{scope}
\begin{scope}[>={Stealth[black]},every edge/.style={draw=black,very thick,dashed}]
    \path [-] (1) edge[bend right=40] (2);
\end{scope}
\end{tikzpicture}}
+
{\tiny\begin{tikzpicture}[baseline={([yshift=-.8ex]current bounding box.center)},scale=0.7]
\begin{scope}[every node/.style={circle,fill,inner sep=0pt,minimum size=3pt}]
    \node (1) at (0,0) {};
    \node (2) at (0,1) {};
    \node (3) at (0,2) {};
\end{scope}
\begin{scope}[>={Stealth[black]},every edge/.style={draw=black,very thick}]
    \path [-] (1) edge node[bar]{\rule{1pt}{6pt}} (2);
    \path [-] (2) edge node[bar]{\rule{1pt}{6pt}}(3);
\end{scope}
\begin{scope}[>={Stealth[black]},every edge/.style={draw=black,very thick,dashed}]
    \path [-] (1) edge[bend right=60] (3);
\end{scope}
\end{tikzpicture}}
+
{\tiny\begin{tikzpicture}[baseline={([yshift=-.8ex]current bounding box.center)},scale=0.7]
\begin{scope}[every node/.style={circle,fill,inner sep=0pt,minimum size=3pt}]
    \node (1) at (0,0) {};
    \node (2) at (0,1) {};
    \node (3) at (0,2) {};
\end{scope}
\begin{scope}[>={Stealth[black]},every edge/.style={draw=black,very thick}]
    \path [-] (1) edge[bend left=40] node[bar]{\rule{1pt}{6pt}} (2);
    \path [-] (2) edge[bend left=40] node[bar]{\rule{1pt}{6pt}} (3);
    \path [-] (2) edge[bend right=40] (3);
\end{scope}
\begin{scope}[>={Stealth[black]},every edge/.style={draw=black,very thick,dashed}]
    \path [-] (1) edge[bend right=40] (2);
\end{scope}
\end{tikzpicture}}
+
{\tiny\begin{tikzpicture}[baseline={([yshift=-.8ex]current bounding box.center)},scale=0.7]
\begin{scope}[every node/.style={circle,fill,inner sep=0pt,minimum size=3pt}]
    \node (1) at (0,0) {};
    \node (2) at (0,1) {};
    \node (3) at (0,2) {};
\end{scope}
\begin{scope}[>={Stealth[black]},every edge/.style={draw=black,very thick}]
    \path [-] (1) edge node[bar]{\rule{1pt}{6pt}} (2);
    \path [-] (2) edge[bend left=40] node[bar]{\rule{1pt}{6pt}}(3);
    \path [-] (2) edge[bend right=30] (3);
\end{scope}
\begin{scope}[>={Stealth[black]},every edge/.style={draw=black,very thick,dashed}]
    \path [-] (1) edge[bend right=60] (3);
\end{scope}
\end{tikzpicture}}\bigg).\]
Finally, appealing to the decay properties of Lemma~\ref{lem:decay-JG}, to the local boundedness of finite-particle flows, cf.\@ Lemma~\ref{lem:conv-cor-finite}, and recalling the decay of correlations, we find that all the terms in the right-hand side are uniformly bounded (and actually converge) as~$\mu\downarrow0$.
More precisely, we can estimate
\begin{multline*}
|T_\mu^{2,2}|
\,\lesssim\,\lambda\int_{\R^d}K_2(x)\big(\lambda_2\wedge|x|^{-2-\beta}\big)\,dx
+\lambda\int_{(\R^d)^2}K_1(x)K_1(x-y)\big(\lambda_2\wedge|y|^{-2-\beta}\big)\,dxdy\\
+\lambda\int_{(\R^d)^2}K_1(x)K_1(x-y)K_0(x-y)\big(\lambda_2\wedge|x|^{-2-\beta}\big)\,dxdy\\
+\int_{(\R^d)^2}K_1(x)K_1(x-y)K_0(x-y)\big(\lambda_2\wedge|y|^{-2-\beta}\big)\,dxdy.
\end{multline*}
Directly estimating the integrals, or using the diagrammatic calculation rules of Section~\ref{sec:diag}, we obtain
\begin{equation*}
|T_\mu^{2,2}|
\,\lesssim\,
\int_{(\R^d)^2}K_2(x)\,\big(\lambda_3\wedge\langle x\rangle^{-2-\beta}\big)\,\log(2+\sqrt\mu|x|)\,dx
\,\equiv\,[\lambda_3]\Lc{\tiny\begin{tikzpicture}[baseline={([yshift=-.8ex]current bounding box.center)},scale=0.7]
\begin{scope}[every node/.style={circle,fill,inner sep=0pt,minimum size=3pt}]
    \node (1) at (0,0) {};
    \node (2) at (0,1) {};
\end{scope}
\begin{scope}[>={Stealth[black]},every edge/.style={draw=black,very thick}]
    \path [-] (1) edge[bend left=40] node[bar]{\rule{1pt}{6pt}\hspace{1.5pt}\rule{1pt}{6pt}} (2);
\end{scope}
\begin{scope}[>={Stealth[black]},every edge/.style={draw=black,very thick,dashed}]
    \path [-] (1) edge[bend right=40] (2);
\end{scope}
\end{tikzpicture}}
\,\lesssim\,\lambda_3^\frac\beta{2+\beta}+o(1),
\end{equation*}
and the claim~\eqref{eq:estim-TL22} follows.
\end{proof}

\subsection{Estimate of third cluster term}\label{sec:subsec-T3}
This section is devoted to the proof of the following error estimate for the third cluster term $T_\mu^3$ in the decomposition~\eqref{eq:VL-decomp}.

\begin{lem}\label{lem:Tmu3}
The third term $T_\mu^3$ in~\eqref{eq:VL-decomp} satisfies
\[\limsup_{\mu\downarrow0}|T_\mu^3|\,\lesssim\,\lambda_3^\frac{\beta}{2+\beta}.\]
\end{lem}

\begin{proof}
By definition of $T_\mu^3$, writing $\mathds1_\Ic=\sum_{x\in\Pc}\mathds1_{B(x)}$, and distinguishing the different intersection cases, we can decompose
\begin{multline*}
T_\mu^3\,=\,
\frac12e\cdot\E\Big[\sum_{x,y,z\in\Pc}^{\ne}\mathds1_{B(x)}\,\delta^{x,y,z}\varphi_\mu^\varnothing\Big]
+\frac1{3!}e\cdot\E\Big[\sum_{x,y,z,w\in\Pc}^{\ne}\mathds1_{B(w)}\,\delta^{x,y,z}\varphi_\mu^\varnothing\Big]\\[-3mm]
-\frac1{3!}\lambda|B|e\cdot\E\Big[\sum_{x,y,z\in\Pc}^{\ne}\delta^{x,y,z}\varphi_\mu^\varnothing\Big],
\end{multline*}
and thus, in terms of density functions, using translation invariance,
\begin{equation}\label{eq:decomp-TL3}
T_\mu^3=T_\mu^{3,1}+T_\mu^{3,2},
\end{equation}
where
\begin{eqnarray*}
T_\mu^{3,1}&:=&\frac12\int_{(\R^d)^2} \Big(e\cdot\int_B\delta^{0,y,z}\varphi_\mu^\varnothing\Big)\,f_3(0,y,z)\,dydz,\\
T_\mu^{3,2}&:=&\frac1{3!}\int_{(\R^d)^3}\Big(e\cdot \int_B\delta^{x,y,z}\varphi_\mu^\varnothing\Big)\,\Big(f_4(x,y,z,0)-\lambda f_3(x,y,z)\Big)\,dxdydz.
\end{eqnarray*}
We split the proof into two steps, separately analyzing the two contributions.

\medskip
\step1 Proof that
\begin{equation}\label{eq:bound-T31}
\limsup_{\mu\downarrow0}|T_\mu^{3,1}|\,\lesssim\,\lambda_3^{\frac\beta{2+\beta}}.
\end{equation}
Decomposing $\delta^{0,y,z}\varphi_\mu^\varnothing=\delta^{y,z}\varphi_\mu^0-\delta^{y,z}\varphi_\mu^\varnothing$, we can write
\[T_\mu^{3,1}\,=\,T_\mu^{3,1,1}- T_\mu^{3,1,2},\]
in terms of
\begin{eqnarray*}
T_\mu^{3,1,1}&:=&\frac12\int_{(\R^d)^2} \Big(e\cdot\int_B\delta^{y,z}\varphi_\mu^0\Big)\,f_3(0,y,z)\,dydz,\\
T_\mu^{3,1,2}&:=&\frac12\int_{(\R^d)^2} \Big(e\cdot\int_B\delta^{y,z}\varphi_\mu^\varnothing\Big)\,f_3(0,y,z)\,dydz.
\end{eqnarray*}
As we shall see, both contributions are divergent as $\mu\downarrow0$, but their divergences compensate exactly: more precisely, we claim that
\begin{eqnarray}
\limsup_{\mu\downarrow0}\Big| T_\mu^{3,1,1}-\tfrac1\mu\lambda^3|B|^3\Big|&\lesssim&\lambda_3^\frac\beta{2+\beta},\label{eq:todo-tildeT31}\\
\limsup_{\mu\downarrow0}\Big| T_\mu^{3,1,2}-\tfrac1\mu\lambda^3|B|^3\Big|&\lesssim&\lambda_3^\frac\beta{2+\beta},\nonumber
\end{eqnarray}
which will indeed conclude the desired bound~\eqref{eq:bound-T31}.
Both estimates on $ T_\mu^{3,1,1}$ and $ T_\mu^{3,1,2}$ can be proven similarly, and we focus on the first one for brevity.
We split the proof into four substeps.

\medskip
\substep{1.1} Decomposition of $T_\mu^{3,1,1}$.\\
Following again our general strategy described in Section~\ref{sec:diag}, we first expand $\delta^{y,z}\varphi_\mu^0$ into reflection blocks.
Successive applications of Lemma~\ref{lem:element-dec} yield
\begin{eqnarray*}
\delta^{y,z}\varphi_\mu^0
&=&
\Jc_{\mu;0}^y\delta^z\varphi_\mu^{0,y}
+\Sym_{y,z}\\
&=&
\Jc_{\mu;0}^y\big(\delta^{0,z}\varphi_\mu^{y}+\delta^{z}\varphi_\mu^{y}\big)
+\Sym_{y,z}\\
&=&
\Jc_{\mu;0}^y\Big(\Jc_{\mu;y}^0\delta^z\varphi_\mu^{0,y}
+\Jc_{\mu;y}^z\delta^0\varphi_\mu^{y,z}
+\Jc_{\mu;y}^z\varphi_\mu^{y,z}
+\Gc_{\mu;y}^z\Big)
+\Sym_{y,z}\\
&=&
\Jc_{\mu;0}^y\Jc_{\mu;y}^0\big(\Jc_{\mu;0,y}^z\varphi_\mu^{0,y,z}+\Gc_{\mu;0,y}^z\big)
+\Jc_{\mu;0}^y\Jc_{\mu;y}^z\big(\Jc_{\mu;y,z}^0\varphi_\mu^{0,y,z}+\Gc_{\mu;y,z}^0\big)\\
&&+\Jc_{\mu;0}^y\big(\Jc_{\mu;y}^z\varphi_\mu^{y,z}+\Gc_{\mu;y}^z\big)
+\Sym_{y,z},
\end{eqnarray*}
where we recall that $\Sym_{y,z}$ stands for the sum of the preceding terms over permutations of the set~$\{y,z\}$.
Further decomposing $\varphi_\mu^{0,y,z}=\varphi_\mu^{z}+\delta^0\varphi_\mu^z+\delta^y\varphi_\mu^{0,z}$ in the first right-hand side term, and  $\varphi_\mu^{y,z}=\varphi_\mu^z+\delta^{y}\varphi_\mu^z$ in the third one, and appealing once more to Lemma~\ref{lem:element-dec} to express finite-particle flows, we are led to
\begin{equation}\label{eq:decomp-deltayzvarphi0}
\delta^{y,z}\varphi_\mu^0
\,=\,\sum_{i=1}^6S_i^{0,y,z}+\Sym_{y,z}
\end{equation}
in terms of
\begin{eqnarray*}
S_1^{0,y,z}&:=&\Jc_{\mu;0}^y\Jc_{\mu;y}^0\Jc_{\mu;0,y}^z\big(\Jc_{\mu;0,z}^y\varphi_\mu^{0,y,z}+\Gc_{\mu;0,z}^y\big),\\
S_2^{0,y,z}&:=&\Jc_{\mu;0}^y\Jc_{\mu;y}^0\Jc_{\mu;0,y}^z\big(\Jc_{\mu;z}^0\varphi_\mu^{0,z}+\Gc_{\mu;z}^0\big),\\
S_3^{0,y,z}&:=&\Jc_{\mu;0}^y\Jc_{\mu;y}^z\big(\Jc_{\mu;y,z}^0\varphi_\mu^{0,y,z}+\Gc_{\mu;y,z}^0\big),\\
S_4^{0,y,z}&:=&\Jc_{\mu;0}^y\Jc_{\mu;y}^0\big(\Jc_{\mu;0,y}^z\varphi_\mu^{z}+\Gc_{\mu;0,y}^z\big),\\
S_5^{0,y,z}&:=&\Jc_{\mu;0}^y\Jc_{\mu;y}^z\big(\Jc_{\mu;z}^y\varphi_\mu^{y,z}+\Gc_{\mu;z}^y\big),\\
S_6^{0,y,z}&:=&\Jc_{\mu;0}^y\big(\Jc_{\mu;y}^z\varphi_\mu^{z}+\Gc_{\mu;y}^z\big).
\end{eqnarray*}
Using diagrammatic notation, recalling the decay of the elementary kernels, cf.~Lemma~\ref{lem:decay-JG}, this reads
\begin{equation*}
\delta^{y,z}\varphi_\mu^0
\,=\,
{\tiny\begin{tikzpicture}[baseline={([yshift=-.8ex]current bounding box.center)},scale=0.7]
\begin{scope}[every node/.style={circle,draw,inner sep=0pt,minimum size=8pt}]
    \node (1) at (-0.6,0) {0};
    \node (2) at (0,1.2) {y};
    \node (3) at (0.6,0) {z};
\end{scope}
\begin{scope}[>={Stealth[black]},every edge/.style={draw=black,very thick}]
    \path [-] (1) edge[bend right=20] (2);
    \path [-] (1) edge[bend left=30] node[bar]{\rule{1pt}{6pt}} (2);
    \path [-] (1) edge (3);
    \path [-] (2) edge node[bar]{\rule{1pt}{6pt}} (3);
\end{scope}
\end{tikzpicture}}
+
{\tiny\begin{tikzpicture}[baseline={([yshift=-.8ex]current bounding box.center)},scale=0.7]
\begin{scope}[every node/.style={circle,draw,inner sep=0pt,minimum size=8pt}]
    \node (1) at (0,0) {0};
    \node (2) at (0,1) {y};
    \node (3) at (0,-1) {z};
\end{scope}
\begin{scope}[>={Stealth[black]},every edge/.style={draw=black,very thick}]
    \path [-] (1) edge[bend right=30] (2);
    \path [-] (1) edge[bend left=30] node[bar]{\rule{1pt}{6pt}} (2);
    \path [-] (1) edge[bend left=30] (3);
    \path [-] (1) edge[bend right=30] node[bar]{\rule{1pt}{6pt}} (3);
\end{scope}
\end{tikzpicture}}
+
{\tiny\begin{tikzpicture}[baseline={([yshift=-.8ex]current bounding box.center)},scale=0.7]
\begin{scope}[every node/.style={circle,draw,inner sep=0pt,minimum size=8pt}]
    \node (1) at (-0.6,0) {0};
    \node (2) at (0,1.2) {y};
    \node (3) at (0.6,0) {z};
\end{scope}
\begin{scope}[>={Stealth[black]},every edge/.style={draw=black,very thick}]
    \path [-] (1) edge node[bar]{\rule{1pt}{6pt}} (2);
    \path [-] (2) edge (3);
    \path [-] (1) edge node[bar]{\rule{1pt}{6pt}} (3);
\end{scope}
\end{tikzpicture}}
+
{\tiny\begin{tikzpicture}[baseline={([yshift=-.8ex]current bounding box.center)},scale=0.7]
\begin{scope}[every node/.style={circle,draw,inner sep=0pt,minimum size=8pt}]
    \node (1) at (0,0) {0};
    \node (2) at (0,1) {y};
    \node (3) at (0,-1) {z};
\end{scope}
\begin{scope}[>={Stealth[black]},every edge/.style={draw=black,very thick}]
    \path [-] (1) edge[bend left=30] node[bar]{\rule{1pt}{6pt}} (2);
    \path [-] (1) edge[bend right=30] (2);
    \path [-] (1) edge node[bar]{\rule{1pt}{6pt}} (3);
\end{scope}
\end{tikzpicture}}
+
{\tiny\begin{tikzpicture}[baseline={([yshift=-.8ex]current bounding box.center)},scale=0.7]
\begin{scope}[every node/.style={circle,draw,inner sep=0pt,minimum size=8pt}]
    \node (3) at (0,-1) {0};
    \node (1) at (0,0) {y};
    \node (2) at (0,1) {z};
\end{scope}
\begin{scope}[>={Stealth[black]},every edge/.style={draw=black,very thick}]
    \path [-] (1) edge[bend left=30] node[bar]{\rule{1pt}{6pt}} (2);
    \path [-] (1) edge[bend right=30] (2);
    \path [-] (1) edge node[bar]{\rule{1pt}{6pt}} (3);
\end{scope}
\end{tikzpicture}}
+
{\tiny\begin{tikzpicture}[baseline={([yshift=-.8ex]current bounding box.center)},scale=0.7]
\begin{scope}[every node/.style={circle,draw,inner sep=0pt,minimum size=8pt}]
    \node (3) at (0,-1) {0};
    \node (1) at (0,0) {y};
    \node (2) at (0,1) {z};
\end{scope}
\begin{scope}[>={Stealth[black]},every edge/.style={draw=black,very thick}]
    \path [-] (1) edge node[bar]{\rule{1pt}{6pt}} (2);
    \path [-] (1) edge node[bar]{\rule{1pt}{6pt}} (3);
\end{scope}
\end{tikzpicture}}
+
\Sym_{y,z},
\end{equation*}
where we emphasize that the expansion has precisely been chosen so as to satisfy the key dependence rule~\eqref{eq:keyrule}: cut edges separate the background variables in the elementary kernels.
Inserting this expansion into the definition of $T_\mu^{3,1,1}$, we get
\begin{equation}\label{eq:dec-tilTmu}
T_\mu^{3,1,1}\,=\,\sum_{i=1}^6\int_{(\R^d)^2} \Big(e\cdot\int_BS_i^{0,y,z}\Big)\,f_3(0,y,z)\,dydz.
\end{equation}
By the decay of the elementary kernels, cf.~Lemma~\ref{lem:decay-JG}, we note that the first three integrals ($1\le i\le3$) are uniformly bounded as $\mu\downarrow0$ in dimension $d>2$: more precisely, we find
\begin{multline*}
\Big|\int_{(\R^d)^2} \Big(e\cdot\int_BS_1^{0,y,z}\Big)\,f_3(0,y,z)\,dydz\Big|\\
\,\lesssim\,
\lambda_3\int_{(\R^d)^2}K_0(y)K_1(y)K_0(z)K_1(y-z)\,dydz
\,\equiv\,
[\lambda_3]
{\tiny\begin{tikzpicture}[baseline={([yshift=-.8ex]current bounding box.center)},scale=0.7]
\begin{scope}[every node/.style={circle,fill,inner sep=0pt,minimum size=3pt}]
    \node (1) at (-0.5,0) {};
    \node (2) at (0,1) {};
    \node (3) at (0.5,0) {};
\end{scope}
\begin{scope}[>={Stealth[black]},every edge/.style={draw=black,very thick}]
    \path [-] (1) edge[bend right=20] (2);
    \path [-] (1) edge[bend left=40] node[bar]{\rule{1pt}{6pt}} (2);
    \path [-] (1) edge (3);
    \path [-] (2) edge node[bar]{\rule{1pt}{6pt}} (3);
\end{scope}
\end{tikzpicture}}
\,\lesssim\,\lambda_3,
\end{multline*}
and similarly
\begin{eqnarray}
\Big|\int_{(\R^d)^2} \Big(e\cdot\int_BS_2^{0,y,z}\Big)\,f_3(0,y,z)\,dydz\Big|&\lesssim&\lambda_3,\nonumber\\
\Big|\int_{(\R^d)^2} \Big(e\cdot\int_BS_3^{0,y,z}\Big)\,f_3(0,y,z)\,dydz\Big|&\lesssim&\lambda_3.\label{eq:est-S3yz}
\end{eqnarray}
The last three integrals in~\eqref{eq:dec-tilTmu}, on the contrary, are a priori unbounded in the limit~$\mu\downarrow0$. We shall exploit cancellations to extract their explicit diverging contribution and bound the remainder. For that purpose, we appeal to the general strategy of Section~\ref{sec:diag}: we decompose the $3$-particle density $f_3$ in terms of correlation functions, cf.~\eqref{eq:3-cluster}, and then take advantage of the explicit cancellation rules of Lemma~\ref{lem:cancel}. The needed estimates are performed in the next three substeps.

\medskip
\substep{1.2} Bound on $S_4^{0,y,z}$:
\[\limsup_{\mu\downarrow0}\Big|\int_{(\R^d)^2} \Big(e\cdot\int_BS_4^{0,y,z}\Big)\,f_3(0,y,z)\,dydz\Big|\,\lesssim\,\lambda_3^\frac{1+\beta}{2+\beta}.\]
Using~\eqref{eq:3-cluster}, we can decompose for instance
\begin{multline}\label{eq:decomp-S4}
\int_{(\R^d)^2} \Big(e\cdot\int_BS_4^{0,y,z}\Big)\,f_3(0,y,z)\,dydz
\,=\,\lambda\int_{(\R^d)^2} \Big(e\cdot\int_BS_4^{0,y,z}\Big) f_2(0,y)\,dydz\\
+\int_{(\R^d)^2} \Big(e\cdot\int_BS_4^{0,y,z}\Big)\Big(\lambda h_2(0,z)+\lambda h_2(y,z)+h_3(0,y,z)\Big)\,dydz,
\end{multline}
Note that the first right-hand side term involves $\int_{\R^d}S_4^{0,y,z}\,dz$, for which, by definition of~$S_4^{0,y,z}$, Lemma~\ref{lem:cancel} yields
\begin{eqnarray*}
\int_{\R^d}S_4^{0,y,z}\,dz
&=&\Jc_{\mu;0}^y\Jc_{\mu;y}^0\int_{\R^d}\big(\Jc_{\mu;0,y}^z\varphi_\mu^{z}+\Gc_{\mu;0,y}^z\big)\,dz\\
&=&\sum_{i=1}^d\Big(\int_B(e+\mu\varphi_{\mu}^0)\Big)_i\Jc_{\mu;0}^y(\Jc_{\mu;y}^0\varphi_{\mu,i}^{0,y}+\Gc_{\mu,i;y}^0).
\end{eqnarray*}
Using diagrammatic notation, by the decay of the elementary operators, cf.~Lemma~\ref{lem:decay-JG}, this means
\begin{equation*}
\Big|\int_{\R^d}\Big(e\cdot\int_BS_4^{0,y,z}\Big)\,dz\Big|
\,\lesssim\, 
{\tiny\begin{tikzpicture}[baseline={([yshift=-.8ex]current bounding box.center)},scale=0.7]
\begin{scope}[every node/.style={circle,draw,inner sep=0pt,minimum size=8pt}]
    \node (1) at (0,0) {0};
    \node (2) at (0,1) {y};
\end{scope}
\begin{scope}[>={Stealth[black]},every edge/.style={draw=black,very thick}]
    \path [-] (1) edge[bend right=30] node[bar]{\rule{1pt}{6pt}} (2);
    \path [-] (1) edge[bend left=30] node[bar]{\rule{1pt}{6pt}} (2);
\end{scope}
\end{tikzpicture}}
\end{equation*}
Inserting this estimate into~\eqref{eq:decomp-S4}, and using the decay of correlations for the remaining terms, we are led to
\begin{equation*}
\Big|\int_{(\R^d)^2} \Big(e\cdot\int_BS_4^{0,y,z}\Big)\,f_3(0,y,z)\,dydz\Big|
\,\lesssim\,[\lambda_3]\bigg(
{\tiny\begin{tikzpicture}[baseline={([yshift=-.8ex]current bounding box.center)},scale=0.7]
\begin{scope}[every node/.style={circle,fill,inner sep=0pt,minimum size=3pt}]
    \node (1) at (0,0) {};
    \node (2) at (0,1) {};
\end{scope}
\begin{scope}[>={Stealth[black]},every edge/.style={draw=black,very thick}]
    \path [-] (1) edge[bend right=40] node[bar]{\rule{1pt}{6pt}} (2);
    \path [-] (1) edge[bend left=40] node[bar]{\rule{1pt}{6pt}} (2);
\end{scope}
\end{tikzpicture}}
+
{\tiny\begin{tikzpicture}[baseline={([yshift=-.8ex]current bounding box.center)},scale=0.7]
\begin{scope}[every node/.style={circle,fill,inner sep=0pt,minimum size=3pt}]
    \node (1) at (0,0) {};
    \node (2) at (0,1) {};
    \node (3) at (0,-1) {};
\end{scope}
\begin{scope}[>={Stealth[black]},every edge/.style={draw=black,very thick}]
    \path [-] (1) edge[bend right=40] (2);
    \path [-] (1) edge[bend left=40] node[bar]{\rule{1pt}{6pt}} (2);
    \path [-] (1) edge[bend right=40] node[bar]{\rule{1pt}{6pt}} (3);
\end{scope}
\begin{scope}[>={Stealth[black]},every edge/.style={draw=black,very thick,dashed}]
    \path [-] (1) edge[bend left=40] (3);
\end{scope}
\end{tikzpicture}}
+
{\tiny\begin{tikzpicture}[baseline={([yshift=-.8ex]current bounding box.center)},scale=0.7]
\begin{scope}[every node/.style={circle,fill,inner sep=0pt,minimum size=3pt}]
    \node (1) at (-0.5,0) {};
    \node (2) at (0,1) {};
    \node (3) at (0.5,0) {};
\end{scope}
\begin{scope}[>={Stealth[black]},every edge/.style={draw=black,very thick}]
    \path [-] (1) edge[bend right=20] (2);
    \path [-] (1) edge[bend left=40] node[bar]{\rule{1pt}{6pt}} (2);
    \path [-] (1) edge node[bar]{\rule{1pt}{6pt}} (3);
\end{scope}
\begin{scope}[>={Stealth[black]},every edge/.style={draw=black,very thick,dashed}]
    \path [-] (2) edge (3);
\end{scope}
\end{tikzpicture}}
\,\bigg).
\end{equation*}
Estimating the integrals, using the diagrammatic calculation rules of Section~\ref{sec:diag}, the claim follows.

\medskip
\substep{1.3} Bound on $S_5^{0,y,z}$:
\[\limsup_{\mu\downarrow0}\Big|\int_{(\R^d)^2} \Big(e\cdot\int_BS_5^{0,y,z}\Big)\,f_3(0,y,z)\,dydz\Big|\,\lesssim\,\lambda_3^\frac{1+\beta}{2+\beta}.\]
Using~\eqref{eq:3-cluster}, we can decompose for instance
\begin{multline}\label{eq:dec-S5}
\int_{(\R^d)^2} \Big(e\cdot\int_BS_5^{0,y,z}\Big)\,f_3(0,y,z)\,dydz
\,=\,\lambda\int_{(\R^d)^2} \Big(e\cdot\int_BS_5^{0,y,z}\Big) f_2(y,z)\,dydz\\
+\int_{(\R^d)^2} \Big(e\cdot\int_BS_5^{0,y,z}\Big)\Big(\lambda h_2(0,z)+\lambda h_2(0,y)+h_3(0,y,z)\Big)\,dydz.
\end{multline}
By translation invariance of $f_2$, the first right-hand side term can be written as
\[\lambda\int_{(\R^d)^2} \Big(e\cdot\int_BS_5^{0,y,z}\Big) f_2(y,z)\,dydz\,=\,\lambda\int_{(\R^d)^2} \Big(e\cdot\int_BS_5^{0,y,y+z}\Big) f_2(0,z)\,dydz,\]
and we note that the definition of $S_5$ allows us to write $S_5^{0,y,y+z}=\Jc_{\mu;0}^y\zeta_{0,z}^y$ in terms of
\[\zeta_{0,z}^y:=\zeta_{0,z}(\cdot-y),\qquad\zeta_{0,z}:=\Jc_{\mu;0}^z(\Jc_{\mu;z}^0\varphi_\mu^{0,z}+\Gc_{\mu;z}^0).\]
Applying \eqref{eq:ident-intJ} in Lemma~\ref{lem:cancel} then yields
\begin{multline*}
\lambda\int_{(\R^d)^2} \Big(e\cdot\int_BS_5^{0,y,z}\Big) f_2(y,z)\,dydz
\,=\,\lambda\int_{\R^d} \Big(\int_B\zeta_{0,z}\Big)_i\Big(e\cdot\int_B(e_i+\mu\varphi_{\mu,i}^0)\Big) f_2(0,z)\,dz\\
\,=\,\lambda\int_{\R^d} \Big(\int_B\Jc_{\mu;0}^z(\Jc_{\mu;z}^0\varphi_\mu^{0,z}+\Gc_{\mu;z}^0)\Big)_i\Big(e\cdot\int_B(e_i+\mu\varphi_{\mu,i}^0)\Big) f_2(0,z)\,dz.
\end{multline*}
Inserting this into~\eqref{eq:dec-S5}, recalling the decay of the elementary operators and the decay of correlations, we get
\begin{equation*}
\Big|\int_{(\R^d)^2} \Big(e\cdot\int_BS_5^{0,y,z}\Big)\,f_3(0,y,z)\,dydz\Big|
\,\le\,[\lambda_3]\Big(
{\tiny\begin{tikzpicture}[baseline={([yshift=-.8ex]current bounding box.center)},scale=0.7]
\begin{scope}[every node/.style={circle,fill,inner sep=0pt,minimum size=3pt}]
    \node (1) at (0,0) {};
    \node (2) at (0,1) {};
\end{scope}
\begin{scope}[>={Stealth[black]},every edge/.style={draw=black,very thick}]
    \path [-] (1) edge[bend right=40] node[bar]{\rule{1pt}{6pt}} (2);
    \path [-] (1) edge[bend left=40] node[bar]{\rule{1pt}{6pt}} (2);
\end{scope}
\end{tikzpicture}}
+
{\tiny\begin{tikzpicture}[baseline={([yshift=-.8ex]current bounding box.center)},scale=0.7]
\begin{scope}[every node/.style={circle,fill,inner sep=0pt,minimum size=3pt}]
    \node (1) at (0,0) {};
    \node (2) at (0,1) {};
    \node (3) at (0,-1) {};
\end{scope}
\begin{scope}[>={Stealth[black]},every edge/.style={draw=black,very thick}]
    \path [-] (1) edge[bend right=40] (2);
    \path [-] (1) edge[bend left=40] node[bar]{\rule{1pt}{6pt}} (2);
    \path [-] (1) edge[bend right=40] node[bar]{\rule{1pt}{6pt}} (3);
\end{scope}
\begin{scope}[>={Stealth[black]},every edge/.style={draw=black,very thick,dashed}]
    \path [-] (1) edge[bend left=40] (3);
\end{scope}
\end{tikzpicture}}
+
{\tiny\begin{tikzpicture}[baseline={([yshift=-.8ex]current bounding box.center)},scale=0.7]
\begin{scope}[every node/.style={circle,fill,inner sep=0pt,minimum size=3pt}]
    \node (1) at (-0.5,0) {};
    \node (2) at (0,1) {};
    \node (3) at (0.5,0) {};
\end{scope}
\begin{scope}[>={Stealth[black]},every edge/.style={draw=black,very thick}]
    \path [-] (1) edge[bend right=20] (2);
    \path [-] (1) edge[bend left=40] node[bar]{\rule{1pt}{6pt}} (2);
    \path [-] (1) edge node[bar]{\rule{1pt}{6pt}} (3);
\end{scope}
\begin{scope}[>={Stealth[black]},every edge/.style={draw=black,very thick,dashed}]
    \path [-] (2) edge (3);
\end{scope}
\end{tikzpicture}}
\,\Big).
\end{equation*}
and the claim follows again using the diagrammatic calculation rules.

\medskip
\substep{1.4} Bound on $S_6^{0,y,z}$:
\[\limsup_{\mu\downarrow0}\bigg|\int_{(\R^d)^2} \Big(e\cdot\int_BS_6^{0,y,z}\Big)\,f_3(0,y,z)\,dydz-\tfrac1\mu\lambda^3|B|^3\bigg|\,\lesssim\,\lambda_3^\frac{\beta}{2+\beta}.\]
Using the cluster expansion~\eqref{eq:3-cluster}, we can decompose for instance
\begin{multline}\label{eq:dec-S6}
\int_{(\R^d)^2} \Big(e\cdot\int_BS_6^{0,y,z}\Big)\,f_3(0,y,z)\,dydz\\
\,=\,\lambda\int_{(\R^d)^2} \Big(e\cdot\int_BS_6^{0,y,z}\Big) h_2(0,y)\,dydz
+\lambda\int_{(\R^d)^2} \Big(e\cdot\int_BS_6^{0,y,z}\Big) f_2(y,z)\,dydz\\
+\int_{(\R^d)^2} \Big(e\cdot\int_BS_6^{0,y,z}\Big)\Big(\lambda h_2(0,z)+h_3(0,y,z)\Big)\,dydz.
\end{multline}
By definition of $S_6^{0,y,z}$, using~\eqref{eq:ident-intJG-re} in Lemma~\ref{lem:cancel} to perform the integral with respect to~$z$, the first right-hand side term is equal to
\begin{multline*}
\lambda\int_{(\R^d)^2} \Big(e\cdot\int_BS_6^{0,y,z}\Big) h_2(0,y)\,dydz\\
\,=\,
\lambda\int_{\R^d} \Big(e\cdot\int_B(\Jc_{\mu;0}^y\varphi_{\mu,i}^y+\Gc_{\mu,i;0}^y)\Big)\Big(\int_B(e+\mu\varphi_{\mu}^0)\Big)_i h_2(0,y)\,dy.
\end{multline*}
By translation invariance of $f_2$, the second right-hand side term in~\eqref{eq:dec-S6} can be written as
\begin{equation*}
\lambda\int_{(\R^d)^2} \Big(e\cdot\int_BS_6^{0,y,z}\Big) f_2(y,z)\,dydz
\,=\,\lambda\int_{(\R^d)^2} \Big(e\cdot\int_B\Jc_{\mu;0}^y\zeta_{0,z}^y\Big) f_2(0,z)\,dydz,
\end{equation*}
in terms of
\[\zeta_{0,z}^y:=\zeta_{0,z}(\cdot-y),\qquad \zeta_{0,z}:=\Jc_{\mu;0}^{z}\varphi_\mu^{z}+\Gc_{\mu;0}^{z},\]
and therefore, applying \eqref{eq:ident-intJ} in Lemma~\ref{lem:cancel},
\begin{multline*}
\lambda\int_{(\R^d)^2} \Big(e\cdot\int_BS_6^{0,y,z}\Big) f_2(y,z)\,dydz\\
\,=\,\lambda\sum_{i=1}^d\int_{\R^d} \Big(\int_B(\Jc_{\mu;0}^{z}\varphi_\mu^{z}+\Gc_{\mu;0}^z)\Big)_i\Big(e\cdot\int_B(e_i+\mu\varphi_{\mu,i}^0)\Big) f_2(0,z)\,dz.
\end{multline*}
Further decomposing $f_2(0,z)=\lambda^2+h_2(0,z)$, and using Lemma~\ref{lem:cancel} in the form
\[\int_{\R^d}(\Jc_{\mu;0}^{z}\varphi_\mu^{z}+\Gc_{\mu;0}^z)\,dz\,=\,\tfrac1\mu\sum_{i=1}^d\Big(\int_B(e+\mu\varphi_{\mu}^0)\Big)_i(e_i+\mu\varphi_{\mu,i}^0),\]
we are led to
\begin{multline*}
\lambda\int_{(\R^d)^2} \Big(e\cdot\int_BS_6^{0,y,z}\Big) f_2(y,z)\,dydz\\
\,=\,\tfrac1\mu\lambda^3\sum_{i,j=1}^d\Big(\int_B(e+\mu\varphi_\mu^0)\Big)_i\Big(\int_B(e_i+\mu\varphi_{\mu,i}^0)\Big)_j\Big(e\cdot\int_B(e_j+\mu\varphi_{\mu,j}^0)\Big)\\
+\lambda\sum_{i=1}^d\int_{\R^d} \Big(\int_B(\Jc_{\mu;0}^{z}\varphi_\mu^{z}+\Gc_{\mu;0}^z)\Big)_i\Big(e\cdot\int_B(e_i+\mu\varphi_{\mu,i}^0)\Big) h_2(0,z)\,dz.
\end{multline*}
Note that we have extracted in this way the explicit diverging contribution $\tfrac1\mu\lambda^3|B|^3$.
With these reformulations of the first two terms in~\eqref{eq:dec-S6}, recalling the decay of the elementary kernels and the decay of correlations, we obtain
\begin{equation*}
\bigg|\int_{(\R^d)^2} \Big(e\cdot\int_BS_6^{0,y,z}\Big)\,f_3(0,y,z)\,dydz-\tfrac1\mu\lambda^3|B|^3\bigg|
\,\lesssim\,[\lambda_3]\Big(1+
{\tiny\begin{tikzpicture}[baseline={([yshift=-.8ex]current bounding box.center)},scale=0.7]
\begin{scope}[every node/.style={circle,fill,inner sep=0pt,minimum size=3pt}]
    \node (1) at (0,0) {};
    \node (2) at (0,1) {};
\end{scope}
\begin{scope}[>={Stealth[black]},every edge/.style={draw=black,very thick}]
    \path [-] (1) edge[bend left=40] node[bar]{\rule{1pt}{6pt}\hspace{1.5pt}\rule{1pt}{6pt}} (2);
\end{scope}
\begin{scope}[>={Stealth[black]},every edge/.style={draw=black,very thick,dashed}]
    \path [-] (1) edge[bend right=40] (2);
\end{scope}
\end{tikzpicture}}
+
{\tiny\begin{tikzpicture}[baseline={([yshift=-.8ex]current bounding box.center)},scale=0.7]
\begin{scope}[every node/.style={circle,fill,inner sep=0pt,minimum size=3pt}]
    \node (1) at (0,0) {};
    \node (2) at (0.5,1) {};
    \node (3) at (1,0) {};
\end{scope}
\begin{scope}[>={Stealth[black]},every edge/.style={draw=black,very thick}]
    \path [-] (1) edge node[bar]{\rule{1pt}{6pt}} (2);
    \path [-] (1) edge node[bar]{\rule{1pt}{6pt}} (3);
\end{scope}
\begin{scope}[>={Stealth[black]},every edge/.style={draw=black,very thick,dashed}]
    \path [-] (2) edge (3);
\end{scope}
\end{tikzpicture}}
\,\Big).
\end{equation*}
Estimating the integrals, using the diagrammatic calculation rules of Section~\ref{sec:diag}, the claim follows.

Combining the bounds of Steps~1.2--1.4, together with the bounds~\eqref{eq:est-S3yz} on the contributions of $S_1,S_2,S_3$, and recalling the decomposition~\eqref{eq:dec-tilTmu}, the desired estimate~\eqref{eq:todo-tildeT31} follows, and therefore the conclusion~\eqref{eq:bound-T31}.

\medskip
\step2 Proof that
\begin{equation}\label{eq:estim-TL33}
\limsup_{\mu\downarrow0}|T_\mu^{3,2}|\,\lesssim\,\lambda_4^\frac{\beta}{2+\beta}.
\end{equation}
Following the general strategy of Section~\ref{sec:diag}, we first expand the third-order difference $\delta^{x,y,z}\varphi_\mu^\varnothing$ into reflection blocks: by Lemma~\ref{lem:element-dec}, it can be decomposed as
\[\delta^{x,y,z}\varphi_\mu^\varnothing\,=\,\Jc_\mu^x\delta^{y,z}\varphi_\mu^x+\Jc_\mu^y\delta^{x,z}\varphi_\mu^y+\Jc_\mu^z\delta^{x,y}\varphi_\mu^z,\]
and then inserting the expansion~\eqref{eq:decomp-deltayzvarphi0} of Step~1 for second-order differences, we get
\begin{equation}\label{eq:decomp-block-3}
\delta^{x,y,z}\varphi_\mu^\varnothing\,=\,\sum_{i=1}^6\Jc_\mu^xS_i^{x,y,z}+\Sym_{x,y,z},
\end{equation}
where $\Sym_{x,y,z}$ stands for the sum of the preceding terms over permutations of the set $\{x,y,z\}$.
Using diagrammatic notation, recalling the decay of the elementary kernels, this reads
\begin{equation*}
\delta^{x,y,z}\varphi_\mu^\varnothing(0)
\,=\,
{\tiny\begin{tikzpicture}[baseline={([yshift=-.8ex]current bounding box.center)},scale=0.7]
\begin{scope}[every node/.style={circle,draw,inner sep=0pt,minimum size=8pt}]
    \node (0) at (0,-1) {0};
    \node (1) at (0,0) {x};
    \node (2) at (-0.6,1) {y};
    \node (3) at (0.6,1) {z};
\end{scope}
\begin{scope}[>={Stealth[black]},every edge/.style={draw=black,very thick}]
    \path [-] (0) edge node[bar]{\rule{1pt}{6pt}} (1);
    \path [-] (1) edge[bend right=20] (2);
    \path [-] (1) edge[bend left=30] (2);
    \path [-] (1) edge (3);
    \path [-] (2) edge node[bar]{\rule{1pt}{6pt}} (3);
\end{scope}
\end{tikzpicture}}
+
{\tiny\begin{tikzpicture}[baseline={([yshift=-.8ex]current bounding box.center)},scale=0.7]
\begin{scope}[every node/.style={circle,draw,inner sep=0pt,minimum size=8pt}]
    \node (0) at (0,-1) {0};
    \node (1) at (0,0) {x};
    \node (2) at (-0.6,1) {y};
    \node (3) at (0.6,1) {z};
\end{scope}
\begin{scope}[>={Stealth[black]},every edge/.style={draw=black,very thick}]
    \path [-] (0) edge node[bar]{\rule{1pt}{6pt}} (1);
    \path [-] (1) edge[bend right=20] (2);
    \path [-] (1) edge[bend left=30] (2);
    \path [-] (1) edge[bend right=30] node[bar]{\rule{1pt}{6pt}} (3);
    \path [-] (1) edge[bend left=20] (3);
\end{scope}
\end{tikzpicture}}
+
{\tiny\begin{tikzpicture}[baseline={([yshift=-.8ex]current bounding box.center)},scale=0.7]
\begin{scope}[every node/.style={circle,draw,inner sep=0pt,minimum size=8pt}]
    \node (0) at (0,-1) {0};
    \node (1) at (0,0) {x};
    \node (2) at (-0.6,1) {y};
    \node (3) at (0.6,1) {z};
\end{scope}
\begin{scope}[>={Stealth[black]},every edge/.style={draw=black,very thick}]
    \path [-] (1) edge (2);
    \path [-] (2) edge (3);
    \path [-] (1) edge node[bar]{\rule{1pt}{6pt}} (3);
    \path [-] (1) edge node[bar]{\rule{1pt}{6pt}} (0);
\end{scope}
\end{tikzpicture}}
+
{\tiny\begin{tikzpicture}[baseline={([yshift=-.8ex]current bounding box.center)},scale=0.7]
\begin{scope}[every node/.style={circle,draw,inner sep=0pt,minimum size=8pt}]
    \node (0) at (0,-1) {0};
    \node (1) at (0,0) {x};
    \node (2) at (-0.6,1) {y};
    \node (3) at (0.6,1) {z};
\end{scope}
\begin{scope}[>={Stealth[black]},every edge/.style={draw=black,very thick}]
    \path [-] (0) edge node[bar]{\rule{1pt}{6pt}} (1);
    \path [-] (1) edge[bend right=20] (2);
    \path [-] (1) edge[bend left=30] (2);
    \path [-] (1) edge node[bar]{\rule{1pt}{6pt}} (3);
\end{scope}
\end{tikzpicture}}
+
{\tiny\begin{tikzpicture}[baseline={([yshift=-.8ex]current bounding box.center)},scale=0.7]
\begin{scope}[every node/.style={circle,draw,inner sep=0pt,minimum size=8pt}]
    \node (0) at (0,-1) {0};
    \node (1) at (0,0) {x};
    \node (2) at (0,1) {y};
    \node (3) at (0,2) {z};
\end{scope}
\begin{scope}[>={Stealth[black]},every edge/.style={draw=black,very thick}]
    \path [-] (2) edge[bend left=30] node[bar]{\rule{1pt}{6pt}} (3);
    \path [-] (2) edge[bend right=30] (3);
    \path [-] (1) edge (2);
    \path [-] (0) edge node[bar]{\rule{1pt}{6pt}} (1);
\end{scope}
\end{tikzpicture}}
+
{\tiny\begin{tikzpicture}[baseline={([yshift=-.8ex]current bounding box.center)},scale=0.7]
\begin{scope}[every node/.style={circle,draw,inner sep=0pt,minimum size=8pt}]
    \node (0) at (0,-1) {0};
    \node (1) at (0,0) {x};
    \node (2) at (0,1) {y};
    \node (3) at (0,2) {z};
\end{scope}
\begin{scope}[>={Stealth[black]},every edge/.style={draw=black,very thick}]
    \path [-] (2) edge node[bar]{\rule{1pt}{6pt}} (3);
    \path [-] (1) edge (2);
    \path [-] (0) edge node[bar]{\rule{1pt}{6pt}} (1);
\end{scope}
\end{tikzpicture}}
+\Sym_{x,y,z},
\end{equation*}
where we emphasize that the expansion has again been precisely chosen so as to satisfy the key dependence rule~\eqref{eq:keyrule}: cut edges separate the background variables in the elementary kernels.
Inserting this into the definition of $T_\mu^{3,2}$, cf.~\eqref{eq:decomp-TL3}, yields
\begin{equation}\label{eq:decomp-Tmu32}
T_\mu^{3,2}=\sum_{i=1}^6\int_{(\R^d)^3}\Big(e\cdot \int_B\Jc_\mu^xS_i^{x,y,z}\Big)\,\Big(f_4(x,y,z,0)-\lambda f_3(x,y,z)\Big)\,dxdydz.
\end{equation}
Expanding the $3$-particle and $4$-particle densities in terms of correlation functions, we find
\begin{multline}\label{eq:cluster-f3f4}
f_{4}(0,x,y,z)-\lambda f_{3}(0,x,y,z)=
\lambda^2\Big(h_{2}(0,x)+h_{2}(0,y)+h_{2}(0,z)\Big)\\
+\Big(h_{2}(0,x)h_{2}(y,z)+h_{2}(0,y)h_{2}(x,z)+h_{2}(0,z)h_{2}(x,y)\Big)\\
+\lambda\Big(h_{3}(0,x,y)+h_{3}(0,x,z)+h_{3}(0,y,z)\Big)+h_{4}(0,x,y,z),
\end{multline}
and thus, by the decay of correlations,
\[|f_4(x,y,z,0)-\lambda f_3(x,y,z)|\,\lesssim\,\lambda_4\wedge\langle x\rangle^{-2-\beta}+\lambda_4\wedge\langle y\rangle^{-2-\beta}+\lambda_4\wedge\langle z\rangle^{-2-\beta}.\]
Inserting this into~\eqref{eq:decomp-Tmu32}, we find that the first three terms ($1\le i\le 3$) are uniformly bounded as $\mu\downarrow0$: more precisely, using the diagrammatic calculation rules of Section~\ref{sec:diag},
\begin{multline}
\sum_{i=1}^3\bigg|\int_{(\R^d)^3}\Big(e\cdot \int_B\Jc_\mu^xS_i^{x,y,z}\Big)\,\Big(f_4(x,y,z,0)-\lambda f_3(x,y,z)\Big)\,dxdydz\bigg|\\
\hspace{-3cm}\,\lesssim\,[\lambda_4]\bigg(~
{\tiny\begin{tikzpicture}[baseline={([yshift=-.8ex]current bounding box.center)},scale=0.7]
\begin{scope}[every node/.style={circle,fill,inner sep=0pt,minimum size=3pt}]
    \node (0) at (0,-1) {};
    \node (1) at (0,0) {};
    \node (2) at (-0.6,1) {};
    \node (3) at (0.6,1) {};
\end{scope}
\begin{scope}[>={Stealth[black]},every edge/.style={draw=black,very thick}]
    \path [-] (0) edge[bend left=30]  node[bar]{\rule{1pt}{6pt}} (1);
    \path [-] (1) edge[bend right=20] (2);
    \path [-] (1) edge[bend left=30] (2);
    \path [-] (1) edge (3);
    \path [-] (2) edge node[bar]{\rule{1pt}{6pt}} (3);
\end{scope}
\begin{scope}[>={Stealth[black]},every edge/.style={draw=black,very thick,dashed}]
    \path [-] (0) edge [bend right=30]  (1);
\end{scope}
\end{tikzpicture}}
+
{\tiny\begin{tikzpicture}[baseline={([yshift=-.8ex]current bounding box.center)},scale=0.7]
\begin{scope}[every node/.style={circle,fill,inner sep=0pt,minimum size=3pt}]
    \node (0) at (0,-1) {};
    \node (1) at (0,0) {};
    \node (2) at (-0.6,1) {};
    \node (3) at (0.6,1) {};
\end{scope}
\begin{scope}[>={Stealth[black]},every edge/.style={draw=black,very thick}]
    \path [-] (0) edge  node[bar]{\rule{1pt}{6pt}} (1);
    \path [-] (1) edge[bend right=20] (2);
    \path [-] (1) edge[bend left=20] (2);
    \path [-] (1) edge (3);
    \path [-] (2) edge node[bar]{\rule{1pt}{6pt}} (3);
\end{scope}
\begin{scope}[>={Stealth[black]},every edge/.style={draw=black,very thick,dashed}]
    \path [-] (0) edge[bend left=20]  (2);
\end{scope}
\end{tikzpicture}}
+
{\tiny\begin{tikzpicture}[baseline={([yshift=-.8ex]current bounding box.center)},scale=0.7]
\begin{scope}[every node/.style={circle,fill,inner sep=0pt,minimum size=3pt}]
    \node (0) at (0,-1) {};
    \node (1) at (0,0) {};
    \node (2) at (-0.6,1) {};
    \node (3) at (0.6,1) {};
\end{scope}
\begin{scope}[>={Stealth[black]},every edge/.style={draw=black,very thick}]
    \path [-] (0) edge  node[bar]{\rule{1pt}{6pt}} (1);
    \path [-] (1) edge[bend right=20] (2);
    \path [-] (1) edge[bend left=30] (2);
    \path [-] (1) edge (3);
    \path [-] (2) edge node[bar]{\rule{1pt}{6pt}} (3);
\end{scope}
\begin{scope}[>={Stealth[black]},every edge/.style={draw=black,very thick,dashed}]
    \path [-] (0) edge[bend right=20]  (3);
\end{scope}
\end{tikzpicture}}
+
{\tiny\begin{tikzpicture}[baseline={([yshift=-.8ex]current bounding box.center)},scale=0.7]
\begin{scope}[every node/.style={circle,fill,inner sep=0pt,minimum size=3pt}]
    \node (0) at (0,-1) {};
    \node (1) at (0,0) {};
    \node (2) at (-0.6,1) {};
    \node (3) at (0.6,1) {};
\end{scope}
\begin{scope}[>={Stealth[black]},every edge/.style={draw=black,very thick}]
    \path [-] (0) edge[bend left=30] node[bar]{\rule{1pt}{6pt}} (1);
    \path [-] (1) edge[bend right=20] (2);
    \path [-] (1) edge[bend left=30] (2);
    \path [-] (1) edge[bend right=30] node[bar]{\rule{1pt}{6pt}} (3);
    \path [-] (1) edge[bend left=20] (3);
\end{scope}
\begin{scope}[>={Stealth[black]},every edge/.style={draw=black,very thick,dashed}]
    \path [-] (0) edge[bend right=30]  (1);
\end{scope}
\end{tikzpicture}}
+
{\tiny\begin{tikzpicture}[baseline={([yshift=-.8ex]current bounding box.center)},scale=0.7]
\begin{scope}[every node/.style={circle,fill,inner sep=0pt,minimum size=3pt}]
    \node (0) at (0,-1) {};
    \node (1) at (0,0) {};
    \node (2) at (-0.6,1) {};
    \node (3) at (0.6,1) {};
\end{scope}
\begin{scope}[>={Stealth[black]},every edge/.style={draw=black,very thick}]
    \path [-] (0) edge node[bar]{\rule{1pt}{6pt}} (1);
    \path [-] (1) edge[bend right=20] (2);
    \path [-] (1) edge[bend left=20] (2);
    \path [-] (1) edge[bend right=30] node[bar]{\rule{1pt}{6pt}} (3);
    \path [-] (1) edge[bend left=20] (3);
\end{scope}
\begin{scope}[>={Stealth[black]},every edge/.style={draw=black,very thick,dashed}]
    \path [-] (0) edge[bend left=20]  (2);
\end{scope}
\end{tikzpicture}}
+
{\tiny\begin{tikzpicture}[baseline={([yshift=-.8ex]current bounding box.center)},scale=0.7]
\begin{scope}[every node/.style={circle,fill,inner sep=0pt,minimum size=3pt}]
    \node (0) at (0,-1) {};
    \node (1) at (0,0) {};
    \node (2) at (-0.6,1) {};
    \node (3) at (0.6,1) {};
\end{scope}
\begin{scope}[>={Stealth[black]},every edge/.style={draw=black,very thick}]
    \path [-] (0) edge node[bar]{\rule{1pt}{6pt}} (1);
    \path [-] (1) edge[bend right=20] (2);
    \path [-] (1) edge[bend left=30] (2);
    \path [-] (1) edge[bend right=20] node[bar]{\rule{1pt}{6pt}} (3);
    \path [-] (1) edge[bend left=20] (3);
\end{scope}
\begin{scope}[>={Stealth[black]},every edge/.style={draw=black,very thick,dashed}]
    \path [-] (0) edge[bend right=20]  (3);
\end{scope}
\end{tikzpicture}}
\\
+
{\tiny\begin{tikzpicture}[baseline={([yshift=-.8ex]current bounding box.center)},scale=0.7]
\begin{scope}[every node/.style={circle,fill,inner sep=0pt,minimum size=3pt}]
    \node (0) at (0,-1) {};
    \node (1) at (0,0) {};
    \node (2) at (-0.6,1) {};
    \node (3) at (0.6,1) {};
\end{scope}
\begin{scope}[>={Stealth[black]},every edge/.style={draw=black,very thick}]
    \path [-] (1) edge (2);
    \path [-] (2) edge (3);
    \path [-] (1) edge node[bar]{\rule{1pt}{6pt}} (3);
    \path [-] (1) edge [bend right=30] node[bar]{\rule{1pt}{6pt}} (0);
\end{scope}
\begin{scope}[>={Stealth[black]},every edge/.style={draw=black,very thick,dashed}]
    \path [-] (0) edge[bend right=30]  (1);
\end{scope}
\end{tikzpicture}}
+
{\tiny\begin{tikzpicture}[baseline={([yshift=-.8ex]current bounding box.center)},scale=0.7]
\begin{scope}[every node/.style={circle,fill,inner sep=0pt,minimum size=3pt}]
    \node (0) at (0,-1) {};
    \node (1) at (0,0) {};
    \node (2) at (-0.6,1) {};
    \node (3) at (0.6,1) {};
\end{scope}
\begin{scope}[>={Stealth[black]},every edge/.style={draw=black,very thick}]
    \path [-] (1) edge (2);
    \path [-] (2) edge (3);
    \path [-] (1) edge node[bar]{\rule{1pt}{6pt}} (3);
    \path [-] (1) edge node[bar]{\rule{1pt}{6pt}} (0);
\end{scope}
\begin{scope}[>={Stealth[black]},every edge/.style={draw=black,very thick,dashed}]
    \path [-] (0) edge[bend left=20]  (2);
\end{scope}
\end{tikzpicture}}
+
{\tiny\begin{tikzpicture}[baseline={([yshift=-.8ex]current bounding box.center)},scale=0.7]
\begin{scope}[every node/.style={circle,fill,inner sep=0pt,minimum size=3pt}]
    \node (0) at (0,-1) {};
    \node (1) at (0,0) {};
    \node (2) at (-0.6,1) {};
    \node (3) at (0.6,1) {};
\end{scope}
\begin{scope}[>={Stealth[black]},every edge/.style={draw=black,very thick}]
    \path [-] (1) edge (2);
    \path [-] (2) edge (3);
    \path [-] (1) edge node[bar]{\rule{1pt}{6pt}} (3);
    \path [-] (1) edge node[bar]{\rule{1pt}{6pt}} (0);
\end{scope}
\begin{scope}[>={Stealth[black]},every edge/.style={draw=black,very thick,dashed}]
    \path [-] (0) edge[bend right=20]  (3);
\end{scope}
\end{tikzpicture}}
~\bigg)
~\lesssim~
[\lambda_4]\Lc{\tiny\begin{tikzpicture}[baseline={([yshift=-.8ex]current bounding box.center)},scale=0.7]
\begin{scope}[every node/.style={circle,fill,inner sep=0pt,minimum size=3pt}]
    \node (1) at (0,0) {};
    \node (2) at (0,1) {};
\end{scope}
\begin{scope}[>={Stealth[black]},every edge/.style={draw=black,very thick}]
    \path [-] (1) edge[bend left=40] node[bar]{\rule{1pt}{6pt}} (2);
\end{scope}
\begin{scope}[>={Stealth[black]},every edge/.style={draw=black,very thick,dashed}]
    \path [-] (1) edge[bend right=40] (2);
\end{scope}
\end{tikzpicture}}
~\lesssim~\lambda_4^{\frac{1+\beta}{2+\beta}}+o(1).\label{eq:estim-T32-S123}
\end{multline}
It remains to examine the last three terms in~\eqref{eq:decomp-Tmu32} ($4\le i\le6$), which are a priori unbounded as $\mu\downarrow0$ and thus require cancellations.
For brevity, let us focus on the last term ($i=6$), which is the most delicate one.
Using~\eqref{eq:cluster-f3f4}, we can decompose
\begin{equation}\label{eq:decomp-T32-S6}
\bigg|\int_{(\R^d)^3}\Big(e\cdot \int_B\Jc_\mu^xS_6^{x,y,z}\Big)\,\Big(f_4(x,y,z,0)-\lambda f_3(x,y,z)\Big)\,dxdydz\bigg|
\,\le\,\sum_{i=1}^4|B_\mu^i|,
\end{equation}
in terms of
\begin{eqnarray*}
B_\mu^1&:=&\int_{(\R^d)^3}\Big(e\cdot \int_B\Jc_\mu^xS_6^{x,y,z}\Big)\,\Big(\lambda^2h_{2}(0,z)+h_{2}(0,y)h_{2}(x,z)+h_{2}(0,z)h_{2}(x,y)\\[-2mm]
&&\hspace{4cm}+\lambda h_{3}(0,x,z)+\lambda h_{3}(0,y,z)+h_{4}(0,x,y,z)\Big)\,dxdydz,\\
B_\mu^2&:=&\lambda^2\int_{(\R^d)^3}\Big(e\cdot \int_B\Jc_\mu^xS_6^{x,y,z}\Big)\,h_2(0,x)\,dxdydz,\\
B_\mu^3&:=&\lambda\int_{(\R^d)^3}\Big(e\cdot \int_B\Jc_\mu^xS_6^{x,y,z}\Big)\,\Big(\lambda h_2(0,y)+h_3(0,x,y)\Big)\,dxdydz,\\
B_\mu^4&:=&\int_{(\R^d)^3}\Big(e\cdot \int_B\Jc_\mu^xS_6^{x,y,z}\Big)\,h_2(0,x)h_2(y,z)\,dxdydz.
\end{eqnarray*}
Recall $\Jc_\mu^xS_6^{x,y,z}=\Jc_\mu^x\Jc_{\mu;x}^y(\Jc_{\mu;y}^z\varphi_\mu^z+\Gc_{\mu;y}^z)$.
The first term $B_\mu^1$ is clearly uniformly bounded as $\mu\downarrow0$: using the diagrammatic calculation rules of Section~\ref{sec:diag},
\[|B_\mu^1|\,\lesssim\,
[\lambda_4]\bigg(~
{\tiny\begin{tikzpicture}[baseline={([yshift=-.8ex]current bounding box.center)},scale=0.7]
\begin{scope}[every node/.style={circle,fill,inner sep=0pt,minimum size=3pt}]
    \node (0) at (0,-1) {};
    \node (1) at (0,0) {};
    \node (2) at (0,1) {};
    \node (3) at (0,2) {};
\end{scope}
\begin{scope}[>={Stealth[black]},every edge/.style={draw=black,very thick}]
    \path [-] (2) edge node[bar]{\rule{1pt}{6pt}} (3);
    \path [-] (1) edge (2);
    \path [-] (0) edge node[bar]{\rule{1pt}{6pt}} (1);
\end{scope}
\begin{scope}[>={Stealth[black]},every edge/.style={draw=black,very thick,dashed}]
    \path [-] (0) edge[bend right=40] (3);
\end{scope}
\end{tikzpicture}}
+
{\tiny\begin{tikzpicture}[baseline={([yshift=-.8ex]current bounding box.center)},scale=0.7]
\begin{scope}[every node/.style={circle,fill,inner sep=0pt,minimum size=3pt}]
    \node (0) at (0,-1) {};
    \node (1) at (0,0) {};
    \node (2) at (0,1) {};
    \node (3) at (0,2) {};
\end{scope}
\begin{scope}[>={Stealth[black]},every edge/.style={draw=black,very thick}]
    \path [-] (2) edge node[bar]{\rule{1pt}{6pt}} (3);
    \path [-] (1) edge (2);
    \path [-] (0) edge node[bar]{\rule{1pt}{6pt}} (1);
\end{scope}
\begin{scope}[>={Stealth[black]},every edge/.style={draw=black,very thick,dashed}]
    \path [-] (0) edge[bend left=40] (2);
    \path [-] (1) edge[bend right=40] (3);
\end{scope}
\end{tikzpicture}}
~\bigg)
~\lesssim~\lambda_4^{\frac{\beta}{2+\beta}}+o(1).\]
We turn to the second term $B_\mu^2$ in~\eqref{eq:decomp-T32-S6}.
Taking advantage of the cancellation identity~\eqref{eq:ident-intJG-re} in Lemma~\ref{lem:cancel} twice, we find
\begin{eqnarray*}
\lefteqn{\int_{(\R^d)^2}\Jc_\mu^xS_6^{x,y,z}\,dydz
~=~\int_{(\R^d)^2}\Jc_\mu^x\Jc_{\mu;x}^y(\Jc_{\mu;y}^z\varphi_\mu^z+\Gc_{\mu;y}^z)\,dydz}\\
&=&\sum_{i=1}^d\Big(\int_B(e+\mu\varphi_\mu^0)\Big)_i\int_{\R^d}\Jc_\mu^x(\Jc_{\mu;x}^y\varphi_{\mu,i}^y+\Gc_{\mu,i;x}^y)\,dy\\
&=&\sum_{i,j=1}^d\Big(\int_B(e+\mu\varphi_\mu^0)\Big)_i\Big(\int_B(e_i+\mu\varphi_{\mu,i}^0)\Big)_j\big(\Jc_{\mu}^x\varphi_{\mu,j}^{x}+\Gc_{\mu,j}^x\big).
\end{eqnarray*}
Hence, using the diagrammatic calculation rules of Section~\ref{sec:diag},
\[|B_\mu^2|\,\lesssim\,[\lambda_4]
{\tiny\begin{tikzpicture}[baseline={([yshift=-.8ex]current bounding box.center)},scale=0.7]
\begin{scope}[every node/.style={circle,fill,inner sep=0pt,minimum size=3pt}]
    \node (1) at (0,0) {};
    \node (2) at (0,1) {};
\end{scope}
\begin{scope}[>={Stealth[black]},every edge/.style={draw=black,very thick}]
    \path [-] (1) edge[bend left=40] node[bar]{\rule{1pt}{6pt}\hspace{1.5pt}\rule{1pt}{6pt}} (2);
\end{scope}
\begin{scope}[>={Stealth[black]},every edge/.style={draw=black,very thick,dashed}]
    \path [-] (1) edge[bend right=40] (2);
\end{scope}
\end{tikzpicture}}
~\lesssim~\lambda_4^{\frac{\beta}{2+\beta}}.\]
We turn to the third term $B_\mu^3$ in~\eqref{eq:decomp-T32-S6}. As the above computation using Lemma~\ref{lem:cancel} yields in particular, for the $z$-integral,
\begin{equation*}
\int_{\R^d}\Jc_\mu^xS_6^{x,y,z}\,dz
\,=\,\sum_{i=1}^d\Big(\int_B(e+\mu\varphi_{\mu}^0)\Big)_i\Jc_\mu^x(\Jc_{\mu;x}^y\varphi_{\mu,i}^y+\Gc_{\mu,i;x}^y),
\end{equation*}
we find
\[|B_\mu^3|
\,\lesssim\,
[\lambda_4]{\tiny\begin{tikzpicture}[baseline={([yshift=-.8ex]current bounding box.center)},scale=0.7]
\begin{scope}[every node/.style={circle,fill,inner sep=0pt,minimum size=3pt}]
    \node (1) at (1,0) {};
    \node (2) at (0,0) {};
    \node (3) at (0.5,0.8) {};
\end{scope}
\begin{scope}[>={Stealth[black]},every edge/.style={draw=black,very thick}]
    \path [-] (1) edge node[bar]{\rule{1pt}{6pt}} (2);
    \path [-] (2) edge node[bar]{\rule{1pt}{6pt}} (3);
\end{scope}
\begin{scope}[>={Stealth[black]},every edge/.style={draw=black,very thick,dashed}]
    \path [-] (1) edge (3);
\end{scope}
\end{tikzpicture}}
\,\le\,
[\lambda_4]\Lc{\tiny\begin{tikzpicture}[baseline={([yshift=-.8ex]current bounding box.center)},scale=0.7]
\begin{scope}[every node/.style={circle,fill,inner sep=0pt,minimum size=3pt}]
    \node (1) at (0,0) {};
    \node (2) at (0,1) {};
\end{scope}
\begin{scope}[>={Stealth[black]},every edge/.style={draw=black,very thick}]
    \path [-] (1) edge[bend left=40] node[bar]{\rule{1pt}{6pt}\hspace{1.5pt}\rule{1pt}{6pt}} (2);
\end{scope}
\begin{scope}[>={Stealth[black]},every edge/.style={draw=black,very thick,dashed}]
    \path [-] (1) edge[bend right=40] (2);
\end{scope}
\end{tikzpicture}}
\,\lesssim\,\lambda_4^{\frac{\beta}{2+\beta}}+o(1).\]
We turn to the last term~$B_\mu^4$ in~\eqref{eq:decomp-T32-S6}. By translation invariance, it can be written as
\[B_\mu^4\,=\,\int_{(\R^d)^2}\Big(e\cdot \int_B\Jc_\mu^x\Jc_{\mu;x}^y\zeta^y
\Big)\,h_2(0,x)\,dxdy,\]
in terms of
\[\zeta^y\,:=\,\zeta(\cdot-y),\qquad\zeta(x)\,:=\,\int_{\R^d}(\Jc_{\mu;0}^{z}\varphi_\mu^{z}+\Gc_{\mu;0}^{z})\,h_2(0,z)\,dz\]
Applying Lemma~\ref{lem:cancel}, we then find
\begin{multline*}
B_\mu^4\,=\,\mu\sum_{i=1}^d\Big(\int_B\zeta\Big)_i\int_{\R^d}\Big(e\cdot \int_B(\Jc_\mu^x\varphi_{\mu,i}^x+\Gc_{\mu,i}^x)
\Big)\,h_2(0,x)\,dx\\
\,=\,\mu\sum_{i=1}^d\bigg(\int_{\R^d}\Big(\int_B(\Jc_{\mu;0}^{z}\varphi_\mu^{z}+\Gc_{\mu;0}^{z})\Big)h_2(0,z)\,dz\bigg)_i~\int_{\R^d}\Big(e\cdot \int_B(\Jc_\mu^x\varphi_{\mu,i}^x+\Gc_{\mu,i}^x)
\Big)h_2(0,x)\,dx.
\end{multline*}
Hence,
\[|B_\mu^4|\,\lesssim\,\mu\bigg([\lambda_2]
{\tiny\begin{tikzpicture}[baseline={([yshift=-.8ex]current bounding box.center)},scale=0.7]
\begin{scope}[every node/.style={circle,fill,inner sep=0pt,minimum size=3pt}]
    \node (1) at (0,0) {};
    \node (2) at (0,1) {};
\end{scope}
\begin{scope}[>={Stealth[black]},every edge/.style={draw=black,very thick}]
    \path [-] (1) edge[bend left=40] node[bar]{\rule{1pt}{6pt}\hspace{1.5pt}\rule{1pt}{6pt}} (2);
\end{scope}
\begin{scope}[>={Stealth[black]},every edge/.style={draw=black,very thick,dashed}]
    \path [-] (1) edge[bend right=40] (2);
\end{scope}
\end{tikzpicture}}
\bigg)^2\,\lesssim\,\mu\lambda_4^\frac{\beta}{2+\beta}.
\]
Combining these different estimates back into~\eqref{eq:decomp-T32-S6}, we conclude
\[\limsup_{\mu\downarrow0}\bigg|\int_{(\R^d)^3}\Big(e\cdot \int_B\Jc_\mu^xS_6^{x,y,z} \Big)\,\Big(f_4(x,y,z,0)-\lambda f_3(x,y,z)\Big)\,dxdydz\bigg|\,\lesssim\,\lambda_4^\frac{\beta}{2+\beta}.\]
Similar, yet slightly simpler, arguments also allow us to bound the fourth and fifth terms in~\eqref{eq:decomp-Tmu32}; we omit the details for brevity. Combined with~\eqref{eq:estim-T32-S123}, this concludes the proof of~\eqref{eq:estim-TL33}.
\end{proof}

\subsection{Estimate of remainder term}\label{sec:remainder}
This subsection is devoted to the proof of the following error estimate on the remainder term~$R_\mu$ in the cluster expansion~\eqref{eq:VL-decomp}.  The proof makes use of Cauchy-Schwarz inequality, which is not sharp and entails a square-root loss compared with the expected size of the error. This explains why the cluster expansion~\eqref{eq:VL-decomp} has to be pushed to third order, and why the third cluster contribution has to be estimated sharply. Indeed, if the cluster expansion were truncated directly after the second order, the resulting error bound would only be
\[O\big(\sqrt\lambda_3+(\sqrt\lambda_4)^\frac{\beta}{2+\beta}\big),\]
up to logarithmic factors. By~\eqref{eq:submult-lambda}, this would be larger than the size $O((\lambda_2)^\frac{\beta}{2+\beta})$ of the second cluster contribution, and would therefore not yield a genuine second-order expansion. It is not clear to us whether and how one could proceed otherwise.

\begin{lem}\label{lem:estim-RL}
The remainder term~$R_\mu$ in~\eqref{eq:VL-decomp} is bounded by
\begin{equation}\label{eq:estim-RL}
\limsup_{\mu\downarrow0}|R_\mu|\,\lesssim\,\sqrt\lambda_5+(\sqrt\lambda_6)^{\frac\beta{2+\beta}}.
\end{equation}
\end{lem}

\begin{proof}
By definition of $R_\mu$, integrating by parts, we can write
\[R_\mu
\,=\,\E[\mathds1_{\Ic}(u_\mu-\E[u_\mu])]
\,=\,\mu\E\big[(u_\mu-\E[u_\mu])(G_\mu\star\mathds1_\Ic)\big]
+\E\big[\nabla u_\mu \cdot (\nabla G_\mu\star\mathds1_\Ic)\big]
,\]
in terms of the Green function $G_\mu$ for the massive Laplacian on $\R^d$, cf.~\eqref{eq:Greenfct-mass}.
Hence, taking advantage of the covariance structure in the first term, and using the Cauchy-Schwarz inequality,
\begin{equation}\label{eq:preest-Rmu}
|R_\mu|\,\lesssim\,\E\Big[\mu|u_\mu-\E[u_\mu]|^2+|\nabla u_\mu|^2\Big]^\frac12\E\Big[\mu|G_\mu\star\mathds1_{\Ic}-\E[G_\mu\star\mathds1_{\Ic}]|^2+|\nabla G_\mu\star\mathds1_{\Ic}|^2\Big]^\frac12.
\end{equation}
Let us examine the second factor. Writing $\mathds1_\Ic(0)=\sum_{x\in\Pc}\mathds1_B(-x)$ and expanding the squares, we find in terms of the pair density and correlation functions,
\begin{multline*}
\E\Big[\mu|G_\mu\star\mathds1_{\Ic}-\E[G_\mu\star\mathds1_{\Ic}]|^2+|\nabla G_\mu\star\mathds1_{\Ic}|^2\Big]\\
\,=\,\lambda\mu\int_{\R^d}|G_\mu\star\mathds1_{B}|^2
+\mu \iint_{\R^d\times\R^d}(G_\mu\star\mathds1_{B})(x)\,(G_\mu\star\mathds1_{B})(y)\,h_2(x,y)\,dxdy\\
+\lambda\int_{\R^d}|\nabla G\star\mathds1_{B}|^2
+\iint_{\R^d\times\R^d}\nabla( G_\mu\star\mathds1_{B})(x)\cdot\nabla( G_\mu\star\mathds1_{B})(y)\,f_2(x,y)\,dxdy.
\end{multline*}
As $\int_{\R^d}\nabla G_\mu\star\mathds1_{B}=0$, note that $f_2$ can be replaced by $h_2=f_2-\lambda^2$ in the last term.
Hence, by the pointwise decay of the massive Green function, cf.~\eqref{e.mass-green}, and by the decay of correlations, we get for $d>2$,
\begin{multline}\label{eq:estim-G1I}
\E\Big[\mu|G_\mu\star\mathds1_{\Ic}-\E[G_\mu\star\mathds1_{\Ic}]|^2+|\nabla G_\mu\star\mathds1_{\Ic}|^2\Big]\\
\,\lesssim\,\lambda
+\iint_{\R^d\times\R^d}\langle x\rangle^{1-d}\langle y\rangle^{1-d}\,|h_2(x,y)|\,dxdy
\,\lesssim\,\lambda+(\lambda_2)^{\frac{\beta}{2+\beta}}.
\end{multline}
Combining this with~\eqref{eq:preest-Rmu}, it is therefore sufficient to show
\begin{equation}\label{eq:estim-RL-red}
\limsup_{\mu\downarrow0}\E\Big[\mu|u_\mu-\E[u_\mu]|^2+|\nabla u_\mu|^2\Big]\,\lesssim\,\lambda_4+(\lambda_5)^\frac{\beta}{2+\beta}.
\end{equation}
By~\eqref{eq:submult-lambda}, this is in fact slightly stronger than the claimed estimate~\eqref{eq:estim-RL}.
We split the proof into six steps.

\medskip
\step1 Variance estimate for $u_\mu$: proof that
\begin{equation}\label{eq:umu-est}
\lim_{\mu\downarrow0}\mu\E[|u_\mu-\E[u_\mu]|^2]\,=\,0.
\end{equation}
By definition of $u_\mu$, cf.~\eqref{eq:def-uL}, we can write
\begin{equation*}
\mu\E[|u_\mu-\E[u_\mu]|^2]
\,\lesssim\,\sum_{i=1}^4C^i_\mu,
\end{equation*}
in terms of
\begin{eqnarray*}
C_\mu^1&:=&\mu\E[|\varphi_\mu|^2],\\
C_\mu^2&:=&\mu\E\bigg[\Big|\sum_{x\in\Pc}\varphi_\mu^x-\E\Big[\sum_{x\in\Pc}\varphi_\mu^x\Big]\Big|^2\bigg],\\
C_\mu^3&:=&\mu\E\bigg[\Big|\sum_{x,y\in\Pc}^{\ne}\delta^{x,y}\varphi_\mu-\E\Big[\sum_{x,y\in\Pc}^{\ne}\delta^{x,y}\varphi_\mu\Big]\Big|^2\bigg],\\
C_\mu^4&:=&\mu\E\bigg[\Big|\sum_{x,y,z\in\Pc}^{\ne}\delta^{x,y,z}\varphi_\mu-\E\Big[\sum_{x,y,z\in\Pc}^{\ne}\delta^{x,y,z}\varphi_\mu\Big]\Big|^2\bigg].
\end{eqnarray*}
It already follows from~\eqref{eq:conv-mass} in the proof of Theorem~\ref{th:mean-velocity} that
\[\lim_{\mu\downarrow0}C^1_\mu\,=\,0.\]
We turn to the analysis of $C_\mu^2$. Expanding the square and using translation invariance, it can be written as
\begin{equation*}
C_\mu^2\,=\,\lambda\mu\int_{\R^d}|\varphi^0_\mu|^2+\mu\iint_{\R^d\times\R^d}\varphi^0_\mu(x)\varphi^0_\mu(y)\,h_2(x,y)\,dxdy.
\end{equation*}
By Lemma~\ref{lem:element-dec}, the decay of elementary operators, and the decay of correlations, we find
\begin{equation*}
|C_\mu^2|\,\lesssim\,\lambda\mu\int_{\R^d}K_2(x)^2dx+\mu\iint_{\R^d\times\R^d}K_2(x)K_2(y)\,\big(\lambda_2\wedge\langle x-y\rangle^{-2-\beta}\big)\,dxdy,
\end{equation*}
or, using diagrammatic notation,
\begin{equation*}
|C_\mu^2|\,\lesssim\,
\lambda\mu
{\tiny\begin{tikzpicture}[baseline={([yshift=-.8ex]current bounding box.center)},scale=0.7]
\begin{scope}[every node/.style={circle,fill,inner sep=0pt,minimum size=3pt}]
    \node (1) at (0,0) {};
    \node (2) at (0,1) {};
\end{scope}
\begin{scope}[>={Stealth[black]},every edge/.style={draw=black,very thick}]
    \path [-] (1) edge[bend left=40] node[bar]{\rule{1pt}{6pt}\hspace{1.5pt}\rule{1pt}{6pt}} (2);
    \path [-] (1) edge[bend right=40] node[bar]{\rule{1pt}{6pt}\hspace{1.5pt}\rule{1pt}{6pt}} (2);
\end{scope}
\end{tikzpicture}}
+\mu[\lambda_2]
{\tiny\begin{tikzpicture}[baseline={([yshift=-.8ex]current bounding box.center)},scale=0.7]
\begin{scope}[every node/.style={circle,fill,inner sep=0pt,minimum size=3pt}]
    \node (1) at (0,0) {};
    \node (2) at (1,0) {};
    \node (3) at (0.5,0.8) {};
\end{scope}
\begin{scope}[>={Stealth[black]},every edge/.style={draw=black,very thick}]
    \path [-] (1) edge node[bar]{\rule{1pt}{6pt}\hspace{1.5pt}\rule{1pt}{6pt}} (2);
    \path [-] (1) edge node[bar]{\rule{1pt}{6pt}\hspace{1.5pt}\rule{1pt}{6pt}} (3);
\end{scope}
\begin{scope}[>={Stealth[black]},every edge/.style={draw=black,very thick,dashed}]
    \path [-] (2) edge (3);
\end{scope}
\end{tikzpicture}}
\end{equation*}
and thus, by the diagrammatic calculation rules of Section~\ref{sec:diag},
\[\lim_{\mu\downarrow0}C^2_\mu=0.\]
We turn to the analysis of $C_\mu^3$.
Expanding the square, it can be written as
\begin{multline}\label{eq:decomp-Cmu3}
C_\mu^3\,=\,\mu\int_{(\R^d)^2}|\delta^{x,y}\varphi_\mu^\varnothing(0)|^2f_2(x,y)\,dxdy
+2\mu\int_{(\R^d)^3}\delta^{x,y}\varphi_\mu^\varnothing(0)\,\delta^{x,z}\varphi_\mu^\varnothing(0)\,f_3(x,y,z)\,dxdydz\\
+\mu\int_{(\R^d)^4}\delta^{x,y}\varphi_\mu^\varnothing(0)\,\delta^{z,w}\varphi_\mu^\varnothing(0)\,\Big(f_4(x,y,z,w)-f_2(x,y)f_2(z,w)\Big)\,dxdydzdw.
\end{multline}
To estimate these integrals, we follow the general strategy described in Section~\ref{sec:diag}:
we expand $\delta^{x,y}\varphi_\mu^\varnothing$ into reflection blocks, cf.~\eqref{eq:expand-deltxy},
we decompose $f_4$ and $f_2$ in terms of correlation functions,
we appeal to the cancellation rules of Lemma~\ref{lem:cancel},
and then we estimate the resulting integrals using the decay of elementary operators and the decay of correlations.
For brevity, let us focus on the last term in~\eqref{eq:decomp-Cmu3}, which is the most involved term to estimate.
The expansion~\eqref{eq:expand-deltxy} into reflection blocks yields
\begin{equation*}
\delta^{x,y}\varphi_\mu^\varnothing(0)\,\delta^{z,w}\varphi_\mu^\varnothing(0)
={\tiny\begin{tikzpicture}[baseline={([yshift=-.8ex]current bounding box.center)},scale=0.7]
\begin{scope}[every node/.style={circle,draw,inner sep=0pt,minimum size=8pt}]
    \node (1) at (0,0) {0};
    \node (2) at (-0.5,0.8) {x};
    \node (3) at (-0.5,1.8) {y};
    \node (4) at (0.5,0.8) {z};
    \node (5) at (0.5,1.8) {w};
\end{scope}
\begin{scope}[>={Stealth[black]},every edge/.style={draw=black,very thick}]
    \path [-] (1) edge node[bar]{\rule{1pt}{6pt}} (2);
    \path [-] (2) edge node[bar]{\rule{1pt}{6pt}} (3);
    \path [-] (1) edge node[bar]{\rule{1pt}{6pt}} (4);
    \path [-] (4) edge node[bar]{\rule{1pt}{6pt}} (5);
\end{scope}
\end{tikzpicture}}
+
{\tiny\begin{tikzpicture}[baseline={([yshift=-.8ex]current bounding box.center)},scale=0.7]
\begin{scope}[every node/.style={circle,draw,inner sep=0pt,minimum size=8pt}]
    \node (1) at (0,0) {0};
    \node (2) at (-0.5,0.8) {x};
    \node (3) at (-0.5,1.8) {y};
    \node (4) at (0.5,0.8) {z};
    \node (5) at (0.5,1.8) {w};
\end{scope}
\begin{scope}[>={Stealth[black]},every edge/.style={draw=black,very thick}]
    \path [-] (1) edge node[bar]{\rule{1pt}{6pt}} (2);
    \path [-] (2) edge node[bar]{\rule{1pt}{6pt}} (3);
    \path [-] (1) edge node[bar]{\rule{1pt}{6pt}} (4);
    \path [-] (4) edge[bend left=40] node[bar]{\rule{1pt}{6pt}} (5);
    \path [-] (4) edge[bend right=40] (5);
\end{scope}
\end{tikzpicture}}
+
{\tiny\begin{tikzpicture}[baseline={([yshift=-.8ex]current bounding box.center)},scale=0.7]
\begin{scope}[every node/.style={circle,draw,inner sep=0pt,minimum size=8pt}]
    \node (1) at (0,0) {0};
    \node (2) at (-0.5,0.8) {x};
    \node (3) at (-0.5,1.8) {y};
    \node (4) at (0.5,0.8) {z};
    \node (5) at (0.5,1.8) {w};
\end{scope}
\begin{scope}[>={Stealth[black]},every edge/.style={draw=black,very thick}]
    \path [-] (1) edge node[bar]{\rule{1pt}{6pt}} (2);
    \path [-] (2) edge[bend left=40] (3);
    \path [-] (2) edge[bend right=40] node[bar]{\rule{1pt}{6pt}} (3);
    \path [-] (1) edge node[bar]{\rule{1pt}{6pt}} (4);
    \path [-] (4) edge[bend left=40] node[bar]{\rule{1pt}{6pt}} (5);
    \path [-] (4) edge[bend right=40] (5);
\end{scope}
\end{tikzpicture}}
+\Sym.
\end{equation*}
Next, expanding $f_4$ and $f_2$ in terms of correlation functions,
\begin{multline*}
f_4(x,y,z,w)-f_2(x,y)f_2(z,w)
=\lambda^2\Big(h_2(x,z)+h_2(x,w)+h_2(y,z)+h_2(y,w)\Big)\\
+\Big(h_2(x,z)h_2(y,w)+h_2(x,w)h_2(y,z)\Big)
+\lambda\Big(h_3(x,y,z)+h_3(x,y,w)+h_3(x,z,w)+h_3(y,z,w)\Big)\\
+h_4(x,y,z,w),
\end{multline*}
we are led to
\begin{multline*}
\int_{(\R^d)^4}\delta^{x,y}\varphi_\mu^\varnothing(0)\,\delta^{z,w}\varphi_\mu^\varnothing(0)\,\Big(f_4(x,y,z,w)-f_2(x,y)f_2(z,w)\Big)\,dxdydzdw\\
=
[\lambda_4]\bigg({\tiny\begin{tikzpicture}[baseline={([yshift=-.8ex]current bounding box.center)},scale=0.7]
\begin{scope}[every node/.style={circle,draw,fill,inner sep=0pt,minimum size=3pt}]
    \node (1) at (0,0) {};
    \node (2) at (-0.5,0.8) {};
    \node (3) at (-0.5,1.8) {};
    \node (4) at (0.5,0.8) {};
    \node (5) at (0.5,1.8) {};
\end{scope}
\begin{scope}[>={Stealth[black]},every edge/.style={draw=black,very thick}]
    \path [-] (1) edge node[bar]{\rule{1pt}{6pt}} (2);
    \path [-] (2) edge node[bar]{\rule{1pt}{6pt}} (3);
    \path [-] (1) edge node[bar]{\rule{1pt}{6pt}} (4);
    \path [-] (4) edge node[bar]{\rule{1pt}{6pt}} (5);
\end{scope}
\begin{scope}[>={Stealth[black]},every edge/.style={draw=black,very thick,dashed}]
    \path [-] (2) edge (4);
\end{scope}
\end{tikzpicture}}
+
{\tiny\begin{tikzpicture}[baseline={([yshift=-.8ex]current bounding box.center)},scale=0.7]
\begin{scope}[every node/.style={circle,draw,fill,inner sep=0pt,minimum size=3pt}]
    \node (1) at (0,0) {};
    \node (2) at (-0.5,0.8) {};
    \node (3) at (-0.5,1.8) {};
    \node (4) at (0.5,0.8) {};
    \node (5) at (0.5,1.8) {};
\end{scope}
\begin{scope}[>={Stealth[black]},every edge/.style={draw=black,very thick}]
    \path [-] (1) edge node[bar]{\rule{1pt}{6pt}} (2);
    \path [-] (2) edge node[bar]{\rule{1pt}{6pt}} (3);
    \path [-] (1) edge node[bar]{\rule{1pt}{6pt}} (4);
    \path [-] (4) edge[bend left=40] node[bar]{\rule{1pt}{6pt}} (5);
    \path [-] (4) edge[bend right=40] (5);
\end{scope}
\begin{scope}[>={Stealth[black]},every edge/.style={draw=black,very thick,dashed}]
    \path [-] (2) edge (4);
\end{scope}
\end{tikzpicture}}
+
{\tiny\begin{tikzpicture}[baseline={([yshift=-.8ex]current bounding box.center)},scale=0.7]
\begin{scope}[every node/.style={circle,draw,fill,inner sep=0pt,minimum size=3pt}]
    \node (1) at (0,0) {};
    \node (2) at (-0.5,0.8) {};
    \node (3) at (-0.5,1.8) {};
    \node (4) at (0.5,0.8) {};
    \node (5) at (0.5,1.8) {};
\end{scope}
\begin{scope}[>={Stealth[black]},every edge/.style={draw=black,very thick}]
    \path [-] (1) edge node[bar]{\rule{1pt}{6pt}} (2);
    \path [-] (2) edge node[bar]{\rule{1pt}{6pt}} (3);
    \path [-] (1) edge node[bar]{\rule{1pt}{6pt}} (4);
    \path [-] (4) edge node[bar]{\rule{1pt}{6pt}} (5);
\end{scope}
\begin{scope}[>={Stealth[black]},every edge/.style={draw=black,very thick,dashed}]
    \path [-] (2) edge (5);
\end{scope}
\end{tikzpicture}}
+
{\tiny\begin{tikzpicture}[baseline={([yshift=-.8ex]current bounding box.center)},scale=0.7]
\begin{scope}[every node/.style={circle,draw,fill,inner sep=0pt,minimum size=3pt}]
    \node (1) at (0,0) {};
    \node (2) at (-0.5,0.8) {};
    \node (3) at (-0.5,1.8) {};
    \node (4) at (0.5,0.8) {};
    \node (5) at (0.5,1.8) {};
\end{scope}
\begin{scope}[>={Stealth[black]},every edge/.style={draw=black,very thick}]
    \path [-] (1) edge node[bar]{\rule{1pt}{6pt}} (2);
    \path [-] (2) edge node[bar]{\rule{1pt}{6pt}} (3);
    \path [-] (1) edge node[bar]{\rule{1pt}{6pt}} (4);
    \path [-] (4) edge node[bar]{\rule{1pt}{6pt}} (5);
\end{scope}
\begin{scope}[>={Stealth[black]},every edge/.style={draw=black,very thick,dashed}]
    \path [-] (3) edge (5);
\end{scope}
\end{tikzpicture}}
+
{\tiny\begin{tikzpicture}[baseline={([yshift=-.8ex]current bounding box.center)},scale=0.7]
\begin{scope}[every node/.style={circle,draw,fill,inner sep=0pt,minimum size=3pt}]
    \node (1) at (0,0) {};
    \node (2) at (-0.5,0.8) {};
    \node (3) at (-0.5,1.8) {};
    \node (4) at (0.5,0.8) {};
    \node (5) at (0.5,1.8) {};
\end{scope}
\begin{scope}[>={Stealth[black]},every edge/.style={draw=black,very thick}]
    \path [-] (1) edge node[bar]{\rule{1pt}{6pt}} (2);
    \path [-] (2) edge[bend left=40] (3);
    \path [-] (2) edge[bend right=20] node[bar]{\rule{1pt}{6pt}} (3);
    \path [-] (1) edge node[bar]{\rule{1pt}{6pt}} (4);
    \path [-] (4) edge node[bar]{\rule{1pt}{6pt}} (5);
\end{scope}
\begin{scope}[>={Stealth[black]},every edge/.style={draw=black,very thick,dashed}]
    \path [-] (2) edge (5);
\end{scope}
\end{tikzpicture}}
+
{\tiny\begin{tikzpicture}[baseline={([yshift=-.8ex]current bounding box.center)},scale=0.7]
\begin{scope}[every node/.style={circle,draw,fill,inner sep=0pt,minimum size=3pt}]
    \node (1) at (0,0) {};
    \node (2) at (-0.5,0.8) {};
    \node (3) at (-0.5,1.8) {};
    \node (4) at (0.5,0.8) {};
    \node (5) at (0.5,1.8) {};
\end{scope}
\begin{scope}[>={Stealth[black]},every edge/.style={draw=black,very thick}]
    \path [-] (1) edge node[bar]{\rule{1pt}{6pt}} (2);
    \path [-] (2) edge[bend left=40] (3);
    \path [-] (2) edge[bend right=40] node[bar]{\rule{1pt}{6pt}} (3);
    \path [-] (1) edge node[bar]{\rule{1pt}{6pt}} (4);
    \path [-] (4) edge[bend left=40] node[bar]{\rule{1pt}{6pt}} (5);
    \path [-] (4) edge[bend right=40] (5);
\end{scope}
\begin{scope}[>={Stealth[black]},every edge/.style={draw=black,very thick,dashed}]
    \path [-] (2) edge (4);
\end{scope}
\end{tikzpicture}}\bigg).
\end{multline*}
For notational simplicity, we have used the trivial bound $K_i\lesssim1$ to remove extra solid edges within cycles, as they are unnecessary to ensure convergence.
Now appealing to the cancellation rules of~Lemma~\ref{lem:cancel}, the first three diagrams simplify to
\begin{multline*}
\int_{(\R^d)^4}\delta^{x,y}\varphi_\mu^\varnothing(0)\,\delta^{z,w}\varphi_\mu^\varnothing(0)\,\Big(f_4(x,y,z,w)-f_2(x,y)f_2(z,w)\Big)\,dxdydzdw\\
=
[\lambda_4]\bigg({\tiny\begin{tikzpicture}[baseline={([yshift=-.8ex]current bounding box.center)},scale=0.7]
\begin{scope}[every node/.style={circle,draw,fill,inner sep=0pt,minimum size=3pt}]
    \node (1) at (0,0) {};
    \node (2) at (-0.5,0.8) {};
    \node (4) at (0.5,0.8) {};
\end{scope}
\begin{scope}[>={Stealth[black]},every edge/.style={draw=black,very thick}]
    \path [-] (1) edge node[bar]{\rule{1pt}{6pt}\hspace{1.5pt}\rule{1pt}{6pt}} (2);
    \path [-] (1) edge node[bar]{\rule{1pt}{6pt}\hspace{1.5pt}\rule{1pt}{6pt}} (4);
\end{scope}
\begin{scope}[>={Stealth[black]},every edge/.style={draw=black,very thick,dashed}]
    \path [-] (2) edge (4);
\end{scope}
\end{tikzpicture}}
+
{\tiny\begin{tikzpicture}[baseline={([yshift=-.8ex]current bounding box.center)},scale=0.7]
\begin{scope}[every node/.style={circle,draw,fill,inner sep=0pt,minimum size=3pt}]
    \node (1) at (0,0) {};
    \node (2) at (-0.5,0.8) {};
    \node (4) at (0.5,0.8) {};
    \node (5) at (0.5,1.8) {};
\end{scope}
\begin{scope}[>={Stealth[black]},every edge/.style={draw=black,very thick}]
    \path [-] (1) edge node[bar]{\rule{1pt}{6pt}\hspace{1.5pt}\rule{1pt}{6pt}} (2);
    \path [-] (1) edge node[bar]{\rule{1pt}{6pt}} (4);
    \path [-] (4) edge[bend left=40] node[bar]{\rule{1pt}{6pt}} (5);
    \path [-] (4) edge[bend right=40] (5);
\end{scope}
\begin{scope}[>={Stealth[black]},every edge/.style={draw=black,very thick,dashed}]
    \path [-] (2) edge (4);
\end{scope}
\end{tikzpicture}}
+
{\tiny\begin{tikzpicture}[baseline={([yshift=-.8ex]current bounding box.center)},scale=0.7]
\begin{scope}[every node/.style={circle,draw,fill,inner sep=0pt,minimum size=3pt}]
    \node (1) at (0,0) {};
    \node (2) at (-0.5,0.8) {};
    \node (4) at (0.5,0.8) {};
    \node (5) at (0.5,1.8) {};
\end{scope}
\begin{scope}[>={Stealth[black]},every edge/.style={draw=black,very thick}]
    \path [-] (1) edge node[bar]{\rule{1pt}{6pt}\hspace{1.5pt}\rule{1pt}{6pt}} (2);
    \path [-] (1) edge node[bar]{\rule{1pt}{6pt}} (4);
    \path [-] (4) edge node[bar]{\rule{1pt}{6pt}} (5);
\end{scope}
\begin{scope}[>={Stealth[black]},every edge/.style={draw=black,very thick,dashed}]
    \path [-] (2) edge (5);
\end{scope}
\end{tikzpicture}}
+
{\tiny\begin{tikzpicture}[baseline={([yshift=-.8ex]current bounding box.center)},scale=0.7]
\begin{scope}[every node/.style={circle,draw,fill,inner sep=0pt,minimum size=3pt}]
    \node (1) at (0,0) {};
    \node (2) at (-0.5,0.8) {};
    \node (3) at (-0.5,1.8) {};
    \node (4) at (0.5,0.8) {};
    \node (5) at (0.5,1.8) {};
\end{scope}
\begin{scope}[>={Stealth[black]},every edge/.style={draw=black,very thick}]
    \path [-] (1) edge node[bar]{\rule{1pt}{6pt}} (2);
    \path [-] (2) edge node[bar]{\rule{1pt}{6pt}} (3);
    \path [-] (1) edge node[bar]{\rule{1pt}{6pt}} (4);
    \path [-] (4) edge node[bar]{\rule{1pt}{6pt}} (5);
\end{scope}
\begin{scope}[>={Stealth[black]},every edge/.style={draw=black,very thick,dashed}]
    \path [-] (3) edge (5);
\end{scope}
\end{tikzpicture}}
+
{\tiny\begin{tikzpicture}[baseline={([yshift=-.8ex]current bounding box.center)},scale=0.7]
\begin{scope}[every node/.style={circle,draw,fill,inner sep=0pt,minimum size=3pt}]
    \node (1) at (0,0) {};
    \node (2) at (-0.5,0.8) {};
    \node (3) at (-0.5,1.8) {};
    \node (4) at (0.5,0.8) {};
    \node (5) at (0.5,1.8) {};
\end{scope}
\begin{scope}[>={Stealth[black]},every edge/.style={draw=black,very thick}]
    \path [-] (1) edge node[bar]{\rule{1pt}{6pt}} (2);
    \path [-] (2) edge[bend left=40] (3);
    \path [-] (2) edge[bend right=20] node[bar]{\rule{1pt}{6pt}} (3);
    \path [-] (1) edge node[bar]{\rule{1pt}{6pt}} (4);
    \path [-] (4) edge node[bar]{\rule{1pt}{6pt}} (5);
\end{scope}
\begin{scope}[>={Stealth[black]},every edge/.style={draw=black,very thick,dashed}]
    \path [-] (2) edge (5);
\end{scope}
\end{tikzpicture}}
+
{\tiny\begin{tikzpicture}[baseline={([yshift=-.8ex]current bounding box.center)},scale=0.7]
\begin{scope}[every node/.style={circle,draw,fill,inner sep=0pt,minimum size=3pt}]
    \node (1) at (0,0) {};
    \node (2) at (-0.5,0.8) {};
    \node (3) at (-0.5,1.8) {};
    \node (4) at (0.5,0.8) {};
    \node (5) at (0.5,1.8) {};
\end{scope}
\begin{scope}[>={Stealth[black]},every edge/.style={draw=black,very thick}]
    \path [-] (1) edge node[bar]{\rule{1pt}{6pt}} (2);
    \path [-] (2) edge[bend left=40] (3);
    \path [-] (2) edge[bend right=40] node[bar]{\rule{1pt}{6pt}} (3);
    \path [-] (1) edge node[bar]{\rule{1pt}{6pt}} (4);
    \path [-] (4) edge[bend left=40] node[bar]{\rule{1pt}{6pt}} (5);
    \path [-] (4) edge[bend right=40] (5);
\end{scope}
\begin{scope}[>={Stealth[black]},every edge/.style={draw=black,very thick,dashed}]
    \path [-] (2) edge (4);
\end{scope}
\end{tikzpicture}}\bigg).
\end{multline*}
Appealing to the diagrammatic calculation rules of Section~\ref{sec:diag}, we then get
\begin{equation*}
{\bigg|\int_{(\R^d)^4}\delta^{x,y}\varphi_\mu^\varnothing(0)\,\delta^{z,w}\varphi_\mu^\varnothing(0)\,\Big(f_4(x,y,z,w)-f_2(x,y)f_2(z,w)\Big)\,dxdydzdw\bigg|}
\lesssim[\lambda_4]\Lc^3{\tiny\begin{tikzpicture}[baseline={([yshift=-.8ex]current bounding box.center)},scale=0.7]
\begin{scope}[every node/.style={circle,draw,fill,inner sep=0pt,minimum size=3pt}]
    \node (1) at (0,0) {};
    \node (2) at (0,1) {};
\end{scope}
\begin{scope}[>={Stealth[black]},every edge/.style={draw=black,very thick}]
    \path [-] (1) edge[bend left=40] node[bar]{\rule{1pt}{6pt}\hspace{1.5pt}\rule{1pt}{6pt}\hspace{1.5pt}\rule{1pt}{6pt}\hspace{1.5pt}\rule{1pt}{6pt}} (2);
\end{scope}
\begin{scope}[>={Stealth[black]},every edge/.style={draw=black,very thick,dashed}]
    \path [-] (1) edge[bend right=40] (2);
\end{scope}
\end{tikzpicture}},
\end{equation*}
which tends to $0$ after multiplying by $\mu$.
Arguing similarly for the two other terms in~\eqref{eq:decomp-Cmu3}, we can conclude
\[\lim_{\mu\downarrow0}C^3_\mu=0.\]
A similar analysis can also be performed for $C^4_\mu$,
thus concluding the proof of~\eqref{eq:umu-est}; we omit the details for brevity.

\medskip
\step2 Energy estimate for $\nabla u_\mu$: proof that
\begin{equation}\label{eq:energy-umu}
\limsup_{\mu\downarrow0}\E[|\nabla u_\mu|^2]
\,\lesssim\,
\sum_{i=1}^4\limsup_{\mu\downarrow0}D^i_\mu,
\end{equation}
in terms of
\begin{eqnarray*}
D^1_\mu&:=&\Big(\lambda+\lambda_2^\frac\beta{2+\beta}\Big)|\mu\E[u_\mu]|^2,\\
D^2_\mu&:=&\mu^3\Var\bigg[G_\mu\star\bigg(\sum_{x\in\Pc}\mathds1_{B(x)}\sum_{y,z,w\in\Pc\setminus\{x\}}^{\ne}\delta^{y,z,w}\varphi_\mu^\varnothing\bigg)\bigg],\\
D^3_\mu&:=&\mu^2\Var\bigg[\nabla G_\mu\star\bigg(\sum_{x\in\Pc}\mathds1_{B(x)}\sum_{y,z,w\in\Pc\setminus\{x\}}^{\ne}\delta^{y,z,w}\varphi_\mu^\varnothing\bigg)\bigg],\\
D^4_\mu&:=&\E\bigg[\sum_{x\in\Pc}\mathds1_{B_+(x)}\Big|\!\sum_{y,z,w\in\Pc\setminus\{x\}}^{\ne}\!\!\!\nabla_\mu\delta^{y,z,w}\varphi_\mu^\varnothing\Big|^2\bigg].
\end{eqnarray*}
By definition of $u_\mu$, cf.~\eqref{eq:def-uL}, combining the weak formulation of~\eqref{eq:correct-sed-mu} and the equation~\eqref{eq:delYphi-dev} for differences of finite-particle flows, we find in the weak sense in $\R^d$, after reorganizing the terms,
\begin{multline*}
\mu u_\mu-\Div(\sigma_\mu(u_\mu))
=-\sum_{x\in\Pc}\delta_{\partial B(x)}\sigma_\mu\Big(\varphi_\mu-\varphi_\mu^x-\sum_{y\in\Pc\setminus\{x\}}\delta^y\varphi_\mu^x-\frac12\sum_{y,z\in\Pc\setminus\{x\}}^{\ne}\delta^{y,z}\varphi_\mu^x\Big)\nu\\
+\mu\sum_{x\in\Pc}\mathds1_{B(x)}\Big(\varphi_\mu+\frac1\mu\alpha e-\varphi_\mu^x-\sum_{y\in\Pc\setminus\{x\}}\delta^y\varphi_\mu^x-\frac12\sum_{y,z\in\Pc\setminus\{x\}}^{\ne}\delta^{y,z}\sigma_\mu^x\Big).
\end{multline*}
Testing with $u_\mu-\E[u_\mu]$, this yields
\begin{equation}\label{eq:decomp-energy-umu}
\mu \E[|u_\mu-\E[u_\mu]|^2]+2\E[|\!\D(u_\mu)|^2]+\tfrac1\mu\E[|\Div(u_\mu)|^2]
\,=\, E_\mu^1+E_\mu^2,
\end{equation}
in terms of
\begin{eqnarray*}
E_\mu^1&\!\!\!:=\!\!\!&-\E\bigg[(u_\mu-\E[u_\mu])\cdot\sum_{x\in\Pc}\delta_{\partial B(x)}\sigma_\mu\Big(\varphi_\mu-\varphi_\mu^x-\sum_{y\in\Pc\setminus\{x\}}\delta^y\varphi_\mu^x-\frac12\sum_{y,z\in\Pc\setminus\{x\}}^{\ne}\delta^{y,z}\varphi_\mu^x\Big)\nu\bigg],\\[-1mm]
E_\mu^2&\!\!\!:=\!\!\!&\mu\E\bigg[(u_\mu-\E[u_\mu])\\[-5mm]
&&\hspace{1cm}\cdot\sum_{x\in\Pc}\mathds1_{B(x)}\Big(\varphi_\mu+\frac1\mu\alpha e-\varphi_\mu^x-\sum_{y\in\Pc\setminus\{x\}}\delta^y\varphi_\mu^x-\frac12\sum_{y,z\in\Pc\setminus\{x\}}^{\ne}\delta^{y,z}\sigma_\mu^x\Big)\bigg].
\end{eqnarray*}
We examine these two terms separately.
As in Step~4 of the proof of Theorem~\ref{th:mean-velocity}, let $\chi\in C^\infty_c(\R^d)$ be nonnegative with $\int_{\R^d}\chi=1$, set $\chi_r:=\chi(\frac\cdot r)$, and modify it into some $\tilde\chi_r$ that is constant in a neighborhood of each inclusion and satisfies
\[|\tilde\chi_r-\chi_r|+|\nabla\tilde\chi_r|\lesssim r^{-1},\qquad\tilde\chi_r|_{\Pc}=\chi_r|_{\Pc}.\]
In these terms, by the ergodic theorem and the properties of $\tilde\chi_r$, the expectation defining~$E_\mu^1$ can be written as follows, almost surely,
\begin{multline*}
E_\mu^1
\,=\,-\lim_{r\uparrow\infty}r^{-d}\sum_{x\in\Pc}\int_{\partial B(x)}\tilde\chi_r (u_\mu-\E[u_\mu])\cdot\sigma_\mu\Big(\varphi_\mu-\varphi_\mu^x-\sum_{y\in\Pc\setminus\{x\}}\delta^y\varphi_\mu^x-\frac12\sum_{y,z\in\Pc\setminus\{x\}}^{\ne}\delta^{y,z}\varphi_\mu^x\Big)\nu\\
\,=\,-\lim_{r\uparrow\infty}r^{-d}\sum_{x\in\Pc}\chi_r(x)\int_{\partial B(x)}(u_\mu-\E[u_\mu])\cdot\sigma_\mu\Big(\varphi_\mu-\varphi_\mu^x-\sum_{y\in\Pc\setminus\{x\}}\delta^y\varphi_\mu^x-\frac12\sum_{y,z\in\Pc\setminus\{x\}}^{\ne}\delta^{y,z}\varphi_\mu^x\Big)\nu.
\end{multline*}
Note that for any $x\in\Pc$ the definition~\eqref{eq:def-uL} of $u_\mu$ can be reorganized as
\begin{equation}\label{eq:umu-reorg}
u_\mu
\,=\,\varphi_\mu+\frac1\mu\alpha e-\varphi_\mu^x-\sum_{y\in\Pc\setminus\{x\}}\delta^{y}\varphi_\mu^x
-\frac1{2}\sum_{y,z\in\Pc\setminus\{x\}}^{\ne}\delta^{y,z}\varphi_\mu^x
-\frac1{3!}\sum_{y,z,w\in\Pc\setminus\{x\}}^{\ne}\delta^{y,z,w}\varphi_\mu^\varnothing.
\end{equation}
Using the rigidity constraints, the above then becomes
\begin{multline*}
E_\mu^1
\,=\,\lim_{r\uparrow\infty}r^{-d}\frac1{3!}\sum_{x\in\Pc}\chi_r(x)\int_{\partial B(x)}\bigg(\sum_{y,z,w\in\Pc\setminus\{x\}}^{\ne}\!\!\delta^{y,z,w}\varphi_\mu^\varnothing-\fint_{B(x)}\sum_{y,z,w\in\Pc\setminus\{x\}}^{\ne}\!\!\delta^{y,z,w}\varphi_\mu^\varnothing\bigg)\\
\cdot\sigma_\mu\Big(u_\mu+\frac1{3!}\sum_{y,z,w\in\Pc\setminus\{x\}}^{\ne}\!\!\delta^{y,z,w}\varphi_\mu^\varnothing\Big)\nu.
\end{multline*}
By the trace estimate of Lemma~\ref{lem:trace-est}, this yields
\begin{multline*}
|E_\mu^1|
\,\lesssim\,\lim_{r\uparrow\infty}r^{-d}\sum_{x\in\Pc}\chi_r(x)\bigg(\int_{B(x)}\Big|\sum_{y,z,w\in\Pc\setminus\{x\}}^{\ne}\nabla_\mu\delta^{y,z,w}\varphi_\mu^\varnothing\Big|^2\bigg)^\frac12\\[-1mm]
\times\bigg(\int_{B_+(x)}|\nabla_\mu u_\mu|^2
+\Big|\sum_{y,z,w\in\Pc\setminus\{x\}}^{\ne}\!\!\nabla_\mu\delta^{y,z,w}\varphi_\mu^\varnothing\Big|^2\bigg)^\frac12.
\end{multline*}
and thus, by the Cauchy-Schwarz inequality and the ergodic theorem,
\begin{multline}\label{eq:estim-Amu1}
|E_\mu^1|
\,\lesssim\,\E\bigg[\sum_{x\in\Pc}\mathds1_{B(x)}\Big|\sum_{y,z,w\in\Pc\setminus\{x\}}^{\ne}\nabla_\mu \delta^{y,z,w}\varphi_\mu^\varnothing\Big|^2\bigg]^\frac12\\[-2mm]
\times\E\bigg[|\nabla_\mu u_\mu|^2+\sum_{x\in\Pc}\mathds1_{B_+(x)}\Big|\sum_{y,z,w\in\Pc\setminus\{x\}}^{\ne}\!\!\nabla_\mu\delta^{y,z,w}\varphi_\mu^\varnothing\Big|^2\bigg]^\frac12.
\end{multline}
We turn to the analysis of $E_\mu^2$. Using identity~\eqref{eq:umu-reorg} again, it takes the form
\begin{eqnarray*}
E_\mu^2
&=&\mu\E\bigg[(u_\mu-\E[u_\mu])
\cdot\sum_{x\in\Pc}\mathds1_{B(x)}\Big(u_\mu+\frac1{3!}\sum_{y,z,w\in\Pc\setminus\{x\}}^{\ne}\delta^{y,z,w}\varphi_\mu^\varnothing\Big)\bigg]\\
&=&\mu\E\big[\mathds1_{\Ic}(u_\mu-\E[u_\mu])
\cdot u_\mu\big]
+\frac1{3!}\mu\,\E\bigg[(u_\mu-\E[u_\mu])
\cdot\sum_{x\in\Pc}\mathds1_{B(x)}\sum_{y,z,w\in\Pc\setminus\{x\}}^{\ne}\delta^{y,z,w}\varphi_\mu^\varnothing\bigg].
\end{eqnarray*}
Let us examine the two right-hand side terms separately. On the one hand, adding and subtracting $\E[u_\mu]$ in the first term, we find
\[\mu\E\big[\mathds1_{\Ic}(u_\mu-\E[u_\mu])
\cdot u_\mu\big]
\,=\,\mu\E\big[\mathds1_\Ic|u_\mu-\E[u_\mu]|^2\big]+\E\big[\mathds1_\Ic(u_\mu-\E[u_\mu])\big]\cdot\mu\E[u_\mu],\]
and thus, recognizing the remainder term $R_\mu=\E[\mathds1_\Ic(u_\mu-\E[u_\mu])]$ itself, and recalling the bounds~\eqref{eq:preest-Rmu}, \eqref{eq:estim-G1I}, and~\eqref{eq:umu-est}, we get
\begin{equation*}
\limsup_{\mu\downarrow0}\mu\big|\E\big[\mathds1_{\Ic}(u_\mu-\E[u_\mu])
\cdot u_\mu\big]\!\big|
\,\lesssim\,\Big(\lambda+(\lambda_2)^\frac\beta{2+\beta}\Big)^\frac12\limsup_{\mu\downarrow0}|\mu\E[u_\mu]|\,\E[|\nabla u_\mu|^2]^\frac12.
\end{equation*}
On the other hand, arguing similarly as in~\eqref{eq:preest-Rmu}, in terms of the massive Green function~$G_\mu$, cf.~\eqref{eq:Greenfct-mass}, an integration by parts yields
\begin{multline*}
\mu\bigg|\E\bigg[(u_\mu-\E[u_\mu])
\cdot\sum_{x\in\Pc}\mathds1_{B(x)}\sum_{y,z,w\in\Pc\setminus\{x\}}^{\ne}\delta^{y,z,w}\varphi_\mu^\varnothing\bigg]\bigg|\\
\,\lesssim\,\E\Big[\mu|u_\mu-\E[u_\mu]|^2+|\nabla u_\mu|^2\Big]^\frac12
\bigg(\mu^3\Var\bigg[G_\mu\star\sum_{x\in\Pc}\mathds1_{B(x)}\sum_{y,z,w\in\Pc\setminus\{x\}}^{\ne}\delta^{y,z,w}\varphi_\mu^\varnothing\bigg]\\
+\mu^2\Var\bigg[\nabla G_\mu\star\sum_{x\in\Pc}\mathds1_{B(x)}\sum_{y,z,w\in\Pc\setminus\{x\}}^{\ne}\delta^{y,z,w}\varphi_\mu^\varnothing\bigg]\bigg)^\frac12.
\end{multline*}
Combining these two estimates, and using~\eqref{eq:umu-est} again, we are led to the following bound on~$E_\mu^2$,
\begin{multline*}
\limsup_{\mu\downarrow0}|E_\mu^2|\,\lesssim\,
\limsup_{\mu\downarrow0}\E[|\nabla u_\mu|^2]^\frac12
\bigg(\Big(\lambda+\lambda_2^\frac\beta{2+\beta}\Big)|\mu\E[u_\mu]|^2\\[-2mm]
+\mu^3\Var\bigg[G_\mu\star\sum_{x\in\Pc}\mathds1_{B(x)}\sum_{y,z,w\in\Pc\setminus\{x\}}^{\ne}\delta^{y,z,w}\varphi_\mu\bigg]\\[-2mm]
+\mu^2\Var\bigg[\nabla G_\mu\star\sum_{x\in\Pc}\mathds1_{B(x)}\sum_{y,z,w\in\Pc\setminus\{x\}}^{\ne}\delta^{y,z,w}\varphi_\mu\bigg]\bigg)^\frac12.
\end{multline*}
Inserting this and~\eqref{eq:estim-Amu1} into~\eqref{eq:decomp-energy-umu}, and recalling Korn's inequality
\[\E[|\nabla u_\mu|^2]=2\E[|\!\D(u_\mu)|^2]-\E[\Div(u_\mu)^2]\le2\E[|\!\D(u_\mu)|^2],\]
the claim~\eqref{eq:energy-umu} follows.

\medskip
\step3 Proof that
\begin{equation*}
\limsup_{\mu\downarrow0}D^1_\mu\,\lesssim\,\lambda_9+\lambda_{10}^\frac{\beta}{2+\beta}.
\end{equation*}
By definition of $D^1_\mu$ and by~\eqref{eq:submult-lambda}, it suffices to show
\begin{equation}\label{eq:conv-Emuumu}
\limsup_{\mu\downarrow0}\mu|\E[u_\mu]|\,\lesssim\,\lambda^4.
\end{equation}
More precisely, we shall show
\begin{eqnarray}
\lim_{\mu\downarrow0}\mu\E\bigg[\sum_{x\in\Pc}\varphi_\mu^x\bigg]&=&\lambda|B|e,\label{eq:exp-phix}\\
\lim_{\mu\downarrow0}\mu\E\bigg[\frac12\sum_{x,y\in\Pc}^{\ne}\delta^{x,y}\varphi_\mu^\varnothing\bigg]&=&(\lambda|B|)^2e,\label{eq:exp-phixy}\\
\lim_{\mu\downarrow0}\mu\E\bigg[\frac1{3!}\sum_{x,y,z\in\Pc}^{\ne}\delta^{x,y,z}\varphi_\mu^\varnothing\bigg]&=&(\lambda|B|)^3e,\label{eq:exp-phixyz}
\end{eqnarray}
which implies, by definition of $u_\mu$ and of $\alpha$, cf.~\eqref{eq:def-uL} and~\eqref{eq:def-alpha},
\[\lim_{\mu\downarrow0}\mu\E[u_\mu]\,=\,\alpha e-\sum_{k=1}^3(\lambda|B|)^ke\,=\,\alpha e-\frac{1-(\lambda|B|)^3}{1-\lambda|B|}\lambda|B|e\,=\,O(\lambda^4),\]
that is, \eqref{eq:conv-Emuumu}.
We turn to the proof of~\eqref{eq:exp-phix}--\eqref{eq:exp-phixyz} and we start with single-particle flows.
By translation invariance, we can write
\[\mu\E\bigg[\sum_{x\in\Pc}\varphi_\mu^x\bigg]\,=\,\lambda\mu\int_{\R^d}\varphi_\mu^0.\]
Recalling~\eqref{eq:int-phi0mu}, we get
\[\mu\E\bigg[\sum_{x\in\Pc}\varphi_\mu^x\bigg]\,=\,\lambda|B|e+\lambda\mu\int_{B}\varphi_\mu^0,\]
which proves~\eqref{eq:exp-phix} by Lemma~\ref{lem:conv-cor-finite}.
We turn to the proof of~\eqref{eq:exp-phixy}. In terms of density functions, using the reflection-block expansion~\eqref{eq:expand-deltxy}, we can write
\begin{eqnarray*}
\mu\E\bigg[\frac12\sum_{x,y\in\Pc}^{\ne}\delta^{x,y}\varphi_\mu^\varnothing\bigg]
&=&\frac\mu2\iint_{(\R^d)^2}\delta^{0,y}\varphi_\mu^\varnothing(x)\,f_2(0,y)\,dxdy\\
&=&\mu\iint_{(\R^d)^2} \Jc^0_\mu(\Jc_{\mu;0}^y\varphi_\mu^{y}+\Gc_{\mu;0}^y)(x)\,f_2(0,y)\,dxdy\\
&&\qquad+\mu\iint_{(\R^d)^2} \Jc^0_\mu\Jc_{\mu;0}^y(\Jc_{\mu;y}^0\varphi_\mu^{0,y}+\Gc_{\mu;y}^0)(x)\,f_2(0,y)\,dxdy.
\end{eqnarray*}
Decomposing $f_2=\lambda^2+h_2$ in the first term, and appealing to the cancellation rules of Lemma~\ref{lem:cancel} in the form
\begin{eqnarray*}
\mu\iint_{(\R^d)^2} \Jc^0_\mu(\Jc_{\mu;0}^y\varphi_\mu^{y}+\Gc_{\mu;0}^y)(x)\,dxdy
&=&\mu\sum_{i=1}^d\Big(\int_B(e+\mu\varphi_\mu^0)\Big)_i\int_{\R^d} (\Jc^0_\mu\varphi_{\mu,i}^0+\Gc_{\mu,i}^0)\\
&=&\sum_{i=1}^d\Big(\int_B(e+\mu\varphi_\mu^0)\Big)_i\Big(\int_B(e_i+\mu\varphi_{\mu,i}^0)\Big)\\
&=&e|B|^2+O(\mu),
\end{eqnarray*}
as well as in the forms
\begin{eqnarray*}
\lefteqn{\mu\iint_{(\R^d)^2} \Jc^0_\mu(\Jc_{\mu;0}^y\varphi_\mu^{y}+\Gc_{\mu;0}^y)(x)\,h_2(0,y)\,dxdy}\\[-2mm]
&\hspace{3cm}=&\mu\int_B\Big(\int_{\R^d}(\Jc_{\mu;0}^y\varphi_\mu^{y}+\Gc_{\mu;0}^y)\,h_2(0,y)\,dy\Big),\\
\lefteqn{\mu\iint_{(\R^d)^2} \Jc^0_\mu\Jc_{\mu;0}^y(\Jc_{\mu;y}^0\varphi_\mu^{0,y}+\Gc_{\mu;y}^0)(x)\,f_2(0,y)\,dxdy}\\[-2mm]
&\hspace{3cm}=&\mu\int_{B}\Big(\int_{\R^d}\Jc_{\mu;0}^y(\Jc_{\mu;y}^0\varphi_\mu^{0,y}+\Gc_{\mu;y}^0)f_2(0,y)dy\Big),
\end{eqnarray*}
we are led to
\begin{multline*}
\bigg|\mu\E\bigg[\frac12\sum_{x,y\in\Pc}^{\ne}\delta^{x,y}\varphi_\mu^\varnothing\bigg]-(\lambda|B|)^2e\bigg|
\,\le\,\mu\int_B\Big|\int_{\R^d}(\Jc_{\mu;0}^y\varphi_\mu^{y}+\Gc_{\mu;0}^y)\,h_2(0,y)\,dy\Big|\\
+\mu\int_{B}\Big|\int_{\R^d}\Jc_{\mu;0}^y(\Jc_{\mu;y}^0\varphi_\mu^{0,y}+\Gc_{\mu;y}^0)f_2(0,y)dy\Big|
+O(\mu),
\end{multline*}
or equivalently, using diagrammatic notation,
\begin{equation*}
\bigg|\mu\E\bigg[\frac12\sum_{x,y\in\Pc}^{\ne}\delta^{x,y}\varphi_\mu^\varnothing\bigg]-(\lambda|B|)^2e\bigg|
\,\le\,\mu{\tiny\begin{tikzpicture}[baseline={([yshift=-.8ex]current bounding box.center)},scale=0.7]
\begin{scope}[every node/.style={circle,fill,inner sep=0pt,minimum size=3pt}]
    \node (1) at (0,0) {};
    \node (2) at (0,0.8) {};
\end{scope}
\begin{scope}[>={Stealth[black]},every edge/.style={draw=black,very thick}]
    \path [-] (1) edge [bend left=40] node[bar]{\rule{1pt}{6pt}\hspace{1.5pt}\rule{1pt}{6pt}} (2);
\end{scope}
\begin{scope}[>={Stealth[black]},every edge/.style={draw=black,very thick,dashed}]
    \path [-] (1) edge [bend right=40] (2);
\end{scope}
\end{tikzpicture}}
+
\mu{\tiny\begin{tikzpicture}[baseline={([yshift=-.8ex]current bounding box.center)},scale=0.7]
\begin{scope}[every node/.style={circle,fill,inner sep=0pt,minimum size=3pt}]
    \node (1) at (0,0) {};
    \node (2) at (0,0.8) {};
\end{scope}
\begin{scope}[>={Stealth[black]},every edge/.style={draw=black,very thick}]
    \path [-] (1) edge [bend left=40] node[bar]{\rule{1pt}{6pt}} (2);
    \path [-] (1) edge [bend right=40] node[bar]{\rule{1pt}{6pt}} (2);
\end{scope}
\end{tikzpicture}}
+O(\mu).
\end{equation*}
By the calculation rules of Section~\ref{sec:diag}, the whole right-hand side is $O(\mu)$, and~\eqref{eq:exp-phixy} follows. The proof of~\eqref{eq:exp-phixyz} follows along the same lines and is skipped for brevity.

\medskip
\step4 Proof that
\begin{equation}\label{eq:A2conv0}
\lim_{\mu\downarrow0}D^2_\mu\,=\,0.
\end{equation}
In terms of density functions, distinguishing the possible intersection cases, we can decompose as follows the variance defining $D_\mu^2$,
\begin{eqnarray}
\lefteqn{\Var\bigg[G_\mu\star\sum_{x\in\Pc}\mathds1_{B(x)}\sum_{y,z,w\in\Pc\setminus\{x\}}^{\ne}\delta^{y,z,w}\varphi_\mu^\varnothing\bigg]}\label{eq:decomp-VarGast}\\
&\lesssim&\Big|\int_{(\R^d)^4}|G_\mu\star(\mathds1_{B(x_1)}\delta^{x_2,x_3,x_4}\varphi_\mu^\varnothing)(0)|^2\,f_4(\bar x)\,d\bar x\Big|\nonumber\\
&+&\Big|\int_{(\R^d)^5}G_\mu\star(\mathds1_{B(x_1)}\delta^{x_2,x_3,x_4}\varphi_\mu^\varnothing)(0)\,G_\mu\star(\mathds1_{B(x_1)}\delta^{x_2,x_3,x_4'}\varphi_\mu^\varnothing)(0)\,f_5(\bar x,\bar x')\,d(\bar x,\bar x')\Big|\nonumber\\
&+&\Big|\int_{(\R^d)^5}G_\mu\star(\mathds1_{B(x_1)}\delta^{x_2,x_3,x_4}\varphi_\mu^\varnothing)(0)\,G_\mu\star(\mathds1_{B(x'_1)}\delta^{x_2,x_3,x_4}\varphi_\mu^\varnothing)(0)\,f_5(\bar x,\bar x')\,d(\bar x,\bar x')\Big|\nonumber\\
&+&\Big|\int_{(\R^d)^6}G_\mu\star(\mathds1_{B(x_1)}\delta^{x_2,x_3,x_4}\varphi_\mu^\varnothing)(0)\,G_\mu\star(\mathds1_{B(x_1)}\delta^{x_2,x_3',x_4'}\varphi_\mu^\varnothing)(0)\,f_6(\bar x,\bar x')\,d(\bar x,\bar x')\Big|\nonumber\\
&+&\Big|\int_{(\R^d)^6}G_\mu\star(\mathds1_{B(x_1)}\delta^{x_2,x_3,x_4}\varphi_\mu^\varnothing)(0)\,G_\mu\star(\mathds1_{B(x'_1)}\delta^{x_2,x_3,x_4'}\varphi_\mu^\varnothing)(0)\,f_6(\bar x,\bar x')\,d(\bar x,\bar x')\Big|\nonumber\\
&+&\Big|\int_{(\R^d)^7}G_\mu\star(\mathds1_{B(x_1)}\delta^{x_2,x_3,x_4}\varphi_\mu^\varnothing)(0)\,G_\mu\star(\mathds1_{B(x_1)}\delta^{x_2',x_3',x_4'}\varphi_\mu^\varnothing)(0)\,f_7(\bar x,\bar x')\,d(\bar x,\bar x')\Big|\nonumber\\
&+&\Big|\int_{(\R^d)^7}G_\mu\star(\mathds1_{B(x_1)}\delta^{x_2,x_3,x_4}\varphi_\mu^\varnothing)(0)\,G_\mu\star(\mathds1_{B(x'_1)}\delta^{x_2,x_3',x_4'}\varphi_\mu^\varnothing)(0)\,f_7(\bar x,\bar x')\,d(\bar x,\bar x')\Big|\nonumber\\
&+&\Big|\int_{(\R^d)^8}G_\mu\star(\mathds1_{B(x_1)}\delta^{x_2,x_3,x_4}\varphi_\mu^\varnothing)(0)\,G_\mu\star(\mathds1_{B(x'_1)}\delta^{x_2',x_3',x_4'}\varphi_\mu^\varnothing)(0)\nonumber\\[-2mm]
&&\hspace{7cm}\times\big(f_8(\bar x,\bar x')-f_4(\bar x)f_4(\bar x')\big)\,d(\bar x,\bar x')\Big|,\nonumber
\end{eqnarray}
where we use $\bar x$ (resp.~$\bar x'$) as a short-hand notation for the vector of the relevant $x_i$'s (resp.~$x_i'$'s) involved in each integral.
Each term in this decomposition can be estimated using the general strategy described in Section~\ref{sec:diag}.
We spell out the argument for the last term, since it contains the largest number of variables and gives rise to the least favorable diagrams. Set
\begin{multline*}
I_\mu\,:=\,\int_{(\R^d)^8}G_\mu\star(\mathds1_{B(x_1)}\delta^{x_2,x_3,x_4}\varphi_\mu^\varnothing)(0)\,G_\mu\star(\mathds1_{B(x'_1)}\delta^{x_2',x_3',x_4'}\varphi_\mu^\varnothing)(0)\\[-2mm]
\times\big(f_8(\bar x,\bar x')-f_4(\bar x)f_4(\bar x')\big)\,d(\bar x,\bar x').
\end{multline*}
We appeal to the reflection-block expansion~\eqref{eq:decomp-block-3} for third-order flow differences.
Bounding the local average of the massive Green function $G_\mu$ by $K_2$, it yields schematically
\begin{multline*}
G_\mu\star(\mathds1_{B(x_1)}\delta^{x_2,x_3,x_4}\varphi_\mu^\varnothing)(0)
\,=\,\sum_{i=1}^6G_\mu\star(\mathds1_{B(x_1)}\Jc_\mu^{x_2}S_i^{x_2,x_3,x_4}\varphi_\mu^\varnothing)(0)\\
\,=\,
{\tiny\begin{tikzpicture}[baseline={([yshift=-.8ex]current bounding box.center)},scale=0.7]
\begin{scope}[every node/.style={circle,draw,inner sep=0pt,minimum size=8pt}]
    \node (-1) at (0,-2) {$0$};
    \node (0) at (0,-1) {$x_1$};
    \node (1) at (0,0) {$x_2$};
    \node (2) at (-0.6,1) {$x_3$};
    \node (3) at (0.6,1) {$x_4$};
\end{scope}
\begin{scope}[>={Stealth[black]},every edge/.style={draw=black,very thick}]
    \path [-] (-1) edge node[bar]{\rule{1pt}{6pt}\hspace{1.5pt}\rule{1pt}{6pt}} (0);
    \path [-] (0) edge node[bar]{\rule{1pt}{6pt}} (1);
    \path [-] (1) edge[bend right=20] (2);
    \path [-] (1) edge[bend left=30] (2);
    \path [-] (1) edge (3);
    \path [-] (2) edge node[bar]{\rule{1pt}{6pt}} (3);
\end{scope}
\end{tikzpicture}}
+
{\tiny\begin{tikzpicture}[baseline={([yshift=-.8ex]current bounding box.center)},scale=0.7]
\begin{scope}[every node/.style={circle,draw,inner sep=0pt,minimum size=8pt}]
    \node (-1) at (0,-2) {$0$};
    \node (0) at (0,-1) {$x_1$};
    \node (1) at (0,0) {$x_2$};
    \node (2) at (-0.6,1) {$x_3$};
    \node (3) at (0.6,1) {$x_4$};
\end{scope}
\begin{scope}[>={Stealth[black]},every edge/.style={draw=black,very thick}]
    \path [-] (-1) edge node[bar]{\rule{1pt}{6pt}\hspace{1.5pt}\rule{1pt}{6pt}} (0);
    \path [-] (0) edge node[bar]{\rule{1pt}{6pt}} (1);
    \path [-] (1) edge[bend right=20] (2);
    \path [-] (1) edge[bend left=30] (2);
    \path [-] (1) edge[bend right=30] node[bar]{\rule{1pt}{6pt}} (3);
    \path [-] (1) edge[bend left=20] (3);
\end{scope}
\end{tikzpicture}}
+
{\tiny\begin{tikzpicture}[baseline={([yshift=-.8ex]current bounding box.center)},scale=0.7]
\begin{scope}[every node/.style={circle,draw,inner sep=0pt,minimum size=8pt}]
    \node (-1) at (0,-2) {$0$};
    \node (0) at (0,-1) {$x_1$};
    \node (1) at (0,0) {$x_2$};
    \node (2) at (-0.6,1) {$x_3$};
    \node (3) at (0.6,1) {$x_4$};
\end{scope}
\begin{scope}[>={Stealth[black]},every edge/.style={draw=black,very thick}]
    \path [-] (-1) edge node[bar]{\rule{1pt}{6pt}\hspace{1.5pt}\rule{1pt}{6pt}} (0);
    \path [-] (1) edge (2);
    \path [-] (2) edge (3);
    \path [-] (1) edge node[bar]{\rule{1pt}{6pt}} (3);
    \path [-] (1) edge node[bar]{\rule{1pt}{6pt}} (0);
\end{scope}
\end{tikzpicture}}
+
{\tiny\begin{tikzpicture}[baseline={([yshift=-.8ex]current bounding box.center)},scale=0.7]
\begin{scope}[every node/.style={circle,draw,inner sep=0pt,minimum size=8pt}]
    \node (-1) at (0,-2) {$0$};
    \node (0) at (0,-1) {$x_1$};
    \node (1) at (0,0) {$x_2$};
    \node (2) at (-0.6,1) {$x_3$};
    \node (3) at (0.6,1) {$x_4$};
\end{scope}
\begin{scope}[>={Stealth[black]},every edge/.style={draw=black,very thick}]
    \path [-] (-1) edge node[bar]{\rule{1pt}{6pt}\hspace{1.5pt}\rule{1pt}{6pt}} (0);
    \path [-] (0) edge node[bar]{\rule{1pt}{6pt}} (1);
    \path [-] (1) edge[bend right=20] (2);
    \path [-] (1) edge[bend left=30] (2);
    \path [-] (1) edge node[bar]{\rule{1pt}{6pt}} (3);
\end{scope}
\end{tikzpicture}}
+~
{\tiny\begin{tikzpicture}[baseline={([yshift=-.8ex]current bounding box.center)},scale=0.7]
\begin{scope}[every node/.style={circle,draw,inner sep=0pt,minimum size=8pt}]
    \node (-1) at (0,-2) {$0$};
    \node (0) at (0,-1) {$x_1$};
    \node (1) at (0,0) {$x_2$};
    \node (2) at (0,1) {$x_3$};
    \node (3) at (0,2) {$x_4$};
\end{scope}
\begin{scope}[>={Stealth[black]},every edge/.style={draw=black,very thick}]
    \path [-] (-1) edge node[bar]{\rule{1pt}{6pt}\hspace{1.5pt}\rule{1pt}{6pt}} (0);
    \path [-] (2) edge[bend left=30] node[bar]{\rule{1pt}{6pt}} (3);
    \path [-] (2) edge[bend right=30] (3);
    \path [-] (1) edge (2);
    \path [-] (0) edge node[bar]{\rule{1pt}{6pt}} (1);
\end{scope}
\end{tikzpicture}}
~+~
{\tiny\begin{tikzpicture}[baseline={([yshift=-.8ex]current bounding box.center)},scale=0.7]
\begin{scope}[every node/.style={circle,draw,inner sep=0pt,minimum size=8pt}]
    \node (-1) at (0,-2) {$0$};
    \node (0) at (0,-1) {$x_1$};
    \node (1) at (0,0) {$x_2$};
    \node (2) at (0,1) {$x_3$};
    \node (3) at (0,2) {$x_4$};
\end{scope}
\begin{scope}[>={Stealth[black]},every edge/.style={draw=black,very thick}]
    \path [-] (-1) edge node[bar]{\rule{1pt}{6pt}\hspace{1.5pt}\rule{1pt}{6pt}} (0);
    \path [-] (2) edge node[bar]{\rule{1pt}{6pt}} (3);
    \path [-] (1) edge (2);
    \path [-] (0) edge node[bar]{\rule{1pt}{6pt}} (1);
\end{scope}
\end{tikzpicture}}
~+~\Sym_{x_2,x_3,x_4},
\end{multline*}
where we recall that $\{S_i^{x_1,x_2,x_3}\}_{1\le i\le 6}$ is defined in~\eqref{eq:decomp-deltayzvarphi0}.
Inserting this expansion into the above, we are led to decompose $I_\mu$ as
\begin{equation}\label{eq:decomp-Imu1-6}
I_\mu\,=\,36\sum_{i,j=1}^6I_\mu^{i,j},
\end{equation}
in terms of
\begin{multline*}
I_\mu^{i,j}:=\int_{(\R^d)^8}G_\mu\star(\mathds1_{B(x_1)}\Jc^{x_2}_\mu S_i^{x_2,x_3,x_4})(0)\,G_\mu\star(\mathds1_{B(x'_1)}\Jc^{x_2'}_\mu S_j^{x_2',x_3',x_4'})(0)\\[-2mm]
\times\big(f_8(\bar x,\bar x')-f_4(\bar x)f_4(\bar x')\big)\,d(\bar x,\bar x').
\end{multline*}
Each term $I_\mu^{i,j}$ can be estimated using the general strategy described in Section~\ref{sec:diag}, and we focus for brevity on the last term $i=j=6$, which requires the most care.
For this purpose, we decompose the factor $f_8(\bar x,\bar x')-f_4(\bar x)f_4(\bar x')$ in terms of correlation functions, cf.~\eqref{eq:correl-expand}, which creates correlation couplings (dashed edges) between the variables in~$\bar x$ and in~$\bar x'$, and we apply the cancellation rules of Lemma~\ref{lem:cancel} at each cut edge.
This leads us to
\begin{multline*}
|I_\mu^{6,6}|\,\lesssim\,
[\lambda_8]\bigg(~~{\tiny\begin{tikzpicture}[baseline={([yshift=-.8ex]current bounding box.center)},scale=0.7]
\begin{scope}[every node/.style={circle,fill,inner sep=0pt,minimum size=3pt}]
    \node (1) at (0,-0.8) {};
    \node (2) at (0.5,0) {};
    \node (3) at (-0.5,0) {};
    \node (4) at (0.5,1) {};
    \node (5) at (-0.5,1) {};
    \node (6) at (0.5,2) {};
    \node (7) at (-0.5,2) {};
    \node (8) at (0.5,3) {};
    \node (9) at (-0.5,3) {};
\end{scope}
\begin{scope}[>={Stealth[black]},every edge/.style={draw=black,very thick}]
    \path [-] (1) edge node[bar]{\rule{1pt}{6pt}\hspace{1.5pt}\rule{1pt}{6pt}}(2);
    \path [-] (1) edge node[bar]{\rule{1pt}{6pt}\hspace{1.5pt}\rule{1pt}{6pt}} (3);
    \path [-] (2) edge node[bar]{\rule{1pt}{6pt}} (4);
    \path [-] (3) edge node[bar]{\rule{1pt}{6pt}} (5);
    \path [-] (4) edge (6);
    \path [-] (5) edge (7);
    \path [-] (6) edge node[bar]{\rule{1pt}{6pt}} (8);
    \path [-] (7) edge node[bar]{\rule{1pt}{6pt}} (9);
\end{scope}
\begin{scope}[>={Stealth[black]},every edge/.style={draw=black,very thick,dashed}]
    \path [-] (8) edge (9);
\end{scope}
\end{tikzpicture}}
+
{\tiny\begin{tikzpicture}[baseline={([yshift=-.8ex]current bounding box.center)},scale=0.7]
\begin{scope}[every node/.style={circle,fill,inner sep=0pt,minimum size=3pt}]
    \node (1) at (0,-0.8) {};
    \node (2) at (0.5,0) {};
    \node (3) at (-0.5,0) {};
    \node (4) at (0.5,1) {};
    \node (5) at (-0.5,1) {};
    \node (6) at (0.5,2) {};
    \node (7) at (-0.5,2) {};
    \node (8) at (0.5,3) {};
\end{scope}
\begin{scope}[>={Stealth[black]},every edge/.style={draw=black,very thick}]
    \path [-] (1) edge node[bar]{\rule{1pt}{6pt}\hspace{1.5pt}\rule{1pt}{6pt}}(2);
    \path [-] (1) edge node[bar]{\rule{1pt}{6pt}\hspace{1.5pt}\rule{1pt}{6pt}} (3);
    \path [-] (2) edge node[bar]{\rule{1pt}{6pt}} (4);
    \path [-] (3) edge node[bar]{\rule{1pt}{6pt}} (5);
    \path [-] (4) edge (6);
    \path [-] (5) edge node[bar]{\rule{1pt}{6pt}} (7);
    \path [-] (6) edge node[bar]{\rule{1pt}{6pt}} (8);
\end{scope}
\begin{scope}[>={Stealth[black]},every edge/.style={draw=black,very thick,dashed}]
    \path [-] (8) edge (7);
\end{scope}
\end{tikzpicture}}
+
{\tiny\begin{tikzpicture}[baseline={([yshift=-.8ex]current bounding box.center)},scale=0.7]
\begin{scope}[every node/.style={circle,fill,inner sep=0pt,minimum size=3pt}]
    \node (1) at (0,-0.8) {};
    \node (2) at (0.5,0) {};
    \node (3) at (-0.5,0) {};
    \node (4) at (0.5,1) {};
    \node (5) at (-0.5,1) {};
    \node (6) at (0.5,2) {};
    \node (8) at (0.5,3) {};
\end{scope}
\begin{scope}[>={Stealth[black]},every edge/.style={draw=black,very thick}]
    \path [-] (1) edge node[bar]{\rule{1pt}{6pt}\hspace{1.5pt}\rule{1pt}{6pt}}(2);
    \path [-] (1) edge node[bar]{\rule{1pt}{6pt}\hspace{1.5pt}\rule{1pt}{6pt}} (3);
    \path [-] (2) edge node[bar]{\rule{1pt}{6pt}} (4);
    \path [-] (3) edge node[bar]{\rule{1pt}{6pt}\hspace{1.5pt}\rule{1pt}{6pt}} (5);
    \path [-] (4) edge (6);
    \path [-] (6) edge node[bar]{\rule{1pt}{6pt}} (8);
\end{scope}
\begin{scope}[>={Stealth[black]},every edge/.style={draw=black,very thick,dashed}]
    \path [-] (8) edge (5);
\end{scope}
\end{tikzpicture}}
+\frac1\mu
{\tiny\begin{tikzpicture}[baseline={([yshift=-.8ex]current bounding box.center)},scale=0.7]
\begin{scope}[every node/.style={circle,fill,inner sep=0pt,minimum size=3pt}]
    \node (1) at (0,-0.8) {};
    \node (2) at (0.5,0) {};
    \node (3) at (-0.5,0) {};
    \node (4) at (0.5,1) {};
    \node (6) at (0.5,2) {};
    \node (8) at (0.5,3) {};
\end{scope}
\begin{scope}[>={Stealth[black]},every edge/.style={draw=black,very thick}]
    \path [-] (1) edge node[bar]{\rule{1pt}{6pt}\hspace{1.5pt}\rule{1pt}{6pt}}(2);
    \path [-] (1) edge node[bar]{\rule{1pt}{6pt}\hspace{1.5pt}\rule{1pt}{6pt}} (3);
    \path [-] (2) edge node[bar]{\rule{1pt}{6pt}} (4);
    \path [-] (4) edge (6);
    \path [-] (6) edge node[bar]{\rule{1pt}{6pt}} (8);
\end{scope}
\begin{scope}[>={Stealth[black]},every edge/.style={draw=black,very thick,dashed}]
    \path [-] (8) edge (3);
\end{scope}
\end{tikzpicture}}
+
{\tiny\begin{tikzpicture}[baseline={([yshift=-.8ex]current bounding box.center)},scale=0.7]
\begin{scope}[every node/.style={circle,fill,inner sep=0pt,minimum size=3pt}]
    \node (1) at (0,-0.8) {};
    \node (2) at (0.5,0) {};
    \node (3) at (-0.5,0) {};
    \node (4) at (0.5,1) {};
    \node (5) at (-0.5,1) {};
    \node (6) at (0.5,2) {};
    \node (7) at (-0.5,2) {};
\end{scope}
\begin{scope}[>={Stealth[black]},every edge/.style={draw=black,very thick}]
    \path [-] (1) edge node[bar]{\rule{1pt}{6pt}\hspace{1.5pt}\rule{1pt}{6pt}}(2);
    \path [-] (1) edge node[bar]{\rule{1pt}{6pt}\hspace{1.5pt}\rule{1pt}{6pt}} (3);
    \path [-] (2) edge node[bar]{\rule{1pt}{6pt}} (4);
    \path [-] (3) edge node[bar]{\rule{1pt}{6pt}} (5);
    \path [-] (4) edge node[bar]{\rule{1pt}{6pt}} (6);
    \path [-] (5) edge node[bar]{\rule{1pt}{6pt}} (7);
\end{scope}
\begin{scope}[>={Stealth[black]},every edge/.style={draw=black,very thick,dashed}]
    \path [-] (6) edge (7);
\end{scope}
\end{tikzpicture}}
+
{\tiny\begin{tikzpicture}[baseline={([yshift=-.8ex]current bounding box.center)},scale=0.7]
\begin{scope}[every node/.style={circle,fill,inner sep=0pt,minimum size=3pt}]
    \node (1) at (0,-0.8) {};
    \node (2) at (0.5,0) {};
    \node (3) at (-0.5,0) {};
    \node (4) at (0.5,1) {};
    \node (5) at (-0.5,1) {};
    \node (6) at (0.5,2) {};
\end{scope}
\begin{scope}[>={Stealth[black]},every edge/.style={draw=black,very thick}]
    \path [-] (1) edge node[bar]{\rule{1pt}{6pt}\hspace{1.5pt}\rule{1pt}{6pt}}(2);
    \path [-] (1) edge node[bar]{\rule{1pt}{6pt}\hspace{1.5pt}\rule{1pt}{6pt}} (3);
    \path [-] (2) edge node[bar]{\rule{1pt}{6pt}} (4);
    \path [-] (3) edge node[bar]{\rule{1pt}{6pt}\hspace{1.5pt}\rule{1pt}{6pt}} (5);
    \path [-] (4) edge node[bar]{\rule{1pt}{6pt}}  (6);
\end{scope}
\begin{scope}[>={Stealth[black]},every edge/.style={draw=black,very thick,dashed}]
    \path [-] (6) edge (5);
\end{scope}
\end{tikzpicture}}
+
\frac1\mu
{\tiny\begin{tikzpicture}[baseline={([yshift=-.8ex]current bounding box.center)},scale=0.7]
\begin{scope}[every node/.style={circle,fill,inner sep=0pt,minimum size=3pt}]
    \node (1) at (0,-0.8) {};
    \node (2) at (0.5,0) {};
    \node (3) at (-0.5,0) {};
    \node (4) at (0.5,1) {};
    \node (6) at (0.5,2) {};
\end{scope}
\begin{scope}[>={Stealth[black]},every edge/.style={draw=black,very thick}]
    \path [-] (1) edge node[bar]{\rule{1pt}{6pt}\hspace{1.5pt}\rule{1pt}{6pt}}(2);
    \path [-] (1) edge node[bar]{\rule{1pt}{6pt}\hspace{1.5pt}\rule{1pt}{6pt}} (3);
    \path [-] (2) edge node[bar]{\rule{1pt}{6pt}} (4);
    \path [-] (4) edge node[bar]{\rule{1pt}{6pt}}  (6);
\end{scope}
\begin{scope}[>={Stealth[black]},every edge/.style={draw=black,very thick,dashed}]
    \path [-] (6) edge (3);
\end{scope}
\end{tikzpicture}}
\\
+
{\tiny\begin{tikzpicture}[baseline={([yshift=-.8ex]current bounding box.center)},scale=0.7]
\begin{scope}[every node/.style={circle,fill,inner sep=0pt,minimum size=3pt}]
    \node (1) at (0,-0.8) {};
    \node (2) at (0.5,0) {};
    \node (3) at (-0.5,0) {};
    \node (4) at (0.5,1) {};
    \node (5) at (-0.5,1) {};
\end{scope}
\begin{scope}[>={Stealth[black]},every edge/.style={draw=black,very thick}]
    \path [-] (1) edge node[bar]{\rule{1pt}{6pt}\hspace{1.5pt}\rule{1pt}{6pt}}(2);
    \path [-] (1) edge node[bar]{\rule{1pt}{6pt}\hspace{1.5pt}\rule{1pt}{6pt}} (3);
    \path [-] (2) edge node[bar]{\rule{1pt}{6pt}\hspace{1.5pt}\rule{1pt}{6pt}} (4);
    \path [-] (3) edge node[bar]{\rule{1pt}{6pt}\hspace{1.5pt}\rule{1pt}{6pt}} (5);
\end{scope}
\begin{scope}[>={Stealth[black]},every edge/.style={draw=black,very thick,dashed}]
    \path [-] (4) edge (5);
\end{scope}
\end{tikzpicture}}
+
\frac1\mu
{\tiny\begin{tikzpicture}[baseline={([yshift=-.8ex]current bounding box.center)},scale=0.7]
\begin{scope}[every node/.style={circle,fill,inner sep=0pt,minimum size=3pt}]
    \node (1) at (0,-0.8) {};
    \node (2) at (0.5,0) {};
    \node (3) at (-0.5,0) {};
    \node (4) at (0.5,1) {};
\end{scope}
\begin{scope}[>={Stealth[black]},every edge/.style={draw=black,very thick}]
    \path [-] (1) edge node[bar]{\rule{1pt}{6pt}\hspace{1.5pt}\rule{1pt}{6pt}}(2);
    \path [-] (1) edge node[bar]{\rule{1pt}{6pt}\hspace{1.5pt}\rule{1pt}{6pt}} (3);
    \path [-] (2) edge node[bar]{\rule{1pt}{6pt}\hspace{1.5pt}\rule{1pt}{6pt}} (4);
\end{scope}
\begin{scope}[>={Stealth[black]},every edge/.style={draw=black,very thick,dashed}]
    \path [-] (4) edge (3);
\end{scope}
\end{tikzpicture}}
+
\frac1{\mu^2}
{\tiny\begin{tikzpicture}[baseline={([yshift=-.8ex]current bounding box.center)},scale=0.7]
\begin{scope}[every node/.style={circle,fill,inner sep=0pt,minimum size=3pt}]
    \node (1) at (0,-0.8) {};
    \node (2) at (0.5,0) {};
    \node (3) at (-0.5,0) {};
\end{scope}
\begin{scope}[>={Stealth[black]},every edge/.style={draw=black,very thick}]
    \path [-] (1) edge node[bar]{\rule{1pt}{6pt}\hspace{1.5pt}\rule{1pt}{6pt}}(2);
    \path [-] (1) edge node[bar]{\rule{1pt}{6pt}\hspace{1.5pt}\rule{1pt}{6pt}} (3);
\end{scope}
\begin{scope}[>={Stealth[black]},every edge/.style={draw=black,very thick,dashed}]
    \path [-] (2) edge (3);
\end{scope}
\end{tikzpicture}}~~\bigg).
\end{multline*}
Now evaluating the remaining integrals, using the diagrammatic calculation rules of Section~\ref{sec:diag}, we find
\begin{eqnarray*}
|I_\mu^{6,6}|
&\lesssim&
[\lambda_8]\Lc^7
{\tiny\begin{tikzpicture}[baseline={([yshift=.8ex]current bounding box.center)},scale=0.7]
\begin{scope}[every node/.style={circle,draw,fill,inner sep=0pt,minimum size=3pt}]
    \node (1) at (0,0) {};
    \node (2) at (2,0) {};
\end{scope}
\begin{scope}[>={Stealth[black]},every edge/.style={draw=black,very thick}]
    \path [-] (1) edge [bend right=30] node[bar]{\rule{1pt}{6pt}\hspace{11pt}\rule{1pt}{6pt}} (2);
\end{scope}
\begin{scope}[>={Stealth[black]},every edge/.style={draw=black,very thick,dashed}]
    \path [-] (1) edge [bend left=30] (2);
\end{scope}
\path [dotted] (1-0.19,0.07-0.3) edge (1+0.19,0.07-0.3);
\path [dotted] (1-0.19,-0.07-0.3) edge (1+0.19,-0.07-0.3);
\node (12) at (1,-0.6) {\text{$8$ times}};
\end{tikzpicture}}
+
\frac1\mu
[\lambda_8]\Lc^4
{\tiny\begin{tikzpicture}[baseline={([yshift=.8ex]current bounding box.center)},scale=0.7]
\begin{scope}[every node/.style={circle,draw,fill,inner sep=0pt,minimum size=3pt}]
    \node (1) at (0,0) {};
    \node (2) at (2,0) {};
\end{scope}
\begin{scope}[>={Stealth[black]},every edge/.style={draw=black,very thick}]
    \path [-] (1) edge [bend right=30] node[bar]{\rule{1pt}{6pt}\hspace{11pt}\rule{1pt}{6pt}} (2);
\end{scope}
\begin{scope}[>={Stealth[black]},every edge/.style={draw=black,very thick,dashed}]
    \path [-] (1) edge [bend left=30] (2);
\end{scope}
\path [dotted] (1-0.19,0.07-0.3) edge (1+0.19,0.07-0.3);
\path [dotted] (1-0.19,-0.07-0.3) edge (1+0.19,-0.07-0.3);
\node (12) at (1,-0.6) {\text{$6$ times}};
\end{tikzpicture}}
+
\frac1{\mu^2}
[\lambda_8]\Lc\,
{\tiny\begin{tikzpicture}[baseline={([yshift=.8ex]current bounding box.center)},scale=0.7]
\begin{scope}[every node/.style={circle,draw,fill,inner sep=0pt,minimum size=3pt}]
    \node (1) at (0,0) {};
    \node (2) at (2,0) {};
\end{scope}
\begin{scope}[>={Stealth[black]},every edge/.style={draw=black,very thick}]
    \path [-] (1) edge [bend right=30] node[bar]{\rule{1pt}{6pt}\hspace{11pt}\rule{1pt}{6pt}} (2);
\end{scope}
\begin{scope}[>={Stealth[black]},every edge/.style={draw=black,very thick,dashed}]
    \path [-] (1) edge [bend left=30] (2);
\end{scope}
\path [dotted] (1-0.19,0.07-0.3) edge (1+0.19,0.07-0.3);
\path [dotted] (1-0.19,-0.07-0.3) edge (1+0.19,-0.07-0.3);
\node (12) at (1,-0.6) {\text{$4$ times}};
\end{tikzpicture}}
\\
&\lesssim&
|\!\log\mu|^8\mu^{-2-(\frac{4\wedge d}2-\frac\beta2-1)\vee0},
\end{eqnarray*}
and thus, as $\beta>0$,
\[\lim_{\mu\downarrow0}\mu^3I_\mu^{6,6}\,=\,0.\]
Arguing similarly for the other terms in~\eqref{eq:decomp-VarGast} and~\eqref{eq:decomp-Imu1-6}, the claim~\eqref{eq:A2conv0} follows.

\medskip
\step5 Proof that
\[\lim_{\mu\downarrow0}D_\mu^3\,=\,0.\]
We proceed to the same decomposition as in~\eqref{eq:decomp-VarGast}, now replacing $G_\mu$ by its gradient, and the claim then follows from the same analysis as in the previous step. We omit the detailss for brevity.

\medskip
\step6 Proof that
\begin{equation}\label{eq:estim-Dmu4}
\limsup_{\mu\downarrow0}D_\mu^4\lesssim\lambda_4+(\lambda_5)^\frac{\beta}{2+\beta}.
\end{equation}
In terms of density functions, expanding the square and distinguishing the possible intersection cases, we can decompose $D^4_\mu$ as follows,
\begin{multline}\label{eq:decomp-D4mu}
D^4_\mu\,=\,
\int_{(\R^d)^3}\Big(\int_{B_+}|\nabla_\mu\delta^{x_1,x_2,x_3}\varphi_\mu^\varnothing|^2\Big)f_4(0,\bar x)\,d\bar x\\
+\int_{(\R^d)^4}\Big(\int_{B_+}\nabla_\mu\delta^{x_1,x_2,x_3}\varphi_\mu^\varnothing\cdot\nabla_\mu\delta^{x_1,x_2,x_3'}\varphi_\mu^\varnothing\Big)f_5(0,\bar x,\bar x')\,d\bar xd\bar x'\\
+\int_{(\R^d)^5}\Big(\int_{B_+}\nabla_\mu\delta^{x_1,x_2,x_3}\varphi_\mu^\varnothing\cdot\nabla_\mu\delta^{x_1,x_2',x_3'}\varphi_\mu^\varnothing\Big)f_6(0,\bar x,\bar x')\,d\bar xd\bar x'\\
+\int_{(\R^d)^6}\Big(\int_{B_+}\nabla_\mu\delta^{x_1,x_2,x_3}\varphi_\mu^\varnothing\cdot\nabla_\mu\delta^{x_1',x_2',x_3'}\varphi_\mu^\varnothing\Big)f_7(0,\bar x,\bar x')\,d\bar xd\bar x'.
\end{multline}
Each term can be estimated using  the general strategy described in Section~\ref{sec:diag}.
We spell out the argument for the last term, since it contains the largest number of variables and gives rise to the least favorable diagrams. Set
\begin{equation*}
J_\mu\,:=\,\int_{(\R^d)^6}\Big(\int_{B_+}\nabla_\mu\delta^{x_1,x_2,x_3}\varphi_\mu^\varnothing\cdot\nabla_\mu\delta^{x_1',x_2',x_3'}\varphi_\mu^\varnothing\Big)f_7(0,\bar x,\bar x')\,d\bar xd\bar x'.
\end{equation*}
By the reflection-block expansion~\eqref{eq:decomp-block-3} of third-order flow differences, we can expand schematically
\begin{multline*}
\nabla_\mu\delta^{x_1,x_2,x_3}\varphi_\mu^\varnothing(0)
\,=\,\sum_{i=1}^6\nabla\Jc_\mu^{x_1}S_i^{x_1,x_2,x_3}\varphi_\mu^\varnothing(0)\\
\,=\,
{\tiny\begin{tikzpicture}[baseline={([yshift=-.8ex]current bounding box.center)},scale=0.7]
\begin{scope}[every node/.style={circle,draw,inner sep=0pt,minimum size=8pt}]
    \node (0) at (0,-1) {0};
    \node (1) at (0,0) {$x_1$};
    \node (2) at (-0.6,1) {$x_2$};
    \node (3) at (0.6,1) {$x_3$};
\end{scope}
\begin{scope}[>={Stealth[black]},every edge/.style={draw=black,very thick}]
    \path [-] (0) edge (1);
    \path [-] (1) edge [bend left=30] (2);
    \path [-] (1) edge [bend right=20] (2);
    \path [-] (1) edge (3);
    \path [-] (2) edge node[bar]{\rule{1pt}{6pt}} (3);
\end{scope}
\end{tikzpicture}}
+
{\tiny\begin{tikzpicture}[baseline={([yshift=-.8ex]current bounding box.center)},scale=0.7]
\begin{scope}[every node/.style={circle,draw,inner sep=0pt,minimum size=8pt}]
    \node (0) at (0,-1) {0};
    \node (1) at (0,0) {$x_1$};
    \node (2) at (-0.6,1) {$x_2$};
    \node (3) at (0.6,1) {$x_3$};
\end{scope}
\begin{scope}[>={Stealth[black]},every edge/.style={draw=black,very thick}]
    \path [-] (0) edge (1);
    \path [-] (1) edge [bend left=30] (2);
    \path [-] (1) edge [bend right=20] (2);
    \path [-] (1) edge [bend left=20] (3);
    \path [-] (1) edge [bend right=30] node[bar]{\rule{1pt}{6pt}} (3);
\end{scope}
\end{tikzpicture}}
+
{\tiny\begin{tikzpicture}[baseline={([yshift=-.8ex]current bounding box.center)},scale=0.7]
\begin{scope}[every node/.style={circle,draw,inner sep=0pt,minimum size=8pt}]
    \node (0) at (0,-1) {0};
    \node (1) at (0,0) {$x_1$};
    \node (2) at (-0.6,1) {$x_2$};
    \node (3) at (0.6,1) {$x_3$};
\end{scope}
\begin{scope}[>={Stealth[black]},every edge/.style={draw=black,very thick}]
    \path [-] (0) edge (1);
    \path [-] (1) edge (2);
    \path [-] (2) edge (3);
    \path [-] (1) edge node[bar]{\rule{1pt}{6pt}} (3);
\end{scope}
\end{tikzpicture}}
+
{\tiny\begin{tikzpicture}[baseline={([yshift=-.8ex]current bounding box.center)},scale=0.7]
\begin{scope}[every node/.style={circle,draw,inner sep=0pt,minimum size=8pt}]
    \node (0) at (0,-1) {0};
    \node (1) at (0,0) {$x_1$};
    \node (2) at (-0.6,1) {$x_2$};
    \node (3) at (0.6,1) {$x_3$};
\end{scope}
\begin{scope}[>={Stealth[black]},every edge/.style={draw=black,very thick}]
    \path [-] (0) edge (1);
    \path [-] (1) edge [bend left=30] (2);
    \path [-] (1) edge [bend right=20] (2);
    \path [-] (1) edge node[bar]{\rule{1pt}{6pt}} (3);
\end{scope}
\end{tikzpicture}}
+~
{\tiny\begin{tikzpicture}[baseline={([yshift=-.8ex]current bounding box.center)},scale=0.7]
\begin{scope}[every node/.style={circle,draw,inner sep=0pt,minimum size=8pt}]
    \node (0) at (0,-1) {0};
    \node (1) at (0,0) {$x_1$};
    \node (2) at (0,1) {$x_2$};
    \node (3) at (0,2) {$x_3$};
\end{scope}
\begin{scope}[>={Stealth[black]},every edge/.style={draw=black,very thick}]
    \path [-] (0) edge (1);
    \path [-] (1) edge (2);
    \path [-] (2) edge [bend left=30] (3);
    \path [-] (2) edge [bend right=30] node[bar]{\rule{1pt}{6pt}} (3);
\end{scope}
\end{tikzpicture}}
~+~
{\tiny\begin{tikzpicture}[baseline={([yshift=-.8ex]current bounding box.center)},scale=0.7]
\begin{scope}[every node/.style={circle,draw,inner sep=0pt,minimum size=8pt}]
    \node (0) at (0,-1) {0};
    \node (1) at (0,0) {$x_1$};
    \node (2) at (0,1) {$x_2$};
    \node (3) at (0,2) {$x_3$};
\end{scope}
\begin{scope}[>={Stealth[black]},every edge/.style={draw=black,very thick}]
    \path [-] (0) edge (1);
    \path [-] (1) edge (2);
    \path [-] (2) edge node[bar]{\rule{1pt}{6pt}} (3);
\end{scope}
\end{tikzpicture}}
~+
\Sym_{x_1,x_2,x_3}.
\end{multline*}
Inserting this expansion into the above, we are led to decompose $J_\mu$ as
\begin{equation}\label{eq:decomp-Jmu-1-6}
J_\mu\,=\,36\sum_{i,j=1}^6J_\mu^{i,j},
\end{equation}
in terms of
\[J_\mu^{i,j}\,:=\,\int_{(\R^d)^6}\Big(\int_{B_+}\nabla_\mu\Jc^{x_1}_\mu S_i^{x_1,x_2,x_3}\cdot\nabla_\mu\Jc^{x_1'}_\mu S_j^{x_1',x_2',x_3'}\Big)f_7(0,\bar x,\bar x')\,d\bar xd\bar x'.\]
Each term $J_\mu^{i,j}$ can be estimated using the strategy described in Section~\ref{sec:diag}, and we focus for brevity on the last term $i=j=6$, which requires the most care. For this purpose, we decompose $f_7$ in terms of correlation functions, cf.~\eqref{eq:correl-expand}, and we apply the cancellation rules of Lemma~\ref{lem:cancel} at each cut edge.
Note in particular that Lemma~\ref{lem:cancel} yields, for any translation-invariant function $h$ on $(\R^d)^3$,
\[\int_{(\R^d)^3}\big(\nabla_\mu\Jc_\mu^{x_1}S_6^{x_1,x_2,x_3}\big)\,h(x_1,x_2,x_3)\,dx_1dx_2dx_3\,=\,0.\]
This cancellation ensures that, in the decomposition of $f_7(0,\bar x,\bar x')$ into correlation functions, only the correlation couplings that mix variables from $\bar x$ and $\bar x'$ give nonzero contributions. This leads us to
\begin{equation*}
|J_\mu^{6,6}|
\,\lesssim\,
[\lambda_7]\bigg(~~
{\tiny\begin{tikzpicture}[baseline={([yshift=-.8ex]current bounding box.center)},scale=0.7]
\begin{scope}[every node/.style={circle,fill,inner sep=0pt,minimum size=3pt}]
    \node (1) at (0,-0.8) {};
    \node (2) at (0.5,0) {};
    \node (3) at (-0.5,0) {};
    \node (4) at (0.5,1) {};
    \node (5) at (-0.5,1) {};
    \node (6) at (0.5,2) {};
    \node (7) at (-0.5,2) {};
\end{scope}
\begin{scope}[>={Stealth[black]},every edge/.style={draw=black,very thick}]
    \path [-] (1) edge (2);
    \path [-] (1) edge (3);
    \path [-] (2) edge (4);
    \path [-] (3) edge (5);
    \path [-] (4) edge node[bar]{\rule{1pt}{6pt}} (6);
    \path [-] (5) edge node[bar]{\rule{1pt}{6pt}} (7);
\end{scope}
\begin{scope}[>={Stealth[black]},every edge/.style={draw=black,very thick,dashed}]
    \path [-] (6) edge (7);
\end{scope}
\end{tikzpicture}}
+
{\tiny\begin{tikzpicture}[baseline={([yshift=-.8ex]current bounding box.center)},scale=0.7]
\begin{scope}[every node/.style={circle,fill,inner sep=0pt,minimum size=3pt}]
    \node (1) at (0,-0.8) {};
    \node (2) at (0.5,0) {};
    \node (3) at (-0.5,0) {};
    \node (4) at (0.5,1) {};
    \node (5) at (-0.5,1) {};
    \node (6) at (0.5,2) {};
\end{scope}
\begin{scope}[>={Stealth[black]},every edge/.style={draw=black,very thick}]
    \path [-] (1) edge (2);
    \path [-] (1) edge (3);
    \path [-] (2) edge (4);
    \path [-] (3) edge node[bar]{\rule{1pt}{6pt}} (5);
    \path [-] (4) edge node[bar]{\rule{1pt}{6pt}} (6);
\end{scope}
\begin{scope}[>={Stealth[black]},every edge/.style={draw=black,very thick,dashed}]
    \path [-] (6) edge (5);
\end{scope}
\end{tikzpicture}}
+
{\tiny\begin{tikzpicture}[baseline={([yshift=-.8ex]current bounding box.center)},scale=0.7]
\begin{scope}[every node/.style={circle,fill,inner sep=0pt,minimum size=3pt}]
    \node (1) at (0,-0.8) {};
    \node (2) at (0.5,0) {};
    \node (3) at (-0.5,0) {};
    \node (4) at (0.5,1) {};
    \node (6) at (0.5,2) {};
\end{scope}
\begin{scope}[>={Stealth[black]},every edge/.style={draw=black,very thick}]
    \path [-] (1) edge (2);
    \path [-] (1) edge node[bar]{\rule{1pt}{6pt}} (3);
    \path [-] (2) edge (4);
    \path [-] (4) edge node[bar]{\rule{1pt}{6pt}} (6);
\end{scope}
\begin{scope}[>={Stealth[black]},every edge/.style={draw=black,very thick,dashed}]
    \path [-] (6) edge (3);
\end{scope}
\end{tikzpicture}}
+
{\tiny\begin{tikzpicture}[baseline={([yshift=-.8ex]current bounding box.center)},scale=0.7]
\begin{scope}[every node/.style={circle,fill,inner sep=0pt,minimum size=3pt}]
    \node (1) at (0,-0.8) {};
    \node (2) at (0.5,0) {};
    \node (3) at (-0.5,0) {};
    \node (4) at (0.5,1) {};
    \node (5) at (-0.5,1) {};
\end{scope}
\begin{scope}[>={Stealth[black]},every edge/.style={draw=black,very thick}]
    \path [-] (1) edge (2);
    \path [-] (1) edge (3);
    \path [-] (2) edge node[bar]{\rule{1pt}{6pt}} (4);
    \path [-] (3) edge node[bar]{\rule{1pt}{6pt}} (5);
\end{scope}
\begin{scope}[>={Stealth[black]},every edge/.style={draw=black,very thick,dashed}]
    \path [-] (4) edge (5);
\end{scope}
\end{tikzpicture}}
+
{\tiny\begin{tikzpicture}[baseline={([yshift=-.8ex]current bounding box.center)},scale=0.7]
\begin{scope}[every node/.style={circle,fill,inner sep=0pt,minimum size=3pt}]
    \node (1) at (0,-0.8) {};
    \node (2) at (0.5,0) {};
    \node (3) at (-0.5,0) {};
    \node (4) at (0.5,1) {};
\end{scope}
\begin{scope}[>={Stealth[black]},every edge/.style={draw=black,very thick}]
    \path [-] (1) edge (2);
    \path [-] (1) edge node[bar]{\rule{1pt}{6pt}} (3);
    \path [-] (2) edge node[bar]{\rule{1pt}{6pt}} (4);
\end{scope}
\begin{scope}[>={Stealth[black]},every edge/.style={draw=black,very thick,dashed}]
    \path [-] (4) edge (3);
\end{scope}
\end{tikzpicture}}
+
{\tiny\begin{tikzpicture}[baseline={([yshift=-.8ex]current bounding box.center)},scale=0.7]
\begin{scope}[every node/.style={circle,fill,inner sep=0pt,minimum size=3pt}]
    \node (1) at (0,-0.8) {};
    \node (2) at (0.5,0) {};
    \node (3) at (-0.5,0) {};
\end{scope}
\begin{scope}[>={Stealth[black]},every edge/.style={draw=black,very thick}]
    \path [-] (1) edge node[bar]{\rule{1pt}{6pt}} (2);
    \path [-] (1) edge node[bar]{\rule{1pt}{6pt}} (3);
\end{scope}
\begin{scope}[>={Stealth[black]},every edge/.style={draw=black,very thick,dashed}]
    \path [-] (2) edge (3);
\end{scope}
\end{tikzpicture}}
~~\bigg).
\end{equation*}
Evaluating the remaining integrals, using diagrammatic calculation rules of Section~\ref{sec:diag}, we find
\[|J_\mu^{6,6}|\,\lesssim\,
[\lambda_7]\Lc^5\,
{\tiny\begin{tikzpicture}[baseline={([yshift=-.8ex]current bounding box.center)},scale=0.7]
\begin{scope}[every node/.style={circle,draw,fill,inner sep=0pt,minimum size=3pt}]
    \node (1) at (0,0) {};
    \node (2) at (1,0) {};
\end{scope}
\begin{scope}[>={Stealth[black]},every edge/.style={draw=black,very thick}]
    \path [-] (1) edge [bend right=40] node[bar]{\rule{1pt}{6pt}\hspace{1.5pt}\rule{1pt}{6pt}} (2);
\end{scope}
\begin{scope}[>={Stealth[black]},every edge/.style={draw=black,very thick,dashed}]
    \path [-] (1) edge [bend left=40] (2);
\end{scope}
\end{tikzpicture}}
\,\lesssim\,(\lambda_7)^{\frac{\beta}{2+\beta}}+o(1).\]
Arguing similarly for the other terms in~\eqref{eq:decomp-D4mu} and~\eqref{eq:decomp-Jmu-1-6}, the claim~\eqref{eq:estim-Dmu4} follows. Combining this with the estimates of Steps~1--5 concludes the proof of~\eqref{eq:estim-RL-red}.
\end{proof}


\section{The physics of sedimentation and interpretation of the results}
\label{sec:phys-intro}

To conclude this work, let us discuss the physics of sedimentation and explain how the above results fit into the classical theory of suspensions. This section aims to clarify the relevance of the relative settling speed, the role of the compensating backflow, and the relation between our renormalized cluster formula and Batchelor's original expression.

\subsection{Sedimentation, backflow, and finite containers}
Suspensions of particles in a viscous fluid are a classical object of study in fluid mechanics, with applications ranging from engineering to geophysics. Among the many phenomena arising in such systems, sedimentation --- the collective settling of heavy particles under gravity --- has played a central role throughout the twentieth century; see, for instance, the classical account~\cite{GM-11}.

In the regime where the inertia of both the fluid and the particles is neglected, the evolution is quasistatic. At each time, the fluid velocity solves the steady Stokes equations in the exterior of the particles, with no-slip condition at the particle boundaries. Each particle is subject to a uniform body force, for instance gravity, which is balanced by the hydrodynamic stress exerted by the surrounding fluid; this is precisely the system~\eqref{eq:Dir}. The resulting velocity field then determines the instantaneous particle velocities.
The statistical theory of sedimentation takes the particle configuration at a fixed time as given, and studies the induced statistics of these velocities.

A first basic question is whether a \emph{mean settling speed} can be defined for a statistically homogeneous suspension. Physically, this corresponds to the average velocity of suspended particles in the limit of a large container. This question remained open for quite some time, even at the formal level, despite early attempts by Smoluchowski~\cite{Smoluchowski-11,Smoluchowski-06} and Burgers~\cite{Burgers-41,Burgers-42}.

In the 1970s, Batchelor~\cite{Batchelor-72,Batchelor_1976} revisited the problem and gave the first argument allowing one to define the mean settling speed of a homogeneous suspension in infinite volume, in the case of a hardcore Poisson process --- a standard model for well-stirred suspensions.
He further identified the first dilute correction to the settling speed of an isolated particle, now known as Batchelor's correction. His analysis relies on a renormalization of divergent integrals, justified heuristically by large-scale cancellations of Debye-screening type.

These cancellations are, however, delicate. Batchelor's argument is intrinsically formulated in infinite volume and does not by itself show whether the same velocity is recovered as the limit of large finite containers. In particular, it leaves open the possibility that the limiting velocity may depend on the shape of the container or on the boundary conditions.

This issue was already raised in the 1940s by Burgers, who noted that a dependence of the settling speed on the shape of the vessel would be physically difficult to accept. Batchelor expressed a similar viewpoint: the velocity computed in the whole-space model should represent the average particle velocity relative to the walls of a vessel containing a statistically homogeneous suspension. In other words, both Burgers and Batchelor expected the large-volume mean settling speed to be independent of the shape of the container.

Subsequent works challenged this expectation. Beenakker and Mazur~\cite{BM-85} argued that container walls may affect the average sedimentation speed. Geigenm\"uller and Mazur~\cite{MR974953} asked whether the conditional convergence in Batchelor's renormalization is merely an artefact of the `unphysical' infinite-volume model, or instead reflects a genuine dependence on the container shape that persists in the large-volume limit. Indeed, when Batchelor's renormalization is implemented directly in finite containers, formal calculations become inconclusive: the cancellations needed to remove infrared divergences are so strong that boundary-layer effects can contribute at the same order as the bulk correction. In particular, modifying the distribution of particles close to the walls may influence the thermodynamic limit.

More precisely, these works suggest that walls induce a large-scale recirculation of the fluid, an effect called \emph{intrinsic convection}, which is indeed of the same order as the Batchelor correction.
Nozi\`eres~\cite{NOZIERES1987219} introduced a simplified model for this phenomenon, later used by Geigenm\"uller and Mazur~\cite{MR974953}; see also~\cite{BFAH}. Experimental support was later provided by Peysson and Guazzelli~\cite{PeyssonGuazzelli}, although the measured effect appears substantially smaller than predicted by dilute theories.

This discussion suggests that the local mean particle velocity should be decomposed into three distinct contributions. Ignoring for the moment the smoothing used in the rigorous statements, let
\[\bar V_R(B_r)\,:=\,\frac{\sum_{x\in\Pc(U_R)}\mathds1_{B_r}(x)V_R(x)}{\sum_{x\in \Pc(U_R)}\mathds1_{B_r}(x)}\]
be the average particle velocity in a mesoscopic region~$B_r$. Then one may decompose
\begin{eqnarray}
\bar V_R(B_r)
&=&
\underbrace{\textstyle
\bar V_R(B_r)-\fint_{B_r}u_R}_{=:\bar V_R^{\operatorname{rel}}(B_r)}
\,+\,
\underbrace{\textstyle
\fint_{B_r}\E[u_R]}_{=:\bar V_R^{\operatorname{conv}}(B_r)}
\,+\,
\underbrace{\textstyle
\fint_{B_r}(u_R-\E[u_R])}_{=:\bar V_R^{\operatorname{fluct}}(B_r)}.
\label{eq:dec-barVR}
\end{eqnarray}
The first term is the \emph{relative mean settling speed}: it measures the particle velocity in the moving frame defined by the locally averaged surrounding fluid. This is the contribution expected to converge, as $R\uparrow\infty$, to an intrinsic infinite-volume settling speed independent of the container shape. The second term is the \emph{ensemble-averaged fluid velocity}; it is the contribution associated with intrinsic convection, namely the large-scale recirculation induced by the particle-depleted boundary layer near the walls. The third term is the \emph{random fluctuation} of the averaged fluid velocity around its ensemble mean. Its size is closely related to the Caflisch--Luke paradox: in the Stokes regime, such fluctuations are predicted to be small in space dimensions $d>4$, whereas in dimension $d=3$ their smallness is expected only under additional self-organization of the particle configuration, for instance in the form of hyperuniformity.

Despite the extensive physics literature, a rigorous mathematical theory of sedimentation remains limited. The difficulty is twofold. First, although the Stokes equations are linear, the particle velocities depend on the full particle configuration in a nonlinear and nonlocal way through a boundary-value problem.
Second, the large-volume limit involves stochastic cancellations at infrared scales, so deterministic estimates alone are insufficient to identify the effective velocity.

\subsection{Relative velocity, intrinsic convection, and fluctuations}\label{sec:mathphys}
For a rigorous analysis of~\eqref{eq:dec-barVR}, we replace sharp averages by smooth averages. Letting
\[\bar V_R(\chi_r)
:=
\frac{\sum_{x\in\Pc(U_R)}\chi_r(x)V_R(x)}{\sum_{x\in\Pc(U_R)}\chi_r(x)},\]
the analogue of~\eqref{eq:dec-barVR} is then
\begin{equation}\label{eq:dec-barVR-re}
\bar V_R(\chi_r)\,=\,\underbrace{\textstyle\bar V_R(\chi_r)-\fint_{\chi_r}u_R}_{=:\bar V_R^{\text{rel}}(\chi_r)}\,+\,\underbrace{\textstyle\fint_{\chi_r}\E[u_R]}_{=:\bar V_R^{\text{conv}}(\chi_r)}\,+\,\underbrace{\textstyle\fint_{\chi_r}(u_R-\E[u_R])}_{=:\bar V_R^{\text{fluct}}(\chi_r)},
\end{equation}
where we recall the notation~\eqref{eq:average-chir}.
This identity separates the intrinsic settling speed relative to the local fluid frame, the deterministic convection induced by the walls, and the residual random fluctuations.

\medskip\noindent{\it Relative velocity.}
The present work focuses on the first contribution
$\bar V_R^{\operatorname{rel}}(\chi_r)$.
In contrast with much of the physics literature, we work directly with the full Stokes suspension problem rather than with a linearized model.
Theorem~\ref{th:mean-velocity} constructs the infinite-volume relative settling speed~$\bar V_\infty$ by means of an infrared-renormalized formulation, under minimal assumptions on the inclusion process. Theorem~\ref{th:mean-velocityR} then shows that this intrinsic quantity is indeed obtained from finite containers in the large-volume limit:
\[\bar V_R^{\operatorname{rel}}(\chi_r)\to\bar V_\infty e,\]
in the regime $1\ll r\ll R$. In particular, the limiting relative settling speed is independent of the shape of the container and of the smooth averaging function, thus confirming the expectation of Burgers and Batchelor at the level of the relative velocity.

\medskip\noindent{\it Intrinsic convection.}
The second contribution $\bar V_R^{\operatorname{conv}}(\chi_r)$ is the ensemble-averaged fluid velocity and is expected to account for the intrinsic convection predicted in the physics literature. In the very dilute regime, where particles are uniformly separated on a scale much larger than their diameter, the predictions of~\cite{MR974953,BFAH} are accessible by means of the method of reflections, in the spirit of~\cite{Hofer-19}. In the non-dilute regime considered here, intrinsic convection is substantially subtler and
we leave this question for future work.

We also mention the recent work of G\'erard-Varet and Mecherbet~\cite{MR4904868}, which rigorously analyzes a related wall-induced mechanism in a different asymptotic regime. More precisely, they show that the effect of a macroscopic particle-depleted layer near a wall can be encoded in an effective Navier-type wall law, thereby providing a rigorous explanation of apparent slip phenomena. While closely related in spirit, this result does not yield a derivation of the microscopic intrinsic convection considered here.

\medskip\noindent{\it Random fluctuations.}
The third contribution $\bar V_R^{\operatorname{fluct}}(\chi_r)$ concerns random fluctuations of the fluid velocity. Estimating fluctuations is a central issue in stochastic homogenization, and is by now well understood in the model setting of divergence-form elliptic equations with random coefficients; see~\cite{DGO}. In sedimentation, however, the fluctuation question is more delicate and is related to the Caflisch-Luke paradox~\cite{Caflisch-Luke}. For suspensions modeled by a hardcore Poisson process, the infinite-volume mean settling speed is predicted to be well defined, whereas the variance of individual particle velocities diverges linearly with the size of the container in 3D. More precisely,
\begin{equation}\label{e.CafLuk}
\fint_{\chi_R}\var{u_R}\simeq \left\{\begin{array}{lll}
1&:&\text{if $d>4$},\\
(\log R)^\frac12&:&\text{if $d=4$},\\
R&:&\text{if $d=3$}.
\end{array}\right.
\end{equation}
This divergence in 3D is at odds with the physical expectation that fluctuations should remain of order one in a statistically stabilized suspension. One possible resolution is that the particle configuration does not remain Poissonian: if the particle process becomes hyperuniform, in the sense that large-scale density fluctuations are suppressed, Koch and Shaqfeh~\cite{KS-91} predicted the improved scaling
\begin{equation}\label{e.KoSha}
\fint_{\chi_R}\var{u_R}\simeq 1,\qquad d>2.
\end{equation}
In~\cite{DG-22}, we proved~\eqref{e.CafLuk} as an upper bound in the Poisson setting, and~\eqref{e.KoSha} under the additional assumptions that the point process satisfies a hyperuniform variance estimate and can be periodized in law. These results indicate that, within the Stokes framework, order-one velocity fluctuations in 3D require some form of large-scale self-organization of the particle ensemble. Whether this is the mechanism actually responsible for the experimentally-observed damping of velocity fluctuations remains unclear~\cite{GH-rev-10}.

The fluctuation term in~\eqref{eq:dec-barVR-re} is not the pointwise variance $\fint_{\chi_R}\var{u_R}$, but rather the variance of the locally averaged fluid velocity:
\[\bar V_R^{\operatorname{fluct}}(\chi_r)=\fint_{\chi_r}(u_R-\E[u_R]).\]
For mesoscopic averaging scales $1\ll r\ll R$, the corresponding prediction is
\begin{equation*}
\limsup_{R\uparrow\infty}
\E\Big[\big|\bar V_R^{\operatorname{fluct}}(\chi_r)\big|^2\Big]
\simeq\left\{\begin{array}{lll}
r^{4-d}&:&\text{in Poisson setting, for $d>4$},\\[1mm]
r^{2-d}&:&\text{in hyperuniform setting, for $d>2$.}
\end{array}\right.
\end{equation*}
This quantifies the smallness of the averaged fluctuation contribution as $r\uparrow\infty$. These scalings can be justified as upper bounds by adapting our analysis in~\cite{DG-22}.

\subsection{Rigorous version of Batchelor's formula}
Batchelor's formula~\cite{Batchelor-72} predicts the first terms in the dilute expansion of the infinite-volume settling speed,
\[\bar V_\infty e= \bar V_\infty^{(1)}e+\bar V_\infty^{(2)}e+\ldots.\]
Here $\bar V_\infty^{(1)}e=V_Se$ is the settling speed
of a single isolated particle under gravity $e$, explicitly given by~\eqref{eq:stokes-vel} for spherical particles, while $\bar V_\infty^{(2)}e$ is the first collective correction, expressed in terms of two-particle hydrodynamic interactions and pair statistics. A rigorous derivation of this correction is the main achievement of Theorem~\ref{th:Batchelor}.

Our formula~\eqref{eq:Batch-form-pr} is written in a form that differs from Batchelor's original expression~\cite{Batchelor-72}. For spherical particles, the equivalence with Batchelor's formula is explained in Appendix~\ref{app:batch-origin}. The advantage of~\eqref{eq:Batch-form-pr} is that it does not rely on the explicit spherical Fax\'en law used in Batchelor's argument, and therefore adapts more naturally to non-spherical particles.

\medskip\noindent{\it Heuristic derivation and infrared renormalization.}
At a purely formal level, the dilute expansion suggests that the mean settling speed should be approximated, to second order, by the one-particle and two-particle clusters only. This leads to the naive approximation
\begin{eqnarray}
\bar V_\infty e&\sim&
\frac1{\lambda}\,\E\bigg[\sum_{x\in\Pc}\bigg(\fint_{B(x)}\varphi_e^x+\sum_{y\in\Pc\setminus\{x\}} \fint_{B(x)}(\varphi_e^{x,y}-\varphi_e^x)\bigg)\bigg]\nonumber\\
&=&
\fint_B\varphi_e^0+\frac1{\lambda}\int_{\R^d}\Big(\fint_B(\varphi_e^{0,y}-\varphi_e^0)\Big)f_2(0,y)\,dy\nonumber\\
&=&
V_Se+\frac1{\lambda}\int_{\R^d}\big(U_e(y)-V_Se\big)f_2(0,y)\,dy,\label{eq:Vinfty-heur}
\end{eqnarray}
with the notation of Theorem~\ref{th:Batchelor}. However, this formula is divergent: the perturbation $U_e(y)-V_Se$ of the velocity of the particle at the origin due to a second particle at~$y$ only decays like $|y|^{2-d}$, while for a mixing point process $f_2(0,y)\to\lambda^2$ does not decay as $|y|\uparrow\infty$.

The key point in Batchelor's argument~\cite{Batchelor-72} is that this divergence is artificial. It comes from the non-integrable far field generated by a settling particle, which induces a large-scale motion of the surrounding fluid. Since the relevant quantity is not an absolute settling speed but a velocity relative to this ambient motion, this far-field contribution has to be subtracted. At the formal level, this is encoded by the prescription
\begin{equation}\label{eq:cancellation-batch}
\int_{\R^d}\varphi_e^0=0.
\end{equation}
This integral is not defined, since $\varphi_e^0$ is not integrable, but it becomes meaningful after introducing an infrared cut-off.

To implement this cancellation, one rewrites the two-particle increment in~\eqref{eq:Vinfty-heur} by subtracting its far field. More precisely, the far-field part of $e\cdot(U_e(y)-V_Se)$ is exactly~$-\frac1{|B|}\int_{\partial B}\varphi_e^y\cdot\sigma(\varphi_e^0,\pi_e^0)\nu$.
Batchelor identified this term using Fax\'en's law and the method of reflections. In our proof, it follows from exact identities for the Stokes problem; more precisely, we show that
\begin{equation}\label{eq:decomp-Batchform-0}
e\cdot(U_e(y)-V_Se)+\frac1{|B|}\int_{\partial B}\varphi_e^y\cdot\sigma(\varphi_e^0,\pi_e^0)\nu\\
\,=\,O(|y|^{2(1-d)})
\end{equation}
which is integrable for $d>2$; see~\eqref{eq:decomp-Batchform}.
Adding and subtracting the far-field term in~\eqref{eq:Vinfty-heur}, we obtain formally
\begin{multline*}
e\cdot \bar V_\infty e\,\sim\, e\cdot V_Se+\frac1{\lambda}\int_{\R^d}\bigg(e\cdot U_e(y)-e\cdot V_Se+\frac1{|B|}\int_{\partial B}\varphi_e^y\cdot\sigma(\varphi_e^0,\pi_e^0)\nu\bigg)f_2(0,y)\,dy\\
-\frac1{\lambda}\int_{\R^d}\bigg(\frac1{|B|}\int_{\partial B}\varphi_e^y\cdot\sigma(\varphi_e^0,\pi_e^0)\nu\bigg)f_2(0,y)\,dy.
\end{multline*}
The first integral is absolutely convergent by~\eqref{eq:decomp-Batchform-0}. The second one still contains a non-integrable $|y|^{2-d}$ tail, but the part associated with the limiting value $\lambda^2$ of~$f_2(0,y)$ is formally proportional to~\eqref{eq:cancellation-batch}.
After this cancellation, only the two-point correlation
\[h_2(0,y)=f_2(0,y)-\lambda^2,\]
remains. This yields the renormalized formula~\eqref{eq:Batch-form-pr}, where all integrals are absolutely convergent under the assumption $|h_2(0,y)|\lesssim|y|^{-2-\beta}$ for some $\beta>0$.

\medskip\noindent{\it Higher-order corrections.}
The proof of Theorem~\ref{th:Batchelor} is based on a cluster expansion, a nonlinear estimate on the truncation error, and a systematic infrared renormalization procedure for hydrodynamic interactions. Although we only state the expansion up to the two-particle term needed for Batchelor's formula, the method can be pursued to arbitrary order. In fact, estimating the error after the second-order correction already requires the analysis of higher-order cluster terms and cancellations beyond those visible in Batchelor's original argument.

At higher orders, the naive cluster integrals contain increasingly complicated infrared-divergent contributions. Their cancellation is no longer captured by the simple far-field subtraction~\eqref{eq:cancellation-batch}; it requires a finer decomposition of hydrodynamic interactions. This is achieved in the proof by a finitary diagrammatic expansion in so-called reflection blocks. We do not write explicit higher-order formulas, since they rapidly become cumbersome, but we emphasize that the expansion is constructive: the $m$th correction can be expressed in terms of $m$-body Stokes problems and correlation functions up to order $m$. This is analogous in spirit to the higher-order expansion obtained in~\cite{DG-22} for Einstein's effective viscosity formula, although the present sedimentation problem is substantially more delicate because of the longer-range divergences and the need for additional cancellations.

\medskip\noindent{\it Numerical value of Batchelor's correction.}
Formula~\eqref{eq:Batch-form-pr} is not explicit, since it involves the two-body Stokes problem, but it is amenable to numerical evaluation. 
In the classical 3D setting of a well-stirred suspension modeled by a hardcore Poisson process, one has, in the unit-particle-radius convention,
\[f_2(0,y)=\lambda^2\mathds1_{|y|>2},\qquad h_2(0,y)=-\lambda^2\mathds1_{|y|\le2}.\]
Writing $\phi=\lambda|B|$ for the volume fraction of the suspension, Batchelor's numerical evaluation of the corresponding two-body problem~\cite{Batchelor-72} gives
\begin{equation}\label{eq:Batch-approx}
\bar V_\infty=\frac29(1-\alpha\phi)\Id+o(\phi),\qquad \alpha\approx 6.55.
\end{equation}
A more precise numerical evaluation gives $\alpha\approx6.5462489$.
This is close to the experimental value $\alpha=5.3\pm1.1$ reported in~\cite{Bruneau-90}.
The sign $\alpha>0$ means that the suspension settles more slowly than an isolated particle, in agreement with the classical phenomenology of hindered settling, for instance as described by the Richardson-Zaki law.

\medskip\noindent{\it Very dilute regime and multipole expansion.}
Formula~\eqref{eq:Batch-form-pr} can also be combined with analytical far-field approximations. If particles are uniformly far apart,
\[\inf_{\substack{x,y\in\Pc\\x\ne y}}|x-y|\gg1,\]
then the one- and two-particle flows can be expanded in multipoles. For~$|y|\gg1$, the two terms entering Batchelor's formula satisfy
\begin{eqnarray*}
e\cdot U_e(y)-e\cdot V_Se+\frac1{|B|}\int_{\partial B}\varphi_e^y\cdot\sigma(\varphi_e^0,\pi_e^0)\nu
&=&-d|B|^2|\!\D(\Gc e)(y)|^2+O(|y|^{1-2d}),\\
\frac1{|B|}\int_{\partial B}\varphi_e^y\cdot\sigma(\varphi_e^0,\pi_e^0)\nu
&=&-|B|(e\cdot\Gc e)(y)+O(|y|^{1-d}),
\end{eqnarray*}
where $\Gc$ denotes the Stokeslet. The second term is the leading far-field interaction of two sedimenting particles, while the first one is of higher order in the large-separation regime.

To make this explicit, suppose that the point process is obtained by rescaling a reference process,
\[\Pc\,=\,\{Lx:x\in\Pc^0\},\]
where $\Pc^0$ has intensity~$\lambda^0$ and two-point correlation~$h_2^0$. Then
\[\lambda\,=\,L^{-d}\lambda^0,\qquad
h_2(0,y)\,=\,L^{-2d}h_2^0(0,y/L).\]
Using the homogeneity of the Stokeslet, the leading far-field contribution to Batchelor's correction becomes
\begin{equation}\label{eq:farfield-Batch}
\bar V_\infty^{(2)}\,=\,L^{2-d}\frac{|B|}{\lambda^0}\int_{\R^d}\mathcal G(y)h_2^0(0,y)\,dy+O(L^{1-d}).
\end{equation}
For a hardcore Poisson reference process with $\lambda^0=1$ and $f_2^0(0,y)=\mathds1_{|y|>1}$, this gives
\begin{equation*}
\bar V_\infty^{(2)}\,=\,-L^{2-d}\frac{(d-1)|B|}{2d(d-2)}\Id + O(L^{1-d}).
\end{equation*}
In 3D, since $\phi=\lambda|B|=|B|L^{-3}$, this reads
\begin{equation}\label{eq:Batch-approx-far}
 \bar V_\infty=\frac29(1-\alpha_{\text{far}}\phi^\frac13)\Id+O(\phi^\frac23),\qquad \alpha_{\text{far}}\,=\,\frac32|B|^{\frac23}\,\approx\,3.89.
\end{equation}
This regime should not be confused with Batchelor's usual dilute regime. Here the rescaling of the point process also makes the correlation length of order~$L$. Since the Stokeslet far field is non-integrable, this produces a contribution of order $\lambda L^2\simeq L^{2-d}$, rather than the usual order $\lambda\simeq \phi$ obtained in~\eqref{eq:Batch-approx} when correlations are localized at the particle scale.

This specific regime is closer to Hasimoto's classical expansion for periodic arrays of spheres~\cite{Hasimoto59}. Formally, if~$\Pc$ is a randomly shifted copy of the dilute lattice~$L\Lambda$, where $\Lambda$ has unit covolume, then
\[\lambda=L^{-d},
\qquad
h_2(0,y)=L^{-d}\sum_{p\in\Lambda\setminus\{0\}}\delta_{Lp}(y)-L^{-2d}.\]
Although this periodic setting falls outside the mixing assumptions used in this work, inserting this expression in~\eqref{eq:farfield-Batch} yields the renormalized lattice sum
\[L^{2-d}|B|\int_{\R^d}\Gc(z)\Big(\sum_{p\in\Lambda\setminus\{0\}}\delta_p(dz)-dz\Big),\]
understood in the periodic finite-part sense. Equivalently, this is
\[L^{2-d}|B|\lim_{x\to0}\big(\Gc_\per^\Lambda(x)-\Gc(x)\big),\]
where $\Gc_\per^\Lambda$ is the mean-zero periodic Stokeslet on $\R^d/\Lambda$. In 3D, for the cubic lattice, this recovers Hasimoto's formula~\cite{Hasimoto59}.

\medskip\noindent{\it Hindered versus cooperative settling.}
Formula~\eqref{eq:Batch-form-pr} also clarifies how the sign of the first correction depends on the pair statistics of the suspension. The naive divergent expression~\eqref{eq:Vinfty-heur} is misleading in this respect. Indeed, two isolated particles typically settle faster than a single particle, so the bare two-particle increment
\[e\cdot\fint_B(\varphi_e^{0,y}-\varphi_e^0)\]
is non-negative; see, for instance,~\cite[Section~6.1]{GM-11}. Taken at face value, the unrenormalized cluster formula~\eqref{eq:Vinfty-heur} would therefore suggest that pair interactions always increase the settling speed. The renormalized formula~\eqref{eq:Batch-form-pr} shows that this conclusion is wrong: after subtracting the large-scale backflow, the sign of the correction can be either positive or negative, depending on pair statistics.

For well-stirred suspensions modeled by a hardcore Poisson process, the correction is negative, in agreement with the Richardson-Zaki law; see~\eqref{eq:Batch-approx}. By contrast, pair statistics favoring particles at preferred separations may yield a positive correction, corresponding to cooperative settling at the level of the two-particle expansion. This is already visible from the far-field formula~\eqref{eq:farfield-Batch}.

As a simple illustration, consider a dimerized suspension. Start from a stationary process of cluster centers with intensity~$L^{-d}$ and mutual separations of order~$L$, and replace each center by a pair of particles separated by a distance~$aL$, for some fixed $a>0$, with uniformly random orientation. Then
\[\lambda=2L^{-d},\qquad h_2(0,y)=2L^{-d}\frac{\delta_{\partial B_{aL}}(y)}{|\partial B_{aL}|}+O\big(L^{-2d}\mathds1_{|y|\lesssim L}\big).\]
Inserting this in~\eqref{eq:farfield-Batch} gives
\[\bar V_\infty^{(2)}
\,=\,|B|\fint_{\partial B_{aL}}\Gc(y)\,dy+O(L^{1-d})
\,=\,\frac{d-1}{d^2(d-2)}(aL)^{2-d}+O(L^{1-d}),\]
hence in 3D,
\[\bar V_\infty=\frac29\big(1+a^{-1}L^{-1}\big)+O(L^{-2}).\]
Thus, for such dimerized statistics, the leading far-field correction is cooperative. This illustrates that Batchelor's correction is not merely a hydrodynamic coefficient: it also probes the large-scale pair structure of the suspension.


\appendix

\section{Explicit calculations for spherical particles}

We collect here explicit Stokes computations for a single spherical inclusion, and we use them to relate formula~\eqref{eq:Batch-form-pr} to Batchelor's original expression for spheres.

\subsection{The single-particle problem}\label{app:explicit}
For $V\in\R^d$, consider the single-particle Stokes problem
\[\left\{\begin{array}{ll}
-\triangle\psi_V+\nabla\pi_V=0,
&\text{in $\R^d\setminus B$},\\
\Div(\psi_V)=0,
&\text{in $\R^d\setminus B$},\\
\psi_V=V,
&\text{on $\partial B$},\\
\psi_V(x)\to0,
&\text{as $|x|\uparrow\infty$}.
\end{array}\right.\]
This can be solved explicitly: in terms of the Stokeslet
\[\Gc(x)
=\frac1{2d(d-2)|B|}
\left(
\frac{\Id}{|x|^{d-2}}
+(d-2)\frac{x\otimes x}{|x|^d}
\right),
\qquad
\Pi(x)=\frac1{d|B|}\frac{x}{|x|^d},\]
the solution is
\begin{eqnarray*}
\psi_V(x)
&=&\frac{d^2(d-2)|B|}{d-1}\Big(\Gc(x)+\frac1{2d}\triangle\Gc(x)\Big)V,\\
\pi_V(x)
&=&\frac{d^2(d-2)|B|}{d-1}\,\Pi(x)\cdot V.
\end{eqnarray*}
More explicitly,
\begin{eqnarray*}
\psi_V(x)
&=&\frac{d}{2(d-1)}
\left(\frac{V}{|x|^{d-2}}+(d-2)\frac{(V\cdot x)x}{|x|^d}\right)
+\frac{d-2}{2(d-1)}
\left(\frac{V}{|x|^d}-d\frac{(V\cdot x)x}{|x|^{d+2}}\right),\\
\pi_V(x)&=&\frac{d(d-2)}{d-1}\frac{V\cdot x}{|x|^d}.
\end{eqnarray*}
A direct computation of the traction gives
\[\int_{\partial B}\sigma(\psi_V,\pi_V)\nu
=-\frac{d^2(d-2)}{d-1}|B|V.\]
In fact, for a sphere, the traction is pointwise constant on the boundary:
\[\sigma(\psi_V,\pi_V)\nu=-\frac{d(d-2)}{d-1}V
\qquad\text{on } \partial B.\]
Consequently, for $e\in\R^d$, the single-particle sedimentation problem
\[\left\{\begin{array}{ll}
-\triangle\varphi^0_e+\nabla\pi^0_e=0,&\text{in $\R^d\setminus B$},\\
\Div(\varphi^0_e)=0,&\text{in $\R^d\setminus B$},\\
\D(\varphi^0_e)=0,&\text{in $B$},\\
e|B|+\int_{\partial B}\sigma(\varphi^0_e,\pi^0_e)=0,&\\
\int_{\partial B}x\times\sigma(\varphi^0_e,\pi^0_e)=0,&
\end{array}\right.\]
is solved by $(\varphi_e^0,\pi_e^0)=(\psi_V,\pi_V)$ with
\begin{equation}\label{eq:determine-VS}
V=V_Se,\qquad
V_S:=\frac{d-1}{d^2(d-2)}.
\end{equation}
In 3D this gives the usual Stokes velocity $V_S=\frac29$.
In particular,
\begin{equation}\label{eq:prop-1-part-trac}
\fint_B\varphi_e^0=V_Se=\frac{d-1}{d^2(d-2)}e,
\qquad
\sigma(\varphi_e^0,\pi_e^0)\nu=-\frac1d e\quad\text{on $\partial B$}.
\end{equation}
We shall also use the following spherical averages, which follow directly from the explicit formula above:
\begin{equation}\label{eq:spherical-average-single}
\fint_{\partial B_r}\varphi_e^0
=\left\{\begin{array}{lll}
V_Se&:&0<r<1,\\[1mm]
V_S r^{2-d}e&:&r>1.
\end{array}\right.
\end{equation}

\subsection{Batchelor's original formula}\label{app:batch-origin}

We now explain how formula~\eqref{eq:Batch-form-pr} relates to Batchelor's original expression for the first correction to the settling speed in the case of spherical particles.

\begin{lem}\label{lem:batch-origin}
For spherical particles, formula~\eqref{eq:Batch-form-pr} is equivalent to
\begin{multline}\label{eq:batch-origin}
\bar V_\infty^{(2)}e
=
\frac1{\lambda}
\int_{\R^d}\Big(U_e(y)-V_Se-\varphi_e^0(y)-\frac1{2d}\triangle\varphi_e^0(y)\Big)f_2(0,y)\,dy\\
+\frac1{\lambda}\int_{\R^d}\varphi_e^0(y)\,h_2(0,y)\,dy
+\frac1{2d\lambda}\int_{|y|>1}\triangle\varphi_e^0(y)\,h_2(0,y)\,dy
+\lambda\Big(\frac d2-1\Big)|B|V_Se,
\end{multline}
which recovers Batchelor's original 3D renormalized formula~\cite{Batchelor-72}.
\end{lem}

\begin{proof}
Since the one-particle traction is constant on the boundary, cf.~\eqref{eq:prop-1-part-trac}, we have
\begin{equation}\label{eq:cst-tract-re}
\frac1{|B|}\int_{\partial B}\varphi_e^y\cdot\sigma(\varphi_e^0,\pi_e^0)\nu=-e\cdot\fint_{\partial B}\varphi_e^y.
\end{equation}
For $|y|>2$, since the field $\varphi_e^y$ solves the homogeneous Stokes equations in~$B$, each component is biharmonic, and the biharmonic mean-value formula then gives
\[\fint_{\partial B}\varphi_e^y
=\varphi_e^y(0)+\frac1{2d}\triangle\varphi_e^y(0).\]
By translation symmetry, we have $\varphi_e^y(0)=\varphi_e^0(y)$, and thus for $|y|>2$,
\begin{equation}\label{eq:faxen}
\frac1{|B|}\int_{\partial B}\varphi_e^y\cdot\sigma(\varphi_e^0,\pi_e^0)\nu
=-e\cdot\Big(\varphi_e^0(y)+\frac1{2d}\triangle\varphi_e^0(y)\Big).
\end{equation}
This is precisely Fax\'en's law in fluid mechanics.
The latter identity can however not be used in the excluded region $|y|\le2$ where the two unit balls overlap. We therefore keep that part in its boundary-average form~\eqref{eq:cst-tract-re}. Since the hardcore condition gives $f_2(0,y)=0$ and thus $h_2(0,y)=-\lambda^2$ for $|y|\le2$, the formula~\eqref{eq:Batch-form-pr} becomes
\begin{multline}\label{eq:Batch-form-pr-re}
\bar V_\infty^{(2)}e\,:=\,
\frac1{\lambda}\int_{\R^d}\Big(U_{e}(y)-V_Se-\varphi_e^0(y)-\frac1{2d}\triangle\varphi_e^0(y)\Big)\,f_{2}(0,y)\,dy\\
+\frac1{\lambda}\int_{|y|>2}\Big(\varphi_e^0(y)+\frac1{2d}\triangle\varphi_e^0(y)\Big)\,h_{2}(0,y)\,dy
-\lambda\int_{|y|\le2}\Big(\fint_{\partial B}\varphi_e^y\Big)\,dy.
\end{multline}
It remains to reorganize the terms in the last line, and we start with the last one.
For this purpose, we first note that it can be written as radial averages of $\varphi_e^0$,
\[\int_{|y|\le2}\Big(\fint_{\partial B}\varphi_e^y\Big)dy
=\int_{|z|\le3}\varphi_e^0(z)\,\omega(|z|)dz
=|\partial B|\int_0^3\Big(\fint_{\partial B_r}\varphi_e^0\Big)\,r^{d-1}\omega(r)\,dr,\]
in terms of
\[\omega(|z|):=\fint_{\partial B}\mathds1_{|x-z|\le2}\,dx.\]
This weight $\omega$ is the surface measure of a spherical cap. It is equal to~$1$ for $0\le r\le1$, to~$0$ for $r\ge3$, and for $1<r<3$ it is given by
\[\omega(r)
=\frac{\int_{a(r)}^1(1-t^2)^{\frac{d-3}{2}}\,dt}
{\int_{-1}^1(1-t^2)^{\frac{d-3}{2}}\,dt},\qquad a(r):=\frac{r^2-3}{2r}.\]
Using~\eqref{eq:spherical-average-single}, the above becomes
\[\int_{|y|\le2}\Big(\fint_{\partial B}\varphi_e^y\Big)dy=
|\partial B|V_Se\bigg(\frac1d+\int_1^3 r\omega(r)\,dr\bigg).\]
A direct integration of the cap formula yields $\int_1^3 r\omega(r)\,dr=1+\frac1d$, and thus
\begin{equation*}
\int_{|y|\le2}\Big(\fint_{\partial B}\varphi_e^y\Big)dy=(d+2)|B|V_Se.
\end{equation*}
Moreover, a direct application of~\eqref{eq:spherical-average-single} yields
\begin{equation*}
\int_{|y|\le2}\varphi_e^0(y)\,dy=|\partial B|\int_0^2\Big(\fint_{B_r}\varphi_e^0\Big)r^{d-1}dr
=\Big(1+\frac{3d}{2}\Big)|B|V_Se.
\end{equation*}
By the divergence theorem, using~\eqref{eq:spherical-average-single} again, we also find
\begin{multline*}
\int_{1<|y|\le2}\triangle\varphi_e^0(y)\,dy
=\int_{\partial B_2}\nu\cdot\nabla\varphi_e^0-\int_{\partial B_1}\nu\cdot\nabla\varphi_e^0\\
=|\partial B_2|\,\partial_r\Big(\fint_{\partial B_r}\varphi_e^0\Big)\Big|_{r=2}-|\partial B|\,\partial_r\Big(\fint_{\partial B_r}\varphi_e^0\Big)\Big|_{r=1}=0.
\end{multline*}
Combined with~\eqref{eq:Batch-form-pr-re}, this proves the claimed formula.
\end{proof}

\section{Decay of elementary operators}\label{app:elementary-pieces}
This appendix is devoted to the proof of Lemma~\ref{lem:decay-JG} on the decay properties of elementary operators.
We start with the analysis of the infrared-regularized Stokeslet. Due to the specific choice~\eqref{eq:cut-sigma-mu} of the penalization of the divergence constraint, the scalings are not standard and require some care.

\subsection{Decay of the infrared-regularized Stokeslet}
We record the pointwise estimates for the Green function associated with the infrared-regularized Stokes operator. The main point is that the artificial bulk-viscosity penalization term in~\eqref{eq:cut-sigma-mu} does not screen all modes at the same length scale. The massive term screens the transverse modes at scale $\mu^{-1/2}$, whereas the longitudinal modes are only screened at the larger scale $\mu^{-1}$. This slower longitudinal screening is however accompanied by a compensating crossover factor~$\langle \sqrt\mu x\rangle^{-2}$, which reflects the cancellation of the singular longitudinal projection at low frequency.

\begin{lem}[Decay of the infrared-regularized Stokeslet]\label{lem:massGreen}
Let $d\ge3$ and $0<\mu\le1$. Let $\Gamma_\mu=(\Gamma_{\mu,ij})_{1\le i,j\le d}$ be the fundamental solution of the infrared-regularized Stokes operator, that is, for $1\le i\le d$,
\begin{equation}\label{eq:def-Gammamu}
\mu \Gamma_{\mu,i}-\Div(\sigma_\mu(\Gamma_{\mu,i}))=\delta_0e_i\qquad\text{in $\R^d$},
\end{equation}
where we recall the modified stress tensor $\sigma_\mu$ defined in~\eqref{eq:cut-sigma-mu}.
There is a constant~$c>0$ (depending only on $d$) such that for all $x\ne0$ and $k\ge0$,
\begin{eqnarray*}
|\nabla^k\Gamma_\mu(x)|
&\lesssim_k& |x|^{2-d-k} \langle\sqrt\mu x\rangle^{-2}e^{-c\mu|x|},\\
\Big|\nabla^k\frac1{\sqrt\mu}\Div(\Gamma_\mu)(x)\Big|
&\lesssim_k& \sqrt\mu|x|^{1-d-k}e^{-c\mu|x|},\\
\Big|\nabla^k\frac1\mu\Div(\Div(\Gamma_\mu))(x)\Big|
&\lesssim_k& |x|^{-d-k}e^{-c\mu|x|}.
\end{eqnarray*}
In particular, recalling the notation~\eqref{eq:nabmu},
\begin{eqnarray*}
|\nabla^k\nabla_\mu\Gamma_\mu(x)|
&\lesssim_k&|x|^{1-d-k}\langle\sqrt\mu x\rangle^{-1}e^{-c\mu|x|},\\
|\nabla^k\nabla_\mu^2\Gamma_\mu(x)|
&\lesssim_k&|x|^{-d-k}e^{-c\mu|x|}.
\end{eqnarray*}
\end{lem}

\begin{proof}
By definition of $\sigma_\mu$, we have
\[-\Div(\sigma_\mu(\varphi))=-\triangle\varphi-\Big(1+\frac1\mu\Big)\nabla\Div(\varphi).\]
To establish the claimed estimates, we write the defining equation for $\Gamma_\mu$ in Fourier space:
\[(\mu+|\xi|^2)\widehat\Gamma_{\mu}+\Big(1+\frac1\mu\Big)(\xi\otimes\xi) \widehat\Gamma_{\mu}=\Id.\]
Separating the transverse and longitudinal modes, this equation can be inverted,
\begin{equation*}
\widehat\Gamma_{\mu}
=\frac1{\mu+|\xi|^2}\Big(\Id-\frac{\xi\otimes\xi}{|\xi|^2}\Big)+\frac1{\mu+(2+\frac1\mu)|\xi|^2}\frac{\xi\otimes\xi}{|\xi|^2},
\end{equation*}
or equivalently,
\begin{equation*}
\widehat\Gamma_{\mu}
=\frac1{\mu+|\xi|^2}\Id-\frac{1}{\mu}\Big(\frac{1}{\nu+|\xi|^2}-\frac{1}{\mu+|\xi|^2}\Big)\xi\otimes\xi,\qquad \nu:=\frac{\mu^2}{1+2\mu}.
\end{equation*}
Denoting by $G_\mu$ the Yukawa kernel~\eqref{eq:Greenfct-mass} with screening length $\mu^{-1/2}$, this means
\[\Gamma_\mu=G_\mu\Id-\frac1\mu\nabla^2(G_\nu-G_\mu).\]
Standard properties of the Yukawa potential then yield the desired decay estimates on gradients of $\Gamma_\mu$. More precisely, we use that
\[|\nabla^kG_\mu(x)|\lesssim|x|^{2-d-k}e^{-c\sqrt\mu|x|},\]
and for $0<\nu\le\mu\le1$ and $k\ge 2$,
\begin{equation}\label{e.diff-green}
|\nabla^k(G_\nu-G_\mu)(x)|\lesssim|x|^{2-d-k}e^{-c\sqrt\nu|x|}\big(1\wedge( (\mu-\nu)|x|^2)\big),
\end{equation}
which follows by time integration of the suitably weighted heat kernel.
Finally, taking the divergence removes the transverse part: we find
\[\frac1{\sqrt\mu}\widehat \Gamma_\mu(\xi)\xi
=\frac{\sqrt\mu}{1+2\mu}\frac{ \xi}{\nu+|\xi|^2},
\qquad
\frac1{\mu}\xi\cdot \widehat \Gamma_\mu(\xi)\xi
=\frac{1}{1+2\mu}\frac{|\xi|^2}{\nu+|\xi|^2},\]
hence
\[\frac1{\sqrt\mu}\Div(\Gamma_\mu)=\frac{\sqrt\mu}{1+2\mu}\nabla G_\nu,
\qquad
\frac1{\mu}\Div(\Div(\Gamma_\mu))=\frac{1}{1+2\mu}\triangle G_\nu\]
and the corresponding estimates follow.
\end{proof}

\subsection{Proof of Lemma~\ref{lem:decay-JG}}

Let $\zeta$ satisfy the admissibility condition~\eqref{eq:admissibility} for some $z,A\in\R^d$, and let~$Y\subset \R^d$ be a finite hardcore configuration. 
Next to the notation~\eqref{eq:not-Ki} for the kernels $K_i$, we use the short-hand notation
\[M_0(v;x):=\bigg(\int_{B_+(x)}\mu|v|^2+|\nabla_\mu v|^2\bigg)^{\frac12},
\qquad
M_1(v;x):=\bigg(\int_{B_+(x)}|v|^2+|\nabla_\mu v|^2\bigg)^{\frac12}.\]
In these terms, we shall prove for all $x\in\R^d$,
\begin{eqnarray}
M_0(\Jc^z_{\mu;Y}\zeta;x)&\lesssim_{\sharp Y}&K_0(x-z)M_0(\zeta;z),\label{eq:new-J-grad-main}\\
M_1(\Jc^z_{\mu;Y}\zeta;x)&\lesssim_{\sharp Y}&K_1(x-z)M_0(\zeta;z),\label{eq:new-J-L2-main}\\
M_0(\Gc^z_{\mu;Y};x)&\lesssim_{\sharp Y}&K_1(x-z),\label{eq:new-G-grad-main}\\
M_1(\Gc^z_{\mu;Y};x)&\lesssim_{\sharp Y}&K_2(x-z).\label{eq:new-G-L2-main}
\end{eqnarray}
We split the proof into five steps. We start with the estimates on the stresslet-type operator~$\Jc^z_{\mu;Y}$ in the first four steps, where we argue by induction on the cardinality of the background configuration~$Y$. Step~5 is dedicated to the Stokeslet-type operator~$\Gc^z_{\mu;Y}$.

\medskip
\step1 Free case $Y=\varnothing$.\\
For $u_0:=\Jc_{\mu;\varnothing}^z\zeta$, we shall prove~\eqref{eq:new-J-grad-main} and~\eqref{eq:new-J-L2-main}, that is, for all $x\in\R^d$,
\begin{equation}\label{eq:est-u0M01}
M_i(u_0;x)\lesssim K_i(x-z)M_0(\zeta;z),\qquad i=0,1.
\end{equation}
By definition~\eqref{eq:def-JLYz-0} with $Y=\varnothing$, the Green representation formula yields
\begin{equation}\label{eq:form-Jnothing}
u_0(x)=\ell_z^\zeta(\Gamma_\mu(x-\cdot)),
\end{equation}
in terms of the infrared-regularized Stokeslet~$\Gamma_\mu$, cf.~\eqref{eq:def-Gammamu}, where we have set for abbreviation
\[\ell_z^\varphi(\psi):=-\int_{\partial B(z)}\Big(\psi-\fint_{B(z)}\psi\Big)\cdot\sigma_\mu(\varphi)\nu+\mu\int_{B(z)}\psi\cdot\varphi.\]
By Lemma~\ref{lem:trace-est}, the latter is bounded by
\begin{equation}\label{eq:new-local-trace}
|\ell_z^\varphi(\psi)|
\lesssim M_0(\varphi;z)\Big(\int_{B(z)}|\nabla_\mu\psi|^2+\mu|\psi|^2\Big)^\frac12.
\end{equation}
Appealing to the decay of the infrared-regularized Stokeslet, cf.~Lemma~\ref{lem:massGreen}, we deduce for all $k\ge0$ and $|x-z|>2$,
\begin{eqnarray*}
|\nabla^ku_0(x)|&\lesssim& M_0(\zeta;z)|x-z|^{1-d-k}\langle\sqrt\mu (x-z)\rangle^{-1}e^{-c\mu|x-z|},\\
|\nabla^k\nabla_\mu u_0(x)|&\lesssim& M_0(\zeta;z)|x-z|^{-d-k}e^{-c\mu|x-z|}.
\end{eqnarray*}
To prove~\eqref{eq:est-u0M01}, it remains to consider the regime $|x-z|\le2$ and show that, for all $x\in\R^d$,
\begin{equation}\label{eq:todo-bound-M01}
M_0(u_0;x)\le M_1(u_0;x)\lesssim M_0(\zeta;z).
\end{equation}
The weak formulation~\eqref{eq:def-JLYz} for $u_0=\Jc_{\mu;\varnothing}^z\zeta$ reads as follows, for all $\psi\in H^1(\R^d)$,
\[\int_{\R^d}\mu\psi\cdot u_0+2\D(\psi):\D(u_0)+\frac1\mu\Div(\psi)\Div(u_0)=\ell_z^\zeta(\psi).\]
Testing with $\psi=u_0$ and using~\eqref{eq:new-local-trace}, we deduce
\begin{equation}\label{eq:energy-u0}
\int_{\R^d}\mu|u_0|^2+|\nabla_\mu u_0|^2\lesssim M_0(\zeta;z)^2.
\end{equation}
Arguing as in~\eqref{eq:estimate-mean-finite-Y-global}, we may deduce
\[\sup_{x\in\R^d}\Big|\int_{B(x)}u_0\Big|\lesssim M_0(\zeta;z),\]
and thus, combined with the above and Poincar\'e's inequality,
\[\sup_{x\in\R^d}\int_{B(x)}|u_0|^2+|\nabla_\mu u_0|^2\lesssim M_0(\zeta;z)^2,\]
that is, \eqref{eq:todo-bound-M01}.

\medskip
\step2 Case of a local cluster.\\
Consider the background configuration close to $z$,
\[Y_z:=\{y\in Y\,:\,B(y)\cap B(z)\ne\varnothing\},\]
and set for abbreviation
\[u_z:=\Jc^z_{\mu;Y_z}\zeta.\]
Note that the hardcore condition ensures $\sharp Y_z\lesssim1$ (only depending on $d$ and on the separation  parameter $\delta$) and
\[B_+(z)\cup\bigcup_{y\in Y_z}B_+(y)\subset B_5(z).\]
In this step, we control $u_z$ by comparison with the free field $u_0=\Jc^z_{\mu;\varnothing}\zeta$ considered in Step~1.  The comparison is made in three parts: first we rigidify $u_0$ inside the finitely many local inclusions, then we estimate the difference with $u_z$, and finally we estimate the rigidification error.  We claim that, for all
$x\in\R^d$,
\begin{equation}\label{eq:local-cluster-new}
M_i(u_z;x)\lesssim K_i(x-z)M_0(\zeta;z),\qquad i=0,1.
\end{equation}

\medskip
\substep{2.1} Rigidification of $u_0$.\\
For each $y\in Y_z$, we use a local rigidification operator $R_y$ on the enlarged ball $B_+(y)$.
More precisely, by subtracting a suitable rigid motion in $B(y)$, appealing to Korn's inequality, and correcting the divergence in the annulus $B_+(y)\setminus B(y)$ by means of Bogovskii's operator similarly as in the proof of Lemma~\ref{lem:trace-est}, we can construct a
linear operator
\[R_y:H^1(B_+(y))^d\to H^1(B_+(y))^d,\]
such that, for all $\psi\in H^1(B_+(y))^d$,
\begin{equation}\label{eq:rigidify-properties-new}
R_y\psi=\psi\quad\text{on $\partial B_+(y)$},
\qquad
\D(R_y\psi)=0\quad\text{in $B(y)$,}
\end{equation}
and
\begin{eqnarray}
\int_{B_+(y)}|\nabla_\mu R_y\psi|^2
&\lesssim&
\int_{B_+(y)}|\nabla_\mu\psi|^2,\label{eq:rigidify-grad-new}\\
\int_{B_+(y)}|R_y\psi|^2
&\lesssim&
\int_{B_+(y)}|\psi|^2+|\nabla_\mu\psi|^2.\label{eq:rigidify-L2-new}
\end{eqnarray}
Next, since the balls $B_+(y)$, $y\in Y_z$, are pairwise disjoint, we may define
\begin{equation}\label{eq:def-Tz}
T_z\psi:=\psi+\sum_{y\in Y_z}\mathds1_{B_+(y)}(R_y\psi-\psi).
\end{equation}
By definition, $T_z\psi$ coincides with $R_y\psi$ in $B_+(y)$ for each $y\in Y_z$, and in particular is rigid in each inclusion. Summing~\eqref{eq:rigidify-grad-new} and~\eqref{eq:rigidify-L2-new} over $y\in Y_z$ gives
\begin{align}
\|\nabla_\mu T_z\psi\|_{L^2(B_5(z))}
&\lesssim
\|\nabla_\mu\psi\|_{L^2(B_5(z))},
\label{eq:Tz-full-grad-new}\\
\|T_z\psi\|_{L^2(B_5(z))}
&\lesssim
\|\psi\|_{L^2(B_5(z))}+
\|\nabla_\mu\psi\|_{L^2(B_5(z))}.
\label{eq:Tz-full-L2-new}
\end{align}

\medskip
\substep{2.2} Proof that the difference
\[w_z:=u_z-T_z u_0\]
satisfies for all $x\in\R^d$,
\begin{equation}\label{eq:wz-local-new}
M_i(w_z;x)\lesssim K_i(x-z)M_0(\zeta;z),
\qquad i=0,1.
\end{equation}
The weak formulation~\eqref{eq:def-JLYz} of the equation for $u_z=\Jc_{\mu;Y_z}^z\zeta$ yields, in the weak sense in~$\R^d$,
\begin{multline*}
\mu u_z-\Div(\sigma_\mu(u_z))=-\delta_{\partial B(z)}\sigma_\mu(\zeta)\nu-A\mathds1_{B(z)}+\mu\mathds1_{B(z)}\zeta\\
+\sum_{y\in Y_z}\Big(-\delta_{\partial B(y)}\sigma_\mu(u_z)\nu+\mu u_z\mathds1_{B(y)}\Big),
\end{multline*}
so that, after subtracting the equation for $u_0=\Jc_{\mu;\varnothing}^z\zeta$, we obtain
\[\mu (u_z-u_0)-\Div(\sigma_\mu(u_z-u_0))=\sum_{y\in Y_z}\Big(-\delta_{\partial B(y)}\sigma_\mu(u_z)\nu+\mu u_z\mathds1_{B(y)}\Big).\]
By definition~\eqref{eq:def-Tz} of $T_z$, we have
\[T_zu_0-u_0=\sum_{y\in Y_z}\mathds1_{B_+(y)}(R_yu_0-u_0),\]
where each summand vanishes on $\partial B_+(y)$. Adding this into the above, we obtain for the difference $w_z=(u_z-u_0)-(T_zu_0-u_0)$, in the weak sense on $\R^d$,
\begin{multline*}
\mu w_z-\Div(\sigma_\mu(w_z))
=\sum_{y\in Y_z}\Big(-\delta_{\partial B(y)}\sigma_\mu(u_z)\nu
+\mu u_z\mathds1_{B(y)}\\[-4mm]
-\mu (R_yu_0-u_0)\mathds1_{B_+(y)}+\Div(\mathds1_{B_+(y)}\sigma_\mu(R_yu_0-u_0))\Big).
\end{multline*}
In terms of the infrared-regularized Stokeslet $\Gamma_\mu$, the Green representation formula yields
\begin{multline*}
w_z(x)=\sum_{y\in Y_z}\bigg(-\int_{\partial B(y)}\Big(\Gamma_\mu(x-\cdot)-\fint_{B(y)}\Gamma_\mu(x-\cdot)\Big)\sigma_\mu(u_z)\nu\\
+\int_{B(y)}\mu\Gamma_\mu(x-\cdot)u_z-\int_{B_+(y)}\mu\Gamma_\mu(x-\cdot)(R_yu_0-u_0)\\
-\int_{B_+(y)}2\D(R_yu_0-u_0)\D(\Gamma_\mu)(x-\cdot)+\frac1\mu\Div(R_yu_0-u_0)\Div(\Gamma_\mu)(x-\cdot) \bigg).
\end{multline*}
By the trace estimate of Lemma~\ref{lem:trace-est}, we thus have that, for all $k\ge0$,
\begin{multline*}
|\nabla^kw_z(x)|\lesssim\sum_{y\in Y_z}\bigg(\int_{B_+(y)}\mu|\nabla^k\Gamma_\mu(x-\cdot)|^2+|\nabla^k\nabla_\mu\Gamma_\mu(x-\cdot)|^2\bigg)^\frac12\\
\times\bigg(\int_{B_+(y)}\mu|R_yu_0-u_0|^2+|\nabla_\mu(R_yu_0-u_0)|^2
+\int_{B_+(y)}\mu|u_z|^2+|\nabla_\mu u_z|^2\bigg).
\end{multline*}
By the decay of the infrared-regularized Stokeslet, cf.~Lemma~\ref{lem:massGreen}, and by the rigidification estimates~\eqref{eq:rigidify-grad-new}--\eqref{eq:rigidify-L2-new}, this entails, for all $|x-z|\ge5$ (ensuring $|x-y|\ge2$ for $y\in Y_z$),
\begin{multline*}
|\nabla^kw_z(x)|\lesssim |x-z|^{1-d-k}\langle\sqrt\mu (x-z)\rangle^{-1}e^{-c\mu|x-z|}\\
\times\Big(\sqrt\mu\|(u_0,u_z)\|_{L^2(B_5(z))}+\|(\nabla_\mu u_0,\nabla_\mu u_z)\|_{L^2(B_5(z))}\Big).
\end{multline*}
By the energy estimate~\eqref{eq:energy-u0} on $u_0$, and by the corresponding energy estimate on $u_z=\Jc^z_{\mu;Y_z}\zeta$, this yields for all $k\ge0$ and $|x-z|\ge5$,
\begin{equation*}
|\nabla^kw_z(x)|\lesssim |x-z|^{1-d-k}\langle\sqrt\mu (x-z)\rangle^{-1}e^{-c\mu|x-z|}M_0(\zeta;z).
\end{equation*}
Similarly,
\begin{equation*}
|\nabla^k\nabla_\mu w_z(x)|\lesssim |x-z|^{-d-k}e^{-c\mu|x-z|}M_0(\zeta;z).
\end{equation*}
Combining these bounds with the energy estimates for $u_0$ and $u_z$ in the range $|x-z|\le 5$, the claim~\eqref{eq:wz-local-new} follows.

\medskip
\substep{2.3} Proof that for all $x\in\R^d$,
\begin{equation}\label{eq:Tzu0-new}
M_i(T_zu_0;x)\lesssim K_i(x-z)M_0(\zeta;z),
\qquad i=0,1.
\end{equation}
If $|x-z|\ge7$, then $B_+(x)\cap B_+(y)=\varnothing$ for all $y\in Y_z$, hence $T_zu_0=u_0$ on $B(x)$, so that~\eqref{eq:Tzu0-new} follows directly from the free estimate~\eqref{eq:est-u0M01} of Step~1.  If $|x-z|<20$, it suffices to prove
\[M_i(T_zu_0;x)\lesssim M_0(\zeta;z),\qquad i=0,1,\]
which simply follows from~\eqref{eq:Tz-full-grad-new}--\eqref{eq:Tz-full-L2-new}, and~\eqref{eq:todo-bound-M01}.

\medskip
\step3 Suboptimal estimate for arbitrary backgrounds: for all $x\in\R^d$,\\
\begin{equation}\label{eq:new-borderline-main}
M_0(\Jc^z_{\mu;Y}\zeta;x)
\lesssim_{\sharp Y}
\sum_{y\in \{x\}\cup(Y\setminus Y_z)}K_0(y-z)M_0(\zeta;z).
\end{equation}
Set for abbreviation
\[u_Y:=\Jc^z_{\mu;Y}\zeta.\]
For each $y\in Y\setminus Y_z$, we use the rigidification operator $R_y$ of Step~2.1
and define the $Y$-rigidification of the local-cluster field $u_z=\Jc^z_{\mu;Y_z}\zeta$ by
\[T_Yu_z:=u_z+\sum_{y\in Y\setminus Y_z}\mathds1_{B_+(y)}(R_yu_z-u_z),\]
which is rigid in each inclusion $B(y)$, $y\in Y$, and coincides with $u_z$ outside $\bigcup_{y\in Y\setminus Y_z}B_+(y)$.
Comparing the equations for $u_Y$ and $u_z$, similarly as in Substep~2.2, we find in the weak sense on $\R^d$,
\begin{multline*}
\mu (u_Y-T_Yu_z)\mathds1_{\R^d\setminus\Ic_Y}-\Div(\sigma_\mu(u_Y-T_Yu_z))
=-\sum_{y\in Y_z}\delta_{\partial B(y)}\sigma_\mu(u_Y-u_z)\nu\\
+\sum_{y\in Y\setminus Y_z}\Big(-\delta_{\partial B(y)}\sigma_\mu(u_Y)\nu+\mu T_Yu_z\mathds1_{B(y)}
-\mu(R_yu_z-u_z)\mathds1_{B_+(y)}+\Div\big(\mathds1_{B_+(y)}\sigma_\mu(R_yu_z-u_z)\big)\Big).
\end{multline*}
testing it with $u_Y-T_Yu_z$, using boundary conditions for $u_Y,u_z$, and using~\eqref{e.1.varphi-trou} to fill the holes, we obtain
\begin{multline*}
\int_{\R^d}\mu |u_Y-T_Yu_z|^2+|\nabla_\mu(u_Y-T_Yu_z)|^2\\
\lesssim\sum_{y\in Y\setminus Y_z}\Big(\int_{B(y)}\mu |T_Yu_z|^2+\int_{B_+(y)}\mu|R_yu_z-u_z|^2+|\nabla_\mu(R_yu_z-u_z)|^2\Big).
\end{multline*}
Thus, recalling the definition of $T_Y$ and the rigidification estimates~\eqref{eq:rigidify-grad-new}--\eqref{eq:rigidify-L2-new},
\begin{equation*}
\int_{\R^d}\mu |u_Y-T_Yu_z|^2+|\nabla_\mu(u_Y-T_Yu_z)|^2
\lesssim\sum_{y\in Y\setminus Y_z}\int_{B_+(y)}\mu|u_z|^2+|\nabla_\mu u_z|^2.
\end{equation*}
By the result~\eqref{eq:local-cluster-new} of Step~2, this entails
\begin{equation*}
\Big(\int_{\R^d}\mu |u_Y-T_Yu_z|^2+|\nabla_\mu(u_Y-T_Yu_z)|^2\Big)^\frac12
\lesssim\sum_{y\in Y\setminus Y_z}K_0(y-z)M_0(\zeta;z).
\end{equation*}
Hence, for all $x\in\R^d$,
\[M_0(u_Y;x)\lesssim M_0(T_Yu_z)+\sum_{y\in Y\setminus Y_z}K_0(y-z)M_0(\zeta;z).\]
Recalling the result~\eqref{eq:local-cluster-new} of Step~2 and the rigidification estimates~\eqref{eq:rigidify-grad-new}--\eqref{eq:rigidify-L2-new}, the first right-hand side term can be bounded by $K_0(x-z)M_0(\zeta;z)$, and the claim~\eqref{eq:new-borderline-main} follows.

\medskip
\step4 Induction over remote inclusions.\\
We prove~\eqref{eq:new-J-grad-main} and~\eqref{eq:new-J-L2-main} by induction on $\sharp (Y\setminus Y_z)$. Step~2 gives the result whenever $Y=Y_z$. Assume $\sharp (Y\setminus Y_z)=n>0$ and that the result is known to hold for any smaller cardinality of background configurations.
For $S\subset Y\setminus Y_z$, we set for abbreviation
\[u_S:=\Jc^z_{\mu;Y_z\cup S}\zeta.\]
Comparing weak formulations on the common background $Y_z\cup S$ gives
\begin{equation}\label{eq:resolvent-first-new}
u_{Y}-u_S=\sum_{y\in (Y\setminus Y_z)\setminus S}\Jc^y_{\mu;Y_z\cup S}(u_{Y}).
\end{equation}
Decomposing $u_{Y}=u_{S\cup\{y\}}+(u_{Y}-u_{S\cup\{y\}})$ in the right-hand side, and iterating the expansion, starting from $u_\varnothing=\Jc_{\mu;Y_z}^z=u_z$, we are led to
\begin{equation}\label{eq:chain-new}
u_{Y}-u_z
=\sum_{\ell=1}^{n}\sum_{y_1,\ldots,y_\ell\in Y\setminus Y_z}^{\ne}
\Jc^{y_1}_{\mu;Y_z}
\Jc^{y_2}_{\mu;Y_z\cup\{y_1\}}
\cdots
\Jc^{y_\ell}_{\mu;Y_z\cup\{y_1,\ldots,y_{\ell-1}\}}
(u_{\{y_1,\ldots,y_\ell\}}).
\end{equation}
We can apply the induction hypothesis successively to estimate each operator $\Jc^{p_\ell}_{\mu;Y_z\cup\{p_1,\ldots,p_{n}\}}$, with $0\le n<\ell$, to the effect of
\begin{multline*}
M_i(u_{Y}-u_z;x)
\lesssim_{\sharp Y}\sum_{\ell=1}^{n}\sum_{y_1,\ldots,y_\ell\in Y\setminus Y_z}^{\ne}
\!\!K_i(x-y_1)K_0(y_1-y_2)\ldots K_0(y_{\ell-1}-y_\ell)\\[-4mm]
\times M_0(u_{\{y_1,\ldots,y_\ell\}};y_\ell).
\end{multline*}
for $i=0,1$. Now applying the result~\eqref{eq:new-borderline-main} of Step~3,
\begin{multline*}
M_i(u_{Y}-u_z;x)
\lesssim_{\sharp Y}\sum_{\ell=1}^{n}\sum_{y_1,\ldots,y_\ell\in Y\setminus Y_z}^{\ne}
K_i(x-y_1)K_0(y_1-y_2)\ldots K_0(y_{\ell-1}-y_\ell)\\[-5mm]
\times\Big(\sum_{j=1}^\ell K_0(y_j-z)\Big)M_0(\zeta;z)
\end{multline*}
Using the trivial bound $K_0\lesssim1$, this can be reorganized as
\begin{multline*}
M_i(u_{Y}-u_z;x)
\lesssim_{\sharp Y}\sum_{\ell=1}^{n}\sum_{y_1,\ldots,y_\ell\in Y\setminus Y_z}^{\ne}
K_i(x-y_1)K_0(y_1-y_2)\ldots K_0(y_{\ell-1}-y_\ell)\\[-4mm]
\times K_0(y_\ell-z)M_0(\zeta;z).
\end{multline*}
Further noting that $K_i(a)K_i(b)\lesssim K_i(a+b)$ and $K_0\lesssim K_1$, we conclude
\[M_i(u_Y-u_z;x)\lesssim_{\sharp Y}K_i(x-z)M_0(\zeta;z),\qquad i=0,1.\]
Combined with the result~\eqref{eq:local-cluster-new} of Step~2, this proves~\eqref{eq:new-J-grad-main} and~\eqref{eq:new-J-L2-main}.

\medskip
\step5 Stokeslet-type operator.\\
We start with the free case $Y=\varnothing$. In terms of the infrared-regularized Stokeslet~$\Gamma_\mu$, cf.~\eqref{eq:def-Gammamu}, we have
\[\Gc^z_{\mu;\varnothing}(x)=\int_{B(z)}\Gamma_\mu(x-y)e\,dy,\]
and Lemma~\ref{lem:massGreen} then yields for~$|x-z|\ge2$,
\begin{eqnarray*}
|\Gc^z_{\mu;\varnothing}(x)|&\lesssim& |x-z|^{2-d}\langle\sqrt\mu (x-z)\rangle^{-2}e^{-c\mu|x-z|},\\
|\nabla_\mu\Gc^z_{\mu;\varnothing}(x)|&\lesssim&|x-z|^{1-d}\langle\sqrt\mu (x-z)\rangle^{-1}e^{-c\mu|x-z|}.
\end{eqnarray*}
In order to control the near field, we simply appeal to an energy estimate for $\Gc^z_{\mu;\varnothing}$, arguing e.g. as in the proof of Lemma~\ref{lem:conv-cor-finite}, which leads us to
\[\sup_{x\in\R^d}\Big(\|\Gc^z_{\mu;\varnothing}\|_{L^2(B(x))}+\|\nabla_\mu\Gc^z_{\mu;\varnothing}\|_{L^2(B(x))}\Big)\lesssim1.\]
This proves for all $x\in\R^d$,
\[M_i(\Gc^z_{\mu;\varnothing};x)\lesssim K_{i+1}(x-z),\qquad i=0,1.\]
Finally, the passage from the free case to an arbitrary finite hardcore background configuration $Y$ is obtained by the same argument as in Steps~2--4 above; we omit the details for brevity.\qed

\section*{Acknowledgements}
MD and AG both acknowledge financial support from the European Research Council (ERC) under the European Union's Horizon 2020 research and innovation programme (Grant Agreements n$^\circ$101075879 and n$^\circ$864066, respectively).\footnote{Views and opinions expressed are however those of the authors only and do not necessarily reflect those of the European Union or the European Research Council Executive Agency. Neither the European Union nor the granting authority can be held responsible for them.}

\bibliographystyle{abbrv}
\bibliography{biblio}

\end{document}